\documentclass[10pt,a4paper,reqno]{article}
\usepackage[normalem]{ulem}
\usepackage{graphicx} %,bbold,bbm,mathbbol,
\usepackage[english]{babel} %for biblatex
\usepackage{csquotes} %for biblatex
\usepackage{filecontents} %for biblatex
\usepackage[dvipsnames]{xcolor}
\usepackage{color}
\usepackage[colorinlistoftodos]{todonotes}

\newcounter{dummy}
\usepackage{enumitem}
\makeatletter
\newcommand\myitem[1][]{\item[#1]\refstepcounter{dummy}\def\@currentlabel{#1}}
\makeatother
\usepackage[
  hypertexnames = false,
  colorlinks    = true,
  citecolor     = blue,
  linkcolor     = blue,
  urlcolor      = blue,
  linktocpage = true,
  pagebackref
  ]{hyperref}
  
\usepackage[alphabetic,backrefs]{amsrefs}

\usepackage[margin=1.8cm,includehead]{geometry}
\usepackage{amsmath,stmaryrd, upgreek}
\usepackage{amsthm,amssymb}
\usepackage{latexsym}
\usepackage{amscd}
\usepackage{mathrsfs}
\usepackage{url}
\usepackage{mathtools}

\usepackage{cleveref}
\usepackage{doi}
\usepackage{tensor}
\usepackage{cancel}
\usepackage[all]{xy}
\usepackage{multicol}

\usepackage[T1]{fontenc}
\usepackage{tikz-cd}
\usepackage{titlesec}

\usepackage[titletoc,toc,title]{appendix}

\makeatletter %gets rid of Contents in toc
\renewcommand\tableofcontents{%
    \@starttoc{toc}%
}
\makeatother

\usepackage{enumerate}
\usepackage{slashed}
\graphicspath{}
\usepackage{fancyhdr} %gives heading and author name on alternate pages

\usepackage{changepage}

\newcommand\shorttitle{Flows of $H$-structures} %title which appear on alternate pages
\newcommand\authors{Fadel--Loubeau--Moreno--S{\'{a}} Earp} %author name appears on alternate pages

\newcounter{commentCounter}
\titleformat{\section}    
       {\normalfont\large\bfseries\center}{\thesection.}{1em}{}
\makeatletter
\newcommand*{\rom}[1]{\expandafter\@slowromancap\romannumeral #1@}
\makeatother

\newcommand{\rd}{{\rm d}}

\newcommand{\rg}{{\rm g}}
\newcommand{\rh}{{\rm h}}

\newcommand{\rG}{{\rm G}}

\newcommand{\rI}{{\rm I}}

\def\cD{\mathcal{D}}
\def\cE{\mathcal{E}}
\def\cF{\mathcal{F}}

\def\cH{\mathcal{H}}
\def\cI{\mathcal{I}}

\def\cL{\mathcal{L}}

\def\cT{\mathcal{T}}

\def\cV{\mathcal{V}}

\newcommand{\sX}{\mathscr{X}}

\newcommand{\fg}{{\mathfrak g}}

\newcommand{\fh}{{\mathfrak h}}

\newcommand{\fm}{{\mathfrak m}}

\newcommand{\fu}{{\mathfrak u}}

\def\bR{\mathbb R}

\def\fgl{\mathfrak{gl}}
\def\fso{\mathfrak{so}}
\def\fsu{\mathfrak{su}}

\newcommand{\spin}{\mathfrak{spin}}

\renewcommand{\O}{{\rm O}}
\newcommand{\SO}{{\rm SO}}
\newcommand{\Sp}{{\rm Sp}}
\newcommand{\SU}{{\rm SU}}
\newcommand{\GL}{\mathrm{GL}}

\newcommand{\SL}{\mathrm{SL}}
\newcommand{\U}{{\rm U}}

\def\G2{\mathrm{G}_2}
\def\S7{\mathrm{Spin}(7)}
\def\Spin7{\mathrm{Spin(7)}}

\DeclareMathOperator\Div{div}
\DeclareMathOperator\Dom{Dom}
\DeclareMathOperator\curl{curl}

\DeclareMathOperator{\Stab}{Stab}

\newcommand{\Ad}{\mathrm{Ad}}

\renewcommand{\det}{\mathop\mathrm{det}\nolimits}

\newcommand{\End}{{\mathrm{End}}}

\renewcommand{\epsilon}{\varepsilon}

\newcommand{\Lie}{\mathrm{Lie}}
\newcommand{\Ric}{{\rm Ric}}

\newcommand{\ad}{\mathrm{ad}}

\DeclareMathOperator\tr{tr}

\DeclareMathOperator\vol{vol}

\newcommand{\sym}{\mathrm{sym}}

\newcommand{\Fr}{{\rm Fr}}

\def\pt{\partial}

\def\g2{\varphi}
\def\s7{\Phi}
\def\ddt{\frac{d}{dt}}

\newcommand{\qandq}{\quad\text{and}\quad}
\newcommand{\qwithq}{\quad\text{with}\quad}
\newcommand{\qforq}{\quad\text{for}\quad}

\def\<{\mathopen{}\left<}
\def\>{\right>\mathclose{}}
\def\({\mathopen{}\left(}
\def\){\right)\mathclose{}}

\usepackage{multicol, color}

\definecolor{gold}{rgb}{0.85,.66,0}
\definecolor{cherry}{rgb}{0.9,.1,.2}
\definecolor{burgundy}{rgb}{0.8,.2,.2}
\definecolor{orangered}{rgb}{0.85,.3,0}
\definecolor{orange}{rgb}{0.85,.4,0}
\definecolor{olive}{rgb}{.45,.4,0}
\definecolor{lime}{rgb}{.6,.9,0}
\definecolor{green}{rgb}{.2,.7,0}
\definecolor{grey}{rgb}{.4,.4,.2}
\definecolor{brown}{rgb}{.4,.3,.1}

\newtheorem{theorem}{Theorem}[section]
\newtheorem*{theorem*}{Theorem}
\newtheorem{corollary}[theorem]{Corollary}
\newtheorem{lemma}[theorem]{Lemma}
\newtheorem{proposition}[theorem]{Proposition}
\newtheorem*{proposition*}{Proposition}

\newtheorem{example}[theorem]{Example}

\theoremstyle{definition}
\newtheorem{definition}[theorem]{Definition}
\newtheorem{remark}[theorem]{Remark}

\numberwithin{equation}{section}
\newtheorem{thmx}{Theorem}

\newtheorem{propx}[thmx]{Proposition}
\AtBeginDocument{%
   \def\MR#1{}
}

\newcommand{\bigzero}{\mbox{\normalfont\Large\bfseries 0}}

\usepackage{nicematrix}

\begin{document}

\title{\textbf{Flows of geometric structures}}
\author{Daniel Fadel, Eric Loubeau, Andrés J. Moreno \& Henrique N. S{\'{a}} Earp}
\date{\today}

\maketitle

\begin{abstract}
    We develop an abstract theory of flows of geometric $H$-structures, i.e., flows of tensor fields defining $H$-reductions of the frame bundle, for a closed and connected subgroup $H \subset \mathrm{SO}(n)$, on any connected and oriented $n$-manifold with sufficient topology to admit such structures. 
    
    The first part of the article sets up a unifying theoretical framework for deformations of $H$-structures, by way of the natural infinitesimal action of $\mathrm{GL}(n,\mathbb{R})$ on tensors combined with various bundle decompositions induced by $H$-structures. We compute evolution equations for the intrinsic torsion under general flows of $H$-structures and, as applications, we obtain general Bianchi-type identities for $H$-structures, and, for closed manifolds, a general first variation formula for the $L^2$-Dirichlet energy functional $\mathcal{E}$ on the space of $H$-structures.

    We then specialise the theory to the negative gradient flow of $\mathcal{E}$ over isometric $H$-structures, i.e., their harmonic flow. The core result is an almost monotonocity formula along the flow for a scale-invariant localised energy, similar to the classical formulae by Chen--Struwe \cites{Struwe1988,Struwe1989} for the harmonic map heat flow. This yields an $\epsilon$-regularity theorem and an energy gap result for harmonic structures, as well as long-time existence for the flow under small initial energy,  {relative to the} $L^\infty$-norm of initial torsion, in the spirit of  Chen--Ding \cite{Chen-Ding1990}. Moreover, below a certain energy level, the absence of a torsion-free isometric $H$-structure in the initial homotopy class imposes the formation of finite-time singularities. These seemingly contrasting statements are illustrated by examples on flat $n$-tori, so long as  {the set $[\mathbb{S}^n,\SO(n)/H]$ of homotopy classes of maps $\mathbb{S}^n\to \SO(n)/H$ contains more than one element and the universal cover of $\SO(n)/H$ is a sphere}; e.g. when $n=7$ and $H=\rm G_2$, or $n=8$ and $H=\rm Spin(7)$.
    
\end{abstract}

\begin{adjustwidth}{0.95cm}{0.95cm}
    \tableofcontents
\end{adjustwidth}

%\listoftodos
\newpage
\section*{Introduction}
\addcontentsline{toc}{section}{Introduction}
\label{sec: intro}

Given an oriented Riemannian manifold $(M^n,g)$, a universal description of compatible $H$-structures, for a closed connected subgroup $H \subset \mathrm{SO}(n)$, can be formulated in terms of sections of the homogeneous fibre bundle obtained by $H$-reduction of the oriented frame bundle \cite{loubeau-saearp}. Such geometric structures are ubiquitous in Riemannian Geometry, and their general existence is a rather topological condition, much weaker than e.g. parallel tensors corresponding to special holonomies such as K{\"a}hler, $\rm G_2$- or $\mathrm{Spin}(7)$-manifolds. 
An important toolbox to establish the existence, or indeed the absence, of special geometric structures on manifolds is provided by the analytical theory of geometric flows. 

Inspired by the foundational work of Karigiannis~\cites{Karigiannis2007,karigiannis-spin7} on flows of $\mathrm{G}_2$- and $\mathrm{Spin}(7)$-structures, this paper first aims to advance our abstract understanding of flows of $H$-structures. By formulating a general flow in terms of the infinitesimal deformations of tensors defining a geometric structure, we obtain equations for the coevolution of the Riemannian metric, the intrinsic torsion, and several of their derived quantities, as well as Ricci- and Bianchi-type identities for arbitrary $H$-flows. While several of these properties are already known in particular for $\mathrm{G}_2$- and $\mathrm{Spin}(7)$-structures, our framework allows for simpler proofs, which are independent of context-specific identities and therefore hold for general $H$-structures.

We then specialise the theory to the natural variational problem on homogeneous sections given by the Dirichlet functional, as an extension of harmonic map theory. Once the associated Euler--Lagrange equation has been derived, hence defining harmonic $H$-structures, one may study the corresponding parabolic negative gradient flow, in order to detect optimal $H$-structures within the same isometric class.
While some features will certainly vary according to the group $H$, we advance considerably the general theory of harmonic $H$-flows, obtaining broad-ranging results which subsume and expand previous developments by
Grigorian~\cites{Grigorian2017,Grigorian2019}, Bagaglini \cite{Bagaglini2019} and Dwiwedi-Karigiannis-Gianniotis~\cite{dgk-isometric} on harmonic $\mathrm{G}_2$-structures, He-Li~\cite{he2021} on harmonic almost complex structures ($H=\mathrm{U}(\frac{n}{2})$), Dwivedi-L-SE~\cite{Dwivedi-Loubeau-SaEarp2021} for harmonic $\mathrm{Spin}(7)$-structures, and Fowdar-SE~\cite{udhav2022quaternionic} for $\mathrm{Sp}(\frac{n}{4})\mathrm{Sp}(1)$-structures. See also \cite{dwivedi2023flows} and \cite{dwivedi2024gradient} for recent developments in the theory of general flows in the $\rm G_2$ and $\rm Spin(7)$ cases, respectively.

\bigskip

Section \ref{sec: H-structures} lays out the theoretical framework for a unified approach to $H$-structures on manifolds and their evolutions. Restricting to subgroups $H\subset\mathrm{SO}(n)$ that are realised as the stabiliser of a (multi-)tensor $\xi_{\circ}$ on $\mathbb{R}^n$, we identify $H$-structures on a manifold $M^n$ with tensor fields $\xi$ pointwise modelled on $\xi_{\circ}$, also known as \emph{geometric structures}. Then, a general deformation of $\xi$ as an $H$-structure can be expressed in terms of the \emph{diamond operator} $\diamond$, defined by the infinitesimal action of $\mathrm{GL}(n,\mathbb{R})$ on tensors:
\begin{equation*}
    A\diamond \xi :=\left.\ddt\right\vert_{t=0} e^{tA}.\xi \nonumber,
    \qwithq 
    A=S+C\in \Gamma(\End(TM))= (\Sigma^2\oplus\Omega^2_\fm)(M),
\end{equation*}
where $S$ and $C$ denote the symmetric and skew-symmetric parts, respectively, and $\fm:=\fh^{\perp}\subset\mathfrak{so}(n)$ is the orthogonal complement of the infinitesimal stabiliser $\fh=\Lie(H)$, relative to the Riemannian metric induced by $\xi$. After deriving the main properties of $\diamond$ in Section \ref{sec: inf_deformations}, in particular the distinct roles of the symmetric and skew-symmetric endomorphisms and the special case of deformations of the Riemannian metric itself (Lemma~\ref{lem: properties_diamond}), we give in Section \ref{sec: inner-prod and torsion} an alternative description of the intrinsic torsion $T$ of an $H$-structure $\xi$ in terms of $\diamond$ and $\nabla\xi$, where $\nabla$ denotes the Levi--Civita connection of that $\xi$-metric  (Lemma~\ref{lem: torsion_equiv_nablaxi}), as well as a useful decomposition of the rough Laplacian $\Delta\xi$ in terms of  $T$ (Lemma \ref{lem: basic_estimate_harmonic}). 

In Section \ref{sec: general_flows}, we define a general flow of $H$-structures by deforming under the $\diamond$-action of a curve in $\mathrm{GL}(n,\mathbb{R})$: 
\begin{equation}
\label{eq: general_H_flow}
\tag{GF}
    \frac{\partial}{\partial t}\xi=A\diamond\xi 
    \qforq   A(t)=S(t)+C(t), 
    \qwithq S(t)\in \Sigma^2 
    \qandq C(t)\in \Omega^2_{\fm}\subset\Omega^2.
\end{equation}
We obtain the corresponding evolution equations for important quantities dependent on the $H$-structure, such as the induced metric (Lemma \ref{lemma: isometric variation}), its volume form and Christoffel symbols (Lemma \ref{lem:volume_Christoffel_variation}). As an application, we characterise and prove short-time existence, but not uniqueness, of the simplest flow of $H$-structures coevolving with the Ricci flow of Riemannian metrics (Lemma \ref{lem: Ricci flow}). Most importantly, we also derive the evolution equations of both $\nabla\xi$ and $T$ under a general flow of $H$-structures: 

\begin{propx}[Proposition~\ref{prop: evo_nabla_xi} and Corollary~\ref{cor: evo_torsion}] %, page~\pageref{cor: evo_torsion})
\label{prop: intro - evol of torsion}
    Let $Q\subset\Fr(M)$ be an $H$-structure on $M$ completely determined by a (multi-)tensor $\xi$. Denote by $g$ its induced Riemannian metric, and let $T\in\Omega^1(M,\fm_{Q})$ be the intrinsic torsion tensor of $\xi$, where $\fm_{Q}\subset\mathfrak{so}(TM)$ is the $H$-submodule of skew-symmetric tensors (with respect to $g$) determined by $\fm:=\fh^{\perp}\subset\mathfrak{so}(n)$.
    
    Under the general flow of $H$-structures \eqref{eq: general_H_flow}, for each coordinate vector field $\partial_l$, the evolution of $\nabla _l\xi:=\nabla_{\partial_l}\xi$ is given by
\begin{align*}%\label{eq: variation_nabla_xi}
\frac{\pt}{\pt t}\nabla_l\xi=A\diamond \nabla_l\xi+(\nabla_lC-\Lambda \nabla S_l)\diamond\xi, 
%\frac{\pt}{\pt t}\nabla_l\xi=A\diamond (T_l\diamond \xi)+\pi_\fm(\Lambda \nabla S_l+\nabla_lC)\diamond\xi, 
\end{align*}
where $(\Lambda\nabla S_l)\indices{_j^i}:=(\Lambda\nabla S_l)_{jk}g^{ik}=g^{ik}(\nabla_j S_{kl}-\nabla_k S\indices{_j_l})$. Moreover, the torsion $T_l:=T_{\partial_l}$ satisfies
\begin{equation*}%\label{eq: ddt T_diamond_xi}
%\Big(\frac{\partial}{\partial t}T_l +[T_l,A]-\pi_\fm(\Lambda\nabla B_l+\nabla_lC)\Big)\diamond \xi_t=0,
    \Big(\frac{\partial}{\partial t}T_l +[A,T_l]+\Lambda\nabla S_l-\nabla_l C\Big)\diamond \xi=0,
\end{equation*} so that, denoting by $\pi_{\fm}:\mathfrak{so}(TM)\to\fm_Q$ the orthogonal projection, we have
\begin{align*}%\label{eq: m_part ddt_T}
\begin{split}
    %\pi_\fh\bigg(\frac{\partial}{\partial t}T_l\bigg)&= \pi_\fh([A,T_l])+\pi_\fh(\Lambda\nabla B_l+\nabla_lC),\\
    \pi_\fm\bigg(\frac{\partial}{\partial t}T_l\bigg)&= \pi_\fm([T_l,C])+\pi_\fm(\nabla_lC-\Lambda\nabla S_l).
\end{split}    
\end{align*}
\end{propx}

Next, as an application of the above evolution equations, the diffeomorphism-invariance of $\nabla\xi$ as a function of $\xi$ leads to a Ricci formula (Proposition~\ref{prop: Bianchi_identity})
and produces, in Corollary~\ref{cor: m-part_Bianchi_identity}, a Bianchi-type identity relating covariant derivatives of the intrinsic torsion and the curvature tensor:

\begin{propx}[Proposition~\ref{cor: m-part_Bianchi_identity}] %, page~\pageref{cor: m-part_Bianchi_identity}) 
    Let $R_{la}\in\Gamma(\mathfrak{so}(TM))$ denote the components of the Riemann curvature endomorphism of $(M,g)$. Then
\begin{equation*}
    (\nabla_a T_l-\nabla_lT_a-[T_l,T_a]-R_{la})\diamond \xi=0.
\end{equation*}
\end{propx}

This formula yields simpler alternative proofs to a number of well-known curvature properties for various choices of $H \subset \mathrm{SO}(n)$ (Example~\ref{ex sun}---\ref{ex: g2}) and consequences, for the Ricci tensor, of holonomy in the group $H$ (Corollary~\ref{cor holonomy-ricci}).

Starting from Section \ref{sec: Dirichlet}, we specialise the theory of $H$-flows to the Dirichlet energy functional $\mathcal{E}$, defined on the space of $H$-structures $\xi$ over a closed manifold by the squared $L^2$-norm of intrinsic torsion $T$ of $\xi$. We then use our previous general evolution equations to compute the Euler--Lagrange equations of $\mathcal{E}$ under a general variation (Proposition~\ref{general EL}).
Next, we restrict ourselves to subgroups $H=\Stab_{\SO(n)}(\xi_{\circ})\subset\SO(n)$ for which there is a constant $c>0$ such that
\begin{equation}\label{eq: intro_assumption}
\tag{$\ast$}
\langle A\diamond\xi_{\circ}, B\diamond\xi_{\circ}\rangle = c\langle A,B\rangle,\quad\forall A,B\in\mathfrak{m}=\fh^{\perp}\subset\fso(n),
\end{equation} where $\langle A\diamond\xi_{\circ}, B\diamond\xi_{\circ}\rangle$ denotes the inner product on tensors induced by the flat metric in $\mathbb{R}^n$, and $\langle A,B\rangle=-\tr(AB)$ is the canonical bi-invariant metric making $\SO(n)/H$ a normal homogeneous Riemannian manifold. This assumption \eqref{eq: intro_assumption} is satisfied, for instance, when $\SO(n)/H$ is (strongly) isotropy irreducible, i.e. when $\fm$ is an irreducible $H$-module; e.g. for $H=\U(\frac{n}{2})$, $\rm G_2$ ($n=7$), $\rm Spin(7)$ ($n=8$) or $\rm{Sp}(\frac{n}{4})\rm{Sp}(1)$, and also in the reducible case where $H=\{1\}$ is the trivial subgroup (see Section \ref{sec: inner-prod and torsion} for more details). Under the above assumption, we have $|\nabla\xi|^2=c|T|^2$ and the first variation of the energy $\mathcal{E}$ under isometric deformations gives the so-called harmonic \cites{loubeau-saearp,Dwivedi-Loubeau-SaEarp2021} or $\Div T$-flow \cites{dgk-isometric,Grigorian2019,grigorian2020isometric} as the negative gradient flow of $\mathcal{D}:=c\mathcal{E}$,
\[
\frac{\partial}{\partial t}\xi = \Div T\diamond \xi,
\] in particular defining \emph{harmonic} $H$-structures by their divergence-free intrinsic torsion.

Since the homogeneous bundle describing $H$-structures typically has positive sectional curvature, chances are that any harmonic flow will develop singularities. The behaviour of flows near singularities can be understood by means of solitons, so in Section \ref{sec: solitons} we propose a theory of $H$-solitons for arbitrary flows. 
While e.g. for $H = \rm G_2\subset\mathrm{SO}(7)$, the so-called Laplacian flow is known to admit no noncompact shrinking solitons, and the only compact steady solitons must be given by torsion-free $\rm G_2$-structures, these questions remain mostly open in greater generality; as a first step in this direction, we show that arbitrary $H$-solitons induce self-similar solutions (Proposition \ref{prop: soliton_induce_self-similar}), and also that, under mild assumptions, isometric flows of $H$-structures with an underlying complete metric do not admit compact shrinking or expanding solitons (Lemma \ref{lem: no_compact_solitons}).

Section \ref{sec: harmonic_flow} is dedicated to analytical properties of the harmonic flow defined by the negative gradient of the Dirichlet energy $\mathcal{E}$ on isometric $H$-structures, i.e. compatible with a fixed Riemannian metric. While some basic facts on this flow can be deduced from the general theory of harmonic maps, as recalled in Propositions~\ref{prop: short-time_existence} and \ref{prop: Shi-estimates}, the cornerstone of our study is an almost-monotonicity formula for a scale-invariant local version of the functional, extending results of \cites{Struwe1988,Struwe1989} for harmonic maps, \cite{he2021} for almost complex structures and \cite{udhav2022quaternionic} for quaternionic-K{\"a}hler structures (i.e. $\mathrm{Sp}(\frac{n}{4})\mathrm{Sp}(1)$ reductions):

\begin{thmx}[Theorem~\ref{thm: monotonicity_Theta}] \label{thmC}
%page \pageref{thm: monotonicity_Theta})
Let $\{\xi(t)\}_{[0,\tau_0]}$ be a solution to the harmonic $H$-flow \eqref{eq: harmonic_flow} with initial condition $\xi(0)=\xi_0$ and define the (scale-invariant) function
\begin{align*}
    \Theta_{(y,\tau_0)}(t) & := (\tau_0-t)\int_{M}|T|^2(\cdot{},t)G_{(0,\tau_0)}(\cdot{},t)\phi^2\vol_g,
\end{align*} 
where $T$ is the intrinsic torsion, $\phi$ is a cut-off function supported on a small geodesic ball $B_{r_M}(y)$, and $G_{(0,\tau_0)}$ is the Euclidean backward heat kernel with singularity at $(0,\tau_0)$ in normal coordinates on $B_{r_M}(y)$ (see \eqref{eq: heat_kernel}). Then, for any $\tau_0-\min\{\tau_0,1\}<t_1\leqslant t_2<\tau_0$ and $N>1$, we have
\begin{equation*}
    \Theta(t_2)\leqslant c\Theta(t_1)+c\left(N^{n/2}(E_0+\sqrt{E_0})+\frac{1}{\ln^2N}\right)(t_2-t_1),
\end{equation*}
where $c=c(M,g)>0$  and $E_0:=\frac{1}{2}\int_M |T_{\xi_0}|^2\vol_g$.
\end{thmx}

The Bochner estimate of Lemma~\ref{lem: dif_ineq} and this monotonicity property applied to the function 
\[
\Psi_{(y,\tau_0)}(r) := \int^{\tau_0-r^2}_{\tau_0-4r^2} \frac{\Theta_{(y,\tau_0)}(\xi(t))}{\tau_0-t} \, dt,
\]
as in Theorem~\ref{thm: monotonicity_Psi}, are the main tools to establish $\epsilon$-regularity along the harmonic $H$-flow:

\begin{thmx}[Theorem~\ref{thm: e-regularity}]
%, page~\pageref{thm: e-regularity})
There exists a constant $\varepsilon_0 >0$, depending only on $(M^n,g)$, the group $H$, and the energy of the initial data such that, if $\Psi_{(y,\tau_0)}(R)<\varepsilon_0 $, then
\begin{equation*}
    \sup_{P_{\delta R}(y,\tau_0)} |\nabla\xi|^2\leqslant 4(\delta R)^{-2},
\end{equation*} where $P_{\delta R}$ is a parabolic neighbourhood, and the constant $\delta>0$ depends only on the geometry and initial data. 
\end{thmx}

As a fairly direct consequence, we obtain an energy gap theorem.

\begin{propx}[Proposition~\ref{prop: energy_gap}]%, page~\pageref{prop: energy_gap})
	There exists a constant $\varepsilon_0>0$, depending only on the geometry of $(M,g)$ and the group $H$, such that, if $\xi$ is a compatible harmonic $H$-structure on $(M,g)$ satisfying $\cD(\xi):=\frac{1}{2}\|\nabla\xi\|_{L^2(M)}^2<\varepsilon_0$, then $\xi$ is actually torsion-free, i.e. $\nabla\xi = 0$.
\end{propx}

The most important consequence of these results is that, under the hypothesis of small initial energy, relative to the $L^{\infty}$-norm of the initial torsion, we get long-time existence for the harmonic $H$-flow and convergence to a torsion-free limit, which extends the classical result by Chen-Ding \cite{Chen-Ding1990} to homogeneous sections:

\begin{thmx}[Theorem~\ref{thm: long-time_existence}]%, page~\pageref{thm: long-time_existence})
    For any given constant $\kappa>0$, there is a universal constant $\varepsilon(\kappa)>0$, depending only on $\kappa$, $(M,g)$ and $H$, such that, if
\begin{multicols}{2}
\begin{itemize}
    \item[(i)] $\|\nabla\xi_0\|_{L^{\infty}(M)}\leqslant \kappa$, \qandq
    \item[(ii)] $\cD(\xi_0)=\frac{1}{2}\|\nabla\xi_0\|_{L^2(M)}^2<\varepsilon(\kappa)$,
\end{itemize} 
\end{multicols}
\noindent
    then the harmonic $H$-flow with initial condition $\xi_0$ exists for all time and converges to a torsion-free $H$-structure.
\end{thmx}

As a topological counterpoint to the previous statement, we prove existence of a finite-time singularity for the harmonic $H$-flow with small initial energy \emph{when no torsion-free $H$-structure exists} in the homotopy class of initial data, but the infimum of the energy on such class is nonetheless zero.

\begin{thmx}[Theorem~\ref{thm: blow_up}] \label{thmG} %, page~\pageref{thm: blow_up}) 
    Let $\overline{\xi}$ be an $H$-structure, the isometric homotopy class $[\overline{\xi}]$ of which does not contain any torsion-free $H$-structure, but still such that
$\inf_{\xi\in [\overline{\xi}]} \cD(\xi) = 0$. 

Then there exists a constant $\varepsilon_{\ast}>0$, depending only on $(M^n,g)$ and $H$, such that if $\xi_0\in [\overline{\xi}]$ has $\cD(\xi_0)<\varepsilon_{\ast}$, then the harmonic $H$-flow starting at $\xi_0$ develops a finite-time singularity. Moreover, if $\tau(\xi_0)$ denotes the maximal time of existence of the flow, then $\tau(\xi_0)\to 0$ as $\cD(\xi_0)\to 0$.
\end{thmx}

This blow-up phenomenon is clarified with a general construction on tori (Example~\ref{ex: finite_time_blow-up}), based on the work of He--Li \cite{he2021} for the particular case of almost complex structures, where the topology of the closed and connected subgroup $H\subset\mathrm{SO}(n)$ plays an important role, since a finite-time singularity will appear on flat $n$-tori as soon as the set $[\mathbb{S}^n,\SO(n)/H]$ of homotopy classes of maps $\mathbb{S}^n\to \SO(n)/H$ contains more than one element and the universal cover of $\SO(n)/H$ is a sphere; e.g. when $n=7$ and $H=\rm G_2$, or $n=8$ and $H=\rm Spin(7)$. 

We should remark that similar statements, such as Theorems~\ref{thmC}--\ref{thmG}, were shown in \cite{he2021} for the case $H=\mathrm{U}(\frac{n}{2})$, and we prove here not only a generalisation to $H$-harmonic flows but also provide a certain number of details and particular points absent from \cite{he2021}.

Our last highlighted result is long-time existence for small initial intrinsic torsion, with convergence to a harmonic $H$-structure, and a dynamical stability property of torsion-free $H$-structures, allowing for hope of a universal long-time existence and convergence theorem.

\begin{thmx}[Theorem~\ref{thm: stability}]%, page~\pageref{thm: stability})
\begin{itemize}
    \item[(i)] There is a constant $\kappa_0=\kappa_0(M,g,H)>0$ such that if  $\|\nabla\xi_0\|_{L^{\infty}(M)}<\kappa_0$ then the harmonic $H$-flow starting at $\xi_0$ exists for all $t\geqslant 0$ and converges smoothly to a torsion-free $H$-structure $\xi_{\infty}$ as $t\to\infty$.
    
    \item[(ii)] Suppose $(M^n,g)$ admits a torsion-free compatible $H$-structure $\overline{\xi}$. Then for all $\delta>0$ there is some $\overline{\varepsilon}(\delta,M,g,H)>0$ such that for any compatible $H$-structure $\xi_0$ on $(M^n,g)$ with $\|\xi_0-\overline{\xi}\|_{C^2(M)}<\overline{\varepsilon}$ the harmonic $H$-flow with initial condition $\xi_0$ exists for all $t\geqslant 0$, satisfies the estimate $\|\xi_t-\overline{\xi}\|_{C^1(M)}<\delta$ for all $t\geqslant 0$, and converges smoothly to a torsion-free $H$-structure $\xi_{\infty}$ as $t\to\infty$.
\end{itemize} 
\end{thmx}

Many results and theorems in this article are illustrated with examples for $H= \mathrm{U}(\frac{n}{2}), \rm G_2 , \mathrm{Spin}(7)$.

\bigskip

\noindent
\textbf{Notation and Conventions.} We denote by $c>0$ a generic constant, which depends at most on the background geometric data $(M^n,g,H)$, consisting of an oriented Riemannian $n$-manifold $(M^n,g)$ equipped with an $H$-structure, for some closed Lie subgroup $H\subset\mathrm{SO}(n)$. Its value might change from one occurrence to the next, and further dependencies are indicated by subscripts. We use the symbol $\circledast$ to denote generic multilinear expressions bounded by $c$, the precise form of which is unimportant. We frequently use Young's inequality, $ab \leqslant \tfrac{1}{2\varepsilon} a^2 + \tfrac{\varepsilon}{2} b^2$, for any $a, b\in\mathbb{R}$ and $\varepsilon>0$. The symbol $\Delta$ denotes the \emph{negative definite} rough Laplacian, i.e. $\Delta=-\nabla^{\ast}\nabla$. We use the Einstein summation convention throughout the text. In a local coordinate frame, the Riemann curvature $(1,3)$-tensor is given by
\[
R_{ijk}^m \frac{\partial}{\partial x^{m}} 
    = (\nabla_{i} \nabla_{j} - \nabla_{j} \nabla_{i})\frac{\partial}{\partial x^{k}},
\] and we lower the contravariant index by
$
R_{ijkl} := R_{ijk}^m g_{ml}
$.  
We write
\[
R_{ij} := R_{ijk}^m \partial_m\otimes dx^k
\] 
for the curvature endomorphism tensor, and the Ricci curvature is given by
\[
    \mathrm{Ric}_{jk} = R_{ijkl}g^{il}.
\]
The Riemannian first and second Bianchi identities are 
\begin{equation}
\label{eq: riem1stBid}\tag{$\dagger$}
\begin{aligned}
    R_{ijkl} + R_{iklj} + R_{iljk} 
    &= 0, 
    \\
    \nabla_i R_{jkab} + \nabla_j R_{kiab} + \nabla_k R_{ijab}
    &=0,
\end{aligned}
\end{equation}
where the latter contracts in $i,a$ to
\begin{equation} \tag{$\dagger\dagger$}
\label{eq: riem2ndBid}
    g^{ia}\nabla_i R_{abjk} 
    = \nabla_k \mathrm{Ric}_{jb} - \nabla_j \mathrm{Ric}_{kb}.
\end{equation}

\noindent\textbf{Acknowledgements:} 

The authors are grateful to Udhav Fowdar and also to the members of the Math AmSud (21-Math-06) collaboration Geometric Structures and Moduli Spaces, for several stimulating discussions during the Workshop at Universidad Nacional de Córdoba in August-September 2022. The first-named author also would like to thank Andrey Soldatenkov, Misha Verbitsky, Alexis Garcia, Renan Assimos, Gonçalo Oliveira and Jason Lotay for illuminating discussions that helped improve this paper. Finally, the authors are thankful to the anonymous referee for the careful reading, numerous constructive comments and detailed suggestions that considerably benefited the article.

EL and HSE benefited from a CAPES-COFECUB bilateral collaboration (2018-2022), granted by the Brazilian Coordination for the Improvement of Higher Education Personnel (CAPES) – Finance Code 001 [88881.143017/2017-01], and COFECUB [MA 898/18], and from a CAPES-MathAmSud (2021-2023) grant [88881.520221/2020-01]. 
HSE has also been funded by the São Paulo Research Foundation (Fapesp)  \mbox{[2018/21391-1]} and the Brazilian National Council for Scientific and Technological Development (CNPq)  \mbox{[307217/2017-5]}.
DF was funded by the postdoctoral scholarship [88887.643728/2021-00] of the CAPES-COFECUB collaboration.
AM was funded by the São Paulo Research Foundation (Fapesp) [2021/08026-5].
All four authors are members of the  Fapesp-ANR BRIDGES [2021/04065-6] collaboration.

\newpage
%%%%%%%%
\section{General flows of $H$-structures}
\label{sec: H-structures}

\subsection{Homogeneous sections and stabilised tensors}
\label{sec: preliminaries}

Throughout this paper, $M^n$ will denote a connected and orientable smooth $n$-manifold without boundary. Let $\Fr(M)$ denote the \emph{frame bundle} of $M$, i.e., the principal $\mathrm{GL}(n,\mathbb{R})$-bundle whose fibre over $x\in M$ consists of the linear isomorphisms $u:T_xM\to\mathbb{R}^n$, with right action $\mathrm{GL}(n,\mathbb{R})\times \Fr(M)\to \Fr(M)$ given by $(\rg,u)\mapsto \rg.u:= \rg^{-1}\circ u$. Given a Lie subgroup $H\subset\mathrm{GL}(n,\mathbb{R})$, recall that a \emph{$H$-structure} on $M^n$ is an $H$-reduction of $\Fr(M)$, i.e., a principal $H$-subbundle $Q\subset\Fr(M)$. For example, an $\mathrm{SO}(n)$-structure on $M^n$ is equivalent to a choice of a Riemannian metric $g$ and an orientation. Most of the time, we shall fix such a structure on $M^n$, i.e. we shall work with an oriented Riemannian manifold $(M^n,g)$; the associated $\mathrm{SO}(n)$-structure, or principal $\mathrm{SO}(n)$-bundle of oriented orthonormal coframes, will be denoted by $\pi_{\mathrm{SO}(n)}:\Fr(M,g)\to M$. Then, for a Lie subgroup $H\subset\mathrm{SO}(n)$, we say that $Q$ is a \emph{compatible $H$-structure} on $(M^n,g)$ if it is an $H$-reduction of $\Fr(M,g)$, i.e. if $Q\subset\Fr(M,g)$ as principal bundles. 

We shall restrict ourselves to closed and connected subgroups $H\subset\mathrm{SO}(n)$. Note that any such $H$ right-acts freely on $\Fr(M,g)$ and the quotient map $\pi_H:\Fr(M,g)\to \Fr(M,g)/H$ is a principal $H$-bundle. The map $\pi:\Fr(M,g)/H\to M$ such that $\pi_{\mathrm{SO}(n)}=\pi\circ\pi_H$ then defines a fibre bundle with fibre $\mathrm{SO}(n)/H$; indeed, $\pi:\Fr(M,g)/H\to M$ is isomorphic to the associated bundle $\Fr(M,g)\times_{\mathrm{SO}(n)}\mathrm{SO}(n)/H$.

Now observe that compatible $H$-structures $Q\subset \Fr(M,g)$ are in one-to-one correspondence with sections $\sigma\in\Gamma(\Fr(M,g)/H)$: given $Q$, we define $\sigma_Q(x):=\pi_H(u)$ for any frame $u\in Q$ with $\pi_{\mathrm{SO}(n)}(u)=x$; this is well-defined because any two  $u,\tilde{u}\in\pi_{\mathrm{SO}(n)}^{-1}(x)\subset Q$ differ by $\tilde{u}=\rh.u$ for some $\rh\in H$, and therefore $\pi_H(u)=\pi_H(\tilde{u})$. Conversely, to any section $\sigma\in\Gamma(\Fr(M,g)/H)$ we associate the compatible $H$-structure $Q_{\sigma}:=\pi_H^{-1}(\sigma(M))\subset \Fr(M,g)$, and it is easy to see that these assignments are mutually inverse. More generally, any $H$-structure $Q\subset \Fr(M)$ (not necessarily metric-compatible), can be thought of as a section $\sigma_Q\in\Gamma(\Fr(M)/H)$.
 
The existence of an $H$-structure on $M^n$ is a purely topological question. In particular, if $H=\{1\}$ is the trivial group, then a $\{1\}$-structure on $M^n$ is just a global trivialisation of $\Fr(M)$, which exists if and only if the tangent bundle $TM$ is trivialisable, i.e. $M$ is parallelisable. In dimension $n=2$, note that the only proper closed and connected subgroup $H \subset\mathrm{SO}(2)$ is the trivial group, and since the only parallelisable oriented closed surface is the $2$-torus, the discussion about compatible $\{1\}$-structures on such a Riemannian surface $(M^2,g)$ reduces to parallelisms on the $2$-torus $(\mathbb{T}^2,g)$ endowed with an arbitrary Riemannian metric $g$ (see Remark \ref{rmk: 2-torus} for further details). We shall henceforth restrict attention to dimensions $n>2$.

Next we note that the assumption of $H\subset\mathrm{SO}(n)$ being closed and connected implies that the quotient $\mathrm{SO}(n)/H$ is a normal homogeneous Riemannian manifold with the metric induced by the canonical bi-invariant metric on $\mathrm{SO}(n)$ given by $\langle A,B\rangle = - \tr(AB)$. In particular, the $H$-module decomposition
\begin{equation}\label{eq: so_decomp}
\mathfrak{so}(n)=\mathfrak{h}\oplus\mathfrak{m},
\end{equation} where $\mathfrak{m}:=\mathfrak{h}^{\perp}\subset\mathfrak{so}(n)$ is the orthogonal complement of $\mathfrak{h}=\Lie(H)$  with respect to $\langle\cdot{},\cdot{}\rangle$, is a \emph{reductive} decomposition, i.e., it satisfies $\mathrm{Ad}_{\mathrm{SO}(n)}(H)\mathfrak{m}\subseteq\mathfrak{m}$.

Now suppose that $(M^n,g)$ admits a compatible $H$-structure $Q\subset \Fr(M,g)$. Since \eqref{eq: so_decomp} is reductive, the $H$-structure induces an orthogonal $H$-module decomposition on the subbundle $\mathfrak{so}(TM):=\Fr(M,g)\times_{\mathrm{SO}(n)}\mathfrak{so}(n)$ of skew-symmetric endomorphisms in $\mathrm{End}(TM)=T^{\ast}M\otimes TM$: 
\begin{align}
    \mathfrak{so}(TM) = \fh_Q \oplus \fm_Q,\quad\text{where}\label{eq: decomp_so(TM)}\\
    \fh_Q:=Q\times_H\fh \qandq \fm_Q:=Q\times_H\fm.\nonumber
\end{align}
Recall that  a connection $\tilde{\nabla}$ on $TM$ is said to be \emph{compatible with the $H$-structure} $Q$, or simply an \emph{$H$-connection}, if the corresponding connection $1$-form $\tilde{\omega}\in\Omega^1(\Fr(M),\mathfrak{gl}(n,\mathbb{R}))$ on $\Fr(M)$ \emph{reduces} to $Q$, i.e. if $\iota_Q^{\ast}\tilde{\omega}\in\Omega^1(Q,\mathfrak{h})$ is a connection $1$-form on $Q$, where $\iota_Q:Q\hookrightarrow\Fr(M)$ is the $H$-subbundle inclusion. These $H$-connections are in fact precisely the connections on $TM$ which are induced by connections on $Q$, and they form an affine space modelled on $\Gamma({\fh}_Q)$. Since $Q$ is compatible with $g$, any $H$-connection $\tilde{\nabla}$ on $TM$ preserves $g$, and denoting by $\nabla$ the Levi--Civita connection of $(M^n,g)$, it follows that the difference $\tilde{T}_X:=\tilde{\nabla}_X-\nabla_X$ defines a skew-symmetric endomorphism $\tilde{T}_X\in\Gamma(\mathfrak{so}(TM))$, for all $X\in\sX(M)$. Essentially, $\tilde{T}$ is the \emph{torsion} of $\tilde{\nabla}$; indeed, since $\nabla$ is torsion-free, one has 
$$
\tilde{\nabla}_XY - \tilde{\nabla}_Y X - [X,Y] = \tilde{T}_XY - \tilde{T}_YX,
\quad\forall X,Y\in\sX(M).
$$ 
Writing $\tilde{T}_X = \pi_{\fh}(\tilde{T}_{X}) + \pi_{\fm}(\tilde{T}_X)$, where $\pi_{\fh},\pi_{\fm}$ denote the orthogonal projections associated to the decomposition \eqref{eq: decomp_so(TM)}, we can define the $H$-connection $\nabla_X^H:=\tilde{\nabla}_X - \pi_{\fh}(\tilde{T}_X)$. Since the difference between any two $H$-connections lies in $\Gamma({\fh}_Q)$, it follows that $\nabla^H$ is the unique $H$-connection on $M$ whose torsion $T=T^Q$ satisfies
\begin{equation}
\label{eq: def_T}
    T_X = \nabla_X^H - \nabla_X\in \Gamma({\fm}_Q).
\end{equation} The tensor $T\in\Omega^1(M,{\fm}_Q)$ is called the \emph{intrinsic torsion} of the $H$-structure $Q$, and $Q$ is said to be \emph{torsion-free} when $T=0$ identically, which means that the Levi--Civita connection is an $H$-connection and its holonomy is a subgroup of $H$, see e.g. \cite{Gonzalez-Davila2009}*{\textsection 2} and \cite{Joyce2000}*{\textsection 2.6}.

We now characterise $H$-structures on manifolds in terms of their stabilised tensors. The canonical right-action of $\GL(n,\bR)$ on tensors is the natural extension of its respective right-actions on $\mathbb{R}^n$ and $(\mathbb{R}^n)^*$:
\[
(\rg,v)\mapsto \rg^{-1}v \qforq v\in \mathbb{R}^n,
\qandq
(\rg,\alpha)\mapsto \rg^*\alpha=\alpha\circ\rg \qforq \alpha\in(\mathbb{R}^n)^*.
\]
In terms of the canonical basis $\{e_i\}$ on $\mathbb{R}^n$, and its dual basis $\{e^i\}$ on  $(\mathbb{R}^n)^{\ast}$, let us denote the components of a $(p,q)$-tensor $\xi_{\circ}\in\cT^{p,q}(\mathbb{R}^n):=\left(\bigotimes^p\mathbb{R}^n\right)\otimes\left(\bigotimes^q(\mathbb{R}^n)^{\ast}\right)$ by
\[
    \xi_{\circ} = \xi^{i_1 \ldots i_p}_{j_1 \ldots j_q} e_{i_1}\otimes \ldots\otimes e_{i_p}\otimes e^{j_1}\otimes\ldots \otimes e^{j_q},
\]
where $\xi^{i_1 \dots i_p}_{j_1 \dots j_q}:=\xi_{\circ}(e^{i_1},\dots,e^{i_p},e_{j_1},\dots, e_{j_q})\in\mathbb{R}$, and the summation convention is assumed throughout; the sum above is taken over all subsets $\{i_1,\dots,i_p\},\{j_1,...,j_q\}\subseteq\{1,\dots,n\}$. Then $\rg\in\GL(n,\bR)$ acts on $\xi_{\circ}$ by
\begin{align}
\label{diamond_operator}
    \rg.\xi_{\circ} = \xi^{i_1 \dots i_p}_{j_1 \dots j_q} \rg^{-1} e_{i_1}\otimes \dots\otimes \rg^{-1}e_{i_p}\otimes \rg^\ast e^{j_1}\otimes\dots \otimes \rg^\ast e^{j_q}.
\end{align} We shall denote the stabiliser of $\xi_{\circ}$ under this right $\mathrm{GL}(n,\mathbb{R})$-action by
\[
\mathrm{Stab}(\xi_{\circ}):=\{\rg\in\GL(n,\bR): \rg.\xi_{\circ} = \xi_{\circ}\}.
\] More generally, if $\xi_{\circ}=((\xi_{\circ})_1,\ldots,(\xi_{\circ})_k)$ is a (finite) collection of tensors $(\xi_{\circ})_i$, then we let $\GL(n,\mathbb{R})$ act on $\xi_{\circ}$ componentwise, so that 
\[
\mathrm{Stab}(\xi_{\circ})=\bigcap_i \mathrm{Stab}((\xi_{\circ})_i).
\] In particular, the standard (flat) Euclidean  metric and volume form (orientation)
\begin{align}
g_{\circ} &:= \delta_{ij}e^i\otimes e^j,\label{eq: flat_metric}\\
\vol_{\circ} &:= e^1\wedge\ldots\wedge e^n\nonumber,
\end{align} 
are stabilised by $\mathrm{Stab}(g_{\circ},\vol_{\circ})=\mathrm{Stab}(g_{\circ})\cap\mathrm{Stab}(\vol_{\circ}) = \O(n)\cap\SL(n,\mathbb{R}) = \SO(n)$. 

Now given an $H$-structure $\sigma\in\Gamma(\Fr(M)/H)$, a tensor field $\xi\in\Gamma(\cT^{p,q}(TM))$ is said to be \emph{stabilised by $H$} if, for any adapted $H$-coframe $u\in Q_{\sigma}:=\pi_H^{-1}(\sigma(M))\subset\Fr(M)$, where $\pi_H:\Fr(M)\to\Fr(M)/H$, one has $H\subseteq\mathrm{Stab}(u^{-1}.\xi)$. In what follows, we shall be mostly interested in $H$-structures that are completely characterised by their stabilised tensors. This amounts to assuming that $H\subset\mathrm{SO}(n)$ is the stabiliser of one or several tensors on $\mathbb{R}^n$, meaning $H=\mathrm{Stab}(\xi_{\circ})$ for some element $\xi_{\circ}=((\xi_{\circ})_1,\ldots,(\xi_{\circ})_k)$ in a $r$-dimensional $\GL(n,\bR)$-submodule $V\leqslant\oplus\cT^{p,q}(\mathbb{R}^n)$, $V=V_1\oplus\ldots\oplus V_k$ with $V_i\leqslant \cT^{p_i,q_i}(\mathbb{R}^n)$. Indeed, let $\cF \leqslant \bigoplus \cT^{p,q}(TM)$ be a rank $r$ subbundle with fibre $V\cong\bR^r$. 
We have a natural monomorphism of principal bundles
\begin{equation}
\label{eq: monomorphism rho}
    \rho: \Fr(M)\hookrightarrow \Fr(\cF),
    \qquad
    \rho(u_x):\cF_x \tilde\to \;V.
\end{equation}
which identifies, at each $x\in M$, the element $u_x\in \Fr(M)_x$ with a frame on the fibre $\cF_x$, i.e., with a linear isomorphism onto the typical fibre.
A section $\xi\in\Gamma(\cF)$ is a \emph{geometric structure}, modelled on a fixed element $\xi_{\circ} \in V\leqslant \oplus\cT^{p,q}(\mathbb{R}^n)$, if, for each $x\in M$, there exists a frame of $T_x M$ identifying $\xi(x)$ and $\xi_{\circ}$.
Suppose now $H\subset\SO(n)$ fixes the (linear) model structure $\xi_{\circ}$:
\begin{equation}
\label{eq: H=Stab(xi)}
    H=\Stab(\xi_{\circ}).    
\end{equation}

Equation \eqref{eq: H=Stab(xi)} in fact defines the \emph{universal section}  $\Xi\in\Gamma(\pi^*\cF)$, which codifies all smooth $H$-structures, by
\begin{align}
\label{eq: universal section}
    \Xi(y):=y^*\xi_{\circ}.
\end{align}
Explicitly, one assigns to each $H$-class of frames $y\in\Fr(M)/H$ the vector in $\cF_{\pi(y)}$ whose coordinates are given by the model tensor $\xi_{\circ}$ in the frame $\rho(u_{\pi(y)})$, as in \eqref{eq: monomorphism rho}. Now, to each homogeneous section  $\sigma\in\Gamma(\Fr(M)/H)$, defining an $H$-structure, one associates a geometric structure $\xi\in\Gamma(\cF)$ modelled on $\xi_{\circ}$ by
\begin{align}
\label{eq: Xi correspondence}
    \xi_\sigma:=\sigma^*\Xi
    =\Xi\circ \sigma.
\end{align}
Conversely, to a given geometric structure $\xi\in\Gamma(\cF)$ stabilised by $H$, one associates, at each $x\in M$, the $H$-class of frames $\sigma(x)\in\pi^{-1}(x)$ such that $\xi(x)=\sigma(x)^*\xi_{\circ}$. In view of the correspondence \eqref{eq: Xi correspondence}, one often colloquially speaks of geometric structures, $H$-structures, and homogeneous sections interchangeably.

Note that for $H$-structures $\sigma\in\Gamma(\Fr(M,g)/H)$ compatible with a \emph{fixed} background metric and orientation, it suffices to consider a (multi-)tensor $ {\xi_{\circ}}$ lying in an $\SO(n)$-submodule $ V\leqslant\oplus\cT^{p,q}(\mathbb{R}^n)$,
so that $H=\mathrm{Stab}_{\SO(n)}(\xi_{\circ}):=\{\rg\in\SO(n):\rg.\xi_{\circ}=\xi_{\circ}\}$. Then the homogeneous section $\sigma$ corresponds bijectively to a geometric structure $\xi$ modelled on $\xi_{\circ}$, satisfying the necessary compatibility relations with the metric and orientation.

Motivated by Berger's list of the possible holonomy groups of a simply-connected and non-locally symmetric Riemannian manifold (see \cite{Joyce2007}*{Theorem 3.4.1}), we now discuss the three main examples of $H$-structures (defined by geometric structures) in which we shall be mostly interested throughout this paper.

\begin{example}[$\U(m)$-structures]\label{ex: almost_complex_structure}
Let $n=2m\geqslant 4$ and $H=\U(m)\subset \SO(2m)$. We can write
\[
\U(m) = \mathrm{Stab}(J_{\circ})\cap\mathrm{Stab}(g_{\circ}) = \mathrm{Stab}_{\SO(n)}(J_{\circ}),
\] where $J_{\circ}\in\mathrm{End}(\mathbb{R}^{2m})$ is the standard complex structure on $\mathbb{R}^{2m}=\mathbb{R}^m\oplus\mathbb{R}^m$ given in canonical coordinates by the matrix
\[
J_{\circ}=\left(\begin{array}{@{}c|c@{}}
  \textbf{0}
  & -\textbf{1} \\
\hline
  \textbf{1}
  & \textbf{0}
\end{array}\right).
\] 
A $\U(m)$-structure $(g,J)$ on $M^{2m}$, also called an \emph{almost Hermitian structure}, is determined by a Riemannian metric $g$ on $M^{2m}$ and an orthogonal almost complex structure $J$, i.e., an element $J\in\Gamma(\mathrm{End}(TM))$ such that $J^2=-\mathrm{Id}_{TM}$, and satisfying $J^{\ast}g=g$. Note that $|J|_g^2=2m$, and one further has $\frac{1}{m!}\omega^m = \vol_g$, where $\omega:=g(J\cdot{},\cdot{})\in\Omega^2(M)$ is the associated fundamental non-degenerate $2$-form. In fact,
\[
\U(m) = \mathrm{Stab}(\omega_{\circ})\cap\mathrm{Stab}(g_{\circ}) = \mathrm{Stab}_{\SO(n)}(\omega_{\circ}),
\] where $\omega_{\circ}\in\Lambda^2(\mathbb{R}^{2m})^{\ast}$ is the standard symplectic $2$-form associated to $(g_{\circ},J_{\circ})$. Thus, when considering a compatible $\U(m)$-structure on a given oriented Riemannian $2m$-manifold $(M^{2m},g)$, the model structure can be taken to be either $\xi_{\circ}=J_{\circ}$ or $\xi_{\circ}=\omega_{\circ}$. Note that compatible $\U(m)$-structures are in one-to-one correspondence with sections of the $\SO(2m)/\U(m)$-bundle $\pi:\Fr(M^{2m},g)/\U(m)\to M$.

Under the metric identification $\Lambda^2\cong\mathfrak{so}(2m)$, we have the $\mathrm{U}(m)$-irreducible decomposition
\begin{align*}
\Lambda^2 &=\Lambda^2_{\fu(m)}\oplus \Lambda^2_\fm,\quad\text{where}\\     
\Lambda^2_{\fu(m)}
%=\{\alpha\in \Lambda^2 : J^\ast\alpha=\alpha \}\simeq\mathfrak{u}(m)
&\cong\fu(m) =\{A\in \fso(2m) : J_oA=AJ_o\},
%\simeq[\Lambda_0^{1,1}]\oplus\langle\omega\rangle,
\\  %where $[\Lambda_0^{1,1}]=\mathfrak{su}(m)$ is the orthogonal complement of $\langle\omega\rangle$ in $[\Lambda^{1,1}]$, and
\Lambda^2_\fm%=\{\alpha\in \Lambda^2 : J^\ast\alpha=-\alpha\}
&\cong\mathfrak{m}:=\mathfrak{u}(m)^{\perp} =\{A\in \fso(2m) : J_oA=-AJ_o\}.
%[[\Lambda^{2,0}]]=[[\Lambda^{0,2}]]
\end{align*}
%is $\mathrm{U}(m)$-irreducible. Here $[\Lambda^{1,1}]\otimes\mathbb{C} = \Lambda^{1,1}$ and $[[\Lambda^{2,0}]]\otimes\mathbb{C} = \Lambda^{2,0}\oplus\Lambda^{0,2}$.
\end{example}

%\begin{example}[Quaternionic-K\"ahler structures]
%Let $n=4m\geqslant 8$ and $H=\mathrm{Sp}(m)\cdot{\mathrm{Sp}(1)}\subset\mathrm{SO}(4m)$. Then the model structure $\xi_{\circ}:=\Omega_{\circ}\in\Lambda^4(\mathbb{R}^{4m})^{\ast}$ is given by the $4$-form
%\[
%\Omega_{\circ} := \omega_1^2+\omega_2^2+\omega_3^2,
%\] where $\omega_1$, $\omega_2$ and $\omega_3$ are the fundamental $2$-forms respectively associated to the three standard complex structures $J_1,J_2$ and $J_3$ on $\mathbb{R}^{4m}\cong\mathbb{H}^m$ and the flat metric. So a $\mathrm{Sp}(m)\cdot{\mathrm{Sp}(1)}$-structure on $M^{4m}$ is equivalent to a $4$-form $\Omega$ on $M$ linearly isomorphic to $\Omega_{\circ}$ at each point.
%\end{example}

\begin{example}[$\rm G_2$-structures]\label{ex: G2_structures}
Let $n=7$ and $H=\rm G_2\subset\mathrm{SO}(7)$. In terms of the standard basis $(e^1,\ldots,e^7)$ of $(\mathbb{R}^7)^{\ast}$, define the standard $\rm G_2$-structure $\varphi_{\circ}\in\Lambda^3(\mathbb{R}^7)^{\ast}$ by\footnote{Here our sign convention follows e.g. \cite{Karigiannis2007}.}
\[
\varphi_{\circ} = e^{123} + e^1\wedge(e^{45} - e^{67}) + e^2\wedge(e^{46} - e^{75}) + e^3\wedge(e^{47} - e^{56}).
\] Then $\rm G_2 = \Stab(\varphi_{\circ})\subset\SO(7)$, and $\varphi_{\circ}$ induces the standard Euclidean metric $g_{\circ}$ and orientation $\vol_{\circ}$ through the nonlinear algebraic relation
\begin{equation}\label{eq: induced_metric_g2}
(X\lrcorner\varphi_{\circ})\wedge(Y\lrcorner\varphi_{\circ})\wedge\varphi_{\circ} = -6g_{\circ}(X,Y)\vol_{\circ},\quad\forall X,Y\in\sX(\mathbb{R}^7).
\end{equation} Moreover, if $\ast_{\circ}$ denotes the Hodge star operator induced from $(g_{\circ},\vol_{\circ})$, then $\psi_{\circ}:=\ast_{\circ}\varphi_{\circ}\in\Lambda^4(\mathbb{R}^7)^{\ast}$ is given by
\[
\psi_{\circ} = e^{4567} - e^{4523} - e^{4163} - e^{4127} - e^{2637} - e^{1537} - e^{1526},
\] and one also has $\rm G_2 = \Stab_{\SO(7)}(\psi_{\circ})$. Note from the above expressions that $|\varphi_{\circ}|_{\circ}^2=|\psi_{\circ}|_{\circ}^2=7$, or equivalently $\varphi_{\circ}\wedge\psi_{\circ} = 7\vol_{\circ}$. 

A $\rm G_2$-structure on a smooth $7$-manifold $M^7$ is then defined by a $3$-form $\varphi$ which is pointwise linearly identified with $\varphi_{\circ}$, also known as a \emph{positive} $3$-form $\varphi\in\Omega_{+}^3(M)$. It then induces a metric $g$ and orientation $\vol_g$ on $M^7$ via the pointwise algebraic relation \eqref{eq: induced_metric_g2}. It is known that $M^7$ admits a $\rm G_2$-structure if and only if it is both orientable and spinnable. When that is the case, then for any metric $g$ the Riemannian manifold $(M^7,g)$ admits a compatible $\rm G_2$-structure \cite{Bryant2006}*{Remark 3}, i.e. a positive $3$-form $\varphi\in\Omega_{+}^3(M)$ satisfying the compatibility condition
\[
(X\lrcorner\varphi)\wedge(Y\lrcorner\varphi)\wedge\varphi = -6g(X,Y)\vol_{g},\quad\forall X,Y\in\sX(M).
\] The compatible $\rm G_2$ structures on $(M^7,g)$ are in one-to-one correspondence with sections of the fibre bundle $\pi: \Fr(M^7,g)/\mathrm{G}_2\to M$, with fibre $\rm SO(7)/{\rm G_2}\cong\mathbb{RP}^7$. In fact, given a compatible $\rm G_2$-structure $\varphi$ on $(M^7,g)$, then any other compatible $\rm G_2$-structure can be explicitly parametrised by pairs $(f,X)\in C^{\infty}(M)\times\Gamma(TM)$ satisfying $f^2+|X|^2 = 1$, and $\pm(f,X)$ induce the same $\rm G_2$-structure: if we let $\psi:=\ast\varphi$, then the $\rm G_2$-structure $\varphi_{(f,X)}$ corresponding to $(f,X)$ is \cite{Bryant2006}*{(3.6)}
\begin{equation}
\label{eq: explicit_param_g2}
	\varphi_{(f,X)} = (f^2-|X|^2)\varphi - 2f X\lrcorner\psi + 2X\wedge (X\lrcorner\varphi).    
\end{equation}
As for the decomposition \eqref{eq: so_decomp} in the case $H=\rm G_2$, we note that under the metric identification $\Lambda^2\cong\mathfrak{so}(7)$, we have the following irreducible $\rm G_2$-module decomposition:
\begin{align}
\Lambda^2 &= \Lambda_{\fg_2}^2\oplus\Lambda_\fm^2,\quad\text{where}\nonumber\\ 
\Lambda_{\fg_2}^2 &= \{\omega: \ast(\omega\wedge\varphi) = \omega\} = \{\omega: \omega\wedge\ast\varphi = 0\}\cong\mathfrak{g}_2, \nonumber\\
\Lambda_{\fm}^2 &= \{\omega: \ast(\omega\wedge\varphi) = -2\omega\} = \{u\lrcorner\varphi : u\in \bR^7\}\cong\mathfrak{m}.\label{eq: decomp_lambda_g2}
\end{align}
\end{example}

\begin{example}[$\rm Spin(7)$-structures]\label{ex: Spin7_structures}
Let $n=8$ and $H=\mathrm{Spin}(7)\subset\mathrm{SO}(8)$. The model structure here is $\xi_{\circ}=\Phi_{\circ}\in\Lambda^4(\mathbb{R}^8)^{\ast}$, given in terms of the standard basis $(e^0,e^1,\ldots,e^7)$ of $(\mathbb{R}^8)^{\ast}=(\mathbb{R})^{\ast}\oplus(\mathbb{R}^7)^{\ast}$ by
\[
\Phi_{\circ} = e^0\wedge\varphi_{\circ} + \ast_{\mathbb{R}^7}\varphi_{\circ}.
\] A compatible $\mathrm{Spin}(7)$-structure on $(M^8,g)$ is defined by a $4$-form $\Phi\in\Omega^4(M)$ which is pointwise linearly isomorphic to $\Phi_{\circ}$, and which induces the metric $g$ and orientation $\vol_g$; for each $p\in M$, if we extend any non-zero tangent vector $v\in T_pM$ to a local frame $\{v, e_1, \cdots , e_7\}$, and let
\begin{align*}
    B_{ij}(v)&=((e_i\lrcorner v\lrcorner \s7)\wedge (e_j\lrcorner v\lrcorner \s7)\wedge (v\lrcorner \s7))(e_1, \cdots , e_7),\\
    A(v)&=((v\lrcorner \s7)\wedge \s7)(e_1, \cdots, e_7),
\end{align*} then we have
\[
    (g(v,v))^2=-\frac{7^3}{6^{\frac{7}{3}}}\frac{(\textup{det}\ B_{ij}(v))^{\frac 13}}{A(v)^3}.    
\]
The metric and the orientation determine a Hodge star operator $\ast$, and the $4$-form is \emph{self-dual}, i.e., $\ast\Phi =\Phi$. The compatible $\rm Spin(7)$-structures on $(M^8,g)$ are in one-to-one correspondence with sections of the fibre bundle $\pi:\Fr(M^8,g)/\rm Spin(7)\to M$ with fibre $\SO(8)/\rm Spin(7)\cong\mathbb{RP}^7$, and in fact there is an explicit parametrisation of such structures analogous to the one for $\rm G_2$-structures \eqref{eq: explicit_param_g2}, which can be found in \cite{Dwivedi-Loubeau-SaEarp2021}*{Theorem A}.

Under the identification  $\Lambda^2\cong\mathfrak{so}(8)=\mathfrak{spin}(8)$, we have the irreducible $\mathrm{Spin}(7)$-module decomposition
\begin{align*}
\Lambda^2 &= \Lambda_{\spin(7)}^2\oplus\Lambda_\fm^2,\quad\text{where}\\ 
\Lambda_{\spin(7)}^2 &= \{\omega: \ast(\omega\wedge\Phi) = \omega\}\cong\mathfrak{spin}(7)
=\mathfrak{so}(7), \\
\Lambda_{\fm}^2 &= \{\omega: \ast(\omega\wedge\Phi) = -3\omega\}\cong\mathfrak{m}.
\end{align*}
\end{example}

\subsection{Infinitesimal deformations}
\label{sec: inf_deformations}

We will describe infinitesimal deformations of $H$-structures following the perspective adopted by Karigiannis  for $H=\rG_2,\Spin7$ \cites{Karigiannis2007,karigiannis-spin7}, building upon some notation and results established in \cite{Dwivedi-Loubeau-SaEarp2021}. This will allow us to derive several useful identities regarding the infinitesimal action of $\mathrm{GL}(n,\mathbb{R})$ on tensors, and some important facts relating this action with $H$-structures defined by (multi)tensor fields. 

Let $(M^n,g)$ be an oriented Riemannian $n$-manifold. The musical isomorphisms defined by $g$ induce the following decomposition of endomorphisms on $TM$: 
\begin{align*}
    \Gamma(\End(TM))=\Gamma(\sym(TM))\oplus\Gamma(\fso(TM))\simeq\Sigma^2(M)\oplus\Omega^2(M),   
\end{align*}
where $\Gamma(\sym(TM))$ (resp. $\Sigma^2(M)$) denotes the space of symmetric endomorphisms (resp. symmetric bilinear forms) on $TM$. Explicitly, for any $A\in\Gamma(\mathrm{End}(TM))$, we let $A_{ij}:= g_{lj}A_i^l$ and we decompose $A=S+C\in \Sigma^2(M)\oplus\Omega^2(M)$, where $S_{ij}=\frac{1}{2}(A_{ij} + A_{ji})$ and $C_{ij}=\frac{1}{2}(A_{ij} - A_{ji})$. 

Let $\xi\in\Gamma(\cT^{p,q}(TM))$ be any $(p,q)$-tensor field on $M$. 
In local coordinates, we write
\[
    \xi = \xi^{i_1 \dots i_p}_{j_1 \dots j_q} \frac{\partial}{\partial x^{i_1}}\otimes \dots\otimes\frac{\partial}{\partial x^{i_p}}\otimes dx^{j_1}\otimes\dots \otimes dx^{j_q},
\]
where $\xi^{i_1 \dots i_p}_{j_1 \dots j_q}:=\xi(dx^{i_1},\dots,dx^{i_p};\frac{\partial}{\partial x^{j_1}},\dots,\frac{\partial}{\partial x^{j_q}})$ are smooth local functions. Now, the canonical right $\GL(n,\bR)$-action \eqref{diamond_operator} on tensors on $\mathbb{R}^n$ extends naturally pointwise to tensors on $M$.
This induces an \emph{infinitesimal action} of endomorphisms $A\in\Gamma(\mathrm{End}(TM))$ on $\cT^{p,q}(TM)$ given by 
\begin{align}
\label{eq: diamond_operator}
    A\diamond \xi &:=\left.\ddt\right\vert_{t=0} e^{tA}.\xi \nonumber\\
    &= \xi^{i_1 \dots i_p}_{j_1 \dots j_q}  \Big(
    \sum_{r=1}^{p} 
    - \frac{\partial}{\partial x^{i_1}}\otimes \dots \otimes A\frac{\partial}{\partial x^{i_r}}\otimes \dots\otimes \frac{\partial}{\partial x^{i_p}}\otimes dx^{j_1}\otimes\dots \otimes dx^{j_q}\\
    &\quad\quad+ \sum_{s=1}^{q}
    \frac{\partial}{\partial x^{i_1}}\otimes \dots\otimes \frac{\partial}{\partial x^{i_p}}\otimes dx^{j_1}\otimes\dots \otimes A^{\ast}dx^{j_s}\otimes\dots\otimes dx^{j_q} \Big)\nonumber.
\end{align} Writing $A=(A^i_j)\in \mathfrak{gl}(n, \mathbb{R})$ (pointwise) in the above coordinates, and reordering terms in \eqref{eq: diamond_operator}, one has
\begin{equation}
\label{eq:diamond_coordinates}
    (A\diamond\xi)^{i_1\dots i_p}_{j_1\dots j_q}=-\sum_{r=1}^p A^{i_r}_m\xi^{i_1\dots m \dots i_p}_{j_1 \dots j_q}+\sum_{s=1}^q A^m_{j_s}\xi_{j_1 \dots m \dots j_q}^{i_1 \dots i_p}.
\end{equation} 
More generally, if $\xi=(\xi_1,\ldots,\xi_k)$ is a multi-tensor, we define the \emph{diamond operator} by the component-wise infinitesimal action of $A\in\Gamma(\End(TM))$:
\[
A\diamond\xi := (A\diamond\xi_1,\ldots,A\diamond\xi_k).
\]

We now collect some useful preliminary facts about the diamond operator.
\begin{lemma}
\label{lem: properties_diamond}
    Let $(M,g)$ be a Riemannian manifold, $A,B \in\Gamma(\mathrm{End}(TM))$ and $\xi,\eta\in \Gamma(\cT^{p,q}(TM))$. Then the operator $\diamond$ given by \eqref{eq: diamond_operator} satisfies the following identities:
\begin{itemize}
    \myitem[(i)] \label{item: diamond_bracket} $A\diamond B=-[A,B]$.
    
    \myitem[(ii)] \label{item: diamond_Jacobi} $A\diamond(B\diamond\xi)-B\diamond(A\diamond\xi)=-[A,B]\diamond\xi$.
    
    \myitem[(iii)] \label{item: diamond_sym_alt} Suppose that $p=0$, i.e. $\xi\in\Gamma(\cT^{0,q}(TM))$. If $\xi$ is a symmetric (resp. alternating) tensor, then $A\diamond\xi$ is a symmetric (resp. alternating) tensor.
    
    \myitem[(iv)] \label{item: diamond_g} \cite{Dwivedi-Loubeau-SaEarp2021}*{Lemma 2.4}  $g\diamond\xi=(q-p)\xi$; we shall call $\ell:=q-p$ the \emph{net degree} of the tensor $\xi$.
    
    \myitem[(v)] \label{item: diamond_S} Decomposing $A=S+C\in \Sigma^2(M)\oplus\Omega^2(M)$, we have $A\diamond g = 2S$. In particular, $\ker(\cdot{}\diamond g) = \Omega^2$.
    
    \myitem[(vi)] \label{item: diamond_vol} $A\diamond\vol_g = \tr(A)\vol_g$; in particular, $\ker(\cdot{}\diamond \vol_g) = \Sigma_0^2\oplus\Omega^2$.
    
    \myitem[(vii)] \label{item: diamond_orth} If $D\in\Omega^2(M)$ then $\langle D\diamond\xi,\xi\rangle_g = 0$.
    
    \myitem[(viii)] \label{item: diamond_adjoint} If $D\in\Omega^2(M)$ then $\langle  {\eta},D\diamond\xi\rangle_g = -\langle D\diamond {\eta},\xi\rangle_g$.
\end{itemize}
\end{lemma}
\begin{proof}\quad
\begin{itemize}
    \item[(i)]  An element $\rg\in\mathrm{GL}(n,\mathbb{R})$ acts on endomorphisms by conjugation, i.e.,  $\rg.B=\rg^{-1}B\rg=\mathrm{Ad}_{\mathrm{GL}(n,\mathrm{R})}(\rg^{-1})(B)$. Along a curve $\rg(t)=e^{tA}$ one has $\mathrm{Ad}_{\mathrm{GL}(n,\mathrm{R})}(e^{-tA})=e^{-t\ad(A)}$, and we get $A\diamond B = -[A,B]$ immediately by differentiation. %Alternatively, one can compute explicitly:  %
    %\begin{align*}
     %   A\diamond B &=\sum B^i_j (-A\frac{\partial}{\partial x^i}\otimes dx^j+\frac{\partial}{\partial x^i}\otimes A^{\ast}dx^j)\\
      %      &=\sum -B^i_j A^k_i \frac{\partial}{\partial x^k}\otimes dx^j+B^i_j A^j_l\frac{\partial}{\partial x^i}\otimes dx^l\\
       %     &=\sum\{ -(AB)^i_j + (BA)^i_j \}\frac{\partial}{\partial x^i}\otimes dx^j\\
        %    &=-[A,B].
    %\end{align*}
    \item[(ii)] Notice that the $\GL(n,\bR)$-action is distributive over the diamond operator: 
    \begin{align*}
        \rg.(B\diamond \xi) &= \left.\ddt\right\vert_{t=0} \rg.(e^{tB}.\xi)
        = \left.\ddt\right\vert_{t=0} (\rg^{-1}e^{tB}\rg).(\rg.\xi) \\
         &= \left.\ddt\right\vert_{t=0} e^{t(\rg.B)}.(\rg.\xi)
         = (\rg.B)\diamond (\rg.\xi).
    \end{align*}
    Applying identity $(i)$ along a curve $\rg(t)=e^{tA}$, we have
    \begin{align*}
        A\diamond(B\diamond\xi) &= \left.\ddt\right\vert_{t=0} (e^{tA}.B)\diamond(e^{tA}.\xi)\\
        &= -[A,B]\diamond\xi+B\diamond(A\diamond \xi).
    \end{align*}
    \item[(iii)] For $\xi\in\Gamma(\cT^{0,q}(TM))$,
    \[
    (A\diamond\xi)_{i_1\ldots i_q} = A_{i_1}^m\xi_{m i_2\ldots i_q} + A_{i_2}^m\xi_{i_1 m\ldots i_q} + \ldots + A_{i_q}^m\xi_{i_1 i_2\ldots m}.  
    \] 
    It is clear that  $\xi\in\Sigma^q(M)$ (resp. $\Omega^q(M)$) implies $A\diamond\xi\in\Sigma^q(M)$ (resp. $\Omega^q(M)$).
    
    \item[(iv)] If $A_{ij} = \delta_{ij}$ is the metric tensor,  in normal coordinates at a point, then
    \[
    A\frac{\partial}{\partial x^{i_r}} = \frac{\partial}{\partial x^{i_r}} ; \quad A^{\ast}dx^{j_s} = dx^{j_s},
    \] 
    and we conclude by counting terms in \eqref{eq: diamond_operator}.
    
    \item[(v)] We compute directly:
    $(A\diamond g)_{ij} 
        = A_i^p g_{pj} + A_j^p g_{ip}
        = (S_{ij}+C_{ij}) + (S_{ji}+C_{ji})
        = 2S_{ij}$.

    \item[(vi)] For any $t\in\bR$, note that
    \[
    e^{tA}.\vol_g = (e^{tA})^{\ast}\vol_g = \det(e^{tA})\vol_g;
    \] hence,
    \[
    A\diamond\vol_g = \frac{d}{dt}\Bigr|_{t=0} \det(e^{tA})\vol_g = \tr(A)\vol_g.
    \]

    \item[(vii)] If $D\in\Omega^2(M)$ then
    \begin{align*}
        2\langle D\diamond\xi,\xi\rangle_g &= \frac{d}{dt}\Bigr|_{t=0}\langle e^{tD}.\xi,e^{tD}.\xi\rangle_g\\
        &= \frac{d}{dt}\Bigr|_{t=0}\langle\xi,\xi\rangle_g\quad\text{(since $e^{tD}\in\mathrm{SO}(TM)$)}\\
        &= 0.
    \end{align*} 
    \item[(viii)] If $D\in\Omega^2(M)$ then
    \begin{align*}
        \langle  {\eta},D\diamond\xi\rangle &= \frac{d}{dt}\Bigr|_{t=0}\langle   {\eta},e^{tD}.\xi\rangle
        = \frac{d}{dt}\Bigr|_{t=0}\langle e^{tD}.(e^{-tD}.\eta),e^{tD}.\xi\rangle\\
        &= \frac{d}{dt}\Bigr|_{t=0}\langle e^{-tD}.\eta,\xi\rangle\quad\text{(since $e^{tD}\in\mathrm{SO}(TM)$)}\\
        &= -\langle D\diamond\eta,\xi\rangle.
        \qedhere
    \end{align*}
\end{itemize}
\end{proof}
\begin{remark}
    The minus sign in the identity $A\diamond B = -[A,B]$ of Lemma \ref{lem: properties_diamond}--\ref{item: diamond_bracket} is a result of the fact that the infinitesimal action comes from a \emph{right} action.
\end{remark}
\begin{remark}
    The identity of Lemma \ref{lem: properties_diamond}--\ref{item: diamond_Jacobi} reduces to the Jacobi identity for the commutator Lie bracket $[\cdot,\cdot]$ of endomorphisms when $\xi$ is a $(1,1)$-tensor, by identity \ref{item: diamond_bracket}. 
\end{remark}

Now suppose $(M^n,g)$ admits a compatible $H$-structure $Q\subset \Fr(M,g)$. From \eqref{eq: decomp_so(TM)}, we get a corresponding $H$-module decomposition on $\Lambda^2(T^{\ast}M)\simeq\mathfrak{so}(TM)$: 
\begin{align*}
    \Lambda^2 = \Lambda^2_\fh \oplus \Lambda^2_\fm,\qwithq
    \Lambda^2_\fh \simeq \fh_Q \qandq \Lambda^2_\fm\simeq \fm_Q.
\end{align*} 
We shall write $\Omega^2_{\fh}:=\Gamma(\Lambda^2_{\fh})$ and $\Omega^2_{\fm}:=\Gamma(\Lambda^2_{\fm})$. Then, splitting out the trivial submodule $\Omega^0$ of $\Sigma^2(M)$ spanned by the Riemannian metric, we have
\begin{align}
\label{eq:splitting TM* x TM-GENERAL}
 \Gamma(\End(TM))\simeq\Omega^0\oplus \Sigma_0^2 \oplus \Omega^2_\fh \oplus \Omega^2_\fm,
\end{align}
where $\Sigma_0^2$ denotes the space of traceless symmetric bilinear forms. Hence, with respect to \eqref{eq:splitting TM* x TM-GENERAL}, we can decompose 
$$A=\frac{1}{\dim M} (\tr A)g+A_0+A_\fh+A_\fm,
$$ 
where $A_0$ is a symmetric traceless $2$-tensor. 
\begin{lemma}
\label{lem: ker_diamond}
    Let $(M^n,g)$ be an oriented Riemannian $n$-manifold and suppose that $\sigma\in\Gamma(\Fr(M,g)/H)$ is a compatible $H$-structure. Then the following hold:
\begin{itemize}
    \myitem[(i)] If $\xi\in\Gamma(\cT^{p,q}(TM))$ is stabilised under the action of $H$, then $\Omega_{\fh}^2\subseteq\ker(\cdot{}\diamond \xi)$.
    \myitem[(ii)] If $H=\mathrm{Stab}(\xi_{\circ})$, as in \eqref{eq: H=Stab(xi)}, so that  $\sigma$ corresponds to a geometric structure $\xi=(\xi_1,\ldots,\xi_k)$ modelled on $\xi_{\circ}$, then
    \[
    \Omega_{\fh}^2=\ker(\cdot{}\diamond\xi)=\ker(\cdot{}\diamond\xi_1)\cap\ldots\cap\ker(\cdot{}\diamond\xi_k).
    \] 
    \myitem[(iii)] \label{item: ker in Omega^2}
    If $H=\mathrm{Stab}_{\SO(n)}(\xi_{\circ})$, so that $\sigma$ corresponds to a geometric structure $\xi$ modelled on $\xi_{\circ}$, which is compatible with $g$ and $\vol_g$, then
    \[
    \Omega_{\fh}^2=\ker(\cdot{}\diamond\xi)\cap\Omega^2.
    \]
\end{itemize}
\end{lemma}
\begin{proof}
\quad
\begin{itemize}
    \item[(i)] By the hypotheses, if $A\in\Omega_{\fh}^2$ then $e^{tA}.\xi = \xi$ for all $t\in\bR$, which implies $A\diamond\xi = 0$.
    
    \item[(ii)] It suffices to prove the claim pointwise, and by (i) it suffices to show that, if $A\in\mathfrak{gl}(n,\bR)$ is such that $\frac{d}{dt}(e^{tA}.\xi_{\circ})|_{t=0}=0$, then $A\in\mathfrak{h}$. Recall that the model structure $\xi_{\circ}$ lies in  a $\GL(n,\bR)$-submodule $V\leqslant\oplus\cT^{p,q}(\mathbb{R}^n)$ under the natural right-action $$
    \phi:\GL(n,\bR)\times V\to V, \quad \phi(\rg,\xi)=\rg.\xi.
    $$ 
    Fixing elements $\rg\in \GL(n,\bR)$ and $\xi\in V$ induces, respectively, partial action maps \[\phi_\rg:=\phi(\rg,\cdot{}): V\rightarrow V
    \qandq
     {\psi}_{\xi}:=\phi(\cdot{},\xi):\GL(n,\bR)\rightarrow V,
    \]
    and the infinitesimal action of $A\in\fgl(n,\bR)$ on $\xi$ is given by
\begin{equation*}
    \widetilde{A}(\xi)   {:=} A\diamond\xi = \left.\ddt\right\vert_{t=0} e^{tA}.\xi = \left.\ddt\right\vert_{t=0}   {\psi}_{\xi}(e^{tA}) = (d {\psi}_{\xi})_1(A).
\end{equation*}
In particular, $\widetilde{A}$ defines a vector field  {on} the orbit $\GL(n,\bR).\xi_{\circ}$, with flow $\phi_{e^{tA}}$. In fact, for any $\xi\in \GL(n,\bR).\xi_{\circ}$, the curve $t\mapsto e^{tA}.\xi$ is the flow line of $\widetilde{A}$  through $\xi$, since
\begin{align}
\label{eq: flow_line_eq}
    \frac{d}{dt}(e^{tA}.\xi)=\left.\frac{d}{ds}\right\vert_{s=0} e^{(s+t)A}.\xi=\widetilde{A}(\phi_{e^{tA}}(\xi)).
\end{align}
On the other hand, the translation by $\rg$ of the vector field $\widetilde{A}$ is given by 
\begin{equation}
\label{eq: Adjoint_action_vectors}
    (d\phi_{\rg^{ {-1}}})_{\phi_{\rg}(\xi)}\Big(\widetilde{A}(\phi_\rg(\xi))\Big)
    =(\widetilde{\Ad(\rg^{ {-1}})(A)})(\xi).
\end{equation}
Thus, if $\widetilde{A}(\xi_{\circ})=A\diamond\xi_{\circ}=0$, then it follows from \eqref{eq: flow_line_eq} and \eqref{eq: Adjoint_action_vectors} that $e^{tA}.\xi_{\circ}$ is constant for any $t\in\bR$:
\begin{align*}
    \frac{d}{dt}(e^{tA}.\xi_{\circ})
    &=\widetilde{A}(\phi_{e^{tA}}(\xi_{\circ})) \\
    &= d(\phi_{e^{tA}} \circ \phi_{e^{-tA}})_{\phi_{e^{tA}}(\xi_{\circ})} \Big(\widetilde{A}(\phi_{e^{tA}}(\xi_{\circ}))\Big) \\
    &=(d\phi_{e^{tA}})_{\xi_{\circ}}\Big((\widetilde{\Ad(e^{-tA})(A)})(\xi_{\circ})\Big)\\
    &=(d\phi_{e^{tA}})_{\xi_{\circ}}(\widetilde{A}(\xi_{\circ}))\\
    &=0,
\end{align*}
where we have used the fact that $ {\Ad(e^{tA}) (A) = A}$.
Now $H=\mathrm{Stab}(\xi_{\circ})$ implies $e^{tA}\in H$, for all $t$, and therefore $A\in \fh$, as claimed.

\item[(iii)] Since $H=\Stab_{\SO(n)}(\xi_{\circ})=\Stab(\xi_{\circ},g_{\circ},\vol_{\circ})$, we conclude from (ii) and Lemma \ref{lem: properties_diamond}--\ref{item: diamond_S},\ref{item: diamond_vol}. \qedhere
\end{itemize}
\end{proof}
%\begin{remark}
 %   From the closed subgroup theorem, the smooth structure of $H$ as an embedded Lie subgroup of $\SO(n)$ is obtained from the structure of $\langle \exp \fh\rangle$, where $$\fh=\{C\in\fso(n);\quad  e^{tC}\in H, \quad \forall \quad t\in \bR\},
  %  $$ 
  %  on the other hand, the stabiliser Lie algebra of $\xi$ into $\fso(n)$ is
  %  $$\mathfrak{stab}(\xi)=\{C\in \fso(n); \quad C\diamond \xi =0\}
  %  $$
   % Then we have $\ker(\cdot{}\diamond \, \xi)\cap \fso(n)=\fh$. 
%\end{remark}
\begin{remark}
\label{rmk: ker_diamond}
     In the situation of Lemma \ref{lem: ker_diamond}, if $H=\Stab_{\SO(n)}(\xi_{\circ})\subsetneq\Stab(\xi_{\circ})$ then $\Sigma^2(M)\cap\ker(\cdot{}\diamond \xi)\neq\{0\}$; see Example \ref{ex: diamond_J} below.
\end{remark}
\begin{example}
\label{ex: diamond_J}
    Consider the compatible almost complex case from Example \ref{ex: almost_complex_structure}.
    %Let $n=2m\geqslant 4$ and $H=U(m)=\mathrm{Stab}_{\mathrm{SO}(2m)}(J_{\circ})$, where $J_{\circ}$ is the standard complex structure in $\mathbb{R}^{2m}$, and consider $\xi=J$ a compatible almost complex structure on $(M^{2m},g)$. 
    Since $J$ is a $(1,1)$-tensor, it follows from Lemma \ref{lem: properties_diamond}--\ref{item: diamond_g} that $g\diamond J = 0$, i.e. $\Omega^0\subseteq\ker(\cdot{}\diamond J)$ in the sense of \eqref{eq:splitting TM* x TM-GENERAL}:
    $$
    \mathrm{End}(TM)=\Omega^0\oplus\Sigma_0^2\oplus\Omega_{\mathfrak{u}(m)}^2\oplus\Omega_{\mathfrak{u}(m)^{\perp}}^2.
    $$
    By Lemma \ref{lem: properties_diamond}--\ref{item: diamond_bracket}, one has $A\in\ker(\cdot{}\diamond J)$ if and only if $A$ commutes with $J$, because $\mathrm{Stab}_{\mathrm{GL}(2m,\mathbb{R})}(J_{\circ}) = \mathrm{GL}(m,\mathbb{C})$, so indeed  $\ker(\cdot{}\diamond J)\simeq\mathfrak{gl}(m,\mathbb{C})\subset\mathfrak{gl}(2m,\mathbb{R})$ pointwise. Moreover, the instance of
    Lemma \ref{lem: ker_diamond}--\ref{item: ker in Omega^2}, giving $\Omega_{\mathfrak{u}(m)}^2=\ker(\cdot{}\diamond J)\cap\Omega^2$, is a reflection of the fact that $\U(m)=\mathrm{GL}(m,\mathbb{C})\cap\mathrm{SO}(2m)$. 
    
    Further decomposing the trace-free symmetric endomorphisms $\Sigma_0^2=U\oplus W$ into the $(m^2-1)$-dimensional subspace $U$, of elements commuting with $J$, and the $m(m+1)$-dimensional subspace $W$, of elements anti-commuting with $J$, then it follows by dimension-counting that $\ker(\cdot{}\diamond J)=\Omega^0\oplus U\oplus\Omega_{\mathfrak{u}(m)}^2$, and therefore $\cdot{}\diamond J$ maps $W\oplus\Omega_{\mathfrak{u}(m)^{\perp}}^2$ isomorphically into itself, respecting the decomposition.
\end{example}

\subsection{Inner-product relations and torsion}
\label{sec: inner-prod and torsion}

We examine how the tensor inner-product behaves under the infinitesimal $\diamond$ action, in terms of the reducibility of the orthogonal complement $\fm=\fh^\perp\subset\Lambda^2$. This in turn will lead to a quantitative study of the intrinsic torsion of an $H$-structure in terms of its covariant derivative, culminating at an analytically useful Laplacian estimate. We illustrate the process by some original computations, and by recovering   {along the way} a number of familiar properties of almost complex, $\rG_2$-, and $\S7$-geometries. The overarching point here is that, while such facts have been derived in the literature by a strong appeal to context-specific contraction identities and algebraic identifications, with no immediately apparent extension to other $H$-structures, in truth they are particular instances of an abstract general theory  {(even though explicit computations will rely on constants $\lambda_i$ peculiar to the group $H$)}.

\begin{lemma}
\label{lem: orthogonality}
    Let $(M^n,g)$ be an oriented Riemannian $n$-manifold and suppose that $\xi$ is a compatible $H$-structure, where $H=\mathrm{Stab}_{\SO(n)}(\xi_{\circ})$. Let $\fm=\fm_1\oplus\ldots\oplus\fm_k$ be an orthogonal decomposition of $\fm$, with respect to the bi-invariant metric $\langle A,B\rangle=-\tr(AB)$, into non-equivalent, 
     {isotypic} $\mathrm{Ad}_{\SO(n)}(H)$-submodules\footnote{By Schur's Lemma, such a decomposition is unique up to ordering.}  {(i.e. direct sums of isomorphic irreducible $\mathrm{Ad}_{\SO(n)}(H)$-submodules)}. Then the following assertions hold.
\begin{itemize}
    \myitem[(i)] \label{item: inner_product} 
    There are positive constants $\lambda_i\in\mathbb{R}_{+}$ such that, for all $A,B\in\Omega_{\fm}^2(M)$,
\begin{equation}
\label{eq: inner_product_decomp}
    \langle A\diamond\xi,B\diamond\xi\rangle 
    = \sum\limits_{i=1}^k \lambda_i\langle A_i,B_i\rangle,
\end{equation} 
    where $A_i:=\pi_{\fm_i}(A)$, $B_i:=\pi_{\fm_i}(B)$, for  $i=1,\ldots,k$.
    
    \myitem[(ii)] \label{item: almost_orthogonality}
    In particular, 
\begin{equation}
\label{eq: almost_orthogonality}
    \langle C\diamond(C\diamond\xi),D\diamond\xi\rangle = \sum\limits_{i=1}^k \lambda_i\langle[C,D],C_i\rangle,
    \quad\forall 
    C,D\in\Omega_{\fm}^2(M).
\end{equation}
    
    \myitem[(iii)]\label{item: orthogonality}
    \begin{tabular}[t]{ll}
    If 
    &
    \begin{NiceTabular}[t]{rl} either  & $C=C_i\in\Omega_{\fm_i}^2(M)$ for some $i\in\{1,\ldots,k\}$, and $D\in\Omega_{\fm}^2(M)$ is arbitrary,  \\ 
    or & $C,D\in\Omega_{\fm}^2(M)$ are arbitrary and $\lambda_1=\ldots=\lambda_k$ (e.g. when $\fm$ is irreducible), \\ 
    \end{NiceTabular}
\end{tabular}
\medskip\\
    then the following orthogonality relation holds:
\begin{equation}
\label{eq: orthogonality}
    \langle C\diamond(C\diamond\xi), D\diamond\xi
    \rangle = 0. 
\end{equation}
\end{itemize}    
\end{lemma}
\begin{proof}\quad
\begin{itemize}
    \item[(i)]  
    If $Q_{\xi}\subset\Fr(M,g)$ is the $H$-subbundle reduction determined by $\xi$, we recall that for each $x\in M$ there is $u_x\in (Q_{\xi})_x$ such that $\xi_x=u_x.\xi_{\circ}$ and $g {_x}=u_x.g_{\circ}$. Hence, it suffices to prove the claim pointwise, for $C,D\in \fm$ and $\xi_{\circ}\in V$, where $V$ is a $\SO(n)$-submodule of $\oplus\cT^{p,q}(\mathbb{R}^n)$.

Note that the orbit $\SO(n).\xi_{\circ}$  is a submanifold of $V$, which is naturally identified with the normal homogeneous Riemmanian manifold $\SO(n)/H$, with the $\SO(n)$-invariant Riemannian metric $\langle A,B\rangle = -\tr(AB)$. As in the proof of Lemma \ref{lem: ker_diamond}, for each $\xi\in \SO(n).\xi_{\circ}$, we can think of $\widetilde{A}(\xi) := A\diamond\xi$ as a tangent vector to $\SO(n).\xi_{\circ}$ at $\xi$. Since $\SO(n)/H$ is reductive, there is a one-to-one correspondence between $\SO(n)$-invariant Riemannian metrics on $\SO(n)/H$ and $\Ad_{\SO(n)}(H)$-invariant inner products on $\fm\cong T_{\xi_{\circ}}(\SO(n)\cdot{\xi_{\circ}})\subset V$. Now, thinking of $\widetilde{A}(\xi)$ as a tensor in $V$, the metric $g_{\circ}$ induces an $\Ad_{\SO(n)}(H)$-invariant inner product on $\fm$: 
\[
\langle\langle A,B\rangle\rangle := \langle A\diamond\xi_{\circ},B\diamond \xi_{\circ}\rangle,\quad\forall A,B\in\fm.
\] Indeed, $\ker(\cdot{}\diamond\xi_{\circ})|_{\fm}$ is injective, by Lemma \ref{lem: ker_diamond}, and for all $h\in H$, using the distributivity $h.(C\diamond\eta) = (h.C)\diamond (h.\eta)$ found in the proof of Lemma \ref{lem: properties_diamond}--\ref{item: diamond_Jacobi}, together with the facts that $\SO(n)$ acts by isometries of $g_o$ and $H=\Stab_{\SO(n)}(\xi_{\circ})$, we have
\begin{align*}
    \langle\langle \Ad_{\SO(n)}(h)A,\Ad_{\SO(n)}(h)B\rangle\rangle 
    &= \langle (h.A)\diamond\xi_{\circ},(h.B)\diamond\xi_{\circ}\rangle 
    %\\&
    = \langle h.(A\diamond h^{-1} {.}\xi_{\circ}),h.(B\diamond h^{-1}.\xi_{\circ})\rangle\\
    &= \langle A\diamond h^{-1} {.}\xi_{\circ},B\diamond h^{-1}.\xi_{\circ}\rangle
    %\\&
    = \langle A\diamond \xi_{\circ},B\diamond \xi_{\circ}\rangle\\
    &= \langle\langle A,B\rangle\rangle.
\end{align*}
In particular, the $\langle\cdot,\cdot\rangle$-orthogonal decomposition $\fm=\fm_1\oplus\ldots\oplus\fm_k$ into (non-equivalent)  {isotypic} $\Ad_{\SO(n)}(H)$-submodules is also orthogonal with respect to $\langle\langle\cdot{},\cdot{}\rangle\rangle$, and $\langle\langle\cdot{},\cdot{}\rangle\rangle$ restricts to give $\Ad_{\SO(n)}(H)$-invariant inner products on each of the (non-equivalent) {isotypic} isotropy summands $\fm_i\subset\fm$. On the other hand, the restriction of the bi-invariant metric $\langle\cdot,\cdot\rangle$ to $\fm_i$ is also canonically an $\Ad_{\SO(n)}(H)$-invariant inner product. Hence, by Schur's lemma,  for each $i=1,\ldots,k$ there is $\lambda_i\in\mathbb{R}_{+}$ such that (see e.g. \cite{Besse2008}*{Theorem 7.44})
\[
\langle\langle\cdot{},\cdot{}\rangle\rangle|_{\fm_i\times\fm_i} = \lambda_i\langle\cdot,\cdot\rangle|_{\fm_i\times\fm_i}.
\] 
In conclusion, given $A,B\in\fm$ and letting $A_i:=\pi_{\fm_i}(A)$, $B_i:=\pi_{\fm_i}(B)$, then
 \begin{align*}
     \langle A\diamond\xi_\circ ,B\diamond\xi_\circ \rangle = \langle\langle A,B\rangle\rangle = \sum_{i=1}^k \langle\langle A_i,B_i\rangle\rangle =\sum_{i=1}^k\lambda_i\langle A_i,B_i\rangle,
 \end{align*} 
 which proves the desired identity \eqref{eq: inner_product_decomp}.
    
    \item[(ii)] Using Lemma \ref{lem: properties_diamond}-\ref{item: diamond_Jacobi},\ref{item: diamond_orth},\ref{item: diamond_adjoint} we get 
 \begin{align}
   \langle C\diamond(C\diamond\xi),D\diamond\xi\rangle &= -\langle C\diamond\xi,C\diamond(D\diamond\xi)\rangle\nonumber
   = -\langle C\diamond\xi,D\diamond(C\diamond\xi) - [C,D]\diamond\xi\rangle\nonumber\\
   &= \langle [C,D]\diamond\xi,C\diamond\xi\rangle.\label{eq: orth_first_part}
 \end{align} 
Then, combining equation \eqref{eq: orth_first_part} with the identity \eqref{eq: inner_product_decomp} of the first part \ref{item: inner_product}, we immediately get \eqref{eq: almost_orthogonality}.
    
    \item[(iii)] Under either one of the hypotheses, using the equation \eqref{eq: almost_orthogonality} of the second part \ref{item: almost_orthogonality} one has $\langle C\diamond(C\diamond\xi),D\diamond\xi\rangle = \mathrm{const.}\langle[C,D],C\rangle$, and by Ad-invariance, $\langle [C,D],C\rangle = -\langle D,[C,C]\rangle = 0$, one gets the desired orthogonality relation \eqref{eq: orthogonality}.
\qedhere
\end{itemize}
\end{proof} 
\begin{remark}
If, instead of either one of the assumptions in Lemma \ref{lem: orthogonality}-\ref{item: orthogonality}, one assumes that $\SO(n)/H$ is a \emph{symmetric space}, i.e., that moreover $[\fm,\fm]\subset\fh$, then we also obtain the orthogonality relation \eqref{eq: orthogonality} immediately from equation \eqref{eq: almost_orthogonality} in item \ref{item: almost_orthogonality}. But we claim that if $\SO(n)/H$ is a symmetric space, then $\SO(n)/H$ must be in fact an \emph{irreducible} symmetric space, i.e. $\fm$ is $H$-irreducible, and thus one is actually in the case of the assumption $\lambda_1=\ldots=\lambda_k$ in Lemma \ref{lem: orthogonality}-\ref{item: orthogonality}. To see this, start noting that by the long exact sequence of homotopy groups it follows that the space $\SO(n)/H$ is simply connected, since $\SO(n)$ is simply connected ($n>2$) and we assume $H$ to be connected. Now, if $n\neq 4$ then $\SO(n)$ is a simple Lie group, and it is well-known that any simply connected Riemannian symmetric space $G/H$ with $G$ simple must be irreducible, thus in this case the assumption forces indeed $\SO(n)/H$ to be an irreducible symmetric space. If $n=4$, one may use the well-known $2:1$ epimorphism $\SU(2)\times\SU(2)\to\SO(4)$ to understand the possible closed and connected subgroups $H\subset\SO(4)$, and confront this with the fact that if $\SO(4)/H$ is a (simply connected) symmetric space then it can be written as a finite product of irreducible simply connected Riemannian symmetric spaces of compact type, and one can check the known classification tables of the later spaces (see e.g. \cite[Chapter X]{helgason1978differential}) to conclude that $\SO(4)/H$ must be one of the following irreducible symmetric spaces: $\SO(4)/\U(2)\cong\mathbb{S}^2$, $\SO(4)/\SO(3)\cong\mathbb{S}^3$, or $\SO(4)/(\SO(2)\times\SO(2))$.
\end{remark}
We now illustrate Lemma \ref{lem: orthogonality} in several cases of interest. We start with the three main examples that we have been considering so far, where $H=\U(m), \rm G_2$ or $\rm Spin(7)$, in which the $H$-module $\fm$ is irreducible:
\begin{example}
\label{ex: inner_prod_J}
    When $H=\U(m)=\Stab_{\SO(2m)}(J_{\circ})$, as in Example \ref{ex: almost_complex_structure}, the complement  $\fm=\fu(m)^{\perp}=\{A\in\mathfrak{so}(n):AJ_{\circ} = -J_{\circ}A \}$ is irreducible, and for any compatible $\U(m)$-structure $\xi=J$ on $(M^{2m},g)$, using Lemma \ref{lem: properties_diamond}--\ref{item: diamond_bracket} we can compute, for all $A,B\in\Omega_{\fm}^2(M)$,
\begin{align}
    \label{eq: inner_prod_J}
    \langle A\diamond J,B\diamond J\rangle 
    &= \langle [A,J],[B,J]\rangle\nonumber
    = \langle 2AJ, (-2)JB\rangle
    = 4\tr(AJJB) \\
    &= 4\langle A,B\rangle.
\end{align}
\end{example}
\begin{example}
\label{ex: inner_prod_G2}
    When $H=\rm G_2\subset\SO(7)$, as in Example \ref{ex: G2_structures}, the complement $\fm=\Lambda^2_7\subset\fso(7)$ is irreducible, and if $\varphi$ is a $\rm G_2$-structure on $M^7$ then  \cite{Karigiannis2007}*{\textsection 2.2}
    \begin{equation}
    \label{eq: inner_prod_G2}
        \langle A\diamond\varphi, B\diamond\varphi\rangle = 6\langle A,B\rangle,
        \quad\forall 
        A,B\in\Omega_{\fm}^2(M).
    \end{equation}
\end{example}
\begin{example}
\label{ex: inner_prod_Spin7}
    When $H=\rm{Spin}(7)\subset\SO(8)$, as in Example \ref{ex: Spin7_structures}, the complement $\fm=\Lambda^2_7\subset\fso(8)$ is irreducible, and if $\Phi$ is a $\rm Spin(7)$-structure on $M^8$ then \cite{karigiannis-spin7}*{Proposition 2.5}:
    \begin{equation}
    \label{eq: inner_prod_Spin7}
        \langle A\diamond\Phi, B\diamond\Phi\rangle = 16\langle A,B\rangle,
        \quad\forall
        A,B\in\Omega_{\fm}^2(M).
    \end{equation} 
\end{example}

As a first instance of the \emph{reducible} complement case, let us derive the corresponding inner-product relation for the infinitesimal action on frame fields.
\begin{example}
\label{ex: trivial_subgroup}
    When $H=\{1\}\subset\SO(n)$ is the trivial subgroup, we have $\fh=\{0\}$ and the $\{1\}$-module $\fm=\fso(n)$ splits completely into the trivial one-dimensional representations generated by each element of the standard basis of $\fso(n)$. In this case,  
    \[
    \lambda_1=\ldots=\lambda_{\mathrm{dim}\text{ }\fso(n)} = 1
    \] 
    in Lemma \ref{lem: orthogonality}. Indeed, $\{1\}=\Stab(\xi_{\circ})$, where $\xi_{\circ}:=(e_1,\ldots,e_n)$ is the canonical basis of $\mathbb{R}^n$, so that a compatible $\{1\}$-structure on an oriented Riemannian manifold $(M^n,g)$ is simply a global oriented orthonormal frame $\xi=(\xi_1,\ldots,\xi_n)$ of $TM$, and then by definition $A\diamond\xi = (-A\xi_1,\ldots,-A\xi_n)$ for every $A\in\Omega^2_\fm(M)=\Omega^2(M)$, so 
    \begin{equation}
    \langle A\diamond\xi,B\diamond\xi\rangle = \sum\limits_{j=1}^n \langle A\xi_j,B\xi_j\rangle = \langle A,B\rangle,\quad\forall A,B\in\Omega^2_\fm(M)=\Omega^2(M).
    \end{equation}
\end{example} 

Last but not least, the following example for $H=\SU(m)\subset\SO(2m)$, $m\geqslant 2$, to the best of our knowledge, is new. In this case, it is well-known that $\fm$ is reducible and splits into two  non-trivial irreducible submodules $\fm_1$ and $\fm_2$, so we find it instructive to compute the corresponding $\lambda_1(m)$ and $\lambda_2(m)$ predicted by Lemma \ref{lem: orthogonality}. We will see in particular that $\lambda_1(m)\neq\lambda_2(m)$, for all $m\geqslant 2$, and moreover that the orthogonality relation \eqref{eq: orthogonality} does not always hold for arbitrary elements in $\fm$.
\begin{example}
\label{ex: inner_prod_SU}
    We now consider $H=\SU(m)\subset\SO(2m)$, $m\geqslant 2$. We adopt the following description of the group $\SU(m)$. Let $(x^1,\ldots,x^m,y^1,\ldots,y^m)$ be the standard coordinates on $\mathbb{R}^{2m}=\mathbb{R}^m\oplus\mathbb{R}^m$, so that the canonical complex structure $J_{\circ}\in\mathrm{End}(\mathbb{R}^{2m})$, the Euclidean metric $g_{\circ}$ and the fundamental $2$-form $\omega_{\circ}$ on $\mathbb{R}^{2m}$ (cf. Example \ref{ex: almost_complex_structure}) are given by:
    $$\begin{array}{c}
        J_{\circ}\frac{\partial}{\partial x^p} = \frac{\partial}{\partial y^p},\quad J_{\circ}\frac{\partial}{\partial y^p}=-\frac{\partial}{\partial x^p},\\ 
        g_{\circ} = \sum\limits_{p=1}^m \left(dx^p\otimes dx^p + dy^p\otimes dy^p\right),\quad\text{and}\quad
        \omega_{\circ} = \sum\limits_{p=1}^m dx^p\wedge dy^p.
    \end{array}$$ 
    Accordingly, let $z^p=x^p+iy^p$ be complex coordinates in $\mathbb{C}^m\cong\mathbb{R}^{2m}$. Then, $\rm SU(m)$ is the subgroup of $\GL(2m,\mathbb{R})$ preserving $g_{\circ}$, $J_{\circ}$ (and $\omega_{\circ}$) and the \emph{complex determinant}, or \emph{complex volume form}, $\Upsilon_{\circ}$, given by
    \begin{equation}
        \Upsilon_{\circ}:=dz^1\wedge\ldots\wedge dz^m\in\Lambda_{\mathbb{C}}^m(\mathbb{C}^m)^{\ast}.
    \end{equation} Thus, we may consider the model structure $\xi_{\circ}:=(J_{\circ},\Upsilon_{\circ})$ and write $\SU(m) = \Stab_{\SO(2m)}(\xi_{\circ})$.
    
    We now note that $\fm:=\fsu(m)^{\perp}\subset\fso(2m)$ is a reducible $H$-module. Indeed, we have the orthogonal $H$-module decompositions 
    \[
    \fso(2m)=\fu(m)\oplus\fu(m)^{\perp}=\fsu(m)\oplus\langle J_{\circ}\rangle\oplus\fu(m)^{\perp},
    \]
    and so we have an orthogonal decomposition $\fm =\fm_1\oplus\fm_2$ into the irreducible submodules $\fm_1$ and $\fm_2$ given by
    \[
    \fm_1 :=\langle J_{\circ}\rangle\subset\fso(2m)\quad\text{and}\quad\fm_2:=\fu(m)^{\perp}=\{A\in\fso(2m):AJ_{\circ} = - J_{\circ}A\}.
    \] 
    By Lemma \ref{lem: orthogonality}, there are positive constants $\lambda_1,\lambda_2\in\mathbb{R}_{+}$ such that, for all $A,B\in\fm$, 
\begin{equation}
\label{eq: inner_prod_SU_general}
    \langle A\diamond\xi,B\diamond\xi\rangle 
    = \lambda_1\langle \pi_{\fm_1}(A),\pi_{\fm_1}(B)\rangle + \lambda_2\langle \pi_{\fm_2}(A),\pi_{\fm_2}(B)\rangle.
\end{equation} 
    In particular, $|J_{\circ}\diamond\xi_{\circ}|^2 =\lambda_1|J_{\circ}|^2$, and since $|J_{\circ}|^2=2m$, we can compute $\lambda_1$ from the tensor norm under the action $J_{\circ}\diamond\xi_{\circ}=(J_{\circ}\diamond J_{\circ},J_{\circ}\diamond\Upsilon_{\circ})$. Now, by Lemma \ref{lem: properties_diamond}--\ref{item: diamond_bracket}, we know that $J_{\circ}\diamond J_{\circ} = [J_{\circ},J_{\circ}]=0$, and naturally extending the definition of the $\diamond$-operator [cf. \eqref{eq: diamond_operator}] to complexified forms, we have
\[
J_{\circ}\diamond dz^p =  J_{\circ}\diamond dx^p + i J_{\circ}\diamond dy^p = -dy^p + i dx^p = idz^p,
\] 
hence
\begin{align*}
    J_{\circ}\diamond\Upsilon_{\circ} 
    &= (J_{\circ}\diamond(dz^1))\wedge dz^2\wedge\ldots\wedge dz^m + dz^1\wedge (J_{\circ}\diamond(dz^2))\wedge\ldots\wedge dz^m\\
    &\quad + \ldots + dz^1\wedge\ldots\wedge dz^{m-1}\wedge (J_{\circ}\diamond(dz^m))\\
    &=m i\Upsilon_{\circ}.
\end{align*} 
Noting that $|\Upsilon_{\circ}|^2=2^m$, the above implies that
\[
|J_{\circ}\diamond\xi_{\circ}|^2 = |J_{\circ}\diamond\Upsilon_{\circ}|^2 = m^2 2^m = m 2^{m-1}|J_{\circ}|^2;
\] hence, $\lambda_1 \equiv \lambda_1(m)=m 2^{m-1}$. 

In order to compute $\lambda_2$, observe that for any $A\in\fm_2$ we have $|A\diamond J_{\circ}|^2 = 4|A|^2$, as in Example \ref{ex: inner_prod_J}, so using \eqref{eq: inner_prod_SU_general} we get
\begin{equation}
\label{eq: lambda_2}
    \lambda_2|A|^2=|A\diamond\xi|^2 = 4|A|^2+|A\diamond\Upsilon_{\circ}|^2.
\end{equation} Our task now is to compute $|A\diamond\Upsilon_{\circ}|^2$ for some convenient choice of $A\in\fm_2$. When the complex dimension $m$ is even, it is easy to verify that $A\in\mathfrak{gl}(2m,\mathbb{R})$ defined by
\[
A\frac{\partial}{\partial x^p} = (-1)^p\frac{\partial}{\partial y^{m-p+1}}\quad\text{and}\quad A\frac{\partial}{\partial y^p} = (-1)^p\frac{\partial}{\partial x^{m-p+1}},
\] is skew-symmetric and anti-commutes with $J_{\circ}$, i.e. $A\in\fm_2$, and also $A^2=-1$, so that $|A|^2=2m$. Moreover,
\[
A\diamond dz^p = (-1)^{p+1}(dy^{m-p+1}+idx^{m-p+1}) = (-1)^{p+1}id\overline{z}^{m-p+1}.
\] In particular,
\begin{align*}
    A\diamond\Upsilon_{\circ} &= (id\overline{z}^m)\wedge dz^2\wedge dz^3\wedge\ldots\wedge dz^m - dz^1\wedge (id\overline{z}^{m-1})\wedge dz^3\wedge\ldots\wedge dz^m\\
    &\quad +\ldots+ (-1)idz^1\wedge dz^2\wedge\ldots\wedge dz^{m-1}\wedge d\overline{z}^1.
\end{align*} 
    Since the terms on the right-hand side are pairwise orthogonal, and $idz^{p}\wedge d\overline{z}^p = 2dx^p\wedge dy^p$, 
\[
|A\diamond\Upsilon_{\circ}|^2 = m.4.2^{m-2} = 2^{m-1}|A|^2.
\] 
Combining with  \eqref{eq: lambda_2}, we conclude that $\lambda_2 \equiv \lambda_2(m) = 4+2^{m-1}$ when $m$ is even. 

When $m\geqslant 2$ is odd, i.e., $m=2k+1$ for some $k\geqslant 1$, the computation is analogous for $A$ defined e.g. by
\begin{align*}
    A\frac{\partial}{\partial x^p} 
    =\begin{cases}
        \frac{\partial}{\partial x^{m-p+1}},
        &1\leqslant p\leqslant k,\\
        0,
        &p = k+1,\\
        -\frac{\partial}{\partial x^{m-p+1}},
        &k+1< p\leqslant m
    \end{cases}
    \qandq
    A\frac{\partial}{\partial y^p} 
    = \begin{cases}
        -\frac{\partial}{\partial y^{m-p+1}},
        &1\leqslant p\leqslant k,\\
        0,
        &p = k+1,\\
        \frac{\partial}{\partial y^{m-p+1}},
        &k+1< p\leqslant m
    \end{cases}.
\end{align*} 
Then it follows that $A\in\fm_2$, $|A|^2=2(m-1)$, and
\begin{align*}
    A\diamond dz^p = \begin{cases}
         {-d\overline{z}^{m-p+1}},
        &1\leqslant p\leqslant k,\\
        0,
        &p = k+1,\\
        d\overline{z}^{m-p+1},
        &k+1< p\leqslant m,
    \end{cases}
\end{align*} so that
\[
|A\diamond\Upsilon_{\circ}|^2 = (m-1)2^m = 2^{m-1}|A|^2.
\] 
Together with \eqref{eq: lambda_2}, this also gives $\lambda_2(m) = 4+2^{m-1}$, which therefore holds for any $m\geqslant 2$. It is easy to conclude that $\lambda_1(m)\neq\lambda_2(m)$ for all $m\geqslant 2$; indeed $\lambda_1(2)<\lambda_2(2)$ and $\lambda_1(m)>\lambda_2(m)$ for  $m\geqslant 3$.

As an application, we show that the orthogonality relation \eqref{eq: orthogonality} of Lemma \ref{lem: orthogonality} does not hold for all $C,D\in\fm$ in this reducible case. Let $C=C_1+C_2\in\fm_1\oplus\fm_2$ and $D=D_1+D_2\in\fm_1\oplus\fm_2$ be arbitrary. Then, by equation \eqref{eq: almost_orthogonality} of Lemma \ref{lem: orthogonality},
    \[
    \langle C\diamond (C\diamond\xi), D\diamond\xi\rangle = \lambda_1\langle[C,D],C_1\rangle + \lambda_2\langle[C,D],C_2\rangle.
        \] Using the bi-invariance of $\langle\cdot{},\cdot{}\rangle=-\tr(\cdot{}\cdot{})$, and the fact that we can write $C_1=aJ_{\circ}$ and $D_1=bJ_{\circ}$, for some constants $a,b\in\mathbb{R}$, we have  {(using skew-symmetry of the adjoint action)}
    \begin{align*}
    \langle[C,D],C_1\rangle 
    & {=\langle a[J_{\circ},D_2] + b[C_2,J_{\circ}] + [C_2,D_2], a J_{\circ}\rangle }\\
    &=a\langle [C_2,D_2], J_{\circ}\rangle \\
    &= -a\langle [J_{\circ},D_2], C_2\rangle=-\langle[C,D],C_2\rangle,
    \end{align*}
    and thus
    \[
    \langle C\diamond (C\diamond\xi), D\diamond\xi\rangle = a(\lambda_1-\lambda_2)\langle [C_2,D_2], J_{\circ}\rangle.
    \]
    In particular, the orthogonality \eqref{eq: orthogonality} holds if $C=C_2\in\fm_2$ (i.e. $a=0$ above). In general, since $\lambda_1\neq\lambda_2$ (for any $m\geqslant 2$), it follows that \eqref{eq: orthogonality} is true for all $C,D\in\Omega_{\fm}^2(M)$ if and only if $[\fm_2,\fm_2]\subset\langle J_{\circ}\rangle^{\perp}$. But observe that in general $[\fm_2,\fm_2]\subset\fu(m)=\fsu(m)\oplus\langle J_{\circ}\rangle$, and in complex dimension $m\geqslant 2$ there are examples of elements $C_2,D_2\in\fm_2$ such that $[C_2,D_2]\in\langle J_{\circ}\rangle$; e.g. for $m=2$ it is easy to check that
    \[
C_2 :=\left(\begin{array}{@{}c|c@{}}
 \bigzero
  & \begin{matrix}
  0 & 1 \\
  -1 & 0
  \end{matrix} \\
\hline
  \begin{matrix}
  0 & 1 \\
  -1 & 0
  \end{matrix} &
  \bigzero
\end{array}\right)\quad\text{and}\quad
D_2 :=\left(\begin{array}{@{}c|c@{}}
  \begin{matrix}
  0 & -1 \\
  1 & 0
  \end{matrix}
  & \bigzero \\
\hline
  \bigzero &
  \begin{matrix}
  0 & 1 \\
  -1 & 0
  \end{matrix}
\end{array}\right),
\] are such that $C_2,D_2\in\fm_2$, $C_2D_2= J_{\circ} = - D_2C_2$ and thus $[C_2,D_2]=2J_{\circ}$. (In fact, whenever $m=2k\geqslant 2$ is even, one can always take $C_2$ and $D_2$ to be the other two almost complex structures of the standard hyperk\"ahler triple in $\mathbb{R}^{4k}$.)
\end{example}

The first part of the following result was also proved in \cite{Dwivedi-Loubeau-SaEarp2021}*{Lemma 2.5}, by a different approach. The second part is new, at this level of generality, and it will play a pivotal role in the analytic study of flows of $H$-structures, particularly in the derivation of `Shi-type' estimates.
\begin{lemma}
\label{lem: torsion_equiv_nablaxi}
    Let $(M^n,g)$ be an oriented Riemannian $n$-manifold admitting a compatible $H$-structure  {$\xi$, where $H=\mathrm{Stab}_{\SO(n)}(\xi_{\circ})$ and $\xi$ is a geometric structure modelled on the (multi-)tensor $\xi_{\circ}$.} Then
\begin{align}
\label{eq: nabla of xi-GENERAL}
    \nabla_X\xi 
    &= T_X \diamond \xi, 
    \quad\forall\, X \in \sX(M),
\end{align} 
    where $T\in\Omega^1(M,{\fm}_{\xi})$ denotes the torsion of the $H$-structure $\xi$. 
    In particular, there are constants $c,\tilde{c}>0$, depending only on $(M,g)$ and $H$, such that
\begin{equation}
\label{ineq: T_equiv_nablaxi}
    \tilde{c}|T|^2\leqslant |\nabla\xi|^2\leqslant c|T|^2.
\end{equation} 
    If furthermore there is $c>0$ such that $\langle A\diamond\xi,B\diamond\xi\rangle = c\langle A,B\rangle$, for all $A,B\in\Omega_\fm^2(M)$, i.e. if $c:=\lambda_1=\ldots=\lambda_k$ in Lemma \ref{lem: orthogonality} (e.g. if $\fm$ is an irreducible $H$-module), then in fact
\begin{equation}
\label{eq: T_equiv_nablaxi}
    |\nabla\xi|^2=c|T|^2.   
\end{equation}
\end{lemma}
\begin{proof}
    Equation \eqref{eq: nabla of xi-GENERAL} follows almost immediately from the fundamental relation \eqref{eq: def_T} for the torsion, we just need to unravel some definitions. Without loss of generality, we can assume $\xi$ is a single tensor (if instead $\xi=(\xi_1,\ldots,\xi_k)$, then one simply works componentwise with each tensor $\xi_i$). Since both $\nabla^H$ and $\nabla$ are metric connections, and under the musical isomorphisms any $(r,s)$-tensor field on $(M,g)$ is metric-equivalent to a $(0,r+s)$-tensor field, it suffices to consider $\xi$ as a covariant $(0,q)$-tensor field. Now, $\nabla^H$ is an $H$-connection and $\xi$ is stabilised by $H$, so $\nabla^H\xi = 0$.  {It follows that}, 
\[
X\xi(Y_1,\ldots,Y_q) 
    = \sum\limits_{j=1}^q\xi(Y_1,\ldots,\nabla_X^H Y_j,\ldots,Y_q),
\quad\forall
    X,Y_1,\ldots,Y_q\in\sX(M),
\]
and we know from \eqref{eq: def_T} that $\nabla_X^H = \nabla_X + T_X$, so
\begin{align*}
    \nabla_X\xi(Y_1,\ldots,Y_q) 
    &= X\xi(Y_1,\ldots,Y_q) - \sum\limits_{j=1}^q \xi(Y_1,\ldots,\nabla_XY_j,\ldots,Y_q)\\
    &= \sum\limits_{i=1}^q\xi(Y_1,\ldots,T_XY_j,\ldots,Y_q) \\
    &= (T_X\diamond\xi)(Y_1,\ldots,Y_q),
\end{align*} 
where the last equality follows the definition of the diamond operator on covariant tensors, cf. \eqref{eq:diamond_coordinates}.

Finally, when $H=\mathrm{Stab}_{\SO(n)}(\xi_{\circ})$, it follows from Lemma \ref{lem: ker_diamond} that the linear operator $(\cdot{}\diamond\xi)|_{\Omega_{\fm}^2}$ is injective, and since $T_X\in\Omega_{\fm}^2$ the inequality \eqref{ineq: T_equiv_nablaxi} immediately follows from \eqref{eq: nabla of xi-GENERAL}. In fact, more explicitly, combining \eqref{eq: nabla of xi-GENERAL} with Lemma \ref{lem: orthogonality}, it follows that if $\fm=\fm_1\oplus\ldots\fm_k$ is an orthogonal decomposition into (non-equivalent) {isotypic} $H$-submodules, there are positive constants $\lambda_1,\ldots,\lambda_k$ such that
\[
|\nabla\xi|^2 = \sum\limits_{i=1}\lambda_i|\pi_{\fm_i}(T)|^2. 
\] Thus, if we let $\lambda_{\max}:=\max\limits_{1\leqslant i\leqslant k}\lambda_i$ and $\lambda_{\min}:=\min\limits_{1\leqslant i\leqslant k}\lambda_i$, then $\lambda_{\min}|T|^2\leqslant |\nabla\xi|^2\leqslant \lambda_{\max}|T|^2$. In particular, if furthermore $\lambda_1=\ldots=\lambda_k$ (e.g. if $\fm$ is $\Ad_{\SO(n)}(H)$-irreducible) then we get $|\nabla\xi|^2=\lambda_{\max} |T|^2$, as we wanted.
\end{proof}
    
\begin{example}
\label{ex: nabla_J}
    When $H=\U(m)\subset \SO(2m)$, as in Example \ref{ex: almost_complex_structure},   and the geometric structure is an almost complex structure $\xi=J$, the identity \eqref{eq: nabla of xi-GENERAL} of Lemma \ref{lem: torsion_equiv_nablaxi} and Lemma \ref{lem: properties_diamond}--\ref{item: diamond_bracket} give
    \[
        \nabla_X J = (T_X\diamond J) = -[T_X,J] = 2JT_X, \quad \forall X \in \sX(M),
    \] since $T_X\in\Omega_{\mathfrak{u}(m)^{\perp}}^2\simeq\{A\in\mathfrak{so}(M):AJ=-JA\}$. Thus, we have
    \begin{equation}\label{eq: torsion_almost_complex}
    T_X=-\frac{1}{2}J\nabla_X J, \quad \forall X \in \sX(M).  
\end{equation} 
    In particular, 
\begin{equation}
\label{eq: norm_nabla_J_T}
    |\nabla J|^2 = 4|T|^2.
\end{equation} 
    Moreover, we know from Example \ref{ex: diamond_J} that $\cdot{}\diamond J$ maps $\Omega_{\fu(m)^{\perp}}$ into itself, so that $\nabla_X J\in\Omega_{\fu(m)^{\perp}}^2$, for all $X\in\sX(M)$. Alternatively, the identity \eqref{eq: norm_nabla_J_T} follows by combining \eqref{eq: nabla of xi-GENERAL} with the explicit form of the inner product \eqref{eq: inner_prod_J} derived in Example \ref{ex: inner_prod_J}.
\end{example}

\begin{example}
\label{ex: Intrinsic_torsion_G2}
    When $H=\rG_2\subset\SO(7)$, as in Example \ref{ex: G2_structures}, given a $\rm G_2$-structure $\varphi$ on $M^7$, then in addition to the irreducible $\rG_2$-decomposition of $\Omega^2$, the space of $3$-forms is decomposed into irreducible $\rG_2$-modules
   $\Omega^3=\Omega_1^3\oplus\Omega_7^3\oplus\Omega_{27}^3$. By Lemma \ref{lem: ker_diamond}, we know that $\ker(\cdot{}\diamond\varphi)=\Omega_{\fg_2}^2$, and since $\cdot{}\diamond\varphi$ is $\rm G_2$-equivariant, thus respecting the $\rm G_2$-module decompositions,  it follows by dimension counting that $\cdot{}\diamond\varphi$ maps $\Omega_{\fm}^2$ isomorphically into $\Omega_7^3$. Since $T_X\in\Omega_\fm^2$, we see from equation  \eqref{eq: nabla of xi-GENERAL} of Lemma \ref{lem: torsion_equiv_nablaxi} that $\nabla_X\varphi\in\Omega_7^3$ for every $X \in \sX(M)$, thus recovering a well-known fact dating back to the work of Fernández and Gray \cite{Fernandez1982} (see also \cite[Lemma 2.14]{Karigiannis2009}). 
   
   Quantitatively, it follows from the inner product relation \eqref{eq: inner_prod_G2} of Example \ref{ex: inner_prod_G2}  that
\begin{equation}
\label{eq: norm_nabla_varphi_T}
    |\nabla \varphi|^2 
    = 6|T|^2.
\end{equation}
    Because of the description of $\Omega_{\fm}^2(M)$  as in \eqref{eq: decomp_lambda_g2}, it is common in $\rm G_2$-geometry to identify the intrinsic torsion $T\in\Omega^1(M,\Lambda_\fm^2)$ with the endomorphism $\mathcal{T}_{lm}$ defined by  $T_{l;ij}=:-\frac{1}{3}\mathcal{T}_{lm}\varphi_{mij}$. Applying a well-known self-contraction identity for the $3$-form $\varphi$ \cite[Lemma A.8]{Karigiannis2007}, one has $\langle X\lrcorner\varphi,X\lrcorner\varphi\rangle = 6|X|^2$, and thus
\begin{equation}
\label{eq: relation_torsions_norms_G2}
    |T|^2 = \frac{1}{9}\sum\limits_{i=1}^7\langle \cT(e_i)\lrcorner\varphi,\cT(e_i)\lrcorner\varphi\rangle = \frac{2}{3}|\cT|^2.
\end{equation} 
    
    Moreover, the dual $4$-form $\psi:=\ast\varphi$ is also stabilised by $\rm G_2$, and it follows from contraction identities between $\varphi$ and $\psi$ \cite[Appendix A.3]{Karigiannis2007} that the equations $\nabla\varphi = T\diamond\varphi$ and $\nabla\psi = T\diamond\psi$ that we obtain from \eqref{eq: nabla of xi-GENERAL} applied to $\xi=\varphi$ and $\xi=\psi$, are equivalent respectively to:
\begin{align}
	\nabla_p\varphi_{ijk} 
	&= \mathcal{T}_{pm}\psi_{mijk}
	\label{eq: nabla_varphi}\\
	\nabla_p \psi_{mijk} 
	&= -\mathcal{T}_{pm}\varphi_{ijk} + \mathcal{T}_{pi}\varphi_{mjk} - \mathcal{T}_{pj}\varphi_{mik} + \mathcal{T}_{pk}\varphi_{mij}.
	\label{eq: nabla_psi}
\end{align} 
    Finally, we may invert \eqref{eq: nabla_varphi} to express 
\begin{equation}
\label{eq: g2_torsion}
    \mathcal{T}_{pq} 
    = \frac{1}{24}\nabla_p\varphi_{ijk}\psi_{qijk}.
\end{equation} 
\end{example}

\begin{example}
    When $H=\S7\subset\SO(8)$, as in Example \ref{ex: Spin7_structures}, a $\rm Spin(7)$-structure $\Phi$ on $M^8$ induces a decomposition on the space of $4$-forms into irreducible $\S7$-submodules $\Omega^4=\Omega^4_1\oplus\Omega^4_7\oplus\Omega^4_{27}\oplus\Omega^4_{35}$. Then, arguing as in the previous example, $\cdot\diamond\Phi$ maps $\Omega_\fm^2$ isomorphically into $\Omega_7^4$, and \eqref{eq: nabla of xi-GENERAL} yields $\nabla_X \Phi = T_X\diamond\Phi$, implying that $\nabla_X\Phi\in \Omega^4_7$ for every $X\in\sX(M)$ \cite{fernandez-spin7}. Moreover, from equation \eqref{eq: inner_prod_Spin7} of Example \ref{ex: inner_prod_Spin7} we obtain
\begin{equation}
\label{eq: norm_nabla_Phi_T}
        |\nabla \Phi|^2 = 16|T|^2.
\end{equation}
\end{example}

Henceforth, we shall denote by $\Delta$ the negative definite rough Laplacian, i.e. $\Delta := -\nabla^{\ast}\nabla$, so that at the center of normal coordinates $\Delta=\nabla_k\nabla_k$. 
\begin{lemma}[cf. {\cite[Lemma 3.14]{Gonzalez-Davila2009}}]\label{lem: basic_estimate_harmonic}
Let $(M^n,g)$ be an oriented Riemannian $n$-manifold admitting a compatible  {(geometric) $H$-structure $\xi$ with torsion $T$; we assume that $H=\mathrm{Stab}_{\SO(n)}(\xi_{\circ})$ and $\xi$ is modelled on $\xi_{\circ}$.} Then
\begin{equation}\label{eq: Laplacian_of_xi}
    \Delta\xi = \Div T\diamond\xi + T_k\diamond (T_k\diamond \xi),
\end{equation} 
    where $(\Div T)_{ij}:=\nabla_k T_{k;ij}\in\Omega_{\fm}^2(M)$. 

    In particular, there is a constant $c>0$, depending only on $(M,g)$ and $H$, such that if $\Div T = 0$ then
\begin{equation}\label{ineq: basic_estimate_harmonic}
    |\Delta\xi|\leqslant c|\nabla\xi|^2.
\end{equation} 
    If furthermore there is $c>0$ such that $\langle A\diamond\xi,B\diamond\xi\rangle = c\langle A,B\rangle$ for all $A,B\in\Omega_\fm^2(M)$, i.e. if $c:=\lambda_1=\ldots=\lambda_k$ in Lemma \ref{lem: orthogonality} (e.g. if $\mathfrak{m}$ is an irreducible $H$-module), then the decomposition \eqref{eq: Laplacian_of_xi} of $\Delta\xi$ is orthogonal.
\end{lemma}
\begin{proof}
    To prove \eqref{eq: Laplacian_of_xi}, we apply the `Leibniz rule'  $\nabla_k(A\diamond \xi)=(\nabla_k A)\diamond\xi+A\diamond\nabla_k\xi$, for $A\in\Gamma(\mathrm{End}(TM))$, to $A=T_k$, together with \eqref{eq: nabla of xi-GENERAL}:
\begin{align*}
    \Delta\xi 
    &= \nabla_k\nabla_k\xi = \nabla_k(T_k\diamond\xi)= \nabla_k T_k\diamond\xi + T_k\diamond\nabla_k\xi\\
    &= \Div T\diamond\xi + T_k\diamond (T_k\diamond \xi).
\end{align*} 
    For the second part, we combine \eqref{ineq: T_equiv_nablaxi}, \eqref{eq: Laplacian_of_xi} and $\Div T = 0$:
\[
|\Delta\xi| = |T_k\diamond(T_k\diamond\xi)|\leqslant c|\xi||T|^2\leqslant c|\nabla\xi|^2.
\] 
    Finally, under the last assumption, it follows from the orthogonality relation \eqref{eq: orthogonality} of Lemma \ref{lem: orthogonality} that $\langle \Div T\diamond\xi , T_k\diamond (T_k\diamond\xi)\rangle = 0$, as claimed. 
\end{proof}
\begin{example}
    When $H=\U(m)=\Stab_{\SO(2m)}(J_{\circ})$, recall from Example \ref{ex: diamond_J} that $\cdot{}\diamond J\colon\mathrm{End}(TM)\to\mathrm{End}(TM)$ is not surjective; it rather maps $\mathrm{End}(TM)=\Omega^0\oplus U\oplus W\oplus\Omega_{\fu(m)}^2\oplus\Omega_{\fu(m)^{\perp}}^2$ onto $W\oplus\Omega_{\fu(m)^{\perp}}^2$. Now, for any $C\in\Omega_{\fu(m)^{\perp}}^2$, using $JC=-CJ$ and Lemma \ref{lem: properties_diamond}--\ref{item: diamond_bracket}, we have:
\begin{equation}
\label{eq: c_c_J}
    E:=C\diamond(C\diamond J) = [C,[C,J]] = 4C^2 J.
\end{equation} 
    Recalling also that $J^{-1}=J^t=-J$, we see immediately that $[E,J]=0$ and $E\in\Omega^2$, i.e. \[E\in\ker(\cdot{}\diamond J)\cap\Omega^2=\Omega_{\fu(m)}^2.\] 
    In particular, there is no $A\in\mathrm{End}(TM)$ such that $E=A\diamond J$. Moreover, if $D\in\Omega_{\fu(m)^{\perp}}^2$, then $D\diamond J\in\Omega_{\fu(m)^{\perp}}^2$ and it is clear that $\langle E,D\diamond J\rangle = 0$.
    
    Using \eqref{eq: torsion_almost_complex} from Example \ref{ex: nabla_J}, together with  \eqref{eq: c_c_J}, and recalling that $\nabla_lJ\in \Omega_{\fu(m)^{\perp}}^2$, we can compute
    \[
    T_l\diamond(T_l\diamond J) = 4 T_l T_l J = J(\nabla_l J)J(\nabla_l J)J = J(\nabla_l J)(\nabla_l J).
    \] 
    Therefore, since $\Div T\diamond J=-[\Div T,J]$, equation \eqref{eq: Laplacian_of_xi} of Lemma \ref{lem: basic_estimate_harmonic} becomes
    \[
    \Delta J = -[\Div {T},J] + J(\nabla_l J)(\nabla_l J).
    \] 
    This yields an alternative  proof of the orthogonality results in \cite[Lemma 3.2]{he2021}.
\end{example}

\begin{example}
\label{ex: G2_Laplacian_of_phi}
    When $H=\rG_2$, as in Examples \ref{ex: G2_structures} and \ref{ex: Intrinsic_torsion_G2},
    a direct computation using equations \eqref{eq: nabla_varphi} and \eqref{eq: nabla_psi} gives
\begin{align}
	(\Delta \varphi)_{ijk} 
	&= \nabla_p\nabla_p\varphi_{ijk} \nonumber
	%\\&
	= (\Div\mathcal{T})_m \psi_{mijk} + \mathcal{T}_{pm}\nabla_p\psi_{mijk}\nonumber\\
	&= (\Div\mathcal{T})_m \psi_{mijk} - |\mathcal{T}|^2\varphi_{ijk} + \mathcal{T}_{pm}\mathcal{T}_{pi}\varphi_{mjk} - \mathcal{T}_{pm}\mathcal{T}_{pj}\varphi_{mik} + \mathcal{T}_{pm}\mathcal{T}_{pk}\varphi_{mij}.
	\label{eq: laplacian_phi}
\end{align} 
    In particular, for \emph{harmonic} $\rG_2$-structures, i.e. when $\Div T = 0$, a \emph{uniform estimate} follows from \eqref{eq: g2_torsion}, $|\varphi|^2 = |\psi|^2 = 7$,  and \eqref{eq: laplacian_phi}: 
\begin{equation}
\label{ineq: estimate_laplacian}
	|\Delta\varphi| \leqslant  
  {\sqrt{\frac{3}{2}}}|\nabla\varphi|^2.
\end{equation} 
	On the other hand, combining the above with Lemma \ref{lem: basic_estimate_harmonic}, it follows that $S:=-\frac{1}{3}|\mathcal{T}|^2g + \mathcal{T}\mathcal{T}^t\in\Sigma^2(M)$ satisfies
    \[
    S\diamond\varphi = T_l\diamond(T_l\diamond\varphi).
    \] 
    Letting $C:=\Div T\in\Omega_{\fm}^2(M)$, then the endomorphism $A=A^i_j\in\mathrm{End}(TM)$ defined by $A_{ij}=S_{ij}+C_{ij}$ describes the Laplacian completely as an infinitesimal action:
    \[
    \Delta\varphi = A\diamond\varphi.
    \]
\end{example}

\subsection{General flows}
\label{sec: general_flows}

Let $M^n$ be a connected, orientable $n$-manifold. Recall from Lemma \ref{lem: ker_diamond}, that whenever $M$ admits an $H$-structure $\xi$ defined by one or several tensor fields which are stabilised by $H\subset \SO(n)$, then the $H$-submodule $\Omega^2_\fh$ is a subspace of $\ker (\cdot{}\diamond\xi)$.  {At each point $p\in M$, a representative frame in the class of the homogeneous section $\sigma$ of Section~\ref{sec: preliminaries} identifies $T_{p}M$ with $\mathbb{R}^n$ and the $H$-structure $\xi_{p}$ with the model structure 
$\xi_{\circ}$. This correspondence not only allows us to describe all $H$-structures but also their infinitesimal deformations.}

Consequently, a general $\mathrm{GL}(n,\bR)$-variation of $\xi$ can be written as (cf. \cite[Proposition 2.3]{Dwivedi-Loubeau-SaEarp2021}):
\begin{equation}
    \frac{\partial}{\partial t}\xi=A\diamond\xi \qforq  A\equiv A(t)=S(t)+C(t), \qwithq S(t)\in \Sigma^2 \qandq C(t)\in \Omega^2_{\fm}\subset\Omega^2.
\end{equation}
From now on, we shall restrict ourselves to $H$-structures defined by tensors, i.e. we shall assume that $H=\mathrm{Stab}(\xi_{\circ})$. The evolution equation \eqref{eq: general_H_flow} is then the \emph{general flow equation} for such an $H$-structure $\xi$.

We now want to derive the evolution equations of the main quantities  related to an $H$-structure $\xi$ under the general flow \eqref{eq: general_H_flow}. We start with the evolution of the associated Riemannian $\xi$-metric.
\begin{lemma}
\label{lemma: isometric variation}
    Suppose that $\{\xi(t)\}_{t\in I\ni 0}$ is a family of $H$-structures evolving under \eqref{eq: general_H_flow}. If $g(t)$ is the unique Riemannian metric on $M^n$ determined by $\xi(t)$, then
    \[
    \frac{\partial}{\partial t} g(t)=A(t)\diamond g(t)=2S(t).
    \]
\end{lemma}
\begin{proof}
    Denote by $\xi_{\circ}$ the pointwise linear model of $\xi$, so that $H=\{h\in\mathrm{GL}(n,\mathbb{R}): h.\xi_{\circ} = \xi_{\circ}\}$. Let
\[
Q_{\xi(t)}:=\{u\in\Fr(M): u.\xi_{\circ}=\xi(t)\}
\] 
    be the principal $H$-subbundle of $\Fr(M)$ uniquely determined by $\xi(t)$. Since $H\subset\mathrm{SO}(n)$, we know that $\xi(t)$ also determines uniquely a principal $\mathrm{SO}(n)$-subbundle $P_{\xi(t)}:=\mathrm{SO}(n)\cdot{Q_{\xi(t)}}\subset\Fr(M)$ containing $Q_{\xi(t)}$, which corresponds to a unique metric $g(t)$ and orientation $\vol_{g(t)}$ on $M^n$. Note that
\[
g(t) = u.g_{\circ},\quad\forall u\in P_{\xi(t)},
\] 
    and write $\xi:=\xi(0)$, and $g:=g(0)$. 
    
    We compute the first order variation $\frac{\partial}{\partial t}g(t)|_{t=0}$ of $g$, given the first order variation $\frac{\partial}{\partial t}\xi(t)|_{t=0} = A\diamond\xi$ of $\xi$. It suffices to consider any path $\xi(t)$ of $H$-structures satisfying the latter equation, of the form
\[
\xi(t) = e^{tA}.\xi.
\] 
    Then, for any $u\in Q_{\xi}\subset P_{\xi}$, we have
\[
\xi(t) = e^{tA}.\xi = e^{tA}.(u.\xi_{\circ}) = (ue^{tA}).\xi_{\circ}.
\] 
    Thus $ue^{tA}\in Q_{\xi(t)}\subset P_{\xi(t)}$ for all $t$. In particular, since $u\in P_{\xi}$ and $ue^{tA}\in P_{\xi(t)}$,
\[
g(t) = (ue^{tA}).g_{\circ} = e^{tA}.(u.g_{\circ}) = e^{tA}.g.
\] 
    Therefore, $\frac{\partial}{\partial t} g(t)|_{t=0} = A\diamond g = 2S$, 
    where the last equality follows from Lemma \ref{lem: properties_diamond}--\ref{item: diamond_S}. \qedhere
\end{proof}
\begin{remark}
    If we consider the metric $g(t)$ itself as one of the tensors among the components of $\xi(t)$, then since the diamond in \eqref{eq: general_H_flow} acts componentwise, we already get the evolution $\frac{\partial}{\partial t}g = A\diamond g = 2S$ directly from Lemma \ref{lem: properties_diamond}--\ref{item: diamond_S}. This is true, for instance, when $H=\U(m)$ and $\xi(t)=(g(t),J(t))$, where $J(t)$ is an almost complex structure compatible with $(M^{2m},g(t))$. In other cases, for instance when $H=\rm G_2$, we consider $\xi(t)=\varphi(t)\in\Omega_{+}^3(M)$ as a single tensor, instead of a coevolving pair $\xi(t)=(g(t),\varphi(t))$, so the above proof of Lemma \ref{lemma: isometric variation} is a more general way to deduce the evolution of the associated metric.
\end{remark}

\begin{example}[Ricci $H$-flow]
\label{ex: Ricci flow}
    According to Lemma \ref{lemma: isometric variation}, the simplest flow of $H$-structures $\xi(t)$ inducing the Ricci flow
\[
\partial_t g_{\xi(t)} = - 2\mathrm{Ric}(g_{\xi(t)})
\] on the corresponding co-evolving metrics $\{g_{\xi(t)}\}$ is given by
\[
\frac{\partial}{\partial t}\xi = -\mathrm{Ric}(g_{\xi(t)})\diamond\xi(t).
\] 
 {When $H$ is one of the special holonomy groups, the corresponding  Ricci $H$-flow seems to be a promising tool in the search for the Ricci-flat metrics induced by torsion-free $H$-structures. Paraphrasing the argument by Dwivedi et al. in \cite{dgk-isometric}*{p.1856}, in the context $H=\rG_2$, one possible motivation for studying the Ricci $H$-flow is that it induces precisely the Ricci flow on metrics, in contrast to the Laplacian flow, which induces Ricci flow plus lower order terms involving the torsion. Therefore, one may hope to first flow the $H$-structure in a way that improves the metric, and then flow it isometrically in a way that still decreases the torsion, eg. by \eqref{eq: harmonic_flow} below -- possibly with some additional skew-symmetric terms in $C(t)$, for greater regularity. }
\end{example}

One can envisage leveraging some of the celebrated literature on the standard Ricci flow in order to explore, under suitable assumptions,  the properties  of the Ricci $H$-flow. As a first illustrative step, let us use the short-time existence and uniqueness of the Ricci flow to prove the \emph{existence} of short-time solutions (though not, in general, their uniqueness) to the Ricci $H$-flow:

\begin{lemma}[Short time existence of Ricci $H$-flow]
\label{lem: Ricci flow}
    Let $(M^n,g_0)$ be a closed Riemannian manifold admitting a compatible $H$-structure $\xi_0$. Then there is $\tau>0$ such that there exists a solution $\xi(t)$, defined for all $t\in [0,\tau)$, to the problem
\begin{equation}
\label{eq: Ricci flow_H-str}
\begin{cases}
    \frac{\partial}{\partial t}\xi(t) 
    = -\mathrm{Ric}(g_{\xi(t)})\diamond\xi(t)\\
	\xi(0)=\xi_0
\end{cases}.
\end{equation}
\end{lemma}
\begin{proof}
    By the well-known short-time existence and uniqueness theorem for the Ricci flow equation (originally due to Hamilton \cite{hamilton1982three}; see also DeTurck's simplification \cite{deturck1983deforming}), there exist $\tau>0$ and a unique smooth solution $g(t)$, defined for all $t\in [0,\tau)$, to the Ricci flow equation with initial metric $g_0=g_{\xi_0}$:
    \begin{equation}\label{eq: Ricci_flow}
		\begin{cases}
		\frac{\partial}{\partial t}g(t) = -2\mathrm{Ric}(g(t))\\
		g(0) = g_0.\\
		\end{cases}
	\end{equation} For each $t\in [0,\tau)$, define $A(t),\tilde{A}(t)\in\Gamma(\mathrm{End}(TM))$ by
	\[
	A(t)_j^i:= -g(t)^{il}\mathrm{Ric}(g(t))_{lj}\quad\text{and}\quad \tilde{A}(t)_j^i :=\int_0^t A(s)_i^j ds.
	\] Defining $\xi(t):=e^{\tilde{A}(t)}\cdot{\xi_0}$, for every $t\in [0,\tau)$, we then get a solution to the problem
	\begin{equation}
\begin{cases}
    \frac{\partial}{\partial t}\xi(t) = -\mathrm{Ric}(g(t))\diamond\xi(t)\\
	\xi(0)=\xi_0.
\end{cases}
\end{equation} 
It follows from Lemma \ref{lemma: isometric variation} that the Riemannian metrics $\{g_{\xi(t)}\}$ induced by the $H$-structures $\{\xi(t)\}$ are precisely the unique Riemannian metrics $\{g(t)\}$ that solve \eqref{eq: Ricci_flow}, so in fact $\{\xi(t)\}$ is a solution to the problem \eqref{eq: Ricci flow_H-str} for all $t\in [0,\tau)$.
\end{proof}

From the evolution of the metric $g$ obtained in Lemma \ref{lemma: isometric variation}, one can immediately derive the evolution of other objects related to $g$, such as the volume form and the Christoffel symbols, cf. \cite[Corollary 3.3]{Karigiannis2009}:
%\begin{lemma}
%For any $A,B\in \fgl(m,\bR)$ and $X_l\in \cX(M)$, we have
%\begin{align}\label{eq: diamond_composition}
 %   A\diamond(B\diamond \xi)=-\frac{1}{2}[A,B]\diamond\xi
%\end{align}
%\end{lemma}

%\begin{proof}
%For $t$ and $s$ small enough to assure that $tA,sB\in V\subset \fgl(m,\bR)$, where $V$ is a neighbourhood of $0$, we can apply the Backer-Campbell-Hausdorff identity $e^{tA}e^{sB}=e^{c(tA,sB)}$ where
%\begin{equation}\label{eq: BCH identity}
 %   c(tA,sB)=tA+sB-\frac{ts}{2}[A,B]+\frac{
  %  t^2s}{2}[[A,B],B]-\frac{ts^2}{2}[[A,B],A]+...
%\end{equation}
%Then, 
 %   \begin{align*}
  %  A\diamond(B\diamond \xi)=&\left.\frac{d}{d t}\right|_{t=0}\left.\frac{d}{d s}\right|_{s=0}(e^{tA}e^{sB}).\xi\\
   % =&\left.\frac{d}{d t}\right|_{t=0}\left.\frac{d}{d s}\right|_{s=0}(e^{c(tA,sB)}).\xi\\
    %=&=-\frac{1}{2}[A,B]\diamond\xi.
%\end{align*}
%\end{proof}
\begin{lemma}[Evolution of $g^{ij}$, $\vol_g$ and $\Gamma_{ij}^k$]\label{lem:volume_Christoffel_variation}
The evolution of the inverse of the metric $g^{ij}$, the volume form $\vol_g$ and the Christoffel symbols $\Gamma_{ij}^k$ under the flow \eqref{eq: general_H_flow} are given by
\begin{align*}
    \frac{\pt}{\pt t} g^{ij}=-2S^{ij}, \quad \frac{\pt}{\pt t} \vol_g=\tr_g(S)\vol_g \qandq \frac{\pt}{\pt t} \Gamma_{ij}^k=g^{kl}(\nabla_i S_{jl}+\nabla_j S_{il}-\nabla_l S_{ij}).
\end{align*}
\end{lemma}

Let us derive the evolution of the torsion tensor $T$ of the $H$-structure $\xi$ under the general flow \eqref{eq: general_H_flow}. We begin with the evolution of $\nabla_l\xi$. In what follows, using the definition of $\diamond$ given in \eqref{eq: diamond_operator}, note that we can write the general flow \eqref{eq: general_H_flow} in coordinates as
%\begin{equation}
 %   \frac{\partial}{\partial t}\xi^{i_1,...,i_p}_{j_1,...,j_q}=-\sum_{r=1}^pA\indices{^{i_r}_m}\xi^{i_1,...,m,...,i_p}_{j_1,...,j_q}+\sum_{s=1}^q A\indices{^m_{j_s}}\xi_{j_1,...,m,...,j_q}^{i_1,...,i_p}.
%\end{equation}
\begin{equation}
    \frac{\partial}{\partial t}\xi^I_J=-\sum_{i\in I}A^i_m\xi^{I_m(i)}_J+\sum_{j\in J} A^m_j\xi^I_{J_m(j)},
\end{equation}
%\begin{equation}
 %   \frac{\partial}{\partial t}\xi^I_J=-\sum_{a=1}^pA\indices{^{i_a}_m}\xi\indices{^{I_m(\hat{i_a})}_J}+\sum_{b=1}^q A\indices{^m_{j_b}}\xi\indices{_{J_m(\hat{j_b})}^I},
%\end{equation}
where $I=\{i_1,...,i_p\}$, $J=\{j_1,\dots,j_q\}$ and, $I_m(i)$ and $J_m(j)$ denote, respectively, the index sets $I$ and $J$ after replacing the indexes $i\in I$ and $j\in J$ with $m$.
\begin{proposition}\label{prop: evo_nabla_xi}
The evolution of $\nabla _l\xi$ under the flow \eqref{eq: general_H_flow} is given by
\begin{align}\label{eq: variation_nabla_xi}
\frac{\pt}{\pt t}\nabla_l\xi=A\diamond \nabla_l\xi+(\nabla_lC-\Lambda \nabla S_l)\diamond\xi, 
%\frac{\pt}{\pt t}\nabla_l\xi=A\diamond (T_l\diamond \xi)+\pi_\fm(\Lambda \nabla S_l+\nabla_lC)\diamond\xi, 
\end{align}
where $ {(\Lambda\nabla S_l)\indices{_j^i}}:=(\Lambda\nabla S_l)_{jk}g^{ik}=g^{ik}(\nabla_j S_{kl}-\nabla_k S\indices{_j_l})$.
\end{proposition}
\begin{proof}
Using the evolution of Christoffel symbols from Lemma \ref{lem:volume_Christoffel_variation}, we compute:
\begin{align*}
    \frac{\pt}{\pt t}(\nabla_l\xi)^I_J=& \nabla_l\left(\frac{\pt}{\pt t}(\xi_t)^I_J\right)-\sum_{j\in J}\left(\frac{\pt}{\pt t}\Gamma_{lj}^n\right)\xi^I_{J_n(j)}+\sum_{i\in I}\left(\frac{\pt}{\pt t}\Gamma_{lm}^{i}\right)\xi^{I_m(i)}_J \\
    =&\nabla_l\left(-\sum_{i\in I} A^i_m\xi^{I_m(i)}_J+\sum_{j\in J} A^n_j\xi^I_{J_n(j)}\right)-\sum_{j\in J} g^{nk}(\nabla_l S_{jk}+\nabla_j S_{kl}-\nabla_k S_{lj})\xi^I_{J_n(j)}\\
    &+\sum_{i\in I}g^{ik}(\nabla_l S_{mk}+\nabla_m S_{kl}-\nabla_k S_{lm})\xi^{I_m(i)}_J\\
    =&-\sum_{i\in I}A^i_m\nabla_l\xi^{I_m(i_a)}_J+\big(\nabla_l C^i_m+g^{ik}(\nabla_k S_{ml}-\nabla_m S_{kl})\big)\xi^{I_m(i)}_J\\
    &+\sum_{j\in J} A^n_j\nabla_l\xi^I_{J_n(j)}+\big(\nabla_lC^n_j+g^{nk}(\nabla_k S_{jl}-\nabla_j S_{kl})\big)\xi^I_{J_n(j)}\\
    =&(A\diamond\nabla_l\xi)^I_J+(\nabla_l C\diamond\xi)^I_J-\sum_{i\in I}g^{ik}(\Lambda\nabla S_l)_{km}\xi^{I_m(i)}_J+\sum_{i\in J}g^{nk}(\Lambda\nabla S_l)_{kj}\xi^I_{J_n(j)}
\end{align*}
where ${(\Lambda\nabla S_l)\indices{_j^i}}=(\Lambda\nabla S_l)_{jk}g^{ki}=(\nabla_j S_{kl}-\nabla_k S\indices{_j_l})g^{ki}$. 
\end{proof}

Combining the above with Lemmas \ref{lem: properties_diamond} and \ref{lem: ker_diamond} we conclude the proof of Proposition \ref{prop: intro - evol of torsion}:

\begin{corollary}[Evolution of the torsion]
\label{cor: evo_torsion}
    Under the flow \eqref{eq: general_H_flow}, for each coordinate vector field $\partial_l$, the torsion $T_l:=T_{\partial_l}$ satisfies
\begin{equation}\label{eq: ddt T_diamond_xi}
%\Big(\frac{\partial}{\partial t}T_l +[T_l,A]-\pi_\fm(\Lambda\nabla B_l+\nabla_lC)\Big)\diamond \xi_t=0,
    \Big(\frac{\partial}{\partial t}T_l +[A,T_l]+\Lambda\nabla S_l-\nabla_l C\Big)\diamond \xi=0.
\end{equation} If $\pi_{\fm}:\Omega^2\to\Omega_{\fm}^2$ 
%If $\pi_\fm:\fso(TM)\to \\fm_Q$ 
denotes the orthogonal projection, we then have
\begin{align}\label{eq: m_part ddt_T}
\begin{split}
    %\pi_\fh\bigg(\frac{\partial}{\partial t}T_l\bigg)&= \pi_\fh([A,T_l])+\pi_\fh(\Lambda\nabla B_l+\nabla_lC),\\
    \pi_\fm\bigg(\frac{\partial}{\partial t}T_l\bigg)&= \pi_\fm([T_l,C])+\pi_\fm(\nabla_lC-\Lambda\nabla S_l).
\end{split}    
\end{align}
In particular, 
\begin{equation}\label{eq: ddt_norm_T}
    \frac{\partial}{\partial t}|T|^2=2\langle \nabla C-\Lambda\nabla S,T\rangle-2\langle T_m,T_n\rangle S^{mn}%-4\langle T_m,T_n S\rangle g^{mn}.
\end{equation}
\end{corollary}
\begin{proof}
On one hand, taking $\partial/\partial t$ of $\nabla_l\xi = T_l\diamond\xi$, as in \eqref{eq: nabla of xi-GENERAL}, and using the general flow equation \eqref{eq: general_H_flow}, 
\begin{equation}
\label{eq: evo_torsion_1}
    \frac{\partial}{\partial t}\nabla_l\xi =\Big(\frac{\partial}{\partial t}T_l\Big)\diamond\xi+T_l\diamond\Big(\frac{\partial}{\partial t}\xi\Big)
    =\Big(\frac{\partial}{\partial t}T_l\Big)\diamond\xi+T_l\diamond(A\diamond\xi).
\end{equation} 
    On the other hand, using equation \eqref{eq: variation_nabla_xi} of Proposition \ref{prop: evo_nabla_xi} and  \eqref{eq: nabla of xi-GENERAL}, we have
\begin{equation}
\label{eq: evo_torsion_2}
    \frac{\pt}{\pt t}\nabla_l\xi
    =A\diamond (T_l\diamond\xi)+(\nabla_lC-\Lambda \nabla S_l)\diamond\xi.
\end{equation} 
    Thus, using Lemma \ref{lem: properties_diamond}--\ref{item: diamond_Jacobi} to get $A\diamond(T_l\diamond\xi) - T_l\diamond(A\diamond\xi) = -[A,T_l]\diamond\xi$, and combining equations \eqref{eq: evo_torsion_1} and \eqref{eq: evo_torsion_2} we get the desired evolution equation \eqref{eq: ddt T_diamond_xi}. We then obtain \eqref{eq: m_part ddt_T} from \eqref{eq: ddt T_diamond_xi}, together with $[T_l,S]\in \Sigma^2$ and the fact that $\Omega^2_\fh= \ker(\cdot{}\diamond\xi)|_{\Omega^2}$ [Lemma \ref{lem: ker_diamond}]. 

    Finally, using equation \eqref{eq: m_part ddt_T} and the evolution equations of Lemma \ref{lem:volume_Christoffel_variation}, we get \eqref{eq: ddt_norm_T} by a simple computation:
\begin{align*}
    \frac{\partial}{\partial t}|T|^2
    &= 2 \frac{\partial}{\partial t}T_{m;ab}T_{n;ij}g^{mn}g^{ai}g^{bj}-2T_{m;ab}T_{n;ij}S^{mn}g^{ia}g^{bj}-4T_{m;ab}T_{n;ij}g^{mn}g^{ia}S^{bj}\\
    &= 2\bigg(\frac{\partial}{\partial t}T_{m;ab}-2T_{m;a}^lS_{lb}\bigg)T_{n;ij}g^{mn}g^{ai}g^{bj}-2\langle T_m,T_n\rangle S^{mn}\\
    &= 2\pi_{\fm}\bigg(\frac{\partial}{\partial t}T_m\bigg)_a^lg_{lb}T_{n;ij}g^{mn}g^{ai}g^{bj}-2\langle T_m,T_n\rangle S^{mn}\\
    %&= 2\pi_\fm\bigg(\frac{\partial}{\partial t}T_{m;ab}\bigg)T_{n;ij}g^{mn}g^{ai}g^{bj}-2\langle T_m,T_n\rangle S^{mn}-4\langle T_m,T_n S\rangle g^{mn}\\
    &= 2\pi_\fm\big(\nabla_mC_{ab}-(\Lambda\nabla S_m)_{ab}\big)T_{n;ij}g^{mn}g^{ai}g^{bj}-2\langle T_m,T_n\rangle S^{mn}\\%-4\langle T_m,T_n S\rangle g^{mn}\\
    &= 2\langle \pi_\fm\big(\nabla_m C-\Lambda\nabla S_m\big),T_{n}\rangle g^{mn}-2\langle T_m,T_n\rangle S^{mn}%-4\langle T_m,T_n S\rangle g^{mn},
\end{align*}
    where we have used that 
\begin{align*}
    \langle \pi_\fm([T_m,C]),T_n\rangle g^{mn}
    &=\langle [T_m,C],T_n\rangle g^{mn}
    =-\tr([T_m,C]T_n)g^{mn}
    =-\tr(T_m C T_n-C T_m T_n)g^{mn}\\
    &=0 .\qedhere
\end{align*}    
\end{proof}

\begin{remark}
    Similarly, the %$\Omega^2_
    $\fm$-part of the evolution  {of} $T_{l;ik}=\frac{1}{2}\left({T_{l;}}_i^jg_{jk}-{T_{l;}}_k^jg_{ji}\right)$ can be obtained from Lemma \ref{lem:volume_Christoffel_variation} and \eqref{eq: m_part ddt_T}:
    \begin{equation}
    \label{eq: Omega^2_m-part ddt_T}
         \pi_{\fm}\left(\frac{\partial}{\partial t}T_l\right)_{ik}=\pi_{\fm}(AT_l+T_lA^t)_{ik}+\pi_{\fm}(\nabla_lC-\Lambda\nabla S_l)_{ik},
         %\pi_{\Omega^2_\fm}\left(\frac{\partial}{\partial t}T_l\right)_{ik}=\pi_{\Omega^2_\fm}(AT_l+T_lA^t)_{ik}+\pi_{\Omega^2_\fm}(\nabla_lC-\Lambda\nabla S_l)_{ik},
    \end{equation}
    which agrees with \cite[(3.3)]{karigiannis-spin7} in the context of $\S7$-structures.
\end{remark}

\begin{example}\label{ex: evolution of nabla J}
    Using the torsion of the almost complex structure \eqref{eq: torsion_almost_complex} and applying \eqref{eq: variation_nabla_xi}, we obtain
    \begin{equation*}
         \frac{\partial}{\partial t}\nabla_l J=[A,[T_l,J]]+[\nabla_l C-\Lambda\nabla S_l,J].
    \end{equation*}
    Thus the evolution of the torsion $T_l$ is
    \begin{align*}
         \frac{\partial}{\partial t} T_l &=\frac{1}{2}[A,J]\nabla_l J-\frac{1}{2}J[A,[T_l,J]]-\frac{1}{2}J[\nabla_lC-\Lambda\nabla S_l,J]\\
         &= (T_lA)-(AT_l) {-}\pi_{\fu(m)^{\perp}}(\nabla_lC-\Lambda\nabla S_l)\\
         &= [T_l,A] {-}\pi_{\fu(m)^{\perp}}(\nabla_lC-\Lambda\nabla S_l).
    \end{align*}
\end{example}

\begin{example}[\cite{Karigiannis2007}*{Lemma 3.7}]
    When $H=\rm G_2$, the geometric structure is $\xi=\varphi\in \Omega_+^3(M)$ and equation \eqref{eq: variation_nabla_xi} becomes 
    \begin{equation*}
        \frac{\partial}{\partial t}\nabla_l\varphi=S\diamond\nabla_l\varphi+X\lrcorner\nabla_l\psi-(\Lambda\nabla S_l)\diamond\varphi+\nabla_l X\lrcorner\psi.
    \end{equation*}
\end{example}    

 {
\begin{remark}[The $\fh$-component of the evolution of $T$ under isometric variations]\label{rm: h-part ddt of T}
    Corollary \ref{cor: evo_torsion} only provides the $\fm$-component of the evolution of the torsion $T$. The $\fh$-component depends on the group $H$. For instance, for isometric variations (i.e. with $A=C$)  the $\fh$-part of the evolution of $T$ for the groups $H=\U(n), \S7$ and  $\mathrm{Sp}(2)\mathrm{Sp}(1)$ is (see Example \ref{ex: evolution of nabla J}, \cite[(3.3)]{karigiannis-spin7} and \cite{udhav2022quaternionic}*{(4.3)}, respectively)
    \[
    \pi_{\fh}\left(\frac{\partial}{\partial t} T_l\right)=\pi_{\fh}([T_l,C]).
    \]
    Finally, for $H=\rm G_2$, notice that the evolution of $\cT$ (see \cite{dgk-isometric}*{(2.13)}) gives
    $$
    \frac{\partial}{\partial t}T_{p,mn}=-\frac{1}{3}X_m\cT_{pn}+\frac{1}{3}X_n\cT_{pm}-\frac{1}{3}\nabla_pX_q\varphi_{qmn},
    $$
    where $C=-\frac{1}{3}X\lrcorner\varphi$. Thus, using the projection $\pi_{\fg_2}(\beta)_{mn}=\frac{2}{3}\beta_{mn}+\frac{1}{6}\beta_{ij}\psi_{ijmn}$ (cf. \cite{Karigiannis2007}*{(2.10)}), one can show that $\pi_{\fg_2}\left(\frac{\partial}{\partial t} T_l\right)=\pi_{\fg_2}([T_l,C])$. Alternatively, $\pi_{\fg_2}\left(\frac{\partial}{\partial t} T_l\right)$ can be obtained from \cite{chemtov2022}*{(4.6)}, but beware that our sign convention for the bracket $[T_l,C]_{ij}=[T_l,C]_i^kg_{kj}$ differs from \cite{chemtov2022}*{(4.1)}.
\end{remark}}

Finally, following the proof of Corollary \ref{cor: evo_torsion}, and using Lemma \eqref{lem:volume_Christoffel_variation}, we obtain the variation of $|\nabla T|^2$:

\begin{corollary}
    Under the flow \eqref{eq: general_H_flow},  we have:
    \begin{align}\label{eq: evolution_norm_nablaT}
    \frac{1}{2}\frac{\partial}{\partial t}|\nabla T|^2
    =&\langle \nabla\dot{T}-(\nabla S)\diamond T-(\Lambda\nabla S)\diamond T-S\diamond \nabla T, \nabla T\rangle,
\end{align}
    where $\dot{T}=\frac{\partial}{\partial t}T$ and $\langle \cdot, \cdot\rangle$ is the induced metric on $\cT^{2,0}(TM)\otimes\Omega^2_\fm$.
\end{corollary}

 {
\begin{proof}
Since, summing on repeated indices,
\begin{align*}
    |\nabla T|^2 &= (\nabla_{l}T)_{p;ab} (\nabla_{m}T)_{q;ij} g^{ai} g^{bj} g^{pq} g^{lm} ,
\end{align*}
we have that
\begin{align} \label{neweq}
 \frac{1}{2}  \frac{\partial}{\partial t} |\nabla T|^2 
 &= \frac{\partial}{\partial t} \Big(\nabla_{l}T_{p;ab}\Big) \nabla_{m}T_{q;ij} g^{ai} g^{bj} g^{pq} g^{lm} - 2 \nabla_{l}T_{p;ab} \nabla_{m}T_{q;ij} S^{ai} g^{bj} g^{pq} g^{lm} \\
 & -  \nabla_{l}T_{p;ab} \nabla_{m}T_{q;ij} S^{lm} g^{bj} g^{pq} g^{ai} -  \nabla_{l}T_{p;ab} \nabla_{m}T_{q;ij} S^{pq} g^{bj} g^{lm} g^{ai} . \notag
\end{align}
But
$$
\nabla_{l}T_{p;ab}= \frac{\partial}{\partial x_l} (T_{p;ab})
    - \Gamma^{n}_{al} T_{p;nb} - \Gamma^{n}_{bl} T_{p;an} - \Gamma^{n}_{{p}l} T_{n;ab} ,$$
so
\begin{align}\label{eq: ddt_nabla_T}
 & \frac{\partial}{\partial t} \nabla_{l}T_{p;ab} =
\nabla_{l} \dot{T}_{p;ab}  - T_{p;nb} g^{nr} ( \nabla_a S_{lr} + \nabla_l S_{ar} -\nabla_r S_{la}) \\ \nonumber
& - T_{p;an} g^{nr} ( \nabla_b S_{lr} + \nabla_l S_{br} -\nabla_r S_{lb}) 
 - T_{n;ab} g^{nr} ( \nabla_p S_{lr} + \nabla_l S_{pr} -\nabla_r S_{lp}) .
\end{align}
Replacing \eqref{eq: ddt_nabla_T} into \eqref{neweq}, the first term yields $\langle \nabla\dot{T}, \nabla T\rangle$ while, within the {parentheses}, the middle terms give $-\langle (\nabla S)\diamond T, \nabla T\rangle$ and the rest is equal to $-\langle(\Lambda\nabla S)\diamond T, \nabla T\rangle$.\\
Finally, the last three terms of \eqref{neweq} sum up to $-\langle S\diamond \nabla T, \nabla T\rangle$.
\end{proof}
}

\begin{remark}
    We cannot, in general, formulate an explicit `Shi-type' estimate from \eqref{eq: evolution_norm_nablaT}, since the projection onto $\Omega^2_\fm$ in \eqref{eq: Omega^2_m-part ddt_T} depends specifically on the subgroup $H\subset\SO(n)$. 
    
    Nevertheless, we can say a little more in the particular case of \emph{oriented frame fields}, considered as $\{1\}$-structures, i.e. when $H=\{1\}\subset\SO(n)$  is the trivial subgroup. Then $\pi_\fm=\rI\rd$, and \eqref{eq: Omega^2_m-part ddt_T} becomes 
$$
 \frac{\partial}{\partial t}T_{b;cd}=A_{cm}T_{b;md}-A_{dm}T_{b;mc}+\nabla_bC_{cd}-\nabla_c S_{db}+\nabla_d S_{cb},
$$
 {
so, using $A= S+C$, a computation yields
\begin{align*}
   & \langle \nabla\dot{T} -(\nabla S)\diamond T-S\diamond \nabla T, \nabla T\rangle =
    \Big( \nabla_a C_{mr} T_{b;rn} - \nabla_a C_{nr} T_{b;rm} + C_{mr} \nabla_a T_{b;rn} -  C_{nr} \nabla_a T_{b;rm} + \nabla_a\nabla_b C_{mn} \\
    &- \nabla_a\nabla_m S_{nb} + \nabla_a\nabla_n S_{mb} - \nabla_a S_{br} T_{r;mn} - S_{ar} \nabla_r T_{b;mn} - S_{br} \nabla_a T_{r;mn} \Big) \nabla_i T_{j;pq} g^{ia}g^{bj}g^{mp}g^{nq} \\
    &= \langle \nabla_aC\diamond T_b + C\diamond \nabla_a T_b + \nabla_a\nabla_b C -\nabla_a\Lambda\nabla S_b - \nabla_a S_{br} T_{r} - S_{ar} \nabla_r T_{b} - S_{br} \nabla_a T_{r}, \nabla_i T_j \rangle g^{ai}g^{bj} .
\end{align*}
}

Hence \eqref{eq: evolution_norm_nablaT} has an explicit form amenable to such estimates:
\begin{align*}
    \frac{1}{2}\frac{\partial}{\partial t}|\nabla T|^2
    =&\langle\nabla_a\nabla_b C-\nabla_a\Lambda\nabla S_b+\nabla_aC\diamond T_b 
      {+ C\diamond \nabla_a T_b}
    -\nabla_aS_{bm}T_m\\
    &-S_{am}\nabla_mT_{b}-S_{bm}\nabla_aT_{m}-(\Lambda\nabla S_a\diamond T)_{b} ,\nabla_iT_{j}\rangle g^{ai}g^{bj} .
\end{align*} 
Flows of frame fields are a rather unexplored tool, with possibly relevant applications. For instance, if a flow converges to a smooth frame field with parallel torsion, then it endows the manifold with a Lie group structure.
\end{remark}

\subsection{General Bianchi-type identities and applications}
\label{sec: Bianchi}

We apply the evolution of torsion \eqref{eq: variation_nabla_xi} to derive a Ricci identity and a Bianchi-type identity for manifolds with an $H$-structure $\xi$. Our approach follows the exposition in \cite[\textsection 4]{Karigiannis2007}, using the diffeomorphism-invariance of the intrinsic torsion as a function of $\xi$.

Let $(M^n,g)$ be an oriented Riemannian manifold with compatible $H$-structure $\xi$. First, we note that given $Y\in \sX(M)$, in view of the splitting \eqref{eq:splitting TM* x TM-GENERAL} of $\End(TM)$, we may write 
\[
\nabla Y=\frac 12 \cL_Yg +\nabla_\fm (Y) +\nabla_\fh(Y),
\]
where $\nabla_k(Y):=\pi_k(\nabla Y)\in \Omega^2_k$. In particular  $\nabla_\fh Y\in\Omega^2_\fh \subset \ker(\cdot{}\diamond\, \xi)$, and one can prove the following:
\begin{lemma}[{\cite[Lemma 2.6]{Dwivedi-Loubeau-SaEarp2021}}]
    In terms of the torsion $T$, the Lie derivative of the $H$-structure $\xi$ with respect to a vector field $Y\in \sX(M)$ is given by
\begin{align}
\label{eq: Lie of xi-GENERAL}
    \cL_Y \xi
    &= (Y\lrcorner T +\frac 12 \cL_Yg +\nabla_\fm (Y))\diamond \xi.
\end{align}    
\end{lemma}

\begin{proposition}[Ricci identity]
\label{prop: Bianchi_identity}
The diffeomorphism invariance of the tensor $\nabla \xi$ as a function of the geometric structure $\xi$ is equivalent to the Ricci identity: 
\begin{equation}
\label{eq: Ricci_identity}
    \nabla_a\nabla_l\xi-\nabla_l\nabla_a\xi=R_{la}\diamond \xi,
\end{equation}
where $R_{al}\in\Gamma(\mathfrak{so}(TM))$ is the endomorphism given by the Riemannian curvature tensor.
\end{proposition}

\begin{proof}
  From the diffeomorphism invariance of $\nabla_l\xi$, we have
\begin{equation}\label{eq: diffeo_invariance_nabla_xi}
    \cL_Y\nabla_l\xi=\nabla_Y\nabla_l\xi+\nabla Y\diamond \nabla_l \xi=(D\nabla_l \xi)(\cL_Y\xi),
\end{equation}
where $D\nabla_l\xi$ is the linearisation of the tensor $\nabla_l\xi$ as a function of $\xi$, which means $(D\nabla_l \xi)(\cL_Y\xi)=\frac{\partial}{\partial t}\nabla_l\xi$ under the flow $\frac{\partial\xi}{\partial t}=\cL_Y\xi$. Using \eqref{eq: Lie of xi-GENERAL} and the fact that $(\cL_Y g)_{ij} = \nabla_i Y_j + \nabla_j Y_i$, we see that $\cL_Y\xi = A\diamond\xi$, where $A=S+C$ is the endomorphism with symmetric $S$ and skew-symmetric $C$ parts given by
\begin{align*}
    S^i_j&=\frac{1}{2}(\nabla_jY_k+\nabla_kY_j)g^{ik},\\
    C^i_j&=(Y\lrcorner T)_{jk}g^{ik} +\frac{1}{2}(\nabla_jY_k-\nabla_kY_j)g^{ik}. 
\end{align*}
Applying the evolution equation \eqref{eq: variation_nabla_xi} for such $A$, we get
\begin{align*}
    (D\nabla_l\xi)(\cL_Y\xi) &= (\nabla Y)\diamond \nabla_l\xi+(Y\lrcorner T)\diamond\nabla_l\xi+\big(\nabla_l(Y\lrcorner T)\big)\diamond\xi+\big(\nabla_l\pi_{\fso}(\nabla Y)\big)\diamond\xi-(\Lambda\nabla S_l)\diamond \xi\\
    &= (\nabla Y)\diamond \nabla_l\xi+\nabla_l(Y\lrcorner T\diamond\xi)+\big(\nabla_l\pi_{\fso}(\nabla Y)-\Lambda\nabla S_l)\big)\diamond \xi.
\end{align*}
    Now, from \eqref{eq: diffeo_invariance_nabla_xi} we obtain
    \begin{align*}
        \nabla_Y\nabla_l\xi-\nabla_l\nabla_Y\xi=\big(\nabla_l\pi_{\fso}(\nabla Y)-\Lambda\nabla S_l)\diamond \xi,
    \end{align*}
    and we conclude using the first Bianchi identity \eqref{eq: riem1stBid}:
\begin{align*}
    \nabla_l(\pi_\fso(\nabla Y))^i_j-(\Lambda\nabla S_l)^i_j &= 
    \frac{1}{2}\big(  {(}\nabla_l\nabla_jY  {)}_k-  {(}\nabla_l\nabla_kY {)}_j-  {(}\nabla_j\nabla_kY {)}_l-  {(}\nabla_j\nabla_lY {)}_k+  {(}\nabla_k\nabla_jY {)}_l+  {(}\nabla_k\nabla_lY {)}_j \big) g^{ki}\\
    &= \frac{Y^a}{2}\left(R_{ljak}+R_{kjal}+R_{klaj}\right)g^{ki}=Y^aR_{alkj}g^{ki}.\qedhere
\end{align*}
\end{proof}

From the defining relation \eqref{eq: def_T} of the torsion $T_X=\nabla^H_X-\nabla_X$, the Riemannian curvature tensor $R\in\Omega^2(M,\mathfrak{so}(TM))$ can be expressed in terms of the curvature $R^H$ of $\nabla^H$ by
\begin{equation}
\label{eq: curvature_gauge_formula}
    R=R^H-d_{\nabla^H}T+\frac{1}{2}[T\wedge T],
\end{equation}
where $d_{\nabla^H}T(X,Y)=\nabla^H_XT_Y-\nabla^H_YT_X$ and $[T\wedge T](X,Y)=2[T_X,T_Y]$. Since $R^H\in \Omega^2\otimes \Omega^2_\fh $ and $\nabla^H_XT_Y\in \Omega^2_\fm$,  taking the $\fm$-projection of \eqref{eq: curvature_gauge_formula} one obtains the following Bianchi-type identity, cf. \cite[Lemma 3.9]{Gonzalez-Davila2009}:
\begin{align}
\label{eq: Cabrera_Bianchi_ident}
    \begin{split}
    \pi_{\fm}(R_{ij})
    =&-\nabla^H_iT_j+\nabla^H_jT_i+\pi_\fm([T_i,T_j])\\
    =&-\nabla_iT_j+\nabla_jT_i-2[T_i,T_j]+\pi_\fm([T_i,T_j]),
    \end{split}
\end{align}
Alternatively, combining the relation $\nabla_X\xi = T_X\diamond\xi$ with the Ricci identity \eqref{eq: Ricci_identity} in Proposition \ref{prop: Bianchi_identity}, we also derive \eqref{eq: Cabrera_Bianchi_ident}:

\begin{corollary}[Bianchi-type identity]
\label{cor: m-part_Bianchi_identity}
The torsion $T_l$ satisfies the following Bianchi-type identity
\begin{equation}
\label{eq: m-part_Bianchi_identity}
    (\nabla_a T_l-\nabla_lT_a-[T_l,T_a]-R_{la})\diamond \xi=0.
\end{equation}
Equivalently,
\[
  \nabla_a T_l-\nabla_lT_a=-2[T_a,T_l]+\pi_\fm([T_a,T_l])+\pi_\fm(R_{la}).
\]
\end{corollary}
\begin{proof}
As already mentioned above, equation \eqref{eq: m-part_Bianchi_identity} follows directly by combining equations \eqref{eq: nabla of xi-GENERAL} and \eqref{eq: Ricci_identity}, and applying Lemma \ref{lem: properties_diamond}-\ref{item: diamond_Jacobi}. As for the second assertion, {since $\ker(\cdot \diamond \xi)|_{\Omega^2}=\Omega_{\fh}^2$, equation \eqref{eq: m-part_Bianchi_identity} implies 
\begin{equation}\label{eq: m-part_skew_nabla_T}
    \pi_{\fm}(\nabla_a T_l-\nabla_l T_a)=-\pi_{\fm}([T_a,T_l])+\pi_\fm(R_{la}).
\end{equation}
Moreover, from \eqref{eq: def_T} we get $\nabla_a T_l-\nabla_lT_a=\nabla^H_a T_l-\nabla^H_lT_a-2[T_a,T_l]$ and taking its $\fh$-part, it follows 
\begin{equation}\label{eq: h-part_skew_nabla_T}
\pi_{\fh}(\nabla_a T_l-\nabla_lT_a)=-2\pi_{\fh}([T_a,T_l]),
\end{equation}
since $\nabla^H_a T_l\in\Omega^2_\fm$. Hence, combining \eqref{eq: m-part_skew_nabla_T} and \eqref{eq: h-part_skew_nabla_T} the result follows}. 
\end{proof}

\begin{example}[\cite{nagy2011}*{(2.6)}] 
\label{ex sun}
    When $H=\U(m)\subset\SO(2m)$, we can follow Example \ref{ex: nabla_J} to compute
\begin{align*}
    \nabla_a T_l=&-\frac{1}{2}\nabla_aJ\nabla_lJ-\frac{1}{2}J\nabla_a\nabla_lJ\\
    =&-2T_aT_l-\frac{1}{2}J\nabla_a\nabla_lJ.
\end{align*}
    Using the abstract Ricci identity \eqref{eq: Ricci_identity} of Proposition \ref{prop: Bianchi_identity} and Lemma \ref{lem: properties_diamond}-\ref{item: diamond_bracket}, we recover Nagy's formula:
\begin{align*}
    \nabla_aT_l-\nabla_lT_a=&-2T_aT_l+2T_lT_a-\frac{1}{2}J(R_{la}\diamond J)\\
    =&-2[T_a,T_l]+\frac{1}{2}J[R_{la},J],
\end{align*}
where $\pi_\fm(R_{la})=\frac{1}{2}J[R_{la},J]$ and $\pi_\fm([T_a,T_l])=0$ since $[\fm,\fm]\subset \fu(m)$.
\end{example}

\begin{example}[\cite{Karigiannis2007}*{Theorem 4.2}]
\label{ex: g2}
When $H=G_2\subset\SO(7)$, we have seen in Example \ref{ex: G2_Laplacian_of_phi} that the intrinsic torsion is identified with the endomorphism $\cT$ defined by $\nabla_a\varphi_{ijk}=\cT\indices{_a^b}\psi_{bijk}$; differentiating this relation, we have 
$$
  \nabla_l\nabla_a\varphi_{ijk}=\nabla_l\cT\indices{_a^b}\psi_{bijk}+\cT\indices{_a^b}(-\cT_{lb}\varphi_{ijk}+\cT_{li}\varphi_{bjk}-\cT_{lj}\varphi_{bik}+\cT_{lk}\varphi_{bij}).
$$
%and $\xi=\varphi\in \Omega_+^3$, the identity \eqref{eq: Ricci_identity} of Proposition \ref{prop: Bianchi_identity} is equivalent to
%Indeed, using the definition of torsion of $\varphi$ 
%\begin{equation*}
 %   \nabla_a\varphi_{ijk}=T\indices{_a^b}\psi_{bijk},
%\end{equation*}
%\begin{align*}
 %   \nabla_l\nabla_a\varphi_{ijk}-\nabla_a\nabla_l\varphi_{ijk}=&\quad \nabla_lT\indices{_a^b}\psi_{bijk}+T\indices{_a^b}\nabla_l\psi_{bijk}-\nabla_aT\indices{_l^b}\psi_{bijk}-T\indices{_l^b}\nabla_a\psi_{bijk}\\
  %  =&\quad \nabla_lT\indices{_a^b}\psi_{bijk}+T\indices{_a^b}(-T_{lb}\varphi_{ijk}+T_{li}\varphi_{bjk}-T_{lj}\varphi_{bik}+T_{lk}\varphi_{bij})\\
   % &-\nabla_aT\indices{_l^b}\psi_{bijk}-T\indices{_l^b}(-T_{ab}\varphi_{ijk}+T_{ai}\varphi_{bjk}-T_{aj}\varphi_{bik}+T_{ak}\varphi_{bij}).
%\end{align*}
Contracting with the dual $4$-form and using the identities $\varphi_{ijk}\psi\indices{_m^i^j^k}=0$, $\varphi_{ijk}\psi\indices{_a_b^j^k}=-4\varphi_{iab}$ and $\psi_{aijk}\psi\indices{_b^{ijk}}=24g_{ab}$ \cite[Lemmas A.13 \& A.14]{Karigiannis2007}, 
\begin{align*}
    \nabla_l\nabla_a\varphi_{ijk}\psi\indices{_m^{ijk}}
    &= 24\nabla_l\cT\indices{_a_m}-12\cT\indices{_a_b}\cT_{li}\varphi\indices{_m^i^b}.
\end{align*}
The left-hand side of the Ricci identity \eqref{eq: Ricci_identity} then becomes
$$
  (\nabla_l\nabla_a\varphi_{ijk}-\nabla_a\nabla_l\varphi_{ijk})\psi\indices{_m^i^j^k}=24(\nabla_l\cT_{am}-\nabla_a\cT_{lm}-\cT_{li}\cT_{ab}\varphi\indices{^i^b_m}).
$$
Finally, the right-hand side of \eqref{eq: Ricci_identity} is given by
\begin{align*}
    (R_{la}\diamond\varphi)_{ijk}\psi\indices{_m^{ijk}}=&\quad 
    (R\indices{_l_a_i^n}\varphi_{njk}+R\indices{_l_a_j^n}\varphi_{ink}+R\indices{_l_a_k^n}\varphi_{ijn})\psi\indices{_m^{ijk}}\\
    =&-4R\indices{_l_a_i^n}\varphi\indices{_n_m^i}-4R\indices{_l_a_j^n}\varphi\indices{_n_m^j}-4R\indices{_l_a_k^n}\varphi\indices{_n_m^k}\\
    =&\quad  {-} 12R\indices{_l_a_i^n}\varphi\indices{_n_m^i}.
\end{align*} 
Hence we recover Karigiannis' Ricci identity:
\begin{equation}
    \nabla_l\cT_{am}-\nabla_a\cT_{lm}=\frac{1}{2}R_{lain}\varphi\indices{^i^n_m}+\cT_{li}\cT_{an}\varphi\indices{^i^n_m}.
\end{equation}
\end{example}

An application of the Bianchi-type identity \eqref{eq: m-part_Bianchi_identity} is the following strong restriction on the Riemann curvature tensor of Riemannian metrics coming from torsion-free $H$-structures.
\begin{corollary}
\label{cor: strong_restrictions}
    If $\xi$ is a torsion-free $H$-structure inducing the Riemannian metric $g$, then its Riemann curvature tensor $R_{abcd}\in \Sigma^2(\Lambda^2)$ actually lies in the subspace $\Sigma^2(\Lambda_{\fh}^2)$.
\end{corollary}
\begin{proof}
Using $T=0$ in \eqref{eq: m-part_Bianchi_identity}, together with Lemma \ref{lem: ker_diamond}, we see that $R_{abcd}=(R_{ab})_{cd}$ lies in $\Omega_{\fh}^2$, as a skew-symmetric tensor in $c,d$. Then the result follows from the standard symmetry $R_{abcd}=R_{cdab}$.
\end{proof}
\begin{remark}
    The content of Corollary \ref{cor: strong_restrictions} is well-known, so the novelty here is the alternative proof, generalising the context-specific arguments formulated by Karigiannis in the $\rm G_2$ and $\rm Spin(7)$ cases  \cites{Karigiannis2007,karigiannis-spin7}. In fact, when an $H$-structure $\xi$ is torsion-free, the Riemannian holonomy group $\mathrm{Hol}(g)$ of the induced metric $g$ is a subgroup of $H$, and it follows from the Ambrose--Singer theorem that the Riemann curvature $R$ lies in the subspace $\Sigma^2(\mathfrak{hol}(g))$, see e.g. \cite[Theorem 3.1.7]{Joyce2007}. 
\end{remark}

Another byproduct of the abstract Ricci identity is a direct proof of the following classical results.  
\begin{proposition}
\label{cor holonomy-ricci}
    Let $\xi$ be a compatible torsion-free $H$-structure on $(M^n,g)$.
\begin{itemize}
    \item[(i)] If $n=2m$ with $m\geqslant 2$ and $H=\U(m)\subset\SO(2m)$, then the Ricci tensor of $g$ is Hermitian.
    \item[(ii)] If $n=2m$ with $m\geqslant 2$ and $H=\SU(m)\subset\SO(2m)$, then $(M^{2m},g)$ is Ricci-flat.
    
    \item[(iii)] If $n=4k$ with $k\geqslant 2$ and $H=\mathrm{Sp}(k)\subset\mathrm{SO}(4k)$, then $(M^{4k},g)$ is Ricci-flat.
    
    %\item If $n=4m$ with $m\geqslant 2$ and $H=\mathrm{Sp}(m)\cdot{\mathrm{Sp}(1)}\subset\mathrm{SO}(4m)$, then $(M^{4m},g)$ is Einstein.
    \item[(iv)] If $n=7$ and $H=\rm G_2\subset\mathrm{SO}(7)$, then $(M^7,g)$ is Ricci-flat.
    \item[(v)]  If $n=8$ and $H=\rm Spin(7)\subset\mathrm{SO}(8)$, then $(M^8,g)$ is Ricci-flat.
\end{itemize}
\end{proposition}

\begin{proof}\quad
\begin{itemize}
    \item[(i)] Since $R_{ij}=\pi_{\fu(n)}(R_{ij})$, it follows from the definitions that $R_{ij}=JR_{ij}J^t$.
    \item[(ii)] According with the reductive decomposition $\fso(2m)=\fsu(m)\oplus\fm_1\oplus\fm_2$, where $\fm_1$ is the trivial $\SU(m)$-submodule generated by $J$ and $\fm_2=\fu(m)^\perp$ is  the $\SU(m)$-submodule of skew-symmetric matrices anti-commuting with $J$, we have the projections: 
    \begin{align*}
        \pi_{\fsu(m)}(R_{ij})
        &=\frac{1}{2}R_{ij}-\frac{1}{2}JR_{ij}J+\frac{1}{2m}\tr(JR_{ij})J,\\
        \pi_{\fm_1}(R_{ij})
        &=-\frac{1}{2m}\tr(JR_{ij})J,\\
        \pi_{\fm_2}(R_{ij})
        &=\frac{1}{2}R_{ij}+\frac{1}{2}JR_{ij}J.
    \end{align*}
    We claim that \begin{equation}\label{eq: m1_prop_projection}
    (JR_{ij}J)^i_k=\frac{1}{2}\tr(JR_{ij})J^i_k.
    \end{equation}
    Indeed, using the first Bianchi identity \eqref{eq: riem1stBid} we have:
    \begin{align*}
        (JR_{ij}J)^i_k &= J_{km}g^{mn}R_{ijnp}g^{pq}J_{ql}g^{il}\\
        &= (-R_{jnip}-R_{nijp})J_{km}g^{mn}J_{ql}g^{pq}g^{il}\\
        &= -(R_{jn})^q_iJ^i_qJ^n_k+J_{km}g^{mn}(R_{pj})_{ni}g^{il}J_{ql}g^{pq}\\
        &= \tr(JR_{nj})J^n_k-(JR_{pj}J)_{kq}g^{qp}.
    \end{align*}
  Then, for the Ricci tensor of a torsion free $\SU(m)$-structure, using Corollary \ref{cor: strong_restrictions} we have
    \begin{align*}
    \Ric_{jk} &= R_{ijkl}g^{il}=\pi_{\fsu(m)}(R_{ij})_{kl}g^{il}\\
            &= \left(\frac{1}{2}R_{ijkl}-\frac{1}{2}(JR_{ij}J)_{kl}+\frac{1}{2m}\tr(JR_{ij})J_{kl}\right)g^{il}\\
            &= \pi_{\fm_2}(R_{ij})_{kl}g^{il}+m\pi_{\fm_1}(R_{ij})_{kl}g^{il}-\pi_{\fm_1}(R_{ij})_{kl}g^{il}=0.
    \end{align*}
    
    \item[(iii)] According with the inclusion $\Sp(k)\subset \SU(2k)$ \cite{Salamon1989}*{(5.8)} and (ii), a compatible torsion-free $\Sp(k)$-structure is Ricci-flat.
    
    \item[(iv)] See \cite{Karigiannis2007}*{Corollary 4.12}, which is a consequence of the Ricci identity for $H=\rG_2$, see also Example \ref{ex: g2}.
    
    \item[(v)] See \cite{karigiannis-spin7}*{Corollary 4.7}, which invokes the same argument as in the $\rG_2$ case.
    \qedhere
\end{itemize}
\end{proof}

\subsection{Dirichlet-type energy functionals and related flows}
\label{sec: Dirichlet}

We are going to consider natural energy functionals on the space of $H$-structures over a closed and oriented $n$-manifold $M$. We start by considering the functional which assigns to each $H$-structure $\xi$ a suitable normalisation of the squared $L^2$-norm of its torsion $T$, with respect to its induced metric $g$:
\begin{equation}
\label{eq: energy_T}
    \cE(\xi):=\frac{1}{2}\int_M|T|^2\vol_g.
\end{equation} 
Let us compute the general first variation of \eqref{eq: energy_T}:
\begin{proposition}
\label{general EL}
    If $\{\xi(t)\}$ is a smooth $1$-parameter family of $H$-structures, satisfying any condition of the form \eqref{eq: general_H_flow}, with $\xi(0)=\xi$, then
\begin{align}
\label{eq: conformal variation}
    \frac{d}{dt}\Bigr|_{t=0}\cE(\xi)=&\int_M\Big(\big(\Div T^t_{ia}+\Div T^t_{ai}-(T\ast T)_{ia}+\frac{1}{2}|T|^2g_{ia}\big)S_{pq}-\Div T_{ia}C_{pq}\Big)g^{ip}g^{aq}\vol_g,
\end{align}
    where $T$ is the intrinsic torsion of the $H$-structure $\xi$, $g$ is its associated metric, the transpose torsion  $T^t\in\Omega_\fm^2\otimes\Omega^1$ is defined by $T^t_{ba;i}:=T_{i;ba}$, and $(T\ast T)_{ia}:=T_{i;jk}T_{a;bc}g^{jb}g^{kc}$.
%\[
%(T\ast T)_{ia}:=T_{i;jk}T_{a;bc}g^{jb}g^{kc}.
%\]
    An $H$-structure $\xi$ on $M$ is a critical point of \eqref{eq: energy_T} if, and only if,
\begin{equation}\label{eq: Euler_Lagrange_eq_of_E}
    \sym (\Div T^t)-T\ast T+\frac{1}{2}|T|^2g=0 \qandq \Div T=0.
\end{equation}
\end{proposition}

%\begin{lemma}
%If $\xi_t$ is a smooth family of $H$-structures with $\xi(0)=\xi_0$ and $\frac{d}{dt}|_{t=0}\xi_t=A\diamond\xi_0$ where $A=\alpha g+C\in \Omega^0\oplus\Omega^2_\fm$, then
%\begin{align}\label{eq: conformal variation}
 %   \frac{d}{dt}\Bigr|_{t=0}\cE(\xi)=&-\int_M \Big(2\alpha\Div(\tr T)+\langle\Div T,C\rangle+(3-\frac{n}{2})\alpha|T|^2\Big)\vol_g,
%\end{align}
%where $\tr T_p=T_{j;pq}g^{jq}$.
%\end{lemma}

\begin{proof}
Using Lemma \ref{lem:volume_Christoffel_variation}, \eqref{eq: ddt_norm_T} and integration by parts, we have:
   \begin{align*}
       \frac{d}{dt}\Bigr|_{t=0}\cE(\xi(t))
       =&\;\frac{1}{2}\frac{d}{dt}\Bigr|_{t=0}\int_M T_{i;jk}(t)T_{a;bc}(t)g^{ia}(t)g^{jb}(t)g^{kc}(t)\vol_{g(t)}\\
       =&\int_M\Big((-2\nabla_jS_{ki}+\nabla_iC_{jk})T_{a;bc}g^{ia}g^{jb}g^{kc}-T_{i;jk}T_{a;bc}g^{jb}g^{kc}S_{pq}g^{ip}g^{aq}+\frac{1}{2}|T|^2g_{ia}S_{pq}g^{ip}g^{aq}\Big)\vol_g.\\
       =&\int_M\Big(\big((\nabla_jT_{i;ba}+\nabla_jT_{a;bi})g^{jb}-(T\ast T)_{ia}+\frac{1}{2}|T|^2g_{ia}\big)S_{pq}-\Div T_{ia}C_{pq}\Big)g^{ip}g^{aq}\vol_g.
       \qedhere
   \end{align*}
%   Replacing $S=\alpha g$ and
 %  using integration by parts:
  % \begin{align*}
   %    \frac{d}{dt}|_{t=0}\cE(\xi)
    %   =&-\int_M\Big(C_{ij}\nabla_mT_{n;pq}g^{mn}g^{ip}g^{jq}-2\alpha\nabla_i(\tr T_{p})g^{ip}+(3-\frac{n}{2})\alpha|T|^2\Big)\vol_g,
   %\end{align*}
   %here $\tr T_{p}:=T_{j;pq}g^{jq}$. Finally, applying Stokes' theorem we obtain the result.
\end{proof}
\begin{corollary}\label{cor: first_iso_variation}
    If $\{\xi(t)\}$ is a smooth family of isometric $H$-structures, inducing the fixed Riemannian metric $g$, with $\xi(0)=\xi$ and  $\frac{d}{dt}|_{t=0}\xi(t)=C\diamond\xi$, for $C\in\Omega^2_\fm$, then
\begin{align}
    \frac{d}{dt}\Bigr|_{t=0}\cE(\xi)=&-\int_M\langle\Div T, C\rangle \vol_g,
\end{align} where $T$ is the torsion of the $H$-structure $\xi$.
\end{corollary}

\begin{example}
\label{ex: general_Dirichlet_flow_G2}
    When $H=\rG_2$, according to Example \ref{ex: Intrinsic_torsion_G2}, the Euler--Lagrange equations \eqref{eq: Euler_Lagrange_eq_of_E}  become:
    \begin{equation*}
        h=\sym(-\curl \cT^t+\cT(\cT\lrcorner \psi))-6\cT^t\cT+3|\cT|^2g \qandq C=\Div \cT\lrcorner\varphi,
    \end{equation*}
    where $\cT\in \End(TM)$ denotes the full torsion tensor of $\varphi$, $(\cT\lrcorner \psi)_{cd}=\cT_{ab}\psi\indices{^a^b_c_d}$ and  $\curl \cT^t_{ab}=\nabla_m\cT_{an}\varphi\indices{_b^m^n}$. Using the identity \cite[Proposition 2.9]{Karigiannis2007},
    \begin{equation*}
        \langle S\diamond\varphi, R\diamond \varphi\rangle=\tr(S)\tr(R)+2\tr(SR) \qforq R,S\in \Sigma^2,
    \end{equation*}
    and writing  $\psi(\cT,\cT):=\cT_{ab}\cT_{cd}\psi^{abcd}$, the corresponding gradient flow of \eqref{eq: energy_T} is
\begin{align*}
    \frac{\partial}{\partial t}\varphi
    &=-\frac{1}{2}h\diamond \varphi+\frac{\tr h}{18}g\diamond\varphi+3\Div \cT\lrcorner\psi\\
    &=\frac{1}{2}(\sym(\curl (\cT^t)-\cT(\cT\lrcorner\psi))+3\cT^t\cT-\frac{1}{9}\left(6|\cT|^2-\tr(\curl \cT)+\psi(\cT,\cT)\right)g)\diamond\varphi+3\Div \cT\lrcorner\psi.
\end{align*}
    
    Furthermore, using the expressions for the Ricci tensor and the scalar curvature in terms of the full torsion tensor \cite{grigorian2020flows}*{Lemma 2.1}, 
\begin{align*}
    \Ric&=-\frac{1}{2}\sym(\curl (\cT^t)-\nabla(\cT\lrcorner\varphi)+\cT^2-\tr(\cT)\cT)\\
    s&=2\tr(\curl \cT)-\psi(\cT,\cT)-\tr(\cT^2)+(\tr (\cT))^2,
\end{align*}
    and writing $L(\nabla \cT)_{ab}:=\nabla_a\cT_{mn}\varphi\indices{^m^n_b}$, the $\rG_2$-gradient flow is
\begin{align*}
    \frac{\partial}{\partial t}\varphi
    =&\;(-\Ric+3\cT^t\cT-\frac{1}{2}\sym(\cT^2-\tr(\cT)\cT-L(\nabla \cT))\\
    &+\frac{1}{9}\left(s+\tr(\cT^2)-(\tr(\cT))^2-6|\cT|^2-\tr(\curl \cT)\right)g)\diamond\varphi+3\Div \cT\lrcorner\psi.
\end{align*}
    Notice that the induced $\rG_2$-metric evolves as a Ricci-like flow, modified by the leading terms $\sym(L(\nabla \cT))$ and $\tr(\curl \cT)g$.
\end{example}

In order to properly connect the discussion with a natural notion of \emph{harmonicity}, understood as criticality of a Dirichlet-type gradient flow, let us consider the following alternative energy functional:
\begin{equation}
\label{eq: Dirichlet_energy}
    \cD(\xi):=\frac{1}{2}\int_M |\nabla\xi|^2\vol_g.
\end{equation} By Lemma \ref{lem: torsion_equiv_nablaxi}, there are $c,\tilde{c}>0$ depending only on $(M,g)$ and $H$ such that
\[
\tilde{c}\cE(\xi)\leqslant \cD(\xi)\leqslant c\cE(\xi).
\] Moreover, under the assumption that $c:=\lambda_1=\ldots=\lambda_k$ in Lemma \ref{lem: orthogonality}, i.e. if there is $c>0$ such that $\langle A\diamond\xi, B\diamond\xi\rangle=c\langle A,B\rangle$ for all $A,B\in\Omega_{\fm}^2(M)$ (e.g. if $\fm$ is an irreducible $H$-module), then
\[
\cD(\xi)=c\cE(\xi).
\]
\begin{lemma} 
\label{lem: gradient_dirichlet_along_iso_var}
    Suppose that $H=\mathrm{Stab}_{\SO(n)}(\xi_{\circ})$ is such that $\lambda_1=\ldots=\lambda_k$ in Lemma \ref{lem: orthogonality} (e.g. when $\fm$ is an irreducible $H$-module). If $\{\xi(t)\}$ is a smooth family of compatible $H$-structures on $(M^n,g)$, with $\xi(0)=\xi$ and $\frac{d}{dt}|_{t=0}\xi_t=C\diamond\xi$, for some $C\in\Omega_{\fm}^2$, then
\begin{align*}
    \frac{d}{dt}\Bigr|_{t=0}\cD(\xi(t)) = -\int_M \langle C\diamond\xi,\Div T\diamond\xi\rangle\vol_g.
\end{align*} 
    Thus, the restriction of the energy \eqref{eq: Dirichlet_energy}  to compatible $H$-structures on $(M^n,g)$ has gradient $-\Div T\diamond\xi$ at each point $\xi$.
\end{lemma}
\begin{proof}
    Using the assumption $\langle A\diamond\xi,B\diamond\xi\rangle = c\langle A,B\rangle$ for all $A,B\in\Omega_{\fm}^2(M)$, one can either derive the result immediately from Corollary \ref{cor: first_iso_variation}, or adopt the following direct proof. Since $C\in\Omega_{\fm}^2$ [Lemma \ref{lemma: isometric variation}], i.e. the variation is isometric, integration by parts gives immediately
\[
    \frac{d}{dt}\Bigr|_{t=0}\cD(\xi(t)) 
    = \int_M\langle \nabla(C\diamond\xi),\nabla\xi\rangle\vol_g 
    = -\int_M \langle C\diamond\xi, \Delta\xi\rangle\vol_g.
\] 
    Using again $C\in\Omega_{\fm}^2$ and the decomposition of $\Delta\xi$ given by Lemma \ref{lem: basic_estimate_harmonic}, together with the orthogonality given by Lemma \ref{lem: orthogonality}, we conclude:
\begin{align*}
    \frac{d}{dt}\Bigr|_{t=0}\cD(\xi(t)) 
    &= -\int_M \langle C\diamond\xi,\Div T\diamond\xi + T_l\diamond(T_l\diamond\xi)\rangle \vol_g\\
    &= -\int_M \langle C\diamond\xi,\Div T\diamond\xi\rangle \vol_g.
    \qedhere
\end{align*}
\end{proof}
This motivates a natural harmonicity theory for arbitrary $H$-structures, initially proposed in \cite{loubeau-saearp}:
\begin{definition}
    Let $(M^n,g)$ be an oriented Riemannian $n$-manifold admitting a compatible $H$-structure. A family of compatible $H$-structures $\{\xi(t)\}_{t\in I}$ on $(M,g)$, parameterised by a non-degenerate interval $I\subset\mathbb{R}$, is a solution to the \emph{harmonic flow of $H$-structures} (or \emph{harmonic $H$-flow} for short) if the following evolution equation holds for every $t\in I$:
    \begin{equation}
    \label{eq: harmonic_flow}
    \tag{HF}
        \frac{\partial}{\partial t}\xi(t)=\Div T(t)\diamond \xi(t),
    \end{equation} 
    where $T(t)$ denotes the torsion of $\xi(t)$. Given a compatible $H$-structure $\xi_0$ on $(M^n,g)$, a solution to the harmonic flow of $H$-structures with \emph{initial condition} (or \emph{starting at}) $\xi_0$ is a solution of \eqref{eq: harmonic_flow} defined for every $t\in [0,\tau_0)$, for some $0<\tau_0\leqslant\infty$, and such that $\xi(0)=\xi_0$.
\end{definition}
\begin{definition}
    Let $(M^n,g)$ be an oriented Riemannian $n$-manifold admitting a compatible $H$-structure $\xi$. We say that $\xi$ is \emph{harmonic} when it has divergence-free torsion:
    \[
    \Div_g T = 0.
    \]
\end{definition}

When $M$ is moreover closed, then under the assumptions of Lemma \ref{lem: gradient_dirichlet_along_iso_var} the harmonic $H$-flow \eqref{eq: harmonic_flow} is the negative gradient flow of the energy functional $\cD$ \eqref{eq: Dirichlet_energy} restricted to compatible $H$-structures on $(M^n,g)$. Furthermore, the critical points of the latter are precisely the harmonic $H$-structures. 

Alternatively, one can describe the harmonic flow and the harmonicity condition for $H$-structures viewed as sections of the bundle $\pi:\Fr(M,g)/H\to M$. Denoting by $\omega\in\Omega^1(\Fr(M,g),\mathfrak{so}(n))$ the connection $1$-form associated to the Levi--Civita connection $\nabla$, the tangent bundle of $\Fr(M,g)$ splits as the sum $\ker(\pi_{\mathrm{SO}(n)})_{\ast}\oplus \ker\omega$. Then, considering the principal $H$-bundle $\pi_H:\Fr(M,g)\to\Fr(M,g)/H$ and writing $N:=\Fr(M,g)/H$, we have the following decomposition of $TN$ into vertical and horizontal distributions:
\begin{align*}
    TN = \cV &\oplus\cH,\quad\text{where}\\ 
    \cV :=(\pi_H)_{\ast}(\ker(\pi_{\mathrm{SO}(n)})_{\ast})\quad &\text{and}\quad\cH:=(\pi_H)_{\ast}(\ker\omega).
\end{align*} 
Note that, with respect to the bundle projection $\pi:N\to M$, we have $\cV=\ker\pi_{\ast}$ and $\pi_{\ast}\cH = TM$, and there is a canonical isomorphism of vector bundles $\mathcal{I}$, from $\cV$ to the vector bundle $\underline{\mathfrak{m}}$ associated to $\pi_H:\Fr(M,g)\to N$ with fibre $\mathfrak{m}$, see \cite[§1.1]{loubeau-saearp}). Together with the Riemannian metric $g$ on $M$, and the natural bi-invariant metric on $\mathfrak{m}=\mathfrak{h}^{\perp}\subset\mathfrak{so}(n)$, $\cI$ induces a metric $\eta$ on $N$ by
\[
\eta(A,B):=\langle\pi_{\ast}A,\pi_{\ast}B\rangle + \langle\mathcal{I}(\mathrm{proj}_{\cV}(A)),\mathcal{I}(\mathrm{proj}_{\cV}(B))\rangle.
\] 
When $(M^n,g)$ is closed, the following Dirichlet energy is defined on compatible $H$-structures on $(M^n,g)$, viewed as homogeneous sections $\sigma\in\Gamma(\Fr(M,g)/H)$: 
\begin{equation}
\label{eq: Dirichlet energy}
    E(\sigma):=\frac{1}{2}\int_M |d^\cV\sigma|_\eta^2 \vol_g,
\end{equation}
where $d^\cV\sigma$ is the projection of $d\sigma$ onto the distribution $\cV=\ker\pi_*\subset TN$. In \cite[Proposition 4]{loubeau-saearp} it was shown that the critical point set of \eqref{eq: Dirichlet energy} is the vanishing locus of the \emph{vertical tension} field $\tau^\cV(\sigma):=\tr_g\nabla^{\cV} d^\cV\sigma$, where $\nabla^\cV$ is the vertical part of the Levi--Civita connection of $(N,\eta)$. Moreover, the negative gradient flow associated to the Dirichlet energy \eqref{eq: Dirichlet energy}, named as the \emph{harmonic section flow}, is given by:
\begin{equation}
\label{eq: HSF}
\tag{HSF}
    \frac{\partial}{\partial t} \sigma = \tau^\cV (\sigma).
\end{equation} On the other hand, it was shown in \cite[Theorem 3.3]{Gonzalez-Davila2009} that %\footnote{Our convention for the isomorphism $\mathcal{I}$ differs from the analogous identification $\phi$ in the reference \cite{Gonzalez-Davila2009} by a minus sign.} $\mathcal{I}(d^\cV\sigma)=T$ and
$|d^\cV\sigma|_\eta^2=|T|_g^2$, where %here we identify $\sigma^{\ast}\underline{\mathfrak{m}}\cong\mathfrak{m}_{\sigma}$, and
$T$ is the intrinsic torsion of $\sigma$. Moreover, by \cite[Theorem 3.6]{Gonzalez-Davila2009}, $\sigma$ is a critical point of \eqref{eq: Dirichlet energy} if and only if $\Div T=0$. In fact,  {since the universal section $\Xi$ establishes the one-to-one correspondence \eqref{eq: Xi correspondence} between homogeneous sections and $H$-invariant (muti-)tensors}, under the assumptions of Lemma \ref{lem: gradient_dirichlet_along_iso_var} there corresponds,  to each solution $\{\sigma(t)\}_{t\in I}$ of \eqref{eq: HSF},  a solution $\{\xi(t)\}_{t\in I}$ of \eqref{eq: harmonic_flow}.

\subsection{Solitons of general flows}
\label{sec: solitons}

Let $M^n$ be an oriented manifold admitting a geometric $H$-structure. We will now formulate a general theory of solitons and self-similarity for arbitrary $H$-flows \eqref{eq: general_H_flow}. In particular, by encompassing non-isometric flows (with $S\neq0$), this framework substantially expands -- while building heavily upon --  the studies of isometric/harmonic  $\rG_2$-solitons in \cite{dgk-isometric}*{\textsection2.5}, $\S7$-solitons in \cite{Dwivedi-Loubeau-SaEarp2021}*{\textsection2.1.3}, and $\mathrm{Sp}(\frac{n}{4})\mathrm{Sp}(1)$-solitons in \cite{udhav2022quaternionic}.

Suppose we have a map $A:\xi\mapsto A(\xi)$ assigning to each $H$-structure $\xi$ on $M$, with induced metric $g$ and corresponding  decomposition \eqref{eq:splitting TM* x TM-GENERAL}, a tensor 
$$
A(\xi)=S(\xi)+C(\xi)\in\Sigma^2(M)\oplus\Omega_{\fm}^2(M),
$$ 
completely determined by $\xi$ and its associated structures, satisfying the following assumptions:
\begin{equation}
\label{eq: assumption_flow}
\begin{cases}
    f^{\ast}A(\xi) = A(f^{\ast}\xi),
    &\forall f\in\mathrm{Diff}(M),
    \quad\text{(diffeomorphism equivariance)}\\
    A(\lambda\xi)\diamond(\lambda\xi) = \lambda^{\alpha}A(\xi)\diamond\xi,
    &\text{for some $\alpha\in\mathbb{R}$, and for all $\lambda>0$}.
    \quad\text{(scaling property)}
\end{cases}
\end{equation}
For example, it is easy to see that the Ricci curvature $A=S=\mathrm{Ric}(g)$ satisfies these conditions with $\alpha=0$; the diffeomorphism equivariance is clear, and the scaling property follows from Lemma \ref{lem: properties_diamond} --\ref{item: diamond_S}:
\[
\mathrm{Ric}(\lambda g)\diamond(\lambda g) 
= 2\mathrm{Ric}(\lambda g)
=2\mathrm{Ric}(g)
=\mathrm{Ric}(g)\diamond g,\quad\forall\lambda>0.
\] 
As another instance, the divergence of torsion map $A:\xi\mapsto\Div T(\xi)$ also satisfies such conditions with $\alpha=1-2/\ell$, where $\ell$ is the net degree of $\xi$, which we shall assume throughout to be \emph{nonzero}; for the scaling property,  \eqref{eq: nabla of xi-GENERAL} implies $T(\lambda\xi)=T(\xi)$, and since the associated Riemannian metrics satisfy $g_{\lambda}=\lambda^{2/\ell}g$, we have $\Div_{g_{\lambda}} T(\lambda\xi) = \lambda^{-2/\ell}\Div_g T(\xi)$, see e.g. the proof of Lemma \ref{lem: rescaling}. It is also easy to check conditions \eqref{eq: assumption_flow} for the more general $S$ arising in Example \ref{ex: general_Dirichlet_flow_G2}, along the general gradient flow of the Dirichlet energy \eqref{eq: energy_T}.

For $A:\xi\mapsto A(\xi)$ as above, satisfying \eqref{eq: assumption_flow}, we shall consider the induced flow of $H$-structures on $M$ given by \eqref{eq: general_H_flow}:

\begin{equation}
\label{eq: general_H_flow_with_assumptions}
    \frac{\partial}{\partial t}\xi(t) = A(\xi(t))\diamond\xi(t) = S(\xi(t))\diamond\xi(t) + C(\xi(t))\diamond\xi(t). 
\end{equation}
For simplicity, in what follows we assume that $H$ is the stabiliser of a single tensor, with net degree $\ell$.
\begin{definition}[Self-similar solutions]
\label{def: self-similar}
    Let $\{\xi(t)\}_{t\in I\ni 0}$ be a solution to the flow \eqref{eq: general_H_flow_with_assumptions}. We say that $\xi(t)$ is a \emph{self-similar solution} if there exist a family of diffeomorphisms $\{f_t:M\to M\}_{t\in I}$, with $f_0=\mathrm{Id}_M$, and a smooth function $\rho:I\to\mathbb{R}\setminus\{0\}$ with $\rho(0)=1$, such that
    \[
    \xi(t) = \rho(t)^{\ell}f_t^{\ast}\xi(0),\quad\forall t\in I.
    \] In this case, we define the \emph{stationary vector field} of $\xi(t)$ by $X_t:=(f_t^{-1})_{\ast}W_t\in\sX(M)$, where $W_t\in\sX(M)$ is the infinitesimal generator of $f_t$, i.e. $\partial_t f_t = W_t f_t$.
\end{definition}

\begin{lemma}
\label{lem: self-similar}
    Let $\{\xi(t)\}_{t\in I}$ be a self-similar solution of the flow \eqref{eq: general_H_flow_with_assumptions}, with stationary vector field $X_t$, as in Definition \ref{def: self-similar}. Then the Riemannian metric  associated to each $\xi(t)$ is given by
\begin{equation}
\label{eq: self-similar_metric}
    g(t)=\rho(t)^2f_t^{\ast}g(0)
\end{equation}
    and, for each $t\in I$, its Lie derivative along the stationary vector field is
\begin{equation}
\label{eq: self-similar_metric_evo}
    \cL_{X_t} g(t) = -2\rho'(t)\rho(t)^{-1}g(t) + 2S(\xi(t)).
\end{equation} 
    Moreover, the torsion $T(t)$ of $\xi(t)$ satisfies the \emph{stationary condition}:
\begin{equation}
\label{self-similar_C_evo}
    C(\xi(t)) = X_t\lrcorner T(t) + \nabla_{\fm}(X_t).
\end{equation}
\end{lemma}
\begin{proof}
    Start noting that, as in the proof of Lemma \ref{lemma: isometric variation}, for $u\in Q_{\xi(0)}$ we have $\xi(t)=(\rho(t)uf_t).\xi_o$, thus 
    \[
    g(t)=(\rho(t)uf_t).g_o=(\rho(t)f_t).(u.g_o)=\rho(t)^2f_t^\ast g(0).
    \] Now,
    using Lemma \ref{lemma: isometric variation} and \eqref{eq: self-similar_metric}, we compute:
\begin{align*}
    2S(\xi(t)) = \frac{\partial}{\partial t} g(t)
    &= 2\rho(t)\rho'(t)f_t^{\ast}g(0) + \rho(t)^2f_t^{\ast}\cL_{W_t}g(0)\\
    &= 2\rho'(t)\rho(t)^{-1}g(t) + \cL_{X(t)}g(t).
\end{align*} 
    Using the self-similarity of $\xi(t)$, equations \eqref{eq: Lie of xi-GENERAL} and \eqref{eq: self-similar_metric_evo}, and Lemma \ref{lem: properties_diamond}-\ref{item: diamond_g} we have:
\begin{align*}
    S(\xi(t))\diamond\xi(t) +C(\xi(t))\diamond\xi(t) 
    &= \frac{\partial}{\partial t}\xi(t) 
    = \ell\rho'(t)\rho(t)^{-1}\xi(t) + \cL_{X_t}\xi(t)\\
    &= \ell\rho'(t)\rho(t)^{-1}\xi(t) + \left(X_t\lrcorner T(t) + \frac{1}{2}\cL_{X_t}g(t) + \nabla_{\fm}(X_t)\right)\diamond\xi(t)\\
    &= S(\xi(t))\diamond\xi(t)+\left(X_t\lrcorner T(t) + \nabla_{\fm}(X_t)\right)\diamond\xi(t),
\end{align*} which, together with Lemma \ref{lem: ker_diamond} implies equation \eqref{self-similar_C_evo}.
\end{proof}

\begin{definition}[Solitons]\label{def solitons}
    A \emph{soliton} for the flow \eqref{eq: general_H_flow_with_assumptions} is a triple $(\xi,X,c)$ consisting of an $H$-structure $\xi$, a vector field $X\in\sX(M)$ and a constant $c\in\mathbb{R}$, such that
    \begin{equation}
    \label{eq: soliton}
        \begin{cases}
            \cL_X g = cg + 2S(\xi),\\
            C(\xi) = X\lrcorner T + \nabla_{\fm}(X),
        \end{cases}
    \end{equation} 
    where $g$ is the Riemannian metric induced by $\xi$, and $T$ denotes the torsion of $\xi$.
    When the scaling constant $\alpha$ from \eqref{eq: assumption_flow}
    satisfies $\ell(\alpha-1)<0$, the soliton $(\xi,X,c)$ is called \emph{shrinking}, \emph{steady} or \emph{expanding}, according to  whether $c>0$, $c=0$ or $c<0$ respectively.
\end{definition}

The adjectives qualifying soliton evolution in Definition \ref{def solitons} are proposed by analogy with the theory of Ricci flow. As seen earlier in this section, when $A=\Ric(g)$ one has $\ell=2$ and $\alpha=0$, which agrees with the  `sign' convention $\ell(\alpha-1)<0$.
\begin{example}\label{ex: harmonic_soliton}
    For the harmonic flow of $H$-structures, i.e., when $S(\xi)=0$ and $C(\xi)=\Div T$, we have $\alpha=1-2/\ell$ and thus $\ell(\alpha-1)=-2<0$, and \eqref{eq: soliton} agrees with the definition of \emph{harmonic solitons} $(\xi,X,c)$ \cite[Definition 2.10]{Dwivedi-Loubeau-SaEarp2021} (see also \cite[Definition 2.16]{dgk-isometric}):
    \begin{equation}%\label{eq: soliton}
        \begin{cases}
            \cL_X g = cg,\\
            \Div T = X\lrcorner T + \nabla_{\fm}(X).
        \end{cases}
    \end{equation}
\end{example}
\begin{example}\label{ex: Ricci_soliton}
    For the Ricci flow of $H$-structures considered in Example \ref{ex: Ricci flow}, we have $S(\xi)=-\mathrm{Ric}(g)$, $C(\xi)=0$, and $\alpha=0$. Thus a triple $(\xi,X,c)$ is a soliton for this flow when:
    \begin{equation}\label{eq: Ric soliton}
        \begin{cases}
            \cL_X g = cg - 2\mathrm{Ric}(g),\\
            X\lrcorner T + \nabla_{\fm}(X) = 0.
        \end{cases}
    \end{equation} Of course, when $H=\mathrm{SO}(n)$ and the structure is just a Riemannian metric $\xi=g$, we have $T=0$ and $\mathfrak{m}=0$, and \eqref{eq: Ric soliton} agrees with the definition of Ricci solitons.
\end{example}
By Lemma \ref{lem: self-similar}, every self-similar solution $\{\xi(t)\}_{t\in I\ni 0}$ of the flow \eqref{eq: general_H_flow_with_assumptions}, as in Definition \ref{def: self-similar}, induces the soliton $(\xi(0),X_0,-2\rho'(0))$. When $\ell(\alpha-1)<0$, note that this soliton is `shrinking', i.e., $c:=-2\rho'(0)>0$, exactly when $\rho'(0)<0$, which explains the choice of wording. For the converse, we argue as in \cite[Section 4.4]{Lauret2016}:

\begin{proposition}
\label{prop: soliton_induce_self-similar}
   Every soliton of the flow \eqref{eq: general_H_flow_with_assumptions} induces a self-similar solution.
\end{proposition}
\begin{proof}
    Let $(\xi,X,c)$ be a soliton of \eqref{eq: general_H_flow_with_assumptions}. For the scaling factor of self-similarity, we need a smooth function $\rho:I\to\mathbb{R}\setminus\{0\}$ such that 
    $$
    (\rho(t)^{\ell})'=-\frac{c}{2}\ell(\rho(t)^{\ell})^{\alpha}
    \qandq \rho(0)=1,
    $$
    which can be solved precisely by
\[
\rho(t)=\left\{
\begin{array}{cc}
    \left(1+\frac{\ell(\alpha-1)}{2}ct\right)^{-\frac{1}{\ell(\alpha-1)}}, & \text{for} \quad \alpha\neq 1,\\
    e^{-\frac{c}{2}t}, &\text{for} \quad \alpha=1.
\end{array}\right.
\] 
\begin{multicols}{2}
    When $\ell(\alpha-1)<0$, the maximal definition interval $I$ of the scaling factor, for each soliton type, is given by the following table, along with the corresponding age. 
    \columnbreak
\[\begin{tabular}{c|c|c|c}
    $c$ & type & $I=\Dom(\rho)$ & age\\\hline
    $c>0$ & shrinking &  $(-\infty,-\frac{2}{\ell(\alpha-1)c})$ & ancient\\
    $c=0$ & steady & $(-\infty,\infty)$ & eternal\\
    $c<0$ & expanding & $(-\frac{2}{\ell(\alpha-1)c},\infty)$ & immortal
\end{tabular}\]
\end{multicols}
     
    In each case, we define the time-dependent smooth vector field
\[
W_t:=\rho(t)^{\ell(\alpha-1)}X,
\] 
    and let $f_t:M\to M$ be the corresponding $1$-parameter family of diffeomorphisms. Now define the $H$-structures $\xi(t):=\rho(t)^{\ell}f_t^{\ast}\xi$. We claim that $\xi(t)$ is a (self-similar) solution to  \eqref{eq: general_H_flow_with_assumptions} for every $t\in I$. Indeed, using  \eqref{eq: Lie of xi-GENERAL}, Lemma \ref{lem: properties_diamond}-\ref{item: diamond_g} and the soliton equations \eqref{eq: soliton}, we get:
\begin{align*}
    \frac{\partial}{\partial t}\xi(t) 
    &= (\rho(t)^{\ell})'f_t^{\ast}\xi + \rho(t)^{\ell}f_t^{\ast}\cL_{W_t}\xi =\rho(t)^{\ell\alpha}\left(-\frac{c}{2}\ell f_t^{\ast}\xi + f_t^{\ast}\cL_{X}\xi\right)\\
    &=\rho(t)^{\ell\alpha}f_t^{\ast}\left(-\frac{c}{2}\ell\xi  + \frac{1}{2}(\cL_X g\diamond\xi) + (X\lrcorner T + \nabla_{\fm} X)\diamond\xi\right)\\
    &=(\rho(t)^{\ell})^\alpha f_t^{\ast}(A(\xi)\diamond\xi)\\
    &= A(\xi(t))\diamond\xi(t),
\end{align*} 
    where in the last line we used the key assumption \eqref{eq: assumption_flow}.
\end{proof}
As a corollary of the above proof, we note that if $M$ is closed and the flow \eqref{eq: general_H_flow_with_assumptions} is the $\pm$-gradient flow of some functional which is invariant under diffeomorphisms, then \emph{steady solitons are critical points of the functional}; e.g., in the case where $S=0$ and $C=\Div T$, which induce the negative gradient flow of $\mathcal{D}$ restricted to isometric $H$-structures (by Lemma \ref{lem: gradient_dirichlet_along_iso_var}), the steady solitons are harmonic structures (generalising e.g. \cite[Remark 2.18]{dgk-isometric}).
\begin{corollary}
    Suppose furthermore that $M^n$ is a closed manifold and that the flow \eqref{eq: general_H_flow_with_assumptions}, satisfying the assumptions \eqref{eq: assumption_flow}, is the $\pm$-gradient flow of some functional $E$ on a certain subspace $\mathcal{H}$ of $H$-structures on $M$, satisfying $E(f^{\ast}\xi) = E(\xi)$ for any $f\in\mathrm{Diff}(M)$; so suppose we have $A(\xi(t))=(\pm)\mathrm{grad}(E)(\xi(t))$, with respect to some natural induced metric on $\mathcal{H}$. Then steady solitons of \eqref{eq: general_H_flow_with_assumptions} are critical points of the functional $E$.
\end{corollary}
\begin{proof}
    If $(\xi,X,c)$ is a steady soliton for the induced flow \eqref{eq: general_H_flow_with_assumptions}, then by the above proof of Proposition \ref{prop: soliton_induce_self-similar} it induces a self-similar solution of the form $\xi(t)=f_t^{\ast}\xi$, defined for all $t\in\mathbb{R}$, for some family $f_t\in\mathrm{Diff}(M)$ with $f_0=\mathrm{Id}_M$. It then follows by the invariance of $E$ under diffeomorphisms that $E(\xi(t))=E(\xi)$ for all $t\in\mathbb{R}$, so that taking derivatives at $t=0$ we get $\mathrm{grad}(E)(\xi)=0$, i.e., $\xi$ is a critical point of $E$.
\end{proof}

 We also note the following:
\begin{lemma}\label{lem: no_compact_solitons}
    Suppose that the flow of $H$-structures \eqref{eq: general_H_flow_with_assumptions}, satisfying the assumptions \eqref{eq: assumption_flow}, is isometric, i.e. $S\equiv 0$ (e.g. as in the case of the harmonic $H$-flow), and suppose the underlying metric $g$ on $M$ is complete. If there is a shrinking or expanding soliton of the flow, then $(M^n,g)$ must be isometric to the Euclidean space $(\mathbb{R}^n,g_{\circ})$. In particular, there is no compact shrinking or expanding soliton of such an isometric flow.
\end{lemma}
\begin{proof}
    This follows from the classical result that a complete Riemannian $n$-manifold $(M,g)$ admitting a nontrivial homothetic vector field, i.e. a vector field $0\neq X\in\sX(M)$ satisfying $\mathcal{L}_X g = cg$, for some constant $c\neq 0$, must be isometric to the Euclidean space $(\mathbb{R}^n,g_{\circ})$ (see e.g. \cite{tashiro1965complete}).
\end{proof}

For solitons in Euclidean space we have the following generalisation of \cite[Proposition 2.19]{dgk-isometric}.

\begin{lemma}
\label{lem: soliton_in_Rn}
    Suppose that $(\xi,X,c)$ is a soliton for the  {isometric flow of $H$-structures 
    \eqref{eq: general_H_flow_with_assumptions} with $S\equiv 0$, satisfying the assumptions \eqref{eq: assumption_flow}, on $M=\mathbb{R}^n$ and with} $\xi$ compatible with the Euclidean metric $g_{\circ}$ \eqref{eq: flat_metric}. Then the components of $X=X_i\partial_i$ have the form
\begin{equation}\label{eq: soliton_solution}
X_i = \frac{c}{2}x^i + \sum_{\substack{1\leqslant j\leqslant n\\ j\neq i}}a_{ij}x^j + b_i,    
\end{equation}
    where $[a_{ij}]$ is a skew-symmetric matrix and $b_i\in\mathbb{R}$. 
\end{lemma}
\begin{proof}
    For $g=g_{\circ}$, the equation $\cL_X g = cg $ becomes $\partial_i X_j + \partial_j X_i = c\delta_{ij} $,  {and its solutions are of the form \eqref{eq: soliton_solution}}.
\end{proof}
\begin{example}
\label{ex: special_solitons}
    For the harmonic flow of $H$-structures, as in  Example \ref{ex: harmonic_soliton}, Lemma \ref{lem: soliton_in_Rn} shows that if $(\xi,X,c)$ is a harmonic $H$-soliton on $(\mathbb{R}^n,g_{\circ})$ then
    \[
    X(x)=\frac{c}{2}x + X_0(x),
    \qwithq
    X_0(x)=\sum_{i,j} a_{ij} x^j\partial_i + \sum_i b_i\partial_i,
    \]
    where $[a_{ij}]$ is skew-symmetric, therefore $\cL_{g_{\circ}}X_0 = 0$ and so $X_0$ is a Killing vector field. A special case is when $X_0(x)=-\frac{c}{2}x_0$ for some $x_0\in\mathbb{R}^n$, in which case $\cL_X g_{\circ} = cg_{\circ}$ is trivially satisfied, and also
    \[
    \nabla_{\fm}(X)=\frac{c}{2}\pi_{\fm}(dx^i\otimes\partial_i)=0,
    \] since $dx^i\otimes\partial_i\in\Omega^0\subset\Sigma^2(M)\subset\Gamma(\mathrm{End}(TM))$. Then solitons $(\xi,X,c)$ for which $X_0=0$ reduce to pairs $(\xi,c)$ satisfying the particular stationary condition
\begin{equation}
\label{eq: special_soliton}
    \Div T = \frac{c}{2}(x-x_0)\lrcorner T.    
\end{equation}
\end{example}

\section{The harmonic flow of $H$-structures}\label{sec: harmonic_flow}

Throughout this section, we assume that $H=\mathrm{Stab}_{\SO(n)}(\xi_{\circ})$, where $\xi_{\circ}$ is an element of a $r$-dimensional $\SO(n)$-submodule $V\leqslant \oplus \cT^{p,q}(\mathbb{R}^n)$, where $V=V_1\oplus\ldots\oplus V_k$ with $V_i\leqslant \cT^{p_i,q_i}(\mathbb{R}^n)$. Thus, a compatible $H$-structure on an oriented Riemannian $n$-manifold $(M^n,g)$ is simply a geometric structure $\xi$ modelled on $\xi_{\circ}$, i.e. $\xi\in\Gamma(\mathcal{F})$, where $\mathcal{F}\subset\oplus\cT^{p,q}(TM)$ is a rank $r$ vector subbundle with typical fibre $V$, such that for all $x\in M$ there is $u\in\Fr(M,g)_x$ with $u.\xi_{\circ}=\xi$. Furthermore, we shall assume that $H$ is such that we have $\lambda_1=\ldots=\lambda_k$ in Lemma \ref{lem: orthogonality}, i.e. we assume there is $c>0$ such that
\begin{equation}
\label{eq: key_assumption}
    \langle A\diamond\xi,B\diamond\xi\rangle = c\langle A,B\rangle,\quad\forall A,B\in\Omega_{\fm}^2(M).
\end{equation} This holds, for instance, when $\fm=\fh^{\perp}\subset\fso(n)$ is an irreducible $H$-module. In fact, all of these conditions are satisfied by the main examples $U(m)\subset\SO(2m)$, $\rm G_2\subset\SO(7)$ and $\rm Spin(7)\subset\SO(8)$ that illustrated Section \ref{sec: H-structures}, as well as by the quaternion-K\"ahler case $\Sp(k)\Sp(1)\subset\SO(4k)$, specifically studied in \cite{udhav2022quaternionic}, and also by the trivial subgroup case $H=\{1\}\subset\SO(n)$, see Example \ref{ex: trivial_subgroup}.
 {By contrast, recall from Example \ref{ex: inner_prod_SU} that this is not the case for $\SU(m)$, which has $\lambda_1(m)\neq\lambda_2(m)$ for all $m\geqslant 2$.}

\subsection{Review of known results: short-time existence and Shi-type estimates}
\label{sec: review}

Let $(M^n,g)$ be an oriented Riemannian $n$-manifold of bounded geometry. In §\ref{sec: preliminaries} we saw that a compatible $H$-structure $\xi$ on $(M^n,g)$ corresponds to a section $\sigma$ of $\pi:\Fr(M,g)/H\to M$. Now, there is a natural isomorphism between $\pi:\Fr(M,g)/H\to M$ and the associated bundle $\Fr(M,g)\times_{\mathrm{SO}(n)} \mathrm{SO}(n)/H$, which fibrewise is an isometry with respect to the bi-invariant metric on $\mathrm{SO}(n)$. The induced one-to-one correspondence between sections $\sigma\in\Gamma(\Fr(M,g)/H)$ and $\mathrm{SO}(n)$-equivariant maps $s:\Fr(M,g)\to\mathrm{SO}(n)/H$ identifies solutions to the harmonic section flow \eqref{eq: HSF} with $\mathrm{SO}(n)$-equivariant solutions to the classical harmonic map heat flow for maps $\Fr(M,g)\to\mathrm{SO}(n)/H$, where the target space $\mathrm{SO}(n)/H$ is considered with its normal homogeneous Riemannian manifold structure. Since the latter flow is known to have short-time existence and uniqueness of solutions, one has the following result \cite{loubeau-saearp}*{Proposition 17}:

\begin{proposition}[Short time existence]
\label{prop: short-time_existence}
   Given any smooth compatible $H$-structure $\xi_0$ on $(M^n,g)$, there is a maximal time $0<\tau(\xi_0)\leqslant\infty$ such that the harmonic $H$-flow \eqref{eq: harmonic_flow} with initial condition $\xi_0$ admits a unique smooth solution $\xi(t)$ for $t\in[0,\tau)$.
\end{proposition}

Another important fact that was also previously proved in the generality of harmonic section flows is the following Bochner-type estimate. In the following, we do not require the assumption \eqref{eq: key_assumption}.
\begin{lemma}[Bochner-type estimate]\label{lem: dif_ineq}
	There is a uniform constant $c>0$, depending only on $(M,g)$ and $H$, such that if $\{\xi(t)\}_{t\in I}$ is a solution to the harmonic $H$-flow \eqref{eq: harmonic_flow} on $B_{r}(y)\subset (M,g)$, defined along a nondegenerate interval $I\subset\mathbb{R}$, then the following differential inequality holds on $B_{r}(y)\times I$:
\begin{equation}
\label{ineq: Bochner_estimate}
	(\partial_t-\Delta)e(\xi)\leqslant c(e(\xi)^2+1),
\end{equation}
    where $e(\xi)$ denotes either $|T|^2$ or $|\nabla\xi|^2$. Moreover, if $g$ is flat then the above estimate improves to
	\[
	(\partial_t-\Delta)e(\xi)\leqslant ce(\xi)^2.
	\]
\end{lemma}
\begin{proof}
    When $e(\xi)=|T|^2$,  the inequality \eqref{ineq: Bochner_estimate} follows directly from \cite[Lemma 2.15]{Dwivedi-Loubeau-SaEarp2021}, which actually gives:
\[
\frac{1}{2}(\partial_t - \Delta)|T|^2 \leqslant c(|T|^4 + 1) -|\nabla T|^2.
\] 
    We shall work out a separate direct proof for the case $e(\xi)=|\nabla\xi|^2$, based on Lemmas \ref{lem: torsion_equiv_nablaxi} and \ref{lem: basic_estimate_harmonic}. We start by computing the Laplacian, at the centre of normal coordinates, using the Ricci identity \eqref{eq: Ricci_identity}:
\begin{align}
    \frac{1}{2}\Delta|\nabla\xi|^2 
    &= \langle\nabla_a\nabla_a(\nabla_l\xi),\nabla_l\xi\rangle + |\nabla^2\xi|^2\nonumber\\
    &= \langle\nabla_a(\nabla_l\nabla_a\xi + R_{la}\diamond\xi),\nabla_l\xi\rangle + |\nabla^2\xi|^2\nonumber\\
    &= \langle\nabla_l\nabla_a\nabla_a\xi + 2R_{la}\diamond\nabla_a\xi + \nabla_a R_{la}\diamond\xi,\nabla_l\xi\rangle + |\nabla^2\xi|^2.\label{eq: Bochner_part_1}
\end{align} 
    When $\xi=\xi(t)$ is a solution of the flow \eqref{eq: harmonic_flow}, it follows from Lemmas \ref{lem: torsion_equiv_nablaxi} and \ref{lem: basic_estimate_harmonic} that
\begin{equation}\label{eq: rough_laplacian_of_harm_flow_sol}
\nabla_a\nabla_a\xi = \partial_t\xi + T_k\diamond(\nabla_k\xi),
\end{equation} which implies
\begin{align}
    \langle \nabla_l\nabla_a\nabla_a\xi,\nabla_l\xi\rangle &= \frac{1}{2}\partial_t|\nabla\xi|^2 + \langle\nabla_l T_k\diamond\nabla_k\xi + T_k\diamond\nabla_l\nabla_k\xi,\nabla_l\xi\rangle\nonumber\\
    &= \frac{1}{2}\partial_t|\nabla\xi|^2+\frac{1}{2}\langle\left(\nabla_lT_k-\nabla_kT_l\right)\diamond\nabla_k\xi,\nabla_l\xi\rangle + \langle T_k\diamond\nabla_l\nabla_k\xi,\nabla_l\xi\rangle,\label{eq: Bochner_part_2}
\end{align} since $\langle\nabla_l T_k\diamond\nabla_k\xi,\nabla_l\xi\rangle = -\langle \nabla_k\xi,\nabla_lT_k\diamond\nabla_l\xi\rangle$, 
by Lemma \ref{lem: properties_diamond}--\ref{item: diamond_adjoint}. Using the Bianchi-type identity (Corollary \ref{cor: m-part_Bianchi_identity}), together with the inequality $|T|\leqslant c|\nabla\xi|$ [by Lemma \ref{lem: torsion_equiv_nablaxi}], and the fact that the curvature is uniformly bounded, we get
\begin{align*}
   |\langle\left(\nabla_lT_k-\nabla_kT_l\right)\diamond\nabla_k\xi,\nabla_l\xi\rangle| 
   &= |\langle \left(2[T_k,T_l] -\pi_\fm([T_k,T_l]) -\pi_\fm(R_{kl})\right)\diamond\nabla_k\xi,\nabla_l\xi\rangle|\\
   &\leqslant (|2[T_k,T_l]|+|\pi_\fm([T_k,T_l])|+|\pi_\fm(R_{kl})|)|\nabla_k\xi||\nabla_l\xi|\\
   &\leqslant c|\nabla\xi|^4+c|\nabla \xi|^2
\end{align*}
Moreover, using Young's inequality and again the inequality $|T|\leqslant c|\nabla\xi|$, we have
\begin{equation}
\label{ineq: Bochner_part_3}
    |\langle T_k\diamond\nabla_l\nabla_k\xi,\nabla_l\xi\rangle| \leqslant c|T||\nabla^2\xi||\nabla\xi| \leqslant |\nabla^2\xi|^2 + c|\nabla\xi|^4.
\end{equation} 
    Finally, since $(M,g)$ has bounded geometry, we have 
\begin{equation}
\label{ineq: Bochner_part_4}
    |\langle 2R_{la}\diamond\nabla_a\xi + \nabla_a R_{la}\diamond\xi,\nabla_l\xi\rangle|\leqslant c |\nabla\xi|^2 + c|\nabla\xi|\leqslant c(|\nabla\xi|^2 + 1).
\end{equation} 
    Combining \eqref{eq: Bochner_part_1}, \eqref{eq: Bochner_part_2}, \eqref{ineq: Bochner_part_3} and \eqref{ineq: Bochner_part_4}, we get
\[
\frac{1}{2}(\partial_t - \Delta)|\nabla\xi|^2 \leqslant c(|\nabla\xi|^4 + |\nabla\xi|^2 + 1)\leqslant c(|\nabla\xi|^4 + 1),
\] 
    as claimed. Note also from the above proof that if $g$ is flat then the constant in \eqref{ineq: Bochner_part_4} can be taken to be zero, so in this case one gets $\frac{1}{2}(\partial_t - \Delta)|\nabla\xi|^2\leqslant c|\nabla\xi|^4$. 
\end{proof}

We also recall the known Shi-type estimates along the harmonic $H$-flow \cite{Dwivedi-Loubeau-SaEarp2021}*{Proposition 2.16}.
\begin{proposition}[Shi-type estimates] 
\label{prop: Shi-estimates}
    Let $\kappa\geqslant 1$ and $\{\xi(t)\}_{t\in [0, \kappa^{-4}]}$ be a solution of the harmonic $H$-flow  \eqref{eq: harmonic_flow} on $(M^n,g)$. Assume that there are constants $B_j$, $0\leqslant j\in\mathbb{Z}$, such that
\[
    |\nabla^j Rm|\leqslant B_j \kappa^{j+2},\quad\forall j\geqslant 0.
\]
    If $|T| \leqslant \kappa$, then, for each $m\in \mathbb{N}$, there is a constant $c_m=c_m(M,g,H)$ such that 
\[
    |\nabla^{m} T| \leqslant c_m \kappa t^{-m/2},
    \quad\forall t\in \left[0, \frac{1}{\kappa^4}\right].
\] 
    If $|\nabla\xi|\leqslant \kappa$, then there exists  $c_0=c_0(M,g,H)\geqslant 1$ with the following property: for each $m\in \mathbb{N}$, there is a constant $c_m'=c_m'(M,g,H)$ such that 
\[
    |\nabla^{m}(\nabla\xi)| \leqslant c_m' \kappa t^{-m/2},
    \quad\forall t\in \left[0, \frac{1}{(c_0 \kappa)^4}\right].
\]
\end{proposition}
\begin{proof}
    The first statement, for the intrinsic torsion $T$, is clear from \cite{Dwivedi-Loubeau-SaEarp2021}*{Proposition 2.16}. The result for $\nabla\xi$ follows from the previous case by taking $c_0:=\max\{c,1\}$, where $c>0$ is such that $|T|\leqslant c|\nabla\xi|$, and using induction on $m\in\mathbb{N}$ together with Lemma \ref{lem: torsion_equiv_nablaxi}. Note that the latter implies in particular that
\begin{align*}
    \nabla(\nabla^m\nabla\xi) 
    &= \sum\limits_{p=0}^m \nabla^{m-p+1}T\circledast\nabla^p\xi + \sum\limits_{p=0}^m\nabla^{m-p}T\circledast\nabla^{p+1}\xi.
    \qedhere
\end{align*}
\end{proof}

\subsection{Parabolic rescaling}
\label{sec: parabolic_rescaling}

In this paragraph, we prove a very useful scaling property of the harmonic $H$-flow \eqref{eq: harmonic_flow}, which will be frequently used in the sequel. The next result generalises \cite{dgk-isometric}*{Lemma 2.13} and \cite{Dwivedi-Loubeau-SaEarp2021}*{Lemma 4.11}, which were proved in the contexts of $\rG_2$ and $\S7$, respectively.

\begin{lemma}[Parabolic rescaling]
\label{lem: rescaling}
    Let $(M^n,g)$ be an oriented Riemannian $n$-manifold admitting a compatible $H$-structure $\xi_0$. Let $\lambda>0$ be a constant and $\{\xi(t)\}_{t\in [0,\tau)}$ be a solution to the harmonic $H$-flow \eqref{eq: harmonic_flow} with initial condition $\xi(0)=\xi_0$. Write $\xi(t)=(\xi_1(t),\ldots,\xi_k(t))$, where each $\xi_i(t)$ is a tensor of type $(p_i,q_i)$ and \emph{net degree} $\ell_i:=q_i-p_i$. Set $\lambda^{\ell}:=(\lambda^{\ell_1},\ldots,\lambda^{\ell_k})$. Then
\begin{equation}
\label{eq: parabolic_rescaling_xi}
    \tilde{\xi}( {\tilde{t}}):=\lambda^{\ell}\cdot{\xi( {\lambda^{2} t})}=(\lambda^{\ell_1}\xi_1( {\lambda^{2} t}),\ldots,\lambda^{\ell_k}\xi_k( {\lambda^{2} t})),
    \qforq
     {\tilde{t}}:=\lambda^2t,\quad t\in [0, {\lambda^{-2}\tau}),
\end{equation} 
    defines a solution to the harmonic $H$-flow \eqref{eq: harmonic_flow} on $(M,\tilde{g}:=\lambda^2g)$ with initial condition $\tilde{\xi}(0)=\lambda^{\ell}\cdot{\xi_0}$.
\end{lemma}
\begin{proof}
    Without loss of generality, we can assume $\xi$ is a single tensor, i.e. $k=1$. Start noting that $\tilde{\nabla}=\nabla$. Next, for every $t\in [0, {\lambda^{-2}\tau})$, we know from  \eqref{eq: nabla of xi-GENERAL}  that
\[
\nabla\xi ( {\lambda^{2} t}) = T( {\lambda^{2} t})\diamond\xi( {\lambda^{2} t}),
\] 
    and multiplying by $\lambda^\ell$ we have
\[
\lambda^{\ell}(\nabla\xi)( {\lambda^{2} t}) = (\tilde{\nabla}\tilde{\xi})( {\tilde{t}}) = \tilde{T}( {\tilde{t}})\diamond\tilde{\xi}( {\tilde{t}}) = \lambda^{\ell}\tilde{T}( {\tilde{t}})\diamond\xi( {\lambda^{2} t}).
\] 
    From the injectivity of $\left(\cdot{}\diamond\xi( {\lambda^{2} t})\right)|_{\Omega_{\fm}^2}$, it follows that $\tilde{T}( {\tilde{t}}) = T( {\lambda^{2} t})$. Therefore, since $\Div_{g} T( {\lambda^{2} t})_{bc} = g^{ia}\nabla_i T_{a,bc}( {\lambda^{2} t})$, the rescaled divergence of torsion is
\[
\Div_{\tilde{g}}\tilde{T}( {\tilde{t}}) = \tilde{g}^{ia} \tilde{\nabla}_i \tilde{T}_{a;bc}( {\tilde{t}}) = \lambda^{-2}g^{ia}\nabla_i T_{a;bc}( {\lambda^{2} t}) = \lambda^{-2}\Div_g T( {\lambda^{2} t})_{bc}.
\] Finally, using  {\eqref{eq: harmonic_flow} for $\{\xi(t)\}_{t\in [0,\tau)}$}, we obtain a rescaled solution as claimed:
%\todo[inline]{AM@All: should be $\tilde\xi(t)=\xi(\lambda^2t)=\xi(\tilde t)$}
\begin{align*}
    (\partial_{ {\tilde{t}}}\tilde{\xi})( {\tilde{t}}) 
    &= \lambda^{\ell}\lambda^{-2}(\partial_t\xi)( {\lambda^{2} t}) = \lambda^{\ell}\lambda^{-2}\Div_g T( {\lambda^{2} t})\diamond\xi( {\lambda^{2} t}) \\
    &= \Div_{\tilde{g}}\tilde{T}( {\tilde{t}})\diamond\tilde{\xi}( {\tilde{t}}).
    \qedhere
\end{align*}
\end{proof}

%\begin{corollary}
%\label{cor: tilde xi scaling}    $e(\tilde\xi)=\lambda^{-2}e(\xi)$.
%\todo[inline]{HSE@All: complete the statement of the corollary}
%\end{corollary}

\subsection{Local version of an almost-monotonicity formula}
\label{subsec: local_monotonicity}

We obtain a general almost-monotonicity formula for harmonic $H$-flows on a complete Riemannian manifold $(M^n,g)$ with bounded geometry. We follow mainly the work on almost complex structures in \cite{he2021}, and the original study of harmonic maps in \cites{Struwe1988,Struwe1989}. We should also point out that the same formula was found simultaneously by a different set of authors in \cite{udhav2022quaternionic}*{\textsection5.2}, initially in the context of quaternion-Kähler structures, but ultimately in the same generality.

From now on, we let $r_M=r_M(g)>0$ be a lower bound to the injectivity radius of $(M^n,g)$ with the following properties. There is a uniform constant $c>0$ such that, for every point $y\in M$, the components $g_{ij}$ in normal coordinates $x=(x^1,\ldots,x^n)$ on the geodesic ball $B_{r_M}(y)$  satisfy:
\begin{align}
    \frac{1}{4}\delta_{ij}\leqslant g_{ij}\leqslant 4\delta_{ij},\quad\text{(as bilinear forms)}\label{ineq: metric_bounds_I}\\
    |g_{ij}-\delta_{ij}|\leqslant c|x|^2\quad\text{and}\quad |\partial_k g_{ij}|\leqslant c|x|,\label{ineq: metric_bounds_II}
\end{align} where $|x|$ is the Euclidean distance in $B_{r_M}(0)\subset T_xM\cong\mathbb{R}^n$. We note that the constants $r_M>0$ and $c>0$ can be chosen to depend only on the injectivity radius and the curvature of $g$. In particular, on flat Euclidean space $(M,g)=(\mathbb{R}^n,g_{\circ})$,  we can take any $r_M\in (0,\infty)$ and $c=0$, see e.g. \cite{Hebey2000}*{Theorem 1.3}.

Suppose that $(M^n,g)$ admits a compatible $H$-structure $\xi_0$ and let $\{\xi(t)\}$ be a solution to the harmonic $H$-flow \eqref{eq: harmonic_flow}, with initial condition $\xi(0)=\xi_0$ and maximal interval of existence and uniqueness $[0,\tau)$, cf. Proposition \ref{prop: short-time_existence}. Then, restricted to the geodesic ball $B_{r_M}(y)$, we can regard $\xi$ in normal coordinates as a tensor on $B_{r_M}(0)\times[0,\tau)\subset\mathbb{R}^n\times[0,\tau)$. Now fix any $\tau_0\in (0,\tau)$ and a cut-off function 
\[
\phi\in C_c^{\infty}(B_{r_M}(0))
\qwithq
\phi|_{B_{r_M/2}(0)}\equiv 1.
\] 
For all $t\in (0,\tau_0)$ and $0<r\leqslant\min\{\sqrt{\tau_0}/2,r_M\}$, we define the following functions associated to the energy functionals $\mathcal{E}$ and $\mathcal{D}$ (see Section \ref{sec: Dirichlet}), according to whether $e(\xi)=|T|^2$ or $|\nabla\xi|^2$, respectively:
\begin{align}
    \Theta^{(\cE,\cD)}(t) 
    &\equiv \Theta_{(y,\tau_0)}^{(\cE,\cD)}(\xi(t)) 
    := (\tau_0-t)\int_{\bR^n}e(\xi)^{(\cE,\cD)}(\cdot{},t)G_{(0,\tau_0)}(\cdot{},t)\phi^2\sqrt{\det(g)}dx,
    \label{eq: dfn_Theta}\\
    \Psi^{(\cE,\cD)}(r) 
    &\equiv \Psi_{(y,\tau_0)}^{(\cE,\cD)}(r;\xi(t)) 
    := \int^{\tau_0-r^2}_{\tau_0-4r^2}\int_{\bR^n}e(\xi)^{(\cE,\cD)}G_{(0,\tau_0)}\phi^2\sqrt{\det(g)}dxdt,
    \label{eq: dfn_Psi}
\end{align} 
where $T=T(\xi(t))$ denotes the torsion tensor of $\xi(t)$, and for any $(x_0,t_0)\in\mathbb{R}^n\times\mathbb{R}$ we denote by
\begin{equation}\label{eq: heat_kernel}
G_{(x_0,t_0)}(x,t):=(4\pi(t_0-t))^{-n/2}\exp\left(-\frac{|x-x_0|^2}{4(t_0-t)}\right)
\end{equation} 
the Euclidean backward heat kernel with singularity at $(x_0,t_0)$. We observe that the above quantities $\Theta(t)$ and $\Psi(r)$ are invariant under the parabolic rescalings of Lemma \ref{lem: rescaling}: indeed, if $\tilde{g}=\lambda^2 g$ and $\tilde{\xi}({\tilde{t}})=\lambda^{\ell}\xi( {\lambda^2t})$, as in \eqref{eq: parabolic_rescaling_xi}, one can readily check that 
\[
|\nabla\tilde{\xi}({\tilde{t}})|_{\tilde{g}}^2 = \lambda^{-2}|\nabla\xi( {\lambda^2t})|_{g}^2
\qandq
|\tilde{T}({\tilde{t}})|_{\tilde{g}}^2 = \lambda^{-2}|T( {\lambda^2t})|_g^2,
\] 
as well as $G_{(0, {\lambda^2}\tau_0)}(x, {\lambda^2t})=\lambda^{-n} G_{(0,\tau_0)}(x,t)$ and $\sqrt{\det{(\tilde{g})}} = \lambda^{n}\sqrt{\det{(g)}}$, thus concluding parabolic invariance. Moreover, in view of our hypothesis \eqref{eq: key_assumption} on the $H$-module $\fm$, the equivalence relation \eqref{eq: T_equiv_nablaxi} holds, so that 
\begin{equation}
\label{ineq: equiv_thetas}
    \Theta^{\cD}(t) = c\Theta^{\cE}(t),
\end{equation} 
and, in particular, also $\Psi^{\cD}(r)=c\Psi^{\cE}(r)$.

Our next results give general monotonicity formulas for $\Theta$ and $\Psi$ along the harmonic $H$-flow, generalising the $H=U(m)\subset\mathrm{SO}(2m)$ case originally proved in \cite[Theorems 3.1 and 3.2]{he2021}. See also Appendix \ref{ap: global_monotonicity} for yet another type of monotonicity formula along the harmonic $H$-flow, which will not be used in this paper but might attract an independent interest. For simplicity, henceforth we shall write $G:=G_{(0,\tau_0)}$.

\begin{theorem}
\label{thm: monotonicity_Theta}
    Write $\Theta(t)$ for either $\Theta^{\cE}(t)$ or $\Theta^{\cD}(t)$, as in \eqref{eq: dfn_Theta}. For any $\tau_0-\min\{\tau_0,1\}<t_1\leqslant t_2<\tau_0$ and $N>1$, the following almost-monotonicity formula holds:
\begin{equation}
\label{ineq: monotonicity_Theta}
    \Theta(t_2) \leqslant e^{c(f(t_2)-f(t_1))} \left(\Theta(t_1) +c\left(N^{n/2}(E_0+\sqrt{E_0}) +\frac{1}{\ln^2N}\right) (t_2-t_1)\right),
\end{equation}
    where $c=c(M,g)>0$ is a constant,  $E_0$ denotes the energy of $\xi(0)$ with respect to the corresponding functional, $\mathcal{E}$ or $\mathcal{D}$, and
\begin{align}
\begin{split}
    \label{def: f}
    f(t) 
    &= \hat{f}(\tau_0-t),\\
    \qwithq \hat{f}(x)
    &:=  {-x(\ln^4(x) - 4\ln^3(x) + 13\ln^2(x) - 26\ln(x) +26)}.
\end{split}
\end{align}
\end{theorem}
\begin{proof}
    In view of \eqref{ineq: equiv_thetas}, it suffices to prove the claim for $\Theta(t)=\Theta^{\cE}(t)$. Alternatively, noting that the proof below relies on the evolution equations for the torsion $T$ and the Bianchi-type identity for $T$, both of which have counterparts for $\nabla\xi$ (see Sections \ref{sec: general_flows} and \ref{sec: Bianchi}), one could follow the same arguments directly for $\Theta(t)=\Theta^{\cD}(t)$.

    We begin by differentiating $\Theta$ under the integral, then using the evolution equation \eqref{eq: ddt_norm_T} of Corollary \ref{cor: evo_torsion}, with $S=0$ and $C=\Div T$, as well as the facts that $g$ is time-independent (by Lemma \ref{lemma: isometric variation}), and $\partial_t G=\left(\frac{n}{2(\tau_0-t)}-\frac{|x|^2}{4(\tau_0-t)^2}\right)G$:
\begin{align*}
    \frac{d}{dt}\Theta(t) 
    &=-\int_M |T|^2 G\phi^2\vol_g 
    + 2(\tau_0-t)\int_M \langle\nabla(\Div T),T\rangle G\phi^2\vol_g\\
    &\quad+\int_M |T|^2\left(\frac{n}{2}-\frac{|x|^2}{4(\tau_0-t)}\right)G\phi^2\vol_g.
\end{align*}
    Let us expand the second summand, noting that
\[
\nabla(G\phi^2 \Div T) 
= (\phi^2 dG + 2G\phi d\phi)\otimes\Div T + G\phi^2\nabla(\Div T),
\] 
    and
\begin{equation}
\label{eq: grad_G}
    \nabla G := (dG)^{\sharp_g} 
    := g^{ij}(\partial_i G)\partial_j = -\frac{g^{ij}x_i\partial_j}{2(\tau_0-t)}G.
\end{equation} 
     {Using integration by parts and \eqref{eq: grad_G}, we get}
\begin{align*}
    \int_M \langle\nabla(\Div T),T\rangle G\phi^2\vol_g 
    &= -\int_M\langle\Div T, (\Div T)G\phi^2 + \phi^2 (\nabla G)\lrcorner T + 2(\nabla\phi)\lrcorner T.G\phi\rangle\vol_g\\
    &= -\int_M|\Div T|^2 G\phi^2\vol_g +\int_M\left\langle \Div T,\frac{(g^{ij}x_i\partial_j)\lrcorner T}{2(\tau_0-t)}\right\rangle G\phi^2\vol\\
    &\quad - \int_M\langle\Div T, 2(\nabla\phi)\lrcorner T.G\phi\rangle\vol_g\\
    &= -\int_M \left\vert\Div T - \frac{(g^{ij}x_i\partial_j)\lrcorner T}{2(\tau_0-t)}\right\vert^2 G\phi^2\vol_g +\int_M \left\vert\frac{(g^{ij}x_i\partial_j)\lrcorner T}{2(\tau_0-t)}\right\vert^2G\phi^2\vol_g\\
    &\quad -\int_M \left\langle\Div T, \frac{(g^{ij}x_i\partial_j)\lrcorner T}{2(\tau_0-t)}\right\rangle G\phi^2\vol_g - \int_M\langle\Div T, 2(\nabla\phi)\lrcorner T\rangle G\phi\vol_g.
\end{align*}
    The original expression can then be arranged as the sum of five terms, which we will estimate separately:
    
\begin{align}
    \frac{d}{dt}\Theta(t)
    &=-\int_M |T|^2G\phi^2\vol_g\nonumber\\
    &\quad-2(\tau_0-t)\int_M\left\vert\Div T - \frac{(g^{ij}x_i\partial_j)\lrcorner T}{2(\tau_0-t)}\right\vert^2 G\phi^2\vol_g\nonumber\\
    &\quad-2(\tau_0-t)\int_M \left(-\left\vert\frac{(g^{ij}x_i\partial_j)\lrcorner T}{2(\tau_0-t)}\right\vert^2+\left\langle\Div T, \frac{(g^{ij}x_i\partial_j)\lrcorner T}{2(\tau_0-t)}\right\rangle\right)G\phi^2\vol_g
    \nonumber\\
    &\quad-2(\tau_0-t)\int \langle \Div T,2(\nabla\phi)\lrcorner T\rangle G\phi \sqrt{\det(g)}dx\nonumber\\
    &\quad+\int_M |T|^2\left(\frac{n}{2}-\frac{|x|^2}{4(\tau_0-t)}\right)G\phi^2\vol_g\nonumber\\
    &=: I+II+III+IV+V.
    \label{eq: I_II_III_IV_V}
\end{align}

For the term $IV$, we have:
\begin{align*}
    |IV| &\leqslant 2(\tau_0-t)\int_M \left\vert \left(\Div T-\frac{(g^{ij}x_i\partial_j)\lrcorner T}{2(\tau_0-t)}\right)\phi \right\vert \left\vert 2(\nabla\phi)\lrcorner T\right\vert G\vol_g\\
    &\quad+ \int_M \left|\langle (g^{ij}x_i\partial_j)\lrcorner T,2(\nabla\phi)\lrcorner T\rangle\right| G\phi \vol_g\\
    &\leqslant (\tau_0-t)\int_M \left\vert \left(\Div T-\frac{(g^{ij}x_i\partial_j)\lrcorner T}{2(\tau_0-t)}\right) \right\vert^2 G\phi^2\vol_g +4(\tau_0-t)\int_M |T|^2 |\nabla\phi|^2G\vol_g\\
    &\quad+2\int_M|T|^2|\langle g^{ij}x_i\partial_j,\nabla\phi\rangle| G\phi\vol_g\\
    &\leqslant {-\frac12 II} + 4(\tau_0-t)\int_{M} |T|^2|\nabla\phi|^2 G \vol_g + c\int_{M} |x||T|^2|\nabla\phi| G\phi\vol_g.
\end{align*} 
    Now, given $N>1$ and a  time difference $\tau_0-t$, there are two cases to consider. If $\tau_0-t>\frac{1}{N}$, then $G<cN^{n/2}$ and $(\tau_0-t)G<cN^{n/2-1}<cN^{n/2}$. We know that $|\nabla\phi|\leqslant c$ and $\mathrm{supp}(|\nabla\phi|)\subset B_{r_M}(0)$, and that the energy $\cD(\xi(t))$ is non-increasing along the negative gradient flow \eqref{eq: harmonic_flow} [Lemma \ref{lem: gradient_dirichlet_along_iso_var}] hence $\cE(\xi(t))\leqslant cE_0$ [by  \eqref{eq: T_equiv_nablaxi}], and we have the estimate
\begin{equation*}
    |IV|\leqslant {-\frac12 II} +cN^{n/2}E_0.
\end{equation*}
    Otherwise, if $\tau_0-t\leqslant \frac{1}{N}<1$, note that $\phi|_{B_{r_M/2}(0)}\equiv 1$ and so $\nabla\phi = 0$ in $B_{r_M/2}(0)$, whereas $G$ is uniformly bounded outside $B_{r_M/2}(0)$, for all $t\in [0,\tau_0)$, so we get
\begin{equation*}
    |IV|\leqslant {-\frac12 II}+cE_0.
\end{equation*} Hence, for any $t\in (0,\tau_0)$, we have either way:
\begin{equation}
\label{eq: estimate_IV}
    |IV|\leqslant {-\frac12 II}+cN^{n/2}E_0.
\end{equation}

    As for the term $III$, integration by parts and again  \eqref{eq: grad_G} yield:
\begin{align}
    III &= \frac{1}{2(\tau_0-t)}\int_M |(g^{ij}x_i\partial_j)\lrcorner T|^2 
    G\phi^2\vol_g - \int_M\langle\Div T, (g^{ij}x_i\partial_j)\lrcorner T\rangle G\phi^2\nonumber\\
    &= \frac{1}{2(\tau_0-t)}\int_M |(g^{ij}x_i\partial_j)\lrcorner T|^2 
    G\phi^2\vol_g + \int_M\langle T, \nabla\left((G\phi^2)(g^{ij}x_i\partial_j)\lrcorner T \right)\rangle\vol_g\nonumber\\
    &= \frac{1}{2(\tau_0-t)}\int_M |(g^{ij}x_i\partial_j)\lrcorner T|^2 
    G\phi^2\vol_g + \int_M\langle T,\nabla\left((g^{ij}x_i\partial_j)\lrcorner T \right)\rangle G\phi^2\vol_g\nonumber\\
    &\quad -\frac{1}{2(\tau_0-t)}\int_M |(g^{ij}x_i\partial_j)\lrcorner T|^2 
    G\phi^2\vol_g + \int_M\langle 2(\nabla\phi)\lrcorner T, (g^{ij}x_i\partial_j)\lrcorner T\rangle G\phi\vol_g\nonumber\\
    &= \int_M\langle T,\nabla\left((g^{ij}x_i\partial_j)\lrcorner T \right)\rangle G\phi^2\vol_g + \int_M\langle 2(\nabla\phi)\lrcorner T, (g^{ij}x_i\partial_j)\lrcorner T\rangle G\phi\vol_g.
    \label{eq: III}
\end{align} Expanding the first term in \eqref{eq: III}, we have
\begin{align}
\label{eq: first_term_expansion}
    \int_M\langle T,\nabla\left((g^{ij}x_i\partial_j)\lrcorner T \right)\rangle G\phi^2\vol_g &= \int_M\langle T,(x_i(\partial_k g^{ij}) dx^k + g^{ij}dx_i)\otimes T_j\rangle G\phi^2\vol_g\nonumber\\
    &\quad+ \int_M\langle T,g^{ij}x_i\nabla T_j\rangle G\phi^2\vol_g,
\end{align} 
    and the first term in \eqref{eq: first_term_expansion} can be further expanded as
\begin{align*}
    \int_M\langle T,(x_i(\partial_k g^{ij}) dx^k + g^{ij}dx_i)\otimes T_j\rangle G\phi^2\vol_g = \int_M x_i(\partial_k g^{ij})g^{lk}\langle T_l,T_j\rangle G\phi^2\vol_g + \frac{1}{n}\int_M g^{ii}|T|^2 G\phi^2\vol_g.
\end{align*} 
    On the other hand, we can use the Bianchi-type identity \eqref{eq: m-part_Bianchi_identity} and $T_l\in\Omega_{\fm}^2(M)$ to develop the second term in \eqref{eq: first_term_expansion} as follows:
\begin{align*}
    \langle T,g^{ij}x_i\nabla T_j\rangle &= g^{la}g^{ij}x_i\langle T_l,\nabla_a T_j\rangle\\
    &= g^{la}g^{ij}x_i\langle T_l, \nabla_j T_a + [T_j,T_a] + R_{ja}\rangle\\
    &=g^{la}g^{ij}x_i\langle T_l, \nabla_j T_a\rangle + g^{la}g^{ij}x_i\langle T_l,R_{ja}\rangle, 
\end{align*} 
    since $g^{la}g^{ij}x_i\langle T_{l},[T_j,T_a]\rangle=-\tr(T_l(g^{ij}x_i\partial_j\lrcorner T)T_a)g^{la}+\tr(T_lT_a(g^{ij}x_i\partial_j\lrcorner T))g^{la}=0$. 
    Noting that $g^{la}g^{ij}x_i\langle T_l, \nabla_j T_a\rangle=\frac{1}{2}g^{ij}x_i\partial_j|T|^2$, integration by parts in $\mathbb{R}^n$ gives:
\begin{align*}
    \int_M g^{la}g^{ij}x_i\langle T_l, \nabla_j T_a\rangle G\phi^2\vol_g &= \frac{1}{2}\int_{\mathbb{R}^n} \partial_j|T|^2 x_ig^{ij} G\phi^2\sqrt{\det(g)}dx\\
    &= -\frac{1}{2}\int_M g^{ii}|T|^2G\phi^2\vol_g -\frac{1}{2}\int_M |T|^2 x_i (\partial_j g^{ij})G\phi^2\vol_g\\
    &\quad + \int_M |T|^2\frac{g^{ij} x_i x_j}{4(\tau_0-t)}G\phi^2\vol_g - \frac{1}{2}\int_M |T|^2 g^{ij} x_i\left(2\partial_j\phi + \frac{1}{2}\phi(\partial_jg_{ab})g^{ab}\right)G\phi\vol_g.
\end{align*} 
    Reinserting all of the above back in \eqref{eq: III}, we conclude that
\begin{align}
    III &= \int_M x_i(\partial_k g^{ij})g^{lk}\langle T_l,T_j\rangle G\phi^2\vol_g + \frac{1}{n}\int_M g^{ii}|T|^2 G\phi^2\vol_g\nonumber\\
    &\quad -\frac{1}{2}\int_M g^{ii}|T|^2G\phi^2\vol_g -\frac{1}{2}\int_M |T|^2 x_i (\partial_j g^{ij})G\phi^2\vol_g\nonumber\\
    &\quad + \int_M |T|^2\frac{g^{ij} x_i x_j}{4(\tau_0-t)}G\phi^2\vol_g - \frac{1}{2}\int_M |T|^2 g^{ij} x_i\left(2\partial_j\phi + \frac{1}{2}\phi(\partial_jg_{ab})g^{ab}\right)G\phi\vol_g\nonumber\\
    &\quad+\int_M g^{la}g^{ij}x_i\langle T_l,R_{ja}\rangle G\phi^2\vol_g + \int_M\langle 2(\nabla\phi)\lrcorner T, (g^{ij}x_i\partial_j)\lrcorner T\rangle G\phi\vol_g.
\end{align}
    
    Combining $III$ with $I$ and $V$, we have
\begin{align}
    I+III+V &= \left(\frac{1}{2}-\frac{1}{n}\right)\int (n-g^{ii})|T|^2G\phi^2\vol_g + \frac{1}{4(\tau_0-t)}\int (g^{ij}x_ix_j-|x|^2)|T|^2G\phi^2\vol_g \nonumber\\
    &\quad +\int_M x_i(\partial_k g^{ij})g^{lk}\langle T_l,T_j\rangle G\phi^2\vol_g -\frac{1}{2}\int_M |T|^2 x_i (\partial_j g^{ij})G\phi^2\vol_g\nonumber\\
    &\quad - \frac{1}{2}\int_M |T|^2 g^{ij} x_i\left(2\partial_j\phi + \frac{1}{2}\phi(\partial_jg_{ab})g^{ab}\right)G\phi\vol_g + \int_M\langle 2(\nabla\phi)\lrcorner T, (g^{ij}x_i\partial_j)\lrcorner T\rangle G\phi\vol_g\nonumber\\
    &\quad +\int_M g^{la}g^{ij}x_i\langle T_l,R_{ja}\rangle G\phi^2\vol_g
\end{align} 
    Now since $(M,g)$ has bounded geometry, inequalities \eqref{ineq: metric_bounds_I} and \eqref{ineq: metric_bounds_II} in normal coordinates on $B_{r_M}(y)$ imply, in particular 
\[
|n-g^{ii}|\leqslant c|x|^2, \quad
 |g^{ij}x_ix_j - |x|^2|\leqslant c|x|^4{,}\quad\text{and}\quad |\partial_k g_{ij}|\leqslant c|x|.
\]
    Therefore
\begin{align}
\label{ineq: pre_bound_I+III+V}
    |I+III+V| &\leqslant c\int_M |x||T| G\phi^2\vol_g + c\int_M |x|^2|T|^2 G\phi^2\vol_g + \frac{c}{\tau_0-t}\int_M |x|^4|T|^2 G\phi^2\vol_g\nonumber\\
    &\quad+ c\int_M|x||T|^2|\nabla\phi|G\phi\vol_g.
\end{align} 
    The last term above is bounded by $c {N^{n/2}}E_0$, by the same argument used for $IV$. Let us look individually at the remaining three terms on the right-hand side of \eqref{ineq: pre_bound_I+III+V}, for any $t\in(0,\tau_0)$. For the second term, we have
\begin{align}
    \int |x|^2|T|^2G\phi^2\vol_g &\leqslant 
    \left(\int _{|x|^2\leqslant  (\tau_0-t)\ln^2(\tau_0-t)}
    +\int _{|x|^2>  (\tau_0-t)\ln^2(\tau_0-t)}
    \right)
    |x|^2|T|^2G\phi^2\vol_g\nonumber\\
    &\leqslant (\tau_0-t)\ln^2(\tau_0-t)\int_{|x|^2\leqslant  (\tau_0-t)\ln^2(\tau_0-t)}|T|^2G\phi^2\vol_g\nonumber\\
    &\quad+\int _{|x|^2>  (\tau_0-t)\ln^2(\tau_0-t)}|x|^2|T|^2\frac{\exp{(-\ln^2(\tau_0-t)/4)}}{(4\pi(\tau_0-t))^{n/2}}\phi^2\vol_g\nonumber\\
    &\leqslant \ln^2(\tau_0-t)\Theta(t)+cE_0,\label{ineq: |x|^2|T|^2}
\end{align}
where in the last inequality we used the fact that $\frac{\exp{(-\ln^2(\tau_0-t)/4)}}{(4\pi(\tau_0-t))^{n/2}}$ is uniformly bounded for $t\in (0,\tau_0)$. Analogously, for the third term,
\begin{align*}
    \frac{1}{\tau_0-t}\int |x|^4|T|^2G\phi^2\vol_g &\leqslant \frac{1}{\tau_0-t} \left( \int _{|x|^2\leqslant  (\tau_0-t)\ln^2(\tau_0-t)}
    + \int _{|x|^2 >  (\tau_0-t)\ln^2(\tau_0-t)}\right) |x|^4|T|^2G\phi^2\vol_g\nonumber\\
    &\leqslant \ln^4(\tau_0-t)\Theta(t)
    + \frac{1}{(\tau_0-t)^{n/2+1}}\exp(-\ln^2(\tau_0-t)/4)cE_0\nonumber\\
    &\leqslant \ln^4(\tau_0-t)\Theta(t) + cE_0,
\end{align*}
where in the last line we used that $\frac{\exp{(-\ln^2(\tau_0-t)/4)}}{(\tau_0-t)^{n/2+1}}$ is uniformly bounded for $t\in (0,\tau_0)$. As to the first term on the right-hand side of \eqref{ineq: pre_bound_I+III+V}, we again distinguish two cases:  if $\tau_0-t>\frac{1}{N}$, then $G<cN^{n/2}$ and H\"older's inequality gives
\begin{align*}
    \int |x||T|G\phi^2\vol_g\leqslant cN^{n/2}\Big(\int |T|^2\vol_g\Big)^{1/2}\Big(\int |x|^2\phi^4\vol_g\Big)^{1/2}\leqslant cN^{n/2}\sqrt{E_0}.
\end{align*}
    Otherwise, if $\tau_0-t\leqslant \frac{1}{N}<1$, then using Young's inequality with  {$\varepsilon=2(\tau_0-t)\ln^4(\tau_0-t)$}, we obtain
     {
\begin{align*}
    \int |x||T|G\phi^2\vol_g &\leqslant \frac{1}{4\ln^4(\tau_0-t)}\int \frac{|x|^2}{(\tau_0-t)}G\phi^2\vol_g+(\tau_0-t)\ln^4(\tau_0-t)\int |T|^2G\phi^2\vol_g\\
    &= \frac{1}{4\ln^4(\tau_0-t)} \left( \int _{|x|^2\leqslant  (\tau_0-t)\ln^2(\tau_0-t)}
    + \int _{|x|^2 >  (\tau_0-t)\ln^2(\tau_0-t)}\right) \frac{|x|^2}{(\tau_0-t)}G\phi^2\vol_g+\ln^4(\tau_0-t)\Theta(t)\\
    &\leqslant \frac{c}{\ln^2N}+\frac{c}{\ln^4N}\frac{\exp(-\ln^2(\tau_0-t)/4)}{(\tau_0-t)^{n/2+1}}+\ln^4(\tau_0-t)\Theta(t)\\
    &\leqslant \frac{c}{\ln^2N} +\ln^4(\tau_0-t)\Theta(t).
\end{align*}}
    In summary so far, for any $t\in (0,\tau_0)$, we have
\begin{align}
\label{eq: estimate_I_III_V}
   |I+III+V|\leqslant cE_0+cN^{n/2}( {E_0}+\sqrt{E_0})+\frac{c}{\ln^2N}+c\left(\ln^4(\tau_0-t)+\ln^2(\tau_0-t)\right)\Theta(t).
\end{align}
Using \eqref{eq: I_II_III_IV_V}, \eqref{eq: estimate_IV}, and \eqref{eq: estimate_I_III_V}, we conclude that there is a uniform constant $c>0$ such that 
\begin{align*}
\frac{d}{dt}\Theta(t) & {\leqslant |I+III+V|+|IV|+II}\\
            &\leqslant {\frac12 II}+cN^{n/2}(E_0+\sqrt{E_0})+\frac{c}{\ln^2N}+c\left(\ln^4(\tau_0-t)+\ln^2(\tau_0-t)\right)\Theta(t).
\end{align*}
    Finally,  {notice that $II\leqslant 0$, and} defining $f(t)$ as in \eqref{def: f} we have $f'(t)=\ln^4(\tau_0-t)+\ln^2(\tau_0 - t) $, hence
\begin{align*}
    \frac{d}{dt}(e^{-cf}\Theta(t)) &= e^{-cf}\Big(\frac{d}{dt}\Theta(t)-cf'(t)\Theta(t)\Big)\\
    &\leqslant ce^{-cf}\Big(N^{n/2}(E_0+\sqrt{E_0})+\frac{1}{\ln^2N}\Big).
\end{align*} The result now follows by integrating the above inequality from $t_1$ to $t_2$ and noting that $f(t)$ is increasing, so that $e^{-cf(t)}\leqslant e^{-cf(t_1)}$, for all $t\in[t_1,t_2]$.
\end{proof}

\begin{remark}
\label{rmk: monotonicity_Theta}
    In the context of Theorem \ref{thm: monotonicity_Theta}, since the function $f$ given by \eqref{def: f} is increasing, one has
    \[
     {-26} = f(\tau_0 - 1) \leqslant f(\tau_0 - \min\{\tau_0,1\})\leqslant f(t_1)\leqslant f(t_2)\leqslant f(\tau_0^{-}) = 0.
    \] 
    Therefore, $0\leqslant f(t_2)-f(t_1)\leqslant  {26}$ and so the  factor $e^{c(f(t_2)-f(t_1))}$ in \eqref{ineq: monotonicity_Theta} is uniformly bounded.
\end{remark}

\begin{theorem}
\label{thm: monotonicity_Psi}
    Write $\Psi(r)$ for either $\Psi^{\cE}(r)$ or $\Psi^{\cD}(r)$, as in \eqref{eq: dfn_Psi}. For any $0<r_2\leqslant r_1<\min\{\sqrt{\tau_0}/2, {1/2}\}$ and $N>1$, the following almost-monotonicity formula holds:
\begin{equation}
\label{ineq: monotonicity_Psi}
    \Psi(r_2)\leqslant e^{c(h(r_2)-h(r_1))}\left(\Psi(r_1)+c\left(N^{n/2}(E_0+\sqrt{E_0})+\frac{1}{\ln^2N}\right)(r_1-r_2)\right)
\end{equation}
where $c=c(M,g)>0$ is a constant, $E_0$ denotes the energy of $\xi(0)$ with respect to the corresponding functional, $\mathcal{E}$ or $\mathcal{D}$, and
 {
\begin{equation}
\label{def: h}
    h(r):= f(\tau_0 - r^2),
    %-108r^2-64r^2\ln^4(2r)+128r^2\ln^3(2r)-208r^2\ln^2(2r)+208r^2\ln(2r).
\end{equation} with $f(t)$ given by \eqref{def: f}.}
\end{theorem}

\begin{proof}
    Setting\footnote{ {Here note that we only need to consider the case where $r_2<r_1$, i.e. $\alpha<1$.}} $\alpha:=\frac{r_2^2}{r_1^2}\in (0,1)$, we will perform a change of time variables defined by  $t=:\alpha\tilde{t}+(1-\alpha)\tau_0$ in the integral $\Psi(r_2)$. At first,  {noting that $\tau_0-\tilde{t} = \frac{1}{\alpha}(\tau_0-t)$,} we have
\begin{equation*}
    \Psi(r_2)=\int_{\tau_0-4r_2^2}^{\tau_0-r_2^2}\frac{\Theta(t)}{\tau_0-t}dt=\int_{\tau_0-4r_1^2}^{\tau_0-r_1^2}\frac{\Theta(t)}{\tau_0-\tilde{t}}d\tilde{t}.
\end{equation*}
 {Moreover, using the hypothesis that $r_1<\min\{\sqrt{\tau_0}/2,1/2\}$, for $\tilde{t}\in(\tau_0-4r_1^2,\tau_0-r_1^2)$ we have $\tau_0>t>\tilde{t}>\tau_0-\min\{\tau_0,1\}$. In particular,} by the almost-monotonicity of $\Theta$ from Theorem \ref{thm: monotonicity_Theta}, between times $t_1=\tilde{t} < t=t_2$, we have
\begin{align}\label{ineq: pre_estimate_psi}
    \Psi(r_2)\leqslant \int_{\tau_0-4r_1^2}^{\tau_0-r_1^2}e^{c(f(t)-f(\tilde{t}))}\left(\frac{\Theta(\tilde{t})}{\tau_0-\tilde{t}}+c\bigg(N^{n/2}(E_0+\sqrt{E_0})+\frac{1}{\ln^2N} \bigg)\frac{t-\tilde{t}}{\tau_0-\tilde{t}}\right)d\tilde{t}.
\end{align}
 {Next, we claim that $\frac{d}{d\tilde{t}}(f(t)-f(\tilde{t}))\geqslant 0$ for $\tilde{t}\in(\tau_0-4r_1^2,\tau_0-r_1^2)$. Indeed, on the one hand, by the chain rule,
\begin{equation}\label{eq: derivative_nonnegativity}
    \frac{d}{d\tilde{t}}(f(t)-f(\tilde{t})) = \alpha\frac{df}{dt}(t) - \frac{df}{d\tilde{t}}(\tilde{t}) = \alpha\left(f'(t) - f'(\tilde{t})\right).
\end{equation} On the other hand, recalling from the proof of Theorem \ref{thm: monotonicity_Theta} that $f'(t)=\ln^4(\tau_0-t) +\ln^2(\tau_0-t)$, we have
\[
f''(t) = -2\frac{\ln(\tau_0-t)}{\tau_0-t}(2\ln^2(\tau_0-t) + 1),
\] so that $f''(t)>0$ exactly when $\tau_0 - t<1$, i.e. $f'(t)$ is strictly increasing when $\tau_0 - t<1$. Thus, for $\tilde{t}\in(\tau_0-4r_1^2,\tau_0-r_1^2)$, we have $f'(t) - f'(\tilde{t})\geqslant 0$ (since $\tilde{t}<t$ and $\tau_0-t<\tau_0-\tilde{t}<1$), and combining this with \eqref{eq: derivative_nonnegativity} we get the desired claim. In particular, $f(t)-f(\tilde{t})\leqslant f(\tau_0-r_2^2)-f(\tau_0-r_1^2)$ for $\tilde{t}\in(\tau_0-4r_1^2,\tau_0-r_1^2)$. This last fact together with inequality \eqref{ineq: pre_estimate_psi} gives}
\begin{align*}
    \Psi(r_2) &\leqslant \int_{\tau_0-4r_1^2}^{\tau_0-r_1^2}e^{c(f(\tau_0- {r_2^2})-f(\tau_0- {r_1^2}))}\left(\frac{\Theta(\tilde{t})}{\tau_0-\tilde{t}}+c\bigg(N^{n/2}(E_0+\sqrt{E_0})+\frac{1}{\ln^2N} \bigg)(1-\alpha)\right)d\tilde{t}\\
    & \leqslant e^{c(h(r_2)-h(r_1))}\left(\Psi(r_1)+c\bigg(N^{n/2}(E_0+\sqrt{E_0})+\frac{1}{\ln^2N} \bigg)(r_1-r_2)\right),
\end{align*}
where $h(t)$ is given by \eqref{def: h}, as we wanted.
\end{proof}

\begin{remark}
\label{rmk: monotonicity_Psi}
    By analogy with Remark \ref{rmk: monotonicity_Theta}, we note that whenever $0<r_2\leqslant r_1\leqslant\min\{\sqrt{\tau_0}/2, {1/2}\}$, the exponential factor $e^{c(h(r_2)-h(r_1))}$ in inequality \eqref{ineq: monotonicity_Psi} is uniformly bounded. Indeed, from the above proof of Theorem \ref{thm: monotonicity_Psi} one has 
    $h(r_2)-h(r_1) = f(\tau_0- {r_2^2})-f(\tau_0- {r_1^2})$,
    and since $f$ is increasing  we have $f(\tau_0- {r_1^2})\leqslant f(\tau_0- {r_2^2})\leqslant f(\tau_0^{-})=0$ and $f(\tau_0- {r_1^2})\geqslant f(\tau_0- {1})\geqslant - {26}$, so that $0\leqslant h(r_2)-h(r_1)\leqslant  {26}$.
\end{remark}
\begin{remark}
\label{rmk: time_transl_invariance}
    The harmonic $H$-flow \eqref{eq: harmonic_flow} is clearly invariant under time-translation, i.e. if $\xi(t)$ is a solution to \eqref{eq: harmonic_flow}  then so is $\xi(t+t_0)$, for any $t_0\in\mathbb{R}$. In particular, given any $\tau_0>0$, by reparametrising $t\mapsto t-\tau_0$ we can translate Theorems \ref{thm: monotonicity_Theta} and \ref{thm: monotonicity_Psi}, previously stated for solutions over a time interval $[0,\tau_0]$, to analogous conclusions for solutions defined over a time interval of the form $[-\tau_0,0]$.
\end{remark}

\subsection{$\varepsilon$-regularity and energy gap}
\label{sec: e-regularity}

We remain in the setting of §\ref{subsec: local_monotonicity}, following mainly the sources \cites{Struwe1988,Struwe1989,he2021}, so that $(M^n,g)$ denotes a complete oriented Riemannian manifold of bounded geometry, admitting a compatible $H$-structure.

In what follows, for any $y\in M$, we shall identify $B_{r_M}(y)\subset (M,g)$ with the Euclidean ball $B_{r_M}(0)\subset\mathbb{R}^n$ via normal coordinates centred at $y$, and regard any solution $\xi(t)$ of the harmonic $H$-flow \eqref{eq: harmonic_flow} on $B_{r_M}(y)\times [0,\tau)$ as defined on $B_{r_M}(0)\times [0,\tau)\subset\mathbb{R}^n\times\mathbb{R}$ under such identification. We shall also use the following notation, for any $(x_0,t_0)\in\mathbb{R}^n\times\mathbb{R}$:
\begin{align*}
    P_r(x_0,t_0) &:= \{(x,t):d(x,x_0)\leqslant r, |t-t_0|\leqslant r^2\},\\
    T_r(t_0) &:= \{(x,t): t_0-4r^2<t<t_0-r^2\}.
\end{align*}
The following $\varepsilon$-regularity result along the harmonic $H$-flow generalises what was proved for $H=\U(m)\subset\SO(2m)$  in \cite{he2021}*{Theorem 3.3}.

\begin{theorem}[$\varepsilon$-regularity]
\label{thm: e-regularity}
    Let $(M^n,g)$ be a complete oriented Riemannian manifold of bounded geometry, admitting a compatible $H$-structure. Then, for any $E_0\in (0,\infty)$, there exists a constant $\varepsilon_0>0$, depending only on $(M^n,g)$, $H$ and $E_0$, with the following significance.
    Suppose that $\{\xi(t)\}$ is a solution to the harmonic $H$-flow \eqref{eq: harmonic_flow} on $B_{r_M}(y)\times[0,\tau)\subset M\times[0,\tau)$, with $\tau\leqslant r_M^2$ and initial energy bounded by $E_0$, and fix any $\tau_0\in (0,\tau)$. 
    
    If, for some $0<R<\min\{\varepsilon_0,\sqrt{\tau_0}/2\}$, the quantity  $\Psi(R):=\Psi_{(y,\tau_0)}(R;\xi(t))$ defined by \eqref{eq: dfn_Psi} satisfies
\begin{equation}
\label{ineq: smallness_Psi}
    \Psi(R)<\varepsilon_0,
\end{equation} 
    then
\begin{equation}
\label{ineq: e-reg}
    \sup_{P_{\delta R}(0,\tau_0)} e(\xi)\leqslant 4(\delta R)^{-2},
\end{equation} 
    where the constant $\delta>0$ depends only on $(M^n,g)$ and $H$, and possibly on $E_0$ and $\min\{1,R\}$. 
\end{theorem}
The main ingredients in the proof of Theorem \ref{thm: e-regularity} will be the local monotonicity formula of Theorem \ref{thm: monotonicity_Psi} and the Bochner-type estimate of Lemma \ref{lem: dif_ineq}. Our argument follows closely the proof of \cite{Struwe1989}*{Lemma 2.4}.

\begin{proof}[Proof of Theorem \ref{thm: e-regularity}]
    By a parabolic rescaling as in Lemma \ref{lem: rescaling}, we can assume $r_M=1$. Moreover, by  time-translation invariance of the harmonic $H$-flow equation (see Remark \ref{rmk: time_transl_invariance}), we may as well prove the analogous statement for a solution $\{\xi(t)\}$ defined over the time interval $[-\tau_0,0]$ instead. This amounts to showing that  {we can find a constant $\varepsilon_0=\varepsilon_0(M,g,H,E_0)>0$, such that if $\{\xi(t)\}$ is a solution to the harmonic $H$-flow \eqref{eq: harmonic_flow} on $B_{1}(y)\times [-\tau_0,0]$, $0<\tau_0<1$, with initial energy bounded by $E_0$, and $\Psi_{(y,0)}(R)<\varepsilon_0$ for some $0<R<\min\{\varepsilon_0,\sqrt{\tau_0}/2\}$, then
\begin{equation}\label{ineq: e-reg_proof}
\sup_{P_{\delta R}(0,0)}e(\xi)\leqslant 4(\delta R)^{-2},
\end{equation} for some constant $\delta=\delta(M,g,H,E_0,R)>0$.}

     {For now, let $\varepsilon_0>0$ and $\delta\in (0,1/4]$ be constants to be determined in the course of the proof, and let $0<R<\min\{\varepsilon_0,\sqrt{\tau_0}/2\}<1/2$. Define $r:=2\delta R$; in particular, $r\leqslant R/2<1/4$}. Since $\xi$ is smooth, there is $\sigma_0\in [0,r)$ such that  {the following maximum is attained:}
\[
(r-\sigma_0)^2\sup_{P_{\sigma_0}(0,0)} e(\xi) = \max_{0\leqslant\sigma\leqslant r}\left\{ (r-\sigma)^2\sup_{P_{\sigma}(0,0)} e(\xi)\right\}.
\] 
    Moreover, the supremum of the energy density is attained at some $(x_0,t_0)\in P_{\sigma_0}(0,0)$:
\[
\sup_{P_{\sigma_0}(0,0)} e(\xi) = e(\xi)(x_0,t_0)=:e_0.
\]  {\textbf{Claim:} For $\varepsilon_0=\varepsilon_0(M,g,H,E_0)>0$ and $\delta=\delta(M,g,H,E_0,R)>0$ sufficiently small, if $\Psi_{(y,0)}(R)<\varepsilon_0$ then
\begin{equation}\label{ineq: key_e-reg_proof}
    (r-\sigma_0)^2 e_0\leqslant 4.
\end{equation} Note that inequality \eqref{ineq: key_e-reg_proof} is equivalent to
\[
(r-\sigma)^2\sup_{P_{\sigma}(0,0)} e(\xi)\leqslant 4,\quad\forall\sigma\in [0,r],
\] and taking $\sigma=\frac{1}{2}r=\delta R$ in this last inequality gives inequality \eqref{ineq: e-reg_proof}, and therefore proves the theorem. Thus, in what follows, we proceed to prove the above claim.} 

Set $\rho_0:=\frac{1}{2}(r-\sigma_0)$. Then, by the above choices of $\sigma_0$ and $(x_0,t_0)$, and further noting that $P_{\rho_0}(x_0,t_0)\subset P_{\rho_0+\sigma_0}(0,0)$ and $\rho_0+\sigma_0<r$, we have
\begin{equation}
\label{ineq: sup_bound}
    \sup_{P_{\rho_0}(x_0,t_0)} e(\xi) \leqslant \sup_{P_{\rho_0+\sigma_0}(0,0)} e(\xi) \leqslant 4e_0.
\end{equation} 
    Now let $r_0:=\rho_0\sqrt{e_0}$ and,  {performing a change of variables}, define a new tensor $\tilde{\xi}$ on $P_{r_0}(0,0)$ by
\[
\tilde{\xi}(x,t):=\xi\left(\frac{x}{\sqrt{e_0}} + x_0, \frac{t}{e_0} + t_0\right).
\] 
     {Writing $\tilde{x}:=\frac{x}{\sqrt{e_0}} + x_0$ and $\tilde{t}:=\frac{t}{e_0} + t_0$, note that $(x,t)\in P_{r_0}(0,0)$ if and only if $(\tilde{x},\tilde{t})\in P_{\rho_0}(x_0,t_0)$. Moreover, the energy density $e(\tilde{\xi})$ of $\tilde{\xi}$ with respect to the metric $\tilde{g}(x):=g(\tilde{x})$ satisfies}
    \begin{equation}\label{eq: new_density}
       {  e(\tilde{\xi})(x,t) = e_0^{-1}e(\xi)(\tilde{x},\tilde{t}).}
    \end{equation}
     {In particular,} $e(\tilde{\xi})(0,0)=e_0^{-1}e(\xi)(x_0,t_0)=1$ and \eqref{ineq: sup_bound} becomes
\begin{equation}
\label{ineq: scaled_sup_bound}
    \sup_{P_{r_0}(0,0)} e(\tilde{\xi}) \leqslant 4.
\end{equation}  
    Denoting by $\tilde{\Delta}$ the Laplacian with respect to the metric $\tilde{g}$, we have
\[
(\partial_t - \tilde{\Delta})e(\tilde{\xi})(x,t) = e_0^{-2}(\partial_t - \Delta)e(\xi)( {\tilde{x},\tilde{t}}).
\] 
    So using the Bochner-type estimate \eqref{ineq: Bochner_estimate} together with \eqref{eq: new_density} and \eqref{ineq: scaled_sup_bound} we get
\begin{align}
(\partial_t - \tilde{\Delta})e(\tilde{\xi})(x,t) &=  {e_0^{-2}(\partial_t - \Delta)e(\xi)(\tilde{x},\tilde{t})}\nonumber\\
&\leqslant  {c e_0^{-2}\left(e(\xi)^2(\tilde{x},\tilde{t})+1\right)}\nonumber\\
&\leqslant  {c\left(e(\tilde{\xi})^2(x,t)+\frac{1}{e_0^2}\right)}\nonumber\\
&\leqslant c\left(e(\tilde{\xi})(x,t)+\frac{1}{e_0^2}\right),\quad\forall (x,t)\in P_{r_0}(0,0).\nonumber
\end{align} 
    Thus, the function $u(x,t):=\exp(-ct)(e(\tilde{\xi})(x,t)+e_0^{-2})$ satisfies
    \begin{equation}\label{ineq: heat_op_bound}
        (\partial_t - \tilde{\Delta})u\leqslant 0\quad\text{in }P_{r_0}(0,0).
    \end{equation} Now \emph{suppose that} $r_0>1$;  {by definition $r_0=\frac{1}{2}(r-\sigma_0)\sqrt{e_0}$, so this assumption is equivalent to saying that the bound \eqref{ineq: key_e-reg_proof} does not hold, and the idea now is to derive a contradiction in case $\Psi_{(y,0)}(R)<\varepsilon_0$, provided we choose $\varepsilon_0$ and $\delta$ sufficiently small; this will prove our main claim.}
    
     {On the one hand, using $r_0>1$, inequality \eqref{ineq: heat_op_bound}} and applying Moser's parabolic Harnack inequality \cite[Theorem 3]{Moser1964} to $u(x,t)$, gives the following a priori estimate:
\begin{equation}
\label{ineq: Harnack_bound}
    1+e_0^{-2} = e(\tilde{\xi})(0,0) + e_0^{-2} \leqslant c\int_{P_1(0,0)}(e(\tilde{\xi})+e_0^{-2}) dxdt.
\end{equation} 
    Setting $\lambda_0:= 1/\sqrt{e_0}$ and scaling back, we have
\begin{equation}
\label{eq: scaling_back}
    \int_{P_1(0,0)} e(\tilde{\xi})dxdt = \lambda_0^{-n}\int_{P_{\lambda_0}(x_0,t_0)} e(\xi)dxdt.
\end{equation}
    
    On the other hand, note that  {$r_0>1$ means $\lambda_0<\rho_0=\frac{1}{2}(r-\sigma_0)$, and so $\lambda_0+\sigma_0<r$. Recalling that $r\leqslant R/2<1/4$, and also that $(x_0,t_0)\in P_{\sigma_0}(0,0)$, we get in particular $\lambda_0<R/2$ and }$B_{\lambda_0}(x_0)\subset B_{1/2}(0)$. Noting furthermore that $G_{(x_0,t_0+2\lambda_0^2)}\geqslant c\lambda_0^{-n}$ in $P_{\lambda_0}(x_0,t_0)$ and that $P_{\lambda_0}(x_0,t_0)\subset \overline{T}_{\lambda_0}(t_0+2\lambda_0^2)$, it follows from the almost-monotonicity formula of Theorem \ref{thm: monotonicity_Psi} that
\begin{align}
\label{ineq: after_monotonicity}
    \lambda_0^{-n} \int_{P_{\lambda_0}(x_0,t_0)} e(\xi)dxdt 
    &\leqslant c\int_{P_{\lambda_0}(x_0,t_0)} e(\xi)G_{(x_0,t_0+2\lambda_0^2)} \phi^2\sqrt{\det (g)} dxdt\nonumber\\
    &\leqslant c\int_{T_{\lambda_0}(t_0+2\lambda_0^2)} e(\xi)G_{(x_0,t_0+2\lambda_0^2)}\phi^2\sqrt{\det (g)} dxdt\nonumber\\
    &\leqslant c\int_{T_{R}(t_0+2\lambda_0^2)} e(\xi)G_{(x_0,t_0+2\lambda_0^2)} \phi^2\sqrt{\det (g)} dxdt+cR(E_0 + 1)\nonumber\\
    &\leqslant c\left(\int_{-4R^2}^{-R^2} + \int_{-R^2}^{t_0+2\lambda_0^2-R^2}\right)\int_{\mathbb{R}^n}e(\xi)G_{(x_0,t_0+2\lambda_0^2)}\phi^2\sqrt{\det (g)} dxdt + cR(E_0 + 1)\nonumber\\
    & {\leqslant c\left(\int_{-4R^2}^{-R^2} + \int_{-R^2}^{-R^2/4}\right)\int_{\mathbb{R}^n}e(\xi)G_{(x_0,t_0+2\lambda_0^2)}\phi^2\sqrt{\det (g)} dxdt + cR(E_0 + 1)}\nonumber\\
    &\leqslant c\int_{T_{R}(0)}e(\xi)G_{(x_0,t_0+2\lambda_0^2)}\phi^2\sqrt{\det (g)}dxdt+cR(E_0 + 1).
\end{align}  {Here, we used\footnote{ {The way we apply Theorem \ref{thm: monotonicity_Psi} is by setting e.g. $N=2$ in its statement and also using the rough bounds $r_1-r_2\leqslant r_1$ and $2\sqrt{E_0}\leqslant E_0 + 1$, as well as Remark \ref{rmk: monotonicity_Psi}, to simplify the inequality given by the theorem to $\Psi(r_2)\leqslant c\Psi(r_1) + cr_1(E_0+1)$.}} Theorem \ref{thm: monotonicity_Psi} in the third inequality and in the last inequality; from the fourth line to the fifth we used that $t_0+2\lambda_0^2 -R^2 < -R^2/4$, since $t_0\leqslant\sigma_0^2<r^2\leqslant R^2/4$ and $\lambda_0^2<R^2/4$, and we used the monotonicity with radii $R/2<R$ to go from the fifth inequality to the last one, bounding the second integral by the first plus $cR(E_0+1)$.}

    Now, given $t\in [-4R^2,-R^2]$ and any point $x$, {we want to estimate $G_{(x_0,t_0+2\lambda_0^2)}(x,t)$ in terms of $G_{(0,0)}(x,t)$, and plug such estimate into inequality \eqref{ineq: after_monotonicity} in order to bound the integral in that inequality in terms of $\Psi_{(y,0)}(R)$. To this end, first note that for $t\in [-4R^2,-R^2]$ we have}
\begin{align}\label{ineq: first_kernel_bound}
    G_{(x_0,t_0+2\lambda_0^2)}(x,t) 
    &\leqslant \frac{{c}}{(4\pi |t|)^{n/2}}\exp\left(-\frac{|x-x_0|^2}{4(t_0+2\lambda_0^2-t)}\right)\nonumber\\
    &\leqslant c\exp\left(\frac{|x|^2}{4|t|}-\frac{|x-x_0|^2}{4|t_0+2\lambda_0^2-t|}\right)G_{(0,0)}(x,t).
\end{align} Next, recall that $|x_0|\leqslant\sigma_0<r=2\delta R$ and $t_0\leqslant \sigma_0^2<r^2=4\delta^2R^2$, as well as that by assumption $\lambda_0<\rho_0\leqslant r/2=\delta R$ and thus $2\lambda_0^2<2\delta^2R^2$. In particular, $|t_0+2\lambda_0^2 -t|\leqslant 6\delta^2R^2 + |t|$. Then, using that $R^2\leqslant |t|\leqslant 4R^2$, and the mean-value theorem on the function $s\mapsto 1/(s+|t|)$ on the interval $[0,6\delta^2R^2]$, we get
\begin{align}\label{ineq: difference_ratios_bounds}
\frac{|x|^2}{|t|}-\frac{|x-x_0|^2}{|t_0+2\lambda_0^2-t|} &\leqslant \frac{|x|^2}{|t|}-\frac{|x-x_0|^2}{6\delta^2R^2 +|t|}\nonumber\\
&\leqslant \frac{1}{|t|}(|x|^2 - |x-x_0|^2) + c\frac{\delta^2R^2}{|t|^2}|x-x_0|^2\nonumber\\
&\leqslant \frac{1}{R^2}(2|x||x_0|-|x_0|^2) + c\frac{\delta^2}{R^2}(|x|^2+2|x||x_0|+|x_0|^2)\nonumber\\
&\leqslant \frac{2}{R^2}|x||x_0| + c\delta^2\frac{|x|^2}{R^2} + c\delta^4\nonumber\\
&\leqslant 4\delta\frac{|x|}{R} + c\delta^2\frac{|x|^2}{R^2} + c\nonumber\\
&\leqslant c\delta^2\frac{|x|^2}{R^2} + c. 
\end{align}
Combining the inequalities \eqref{ineq: first_kernel_bound} and \eqref{ineq: difference_ratios_bounds} yields:
\[
G_{(x_0,t_0+2\lambda_0^2)}(x,t) \leqslant
    c\exp\left(c\delta^2\frac{|x|^2}{4R^2}\right)G_{(0,0)}(x,t).
\] {
If $|x|\leqslant R/\delta$, then}
\begin{equation}\label{ineq: second_kernel_bound}
G_{(x_0,t_0+2\lambda_0^2)}(x,t)\leqslant cG_{(0,0)}(x,t).
\end{equation}  {Otherwise, if $|x|>R/\delta$, then for small enough $\delta=\delta(n)>0$ we have:}
\begin{align}
   G_{(x_0,t_0+2\lambda_0^2)}(x,t) &\leqslant c\exp\left(c\delta^2\frac{|x|^2}{4R^2}\right)G_{(0,0)}(x,t)\nonumber\\
    &\leqslant  {cR^{-n}\exp\left(\frac{|x|^2}{4R^2}(c\delta^2 - \frac{1}{4})\right)}\nonumber\\
    &\leqslant  {cR^{-n}\exp\left(\frac{1}{4\delta^2}(c\delta^2 - \frac{1}{4})\right)}\nonumber\\
    &\leqslant cR^{-n}\exp(-c'\delta^{-2}),
    \label{ineq: third_kernel_bound}
\end{align}  where $c=c(n)>0$ depends only on $n$ and $c'>0$ is an absolute constant. Therefore, for every $\varepsilon\in (0,1)$, and for small enough $\delta=\delta(n,\varepsilon,R)>0$, combining inequalities \eqref{ineq: second_kernel_bound} and \eqref{ineq: third_kernel_bound} it follows that, whenever $t\in [-4R^2,-R^2]$, we have:
\begin{align}\label{ineq: final_kernel_bound}
G_{(x_0,t_0+2\lambda_0^2)}(x,t)\leqslant cG_{(0,0)}(x,t) + c\varepsilon R^{-2},
\end{align} where $c>0$ depends only on $n$, while we choose $\delta\sim(\ln(\varepsilon^{-1} R^{2-n}))^{-1/2}$. Combining \eqref{ineq: Harnack_bound}, \eqref{eq: scaling_back}, \eqref{ineq: after_monotonicity} and \eqref{ineq: final_kernel_bound}, we arrive at the following inequality for every $\varepsilon\in (0,1)$:
\begin{equation}\label{ineq: pre_step_final}
    1+e_0^{-2} \leqslant ce_0^{-2} + c\Psi_{(y,0)}(R)+c\varepsilon R^{-2}\int_{T_R(0)}e(\xi)\phi^2\sqrt{\det(g)}dxdt + cR(E_0+1),
\end{equation} where $c>0$ depends only on $(M^n,g)$ and $H$. In fact, further noting that $E_0$ bounds the initial energy of the solution $\{\xi(t)\}$, and absorbing the $e_0^{-2}$ term into the right-hand side of \eqref{ineq: pre_step_final}, while also noting that $e_0^{-2}\leqslant \rho_0^4 <(\delta R)^4$, we conclude:
\begin{equation}\label{ineq: final_e-reg_contradiction}
1\leqslant c\Psi_{(y,0)}(R) + c\varepsilon E_0 +cR(E_0+1) + c(\delta R)^4. 
\end{equation} Finally, if $\Psi_{(y,0)}(R)<\varepsilon_0$, then by assuming $\varepsilon_0<E_0$ and choosing $\varepsilon:=\varepsilon_0/E_0$, it follows from inequality \eqref{ineq: final_e-reg_contradiction} and the facts that $R<\varepsilon_0$ and $\delta=\delta(n,\varepsilon_0,E_0,R)<1$ that
    \[
    1\leqslant c\varepsilon_0(E_0+1) + c\varepsilon_0^4.
    \] Hence, by taking $\varepsilon_0$ small enough, depending only on $(M,g)$, $H$ and $E_0$, we obtain a \emph{contradiction}. Therefore, for such choices of $\varepsilon_0=\varepsilon_0(M,g,H,E_0)$ and $\delta=\delta(M,g,H,E_0,R)$, we must have $r_0\leqslant 1$, which concludes the proof.
\end{proof}

As a first consequence of $\varepsilon$-regularity, we close this section with a general energy gap result for harmonic $H$-structures, which is a direct analogue of the classical statement  \cite{Chen-Ding1990}*{Lemma 3.1} in harmonic map theory. It also generalises what was shown for  almost Hermitian structures in \cite{he2021}*{Lemma 3.4}.

\begin{proposition}[Energy gap for harmonic $H$-structures]
\label{prop: energy_gap}
	Let $(M^n,g)$ be a closed Riemannian manifold admitting a compatible $H$-structure. Then there is a constant $\varepsilon_0>0$, depending only on $(M^n,g)$ and the group $H$, such that, if $\xi$ is a compatible harmonic $H$-structure whose Dirichlet energy satisfies $\cD(\xi)=\frac{1}{2}\|\nabla\xi\|_{L^2(M)}^2<\varepsilon_0$, then $\xi$ is actually torsion-free, i.e. $\nabla\xi = 0$.
\end{proposition}
\begin{proof}
	Arguing by contradiction, suppose otherwise that there exists a sequence $(\xi_k)$ of harmonic $H$-structures inducing $g$ such that $\cD(\xi_k)\to 0$ as $k\to\infty$, but $\nabla\xi_k\neq 0$ for all $k$.
	
	By the $\varepsilon$-regularity [Theorem \ref{thm: e-regularity}] and Shi-type estimates [Proposition \ref{prop: Shi-estimates}], applied individually to each $\xi_k$ thought of as a static harmonic $H$-flow,  for $k\gg 1$ we eventually have that $|\nabla^m\xi_k|$ is uniformly bounded for each $m\in\mathbb{N}_0$. Therefore $(\xi_k)$ subconverges in the smooth topology to a torsion-free $H$-structure $\xi$. But since $\xi_k$ is harmonic, it follows from Lemma \ref{lem: basic_estimate_harmonic} that $|\Delta\xi_k|\leqslant c|\nabla\xi_k|^2$ for all $k$, where $c>0$ only depends on the geometry (and not on $k$). In particular, since $\nabla\xi = 0$, we get
	\[
	|\Delta (\xi_k - \xi)| \leqslant c|\nabla(\xi_k-\xi)|^2. 
	\] 
	Integrating by parts on $M$ gives
	\begin{align*}
		\int_M |\nabla(\xi_k-\xi)|^2 &= -\int_M \langle\xi_k-\xi,\Delta(\xi_k-\xi)\rangle\\
		&\leqslant c\|\xi_k-\xi\|_{L^{\infty}(M)}\int_M|\nabla(\xi_k-\xi)|^2.
	\end{align*} 
	Since $\xi_k\to\xi$ as $k\to\infty$ in the smooth topology, the above yields $\nabla(\xi_k-\xi)=0$,  i.e. $\nabla\xi_k=\nabla\xi=0$, for all $k\gg 1$. This contradicts our assumption that $\nabla\xi_k\neq 0$.
\end{proof}

\begin{remark}
    Of course, by the equivalence relation \eqref{ineq: T_equiv_nablaxi}, up to a change in the constant $\varepsilon_0$, the same can be stated in terms of the energy $\cE(\xi)=\frac{1}{2}\|T_{\xi}\|_{L^2(M)}^2$ instead.
\end{remark}

\subsection{Long-time existence and finite-time singularity}
\label{sec: long_time_existence}

We shall prove two main results about the harmonic $H$-flow under certain assumptions on the initial condition. Our first result establishes long-time existence and convergence of the harmonic $H$-flow to a torsion-free limit, under a smallness condition on the initial energy, relative to the $L^{\infty}$-norm of torsion. In particular, this answers in the affirmative a question raised by Grigorian \cite[p.8]{grigorian2020isometric} for $H=\rm G_2$, as to the possibility of proving long-time existence given small initial energy, rather than small initial entropy or small pointwise torsion, as obtained in  \cites{Grigorian2019,dgk-isometric}. The following also generalises the almost Hermitian case $H=\U(m)\subset\SO(2m)$, proved in \cite[Theorem 1]{he2021}.

\begin{theorem}[Long-time existence under small initial energy]
\label{thm: long-time_existence}
    Let $(M^n,g)$ be a closed, oriented Riemannian manifold admitting a compatible $H$-structure. For any given constant $\kappa>0$, there exists a universal constant $\varepsilon(\kappa)>0$, depending only on $\kappa$, the geometry $(M,g)$ and $H$, such that, if $\xi_0$ is a compatible $H$-structure on $(M^n,g)$ satisfying
\begin{itemize}
\item[(i)] $\|\nabla\xi_0\|_{L^{\infty}(M)}\leqslant \kappa$ and
\item[(ii)] $\cD(\xi_0)=\frac{1}{2}\|\nabla\xi_0\|_{L^2(M)}^2<\varepsilon(\kappa)$,
\end{itemize} 
    then the harmonic $H$-flow with initial condition $\xi_0$ exists for all time $t\geqslant 0$ and subconverges smoothly to a torsion-free $H$-structure as $t\to\infty$. Moreover, the universal constant can be chosen of the form 
    $$
    \varepsilon(\kappa)=\min\left\{\varepsilon_{\ast},c\left(\arctan{\frac{1}{2\kappa^2}}\right)^{n-2}\right\},
    $$
    where $\varepsilon_{\ast},c>0$ are constants depending only on $(M^n,g)$ and $H$.
\end{theorem}
\begin{remark}
    In the next section, we shall prove in Proposition \ref{prop: uniqueness} that for a (possibly) smaller choice of constant $\varepsilon(\kappa)$ in Theorem \ref{thm: long-time_existence}, with a (possibly) more complicated dependence on $\kappa$ and $(M^n,g)$, the torsion-free limit of the flow as $t\to\infty$ is unique (independent of subsequence).
\end{remark}
\begin{remark}
    In Section \ref{sec: stability}, we shall also prove long-time existence under small initial torsion (see Theorem \ref{thm: long_time_small_torsion}), subsuming the results known in particular for the $H=\rm G_2$ and $\mathrm{Spin}(7)$ cases, cf. \cites{Grigorian2019,dgk-isometric,Dwivedi-Loubeau-SaEarp2021}. Together with the energy gap of Proposition \ref{prop: energy_gap}, this implies that whenever $(M^n,g)$ admits a compatible torsion-free $H$-structure $\overline{\xi}$, then there is a (possibly very small)  {$C^2$}-neighbourhood $\mathcal{U}$ of $\overline{\xi}$ such that the harmonic $H$-flow starting anywhere inside $\mathcal{U}$ exists for all time and converges smoothly to a torsion-free limit; see Theorem \ref{thm: stability}.
\end{remark}

The second main result in this paragraph gives sufficient conditions for the formation of a finite-time singularity along the harmonic $H$-flow, i.e., for the torsion to blow up in finite time. It is a direct generalisation of the result for almost complex structures proved in \cite{he2021}*{Theorem 2}. 
For what follows, we define the \emph{isometric homotopy class} $[\xi]$ of a compatible $H$-structure $\xi$ on $(M^n,g)$ to be the homotopy class of $\xi$ as a section of the fibre bundle $\pi: \Fr(M,g)/H\to M$, i.e. among compatible $H$-structures. 

\begin{theorem}[Finite-time singularity]
\label{thm: blow_up}
    Let $(M^n,g)$ be a closed oriented Riemannian manifold, with $n>2$, endowed with a compatible $H$-structure $\overline{\xi}$, the isometric homotopy class $[\overline{\xi}]$ of which does not contain any torsion-free $H$-structure, but satisfies
\[
\inf_{\xi\in [\overline{\xi}]} \cD(\xi) = 0.
\] 
    Then there exists a constant $\varepsilon_{\ast}>0$, depending only on $(M^n,g)$ and $H$, such that, if $\xi_0\in [\overline{\xi}]$ has $\cD(\xi_0)<\varepsilon_{\ast}$, then the harmonic $H$-flow starting at $\xi_0$ develops a finite-time singularity. Moreover, if $[0,\tau(\xi_0))$ denotes the maximal existence interval for the solution, then $\tau(\xi_0)^{n-2}\lesssim \cD(\xi_0)$; in particular, $\tau(\xi_0)\to 0$ as $\cD(\xi_0)\to 0$.
\end{theorem}

Before we move on to prove these two theorems, let us examine a concrete instance of Theorem \ref{thm: blow_up}, which also illuminates the hypotheses of Theorem \ref{thm: long-time_existence}. The following is an adaptation of the analogous example constructed in \cite[§3.3]{he2021}, in the context of almost complex structures; see also Remark \ref{rmk: gen_example} below for other possible generalisations.

\begin{example}[Finite-time singularity]
\label{ex: finite_time_blow-up}
	Let $M = \mathbb{T}^7:=\mathbb{S}^1\times\ldots\times \mathbb{S}^1$ be the $7$-torus, endowed with the standard $\rm G_2$-structure $\varphi_{\circ}$ inducing the flat metric $g_{\circ}$. Then the frame bundle $\Fr(\mathbb{T}^7,g_{\circ})$ is trivialised by a \emph{parallel} global orthonormal frame, which in turn induces a trivialisation of the homogeneous bundle 
	\[
	\Fr(\mathbb{T}^7,g_{\circ})/\rm{G}_2\cong\mathbb{T}^7\times SO(7)/{\rm G_2}=\mathbb{T}^7\times\mathbb{RP}^7.
    \] 
	Its space of sections $[[\varphi_{\circ}]]$ consists of $\rm G_2$-structures which are \emph{isometric} to $\varphi_{\circ}$, i.e. compatible $\rm G_2$-structures on $(\mathbb{T}^7,g_{\circ})$. Any $\rm G_2$-structure $\varphi\in [[\varphi_{\circ}]]$ can be thought of as a map from $\mathbb{T}^7$ to $\mathbb{RP}^7$ and under such identification the torsion-free compatible $\rG_2$-structures, such as the standard $\varphi_{\circ}$ itself, correspond to constant maps. Moreover, the isometric homotopy class of $\varphi$ is the (unrestricted) homotopy class of the corresponding map. 
	
	Fix $p\in\mathbb{T}^7$ and $0<r_0<\mathrm{inj}(\mathbb{T}^7,g_{\circ})$ so that the geodesic ball $B(p,r_0)\subset (\mathbb{T}^7,g_{\circ})$ is isometric to the Euclidean ball $B(0,r_0)\subset\mathbb{R}^7\cong T_p\mathbb{T}^7$. We first show the following\\
    
    \emph{\textbf{Claim:} There is a smooth $\rm G_2$-structure $\varphi\in[[\varphi_{\circ}]]$ which coincides with the constant map $\varphi_{\circ}\equiv y_0\in\mathbb{RP}^7$ outside $B(p,r_0)$, but whose isometric homotopy class $[\varphi]$ is nontrivial, $[\varphi]\neq 0=[\varphi_{\circ}]$.}\\
    
    In order prove this, we start considering a quotient map
    \[
    q:(\overline{B}(p,r_0),\partial\overline{B}(p,r_0))\to (\mathbb{S}^7,x_0)
    \] identifying the boundary $\partial\overline{B}(p,r_0)$ of the closed ball $\overline{B}(p,r_0)$ with a point $x_0\in\mathbb{S}^7$, inducing a diffeomorphism $B(p,r_0)\cong \mathbb{S}^7\setminus\{x_0\}$ between the open ball and the $7$-sphere minus the point $x_0$. Recall that the set $[\mathbb{S}^7,\mathbb{RP}^7]$ of unrestricted homotopy classes of maps $\mathbb{S}^7\to\mathbb{RP}^7$, can be identified with the quotient of $\pi_7(\mathbb{RP}^7,y_0)$ by the usual action of $\pi_1(\mathbb{RP}^7,y_0)$ (cf. \cite[Proposition 4A.2]{Hatcher2002}), i.e.
    \[
    [\mathbb{S}^7,\mathbb{RP}^7] = \pi_7(\mathbb{RP}^7,y_0)/{\pi_1(\mathbb{RP}^7,y_0)}.
    \] On the other hand, it is a general fact that the action of $\pi_1(\mathbb{RP}^n)$ on $\pi_n(\mathbb{RP}^n)\cong\mathbb{Z}$ is trivial when $n$ is odd, so the set $[\mathbb{S}^7,\mathbb{RP}^7]$ is countably infinite. It follows that we can choose a nontrivial element $0\neq [f]\in [\mathbb{S}^7,\mathbb{RP}^7]$, $f:(\mathbb{S}^7,x_0)\to (\mathbb{RP}^7,y_0)$, which in turn induces the $\rm G_2$-structure $\overline{\varphi}:\mathbb{T}^7\to\mathbb{RP}^7$ isometric to $\varphi_{\circ}$ given by
    \begin{equation}\label{eq: def_varphi}
        \overline{\varphi}(x) := \begin{cases}
            \varphi_{\circ}(x) = y_0,\quad\text{if } x\in \mathbb{T}^7\setminus {B(p,r_0)},\\
            f(q(x)),\quad\text{if } x\in B(p,r_0).
        \end{cases} 
    \end{equation} 
    Note that $\overline{\varphi}$ is continuous on $\mathbb{T}^7$ and, by choosing the representative $f$ to be smooth, we further have that $\overline{\varphi}$ is also smooth outside the measure-zero set $\partial\overline{B}(p,r_0)$. Moreover, we have $[\overline{\varphi}]\neq 0$: indeed, 
    \[
    \mathrm{deg}(\overline{\varphi}:\mathbb{T}^7\to\mathbb{RP}^7)\mathrm{vol}(\mathbb{RP}^7)=\int_{\mathbb{T}^7}\overline{\varphi}^{\ast}\vol = \int_{B(p,r_0)}(f\circ q)^{\ast}\vol = \int_{\mathbb{S}^7}f^{\ast}\vol \neq 0, 
    \] where in the last step we used that $[f]\neq 0$ and the fact that the degree is a complete homotopy invariant for maps $\mathbb{S}^7\to\mathbb{RP}^7$, since such a map always lifts to a map into the universal cover $\mathbb{S}^7$ and one can apply Hopf's theorem\footnote{Hopf's theorem says that the degree is the only homotopy invariant for continuous maps into spheres. It is also interesting to notice, although not needed in the example, that the degree is a complete homotopy invariant for the $\rm G_2$-structure map $\overline{\varphi}:\mathbb{T}^7\to\mathbb{RP}^7$ given by \eqref{eq: def_varphi}, and in the end $[\overline{\varphi}]\neq 0$ if and only if $[f]\neq 0$. This is because the induced homomorphism $\pi_1(\overline{\varphi}):\pi_1(\mathbb{T}^7,p)\to\pi_1(\mathbb{RP}^7,y_0)$ is trivial, and thus $\overline{\varphi}$ lifts to a map $\tilde{\varphi}$ into $\mathbb{S}^7$, so $\mathrm{deg}(\overline{\varphi})=2\mathrm{deg}(\tilde{\varphi})$ and we can apply Hopf's theorem.} to the lifted map.

    By Whitney's approximation theorem, we then find a smooth map $\varphi:\mathbb{T}^7\to\mathbb{RP}^7$ which coincides with the constant map $\varphi_{\circ}\equiv y_0\in\mathbb{RP}^7$ outside $B(p,r_0)$, and is sufficiently $C^0$-close to $\overline{\varphi}$ so that we still have $[\varphi]\neq 0$. This gives us a smooth $\rm G_2$-structure $\varphi\in [[\varphi_{\circ}]]$ with the desired properties, thereby proving the claim.
	
    Next, for each $r\in (0,r_0)$, let $\varphi_r$ be the $\rm G_2$-structure on $\mathbb{T}^7$ defined by
    \begin{equation}\label{eq: def_phi_r}
        \varphi_r(x) := \begin{cases}
            \varphi_{\circ}(x) = y_0,\quad\text{if } x\in \mathbb{T}^7\setminus{\overline{B}(p,r)},\\
            \varphi\left(\frac{xr_0}{r}\right),\quad\text{if } x\in B(p,r_0)\simeq B(0,r_0)\subset\mathbb{R}^7,
        \end{cases} 
    \end{equation}
    where the smooth map $x\mapsto\varphi\left(\frac{xr_0}{r}\right)$ is defined using the normal coordinates giving the isometric identification $B(p,r_0)\simeq B(0,r_0)$. Note that, because $\varphi$ agrees with $\varphi_{\circ}$ outside $B(p,r_0)$, the two smooth maps in the definition of $\varphi_r$ agree in the intersection of their open domains (whose union is $\mathbb{T}^7$ since $0<r<r_0$). Thus, $\varphi_r$ is a smooth $\rm G_2$-structure on $\mathbb{T}^7$, isometric to $\varphi_{\circ}$ and such that $\varphi_r\in[\varphi]\neq 0=[\varphi_{\circ}]$. Next, we compute its Dirichlet energy:
\begin{align*}
	\cD(\varphi_r) 
	&= \frac{1}{2}\int_{B(p,r)}|\nabla\varphi_r|^2(x)dx &\quad\text{(since $\varphi_r=\varphi_{\circ}$ outside $B(p,r)$)}\\
	&= \frac{1}{2}\int_{B(p,r)}\left|\nabla\left(\varphi\left(\frac{xr_0}{r}\right)\right)\right|^2 dx 
	&\quad\text{(by \eqref{eq: def_phi_r})}\\
	&= r_0^2 r^{-2}\frac{1}{2}\int_{B(p,r)}|\nabla\varphi|^2\left(\frac{xr_0}{r}\right) dx &\\
	&= r_0^2r^{-2}r_0^{-7} r^7\frac{1}{2}\int_{B(p,r_0)}|\nabla\varphi|^2(y)dy &\quad\text{(by change of variables)}\\
	&= r_0^{-5}r^5 \cD(\varphi).
	&\quad\text{(since $\varphi=\varphi_{\circ}$ outside $B(p,r_0)$)}
\end{align*} 
	In particular, $\cD(\varphi_r)\to 0$ as $r\to 0$, and therefore
	\[
		\inf_{\tilde{\varphi}\in[\varphi]} \cD(\tilde{\varphi})=0.
	\] 
	On the other hand, since $[\varphi]\neq0$, this class cannot contain a torsion-free $\rm G_2$-structure, which would correspond to a constant map from $\mathbb{T}^7$ to $\mathbb{RP}^7$. 
	
	We claim that, for small enough $r\ll 1$, the harmonic $\rm G_2$-flow starting at $\varphi_r$ has a finite-time singularity, as guaranteed by Theorem \ref{thm: blow_up}. Indeed, if otherwise the flow $\{\varphi(t)\}$ with $\varphi(0)=\varphi_r$ existed for all time $t>0$, then, since $r\ll 1$ and thus $\cD(\varphi_r)\ll 1$, it would follow from the $\varepsilon$-regularity [Theorem \ref{thm: e-regularity}], together with Shi-type estimates [Proposition \ref{prop: Shi-estimates}], that $\varphi(t)$ converges smoothly as $t\to\infty$ to a $\rm G_2$-structure $\varphi_{\infty}\in[\varphi_r]=[\varphi]$ with divergence-free torsion  (note that the flow is isometric and is itself a  homotopy in the class of $\varphi(0)=\varphi_r$). In fact, if $r\ll 1$ is small enough, $\varphi_{\infty}$ would be torsion-free because of the energy gap of Proposition \ref{prop: energy_gap}, since the energy is non-increasing along the flow [Lemma \ref{lem: gradient_dirichlet_along_iso_var}]. But $\varphi_{\infty}$ being torsion-free implies that its homotopy class corresponds to that of a constant map from $\mathbb{T}^7$ to $\mathbb{RP}^7$, contradicting the non-triviality of $[\varphi]$. This gives an instance of Theorem \ref{thm: blow_up}, which moreover asserts that the maximal existence interval $[0,\tau(\varphi_r))$ of the flow starting at $\varphi_r$ shrinks to $\{0\}$ as $r\to 0$; indeed, $\tau(\varphi_r)\lesssim \cD(\varphi_r)^{1/5} = r_0^{-1}r\cD(\varphi)^{1/5}$.
	
	It is also noteworthy that however much $\varphi_r$ may have arbitrarily small energy $\cD(\varphi_r)\to 0$ as $r\to 0$, the $L^{\infty}$-norm of its torsion is actually blowing-up: 
    \[
    \|\nabla\varphi_r\|_{L^{\infty}(M)} = r_0 r^{-1}\|\nabla\varphi\|_{L^{\infty}(B(p,r_0))}\to\infty\quad\text{as }r\to 0.
    \] 
    This also exemplifies why a general result of long-time existence for the harmonic flow under small initial energy should take into account the $L^{\infty}$-norm of the initial torsion, as does Theorem \ref{thm: long-time_existence}.
\end{example} 

\begin{remark}
\label{rmk: gen_example}
    For any dimension $n>2$ and any closed and connected subgroup $H\subset\mathrm{SO}(n)$, the construction in Example \ref{ex: finite_time_blow-up} can be easily generalised for $H$-structures on the flat $n$-torus $\mathbb{T}^n$, provided  {the set of homotopy classes $[\mathbb{S}^n,\SO(n)/H]$, which bijectively corresponds to the quotient of $\pi_n(\SO(n)/H)$ by the action of $\pi_1(\SO(n)/H)$, has more than one element and the universal cover of $\SO(n)/H$ is a sphere. For instance, since $\SO(4)/U(2)\cong \mathbb{S}^2$ is simply connected and $\pi_4(\mathbb{S}^2)\cong\mathbb{Z}_2$, the construction works for $U(2)$-structures on the $4$-torus \cite[§3.3]{he2021}. It can also be reproduced for $\rm Spin(7)$-structures on the $8$-torus, since $\SO(8)/\rm Spin(7)\cong\mathbb{RP}^7$ and $\pi_8(\mathbb{RP}^7)\cong\mathbb{Z}_2\cong\pi_1(\mathbb{RP}^7)$, and the action of $\pi_1\cong\mathbb{Z}_2$ on $\pi_8\cong\mathbb{Z}_2$ must be by group automorphisms, so in this case it must be the trivial action and $[\mathbb{S}^8,\SO(8)/\rm Spin(7)]$ has two elements. On the other hand, since $\pi_6(\SO(6)/U(3))=\pi_6(\mathbb{CP}^3)=\{1\}$, the same construction does not work for $\rm U(3)$-structures on the $6$-torus.}
\end{remark}

\begin{remark}[What happens in dimension $n=2$]
\label{rmk: 2-torus}
    The only proper closed and connected subgroup $H\subset\mathrm{SO}(2)$ is $H=\{1\}$, and so $\pi_2(\mathrm{SO}(2)/H)=\pi_2(\mathrm{SO}(2))=\{1\}$. Let $(M^2,g)$ be a closed, oriented Riemannian surface admitting a compatible $\{1\}$-structure $\xi$, i.e.  a global oriented orthonormal frame $\xi=\{e_1,e_2\}$. Then, letting $\omega_{12}:=\langle\nabla e_1,e_2\rangle_g$, a standard computation gives $d\omega_{12}=-K_g\vol_g$, where $K_g$ is the Gaussian curvature of $(M^2,g)$. By the Gauss--Bonnet and Stokes' theorems, it follows that $\chi(M^2)=0$, and thus $M$ is diffeomorphic to a $2$-torus $\mathbb{T}^2$. 
 
    Without loss of generality, we let $M^2=\mathbb{T}^2$. For any Riemannian metric $g$ on $\mathbb{T}^2$, the orthonormal frame bundle $\Fr(\mathbb{T}^2,g)\cong \mathbb{T}^2\times\mathrm{SO}(2)$ is trivial, moreover there is always a conformal factor $e^{2f}\in C^{\infty}(\mathbb{T}^2)$ such that the metric $g_f:=e^{2f}g$ is flat. Now, one can verify directly that $\xi=(X,Y)$ is a compatible $\{1\}$-structure on $(\mathbb{T}^2,g)$ if and only if $\xi_f:=(e^{-f}X,e^{-f}Y)$ is a compatible $\{1\}$-structure on $(\mathbb{T}^2,g_f)$, and this gives a one-to-one correspondence between solutions of the harmonic $\{1\}$-flow on $(\mathbb{T}^2,g)$ and $(\mathbb{T}^2,g_f)$, because $n=2$ is the critical dimension for the Dirichlet functional.
 
    Suppose henceforth that $(M^2,g)=(\mathbb{T}^2,g_{\circ})$  is a flat torus. The frame bundle $\Fr(\mathbb{T}^2,g_{\circ})\cong \mathbb{T}^2\times\mathrm{SO}(2)$ is trivialised by a parallel global orthonormal frame, so any compatible $\{1\}$-structure on $(\mathbb{T}^2,g_{\circ})$ can be seen as a smooth map $\sigma:\mathbb{T}^2\to\mathrm{SO}(2)$. The isometric homotopy class of a $\{1\}$-structure is then simply the homotopy class of the map $\sigma:\mathbb{T}^2\to\mathrm{SO}(2)$, and its harmonic flow on $(\mathbb{T}^2,g_{\circ})$ corresponds to the harmonic map heat flow for maps $\mathbb{T}^2 \to \mathrm{SO}(2)$. Moreover, $\sigma$ is a harmonic (resp. torsion-free) compatible $\{1\}$-structure on $(\mathbb{T}^2,g_{\circ})$ if and only if the corresponding map $\sigma:\mathbb{T}^2\to\mathrm{SO}(2)$ is harmonic (resp. constant).
 
    Now since $\pi_2(\mathrm{SO}(2))=\{1\}$, a classical result in the theory of harmonic maps from Riemann surfaces (see e.g. \cite{jost2011riemannian}*{Theorem 9.2.1}) guarantees that any smooth map $\overline{\sigma}:\mathbb{T}^2\to\mathrm{SO}(2)$ is homotopic to a harmonic map $\sigma:\mathbb{T}^2\to\mathrm{SO}(2)$, which is energy-minimising in the homotopy class $[\overline{\sigma}]$. In particular, if $[\overline{\sigma}]\neq 0$ then
\[
\inf_{\sigma\in[\overline{\sigma}]} \cD(\sigma) \neq 0,
\] 
    for otherwise one would find a harmonic map $\sigma\in[\overline{\sigma}]$ with $\cD(\sigma)=0$ (i.e. $\sigma$ is a constant map), contradicting $[\overline{\sigma}]\neq 0$. So we cannot find a compatible $\{1\}$-structure $\overline{\sigma}$ in $(\mathbb{T}^2,g_{\circ})$ satisfying the hypotheses of Theorem \ref{thm: blow_up}.  {Moreover, since $\rm SO(2)$ is $1$-dimensional, and therefore its Riemann curvature tensor vanishes, it follows from the classical Eells--Sampson theorem \cite{Eells1964} that the harmonic map heat flow with any given smooth initial map $\sigma_0:\mathbb{T}^2\to\mathrm{SO}(2)$ has a unique smooth solution $\{\sigma(t)\}$ which exists for all time $t\geqslant 0$, and for a suitable sequence $t_i\to\infty$, the sequence $(\sigma(t_i))$ converges smoothly to a harmonic map $\sigma_{\infty}:\mathbb{T}^2\to\mathrm{SO}(2)$. Thus, there are no finite-time singularities for harmonic flows of $\{1\}$-structures on $(\mathbb{T}^2,g_{\circ})$.}
 
    %Moreover, we claim that there are no finite-time singularities for harmonic flows of $\{1\}$-structures on $(\mathbb{T}^2,g_{\circ})$ with sufficiently small initial energy. Indeed, defining\footnote{By the energy gap for non-constant regular harmonic maps, see e.g. \cite[Lemma 3.1]{Chen-Ding1990}, one has $\varepsilon_{\ast}>0$.}
%\[
%\varepsilon_{\ast}:=\inf\{\cD(u):u:\mathbb{S}^2\to\mathrm{SO}(2)\text{ is a non-constant regular harmonic map}\},
%\] 
 %   where $\mathbb{S}^2$ is the round sphere and $\cD(u)$ denotes the Dirichlet energy of the map $u$, then it follows from the classical work of Struwe \cite[Remark 4.4]{struwe1985evolution} that, for any smooth initial map $\sigma_0:\mathbb{T}^2\to\mathrm{SO}(2)$ with energy $\cD(\sigma_0)<\varepsilon_{\ast}$, the harmonic map heat flow with initial condition $\sigma_0$ has a unique smooth solution $\{\sigma(t)\}$ which exists for all time $t\geqslant 0$. Furthermore, for some sequence $t_i\to\infty$, the sequence $(\sigma(t_i))$ converges smoothly to a smooth harmonic map $\sigma_{\infty}:\mathbb{T}^2\to\mathrm{SO}(2)$.
\end{remark}

The remainder of this section is dedicated to the proofs of Theorems \ref{thm: long-time_existence} and \ref{thm: blow_up}. Our approach mostly follows the work of He--Li \cite{he2021} on almost complex structures, as well as the classical work of Chen--Ding \cite{Chen-Ding1990} on harmonic maps. 
Let $\{\xi(t)\}$ be a solution to the harmonic $H$-flow \eqref{eq: harmonic_flow} on $(M,g)$, and let $[0,\tau)$ be its (possibly semi-infinite) maximal time-interval of existence and uniqueness. In what follows, we write 
\[
e(\xi):=|\nabla\xi|^2
\qandq
\overline{e}(t) := \max_M e(\xi(t)).
\]
\begin{lemma}[cf. {\cite{Chen-Ding1990}*{Lemma 2.1}} and {\cite{he2021}*{Lemma 3.3}}]
\label{lem: time_bounds}
    If we let $\delta:=1/c>0$, where $c>0$ is given by Lemma \ref{lem: dif_ineq}, then, for any $t_0\in [0,\tau)$, 
\begin{equation}
	t_0 + \delta\arctan\frac{1}{2\overline{e}_0} <\tau,
	\qwithq 
	\overline{e}_0 :=\overline{e}(t_0),
\end{equation} 
    and 
\begin{equation}
	\overline{e}(t) \leqslant \frac{\overline{e}_0 + \tan{c(t-t_0)}}{1-\overline{e}_0\tan{c(t-t_0)}}\quad\forall t\in\left[t_0,t_0+\delta\arctan\frac{1}{\overline{e}_0}\right).
\end{equation} 
    In particular,
\begin{equation}
	\overline{e}(t) \leqslant 2\overline{e}_0 + \frac{1}{\overline{e}_0},\quad\forall t\in\left[t_0,t_0+\delta\arctan\frac{1}{2\overline{e}_0}\right].
	\end{equation}
\end{lemma}
\begin{proof}
	Using the Bochner-type estimate of Lemma \ref{lem: dif_ineq}, the proof is the same as in \cite[Lemma 3.3]{he2021}.
\end{proof}

\begin{corollary}
\label{cor: at_least}
	If the initial condition $\xi_0$ of a harmonic $H$-flow  \eqref{eq: harmonic_flow} satisfies   $\|\nabla\xi_0\|_{L^{\infty}(M)}\leqslant \kappa$, then the solution $\{\xi(t)\}$  exists at least for all $t\in[0,\delta\arctan\frac{1}{2\kappa^2}]$.
\end{corollary}

We now combine the above maximum time lower bound estimates [Lemma \ref{lem: time_bounds}] with our previous local monotonicity formula [Theorem \ref{thm: monotonicity_Theta}], the $\varepsilon$-regularity mechanism [Theorem \ref{thm: e-regularity}] and the energy gap [Proposition \ref{prop: energy_gap}], to prove the two key results underlying the proofs of Theorems \ref{thm: long-time_existence} and \ref{thm: blow_up}.

\begin{lemma}[Existence and convergence under uniformly bounded torsion]
\label{lem: unif_bound_case}
    Suppose that $\xi_0$ is a compatible $H$-structure on $(M^n,g)$. Let $\{\xi(t)\}_{[0,\tau)}$ be the maximal unique solution to the harmonic $H$-flow \eqref{eq: harmonic_flow} with initial condition $\xi(0)=\xi_0$, and suppose that
\begin{equation}
\label{ineq: torsion_uniform_bound}
    \sup\{\overline{e}(t):t\in [0,\tau)\}<\infty.
\end{equation} 
    Then actually $\tau=\infty$, and the flow $\{\xi(t)\}$ subconverges smoothly when $t\to\infty$. 
    Moreover, any such subsequential limit $\xi_{\infty}$ satisfies $\cD(\xi_{\infty})\leqslant \cD(\xi_0)$ and has divergence-free torsion:
\begin{equation}
\label{eq: div_free_limit}
    \Div T(\xi_{\infty}) = 0.
\end{equation} 
    If furthermore $\cD(\xi_0)<\varepsilon_0$, as in Proposition \ref{prop: energy_gap}, then any subsequential limit $\xi_\infty$ is torsion-free. 
\end{lemma}
\begin{proof}
    It is straightforward to check that the flow $\{\xi(t)\}$ exists for all $t\geqslant 0$, because otherwise the uniform bound \eqref{ineq: torsion_uniform_bound} would lead to a contradiction in Lemma \ref{lem: time_bounds}, when $t_0$ is sufficiently close to $\tau$. Moreover, combining \eqref{ineq: torsion_uniform_bound} with the Shi-type estimates [Proposition \ref{prop: Shi-estimates}] shows that $|\nabla^m\xi|$ is uniformly bounded for all $m$. Thus, for any sequence $t_n\to\infty$, there is a subsequence of $\xi(t_n)$ converging smoothly to a limit $\xi_{\infty}$.

    On the other hand, by Lemma \ref{lem: gradient_dirichlet_along_iso_var}, $\frac{d}{dt}\cD(\xi(t)) = - \int_M |\Div T(t)\diamond\xi(t)|^2\leqslant 0$, 
    and so
\[
0\leqslant \int_{0}^{\infty}\int_M |\Div T(t)\diamond\xi(t)|^2 = \cD(\xi_0) - \cD(\xi_{\infty})\leqslant \cD(\xi_0)<\infty.
\] 
    In particular, any subsequential limit $\xi_{\infty}$ as above satisfies \eqref{eq: div_free_limit};  {indeed, since $\xi(t_n)$ subconverges smoothly to $\xi_{\infty}$, the finiteness $\int_0^{\infty}\|\Div T(t)\diamond\xi(t)\|_{L^2}^2 dt<\infty$ implies that, up to taking a further subsequence, we must have $\|\Div T(t_n)\diamond\xi(t_n)\|_{L^2}^2\to 0=\|\Div T(\xi_{\infty})\diamond\xi_{\infty}\|_{L^2}^2$ as $t_n\to\infty$, which then gives the claim by the injectivity of $(\cdot{}\diamond\xi_{\infty})|_{\Omega_{\fm}^2}$ [Lemma \ref{lem: ker_diamond}]}. Furthermore, since $\cD(\xi_{\infty})\leqslant \cD(\xi_0)$, if $\cD(\xi_0)<\varepsilon_0$ then it follows from Proposition \ref{prop: energy_gap} that $\xi_{\infty}$ is torsion-free.
\end{proof}

\begin{lemma}[Finite-time singularity under unbounded torsion]
\label{lem: unbdound_case}
    There are constants $\varepsilon_1>0$ and $c_1>0$, depending only on $(M^n,g)$ and $H$, with the following significance. Suppose that $\xi_0$ is a compatible $H$-structure on $(M^n,g)$. Let $\{\xi(t)\}_{[0,\tau)}$ be the maximal unique solution to the harmonic $H$-flow \eqref{eq: harmonic_flow} with initial condition $\xi(0)=\xi_0$, and suppose that
\begin{equation}
\label{ineq: torsion_blows_up}
    \sup\{\overline{e}(t):t\in [0,\tau)\}=\infty.
\end{equation} 
    If $\varepsilon:=\cD(\xi_0)<\varepsilon_1$, then
\begin{equation}
\label{ineq: max_time_energy_bound}
    \tau^{\frac{n-2}{2}} \leqslant c_1\sqrt{\varepsilon}.
\end{equation}
\end{lemma}
\begin{proof}
    By the assumption \eqref{ineq: torsion_blows_up}, there is a sequence $(t_i)\subset (0,\tau)$ with $\lim t_i=\tau$, such that $\lim \overline{e}(t_i)=\infty $. In particular, 
\[
\lambda_i^2:=\arctan{\frac{1}{2\overline{e}(t_i)}}\to 0, \quad\text{as }i\to\infty.
\] 
    Now, for each $i\in\mathbb{N}$, let $p_i\in M$ be a point where the supremum is attained,
\[
e(\xi)(p_i,t_i) = \overline{e}(t_i),
\] 
    and let $\{x_{\alpha}\}$ be normal coordinates centred at $p_i$. In such coordinates, we can define $\Theta(t)$ as in \eqref{eq: dfn_Theta}, with $\tau_0:=t_i + \delta\lambda_i^2<\tau$, where $\delta {=\delta(M^n,g,H)}>0$ is given by Lemma \ref{lem: time_bounds}. Possibly after scaling, we can assume $\mathrm{inj}(M,g)>1$, by Lemma \ref{lem: rescaling} and the parabolic scale-invariance of $\Theta$. On $U_i:=B_{\lambda_i^{-1}}(0)\times [-\lambda_i^{-2}t_i,\delta]$ we define
\[
\tilde{\xi}(x,t) := \xi (\lambda_i x, t_i+\lambda_i^2 t).
\] 
    Then $\tilde{\xi}$ satisfies the harmonic $H$-flow on $U_i$, with respect to the scaled metric $\tilde{g}_{\alpha\beta}(x):=g_{\alpha\beta}(\lambda_i x)$.  { Indeed, noting that $\tilde{\nabla}_k\tilde{\xi}(x,t) = \lambda_i\nabla_k\xi(\lambda_i x, t_i+\lambda_i^2 t)$ and using the relation given by Lemma \ref{lem: torsion_equiv_nablaxi}, we see that the torsion $\tilde{T}$ of $\tilde{\xi}$ is given by $\tilde{T}(x,t) = \lambda_i T(\lambda_i x, t_i+\lambda_i^2 t)$. In particular, $(\Div_{\tilde{g}}\tilde{T})_{bc}(x,t) = \tilde{g}^{ka}\tilde{\nabla}_k\tilde{T}_{a,bc}(x,t) = \lambda_i^{2}(\Div_{g}T)_{bc}(\lambda_i x, t_i+\lambda_i^2 t)$, and since $\partial_t\tilde{\xi}(x,t) = \lambda_i^2(\partial_t\xi)(\lambda_i x, t_i+\lambda_i^2 t)$, we do have $\partial_t\tilde{\xi} = \Div_{\tilde{g}}\tilde{T}\diamond\tilde{\xi}$. Moreover, it follows that the energy density $e(\tilde{\xi})$ of $\tilde{\xi}$ with respect to $\tilde{g}$ is given by $e(\tilde{\xi})(x,t) = \lambda_i^2e(\xi)(\lambda_i x, t_i+\lambda_i^2 t)$.}

    Since $\lim\lambda_i^2 \overline{e}(t_i)= \frac{1}{2}$, for $i\gg 1$ we have, on one hand,
\[
e(\tilde{\xi})(0,0) = \lambda_i^2\overline{e}(t_i)>\frac{1}{4}.
\] 
    On the other hand, using Lemma \ref{lem: time_bounds}, 
\[
e(\tilde{\xi})(x,t) \leqslant \lambda_i^2\overline{e}(t_i+\lambda_i^2t)\leqslant\lambda_i^2\left(2\overline{e}(t_i)+\frac{1}{\overline{e}(t_i)}\right)<2,
\quad\forall (x,t)\in U_i,\quad  {i\gg 1}.
\] 
    
    In particular, by the Bochner-type estimate \eqref{ineq: Bochner_estimate},
\[
(\partial_t - \tilde{\Delta})e(\tilde{\xi})(x,t) = \lambda_i^4(\partial_t - \Delta)e(\xi)(\lambda_i x, t_i+\lambda_i^2 t)\leqslant c(e(\tilde{\xi})+\lambda_i^4),\quad\forall (x,t)\in U_i,
\quad i\gg 1.
\] 
    Thus, the function $u(x,t):=\exp(-ct)(e(\tilde{\xi})(x,t)+\lambda_i^4)$ satisfies $(\partial_t - \tilde{\Delta})u\leqslant 0$ on $U_i$. If we consider, for $i\gg 1$, the subset
\[
U:=B_1(0)\times\left(-\min\left\{\frac{\delta}{2},\frac{\delta}{c}\right\},\frac{\delta}{2}\right)\subset U_i,
\] 
    then by Moser's parabolic Harnack inequality there is $\gamma>0$, depending only on $(M^n,g)$, such that
\[
\frac{1}{4}<e(\tilde{\xi})(0,0)\leqslant u(0,0)\leqslant\gamma\left(\frac{1}{\delta\mathrm{Vol}(B_1(0))}\int_U u^2 dxdt\right)^{\frac{1}{2}}.
\] 
    Since $e(\tilde{\xi})<2$ and $\exp(-2ct)\leqslant \exp(2\delta)$ in $U$, we get
\[
1\leqslant 16\gamma^2\left(\frac{(2+\lambda_i^4)\exp{(2\delta)}}{\delta\mathrm{Vol}(B_1(0))}\int_U(e(\tilde{\xi})+\lambda_i^4)dxdt\right),\quad i\gg 1.
\] 
    Recalling that $\lim \lambda_i = 0$, we can assume for $i\gg 1$ that $\sqrt{\det{(\tilde{g}_{\alpha\beta}})}>1/2$ on $B_1(0)$, $\lambda_i\ll 1$ and $16\gamma^2(2+\lambda_i^4)\exp(2\delta)\lambda_i^4<1/4$. Hence, noting further that $\int_U dxdt\leqslant\delta\mathrm{Vol}(B_1(0))$, we get
    \begin{align}
        1 &\leqslant \frac{1}{4} + 16\gamma^2\frac{(2+\lambda_i^4)\exp{(2\delta)}}{\delta\mathrm{Vol}(B_1(0))}\int_U e(\tilde{\xi})dxdt\nonumber\\
        &\leqslant \frac{1}{4} + 16\gamma^2\frac{3\exp{(2\delta)}}{\delta\mathrm{Vol}(B_1(0))}\int_U e(\tilde{\xi})dxdt\nonumber\\
        &\leqslant \frac{1}{4} + 32\gamma^2\frac{3\exp{(2\delta)}}{\delta\mathrm{Vol}(B_1(0))}\int_U e(\tilde{\xi})\sqrt{\det{(\tilde{g}_{\alpha\beta}})}dxdt,\nonumber
    \end{align} yielding the lower bound
\begin{equation}
\label{ineq: crucial_lower_bound}
    1\leqslant \gamma_1\int_U  {e(\tilde{\xi})} \sqrt{\det{(\tilde{g}_{\alpha\beta}})}dxdt,
\end{equation} 
    where $\displaystyle\gamma_1 := \frac{128\gamma^2\exp{(2\delta)}}{\delta\mathrm{Vol}(B_1(0))}$, and therefore $\gamma_1>0$ depends only on $(M^n,g)$  {and $H$}.

    We now invoke the monotonicity of $\Theta(t)$ from Theorem \ref{thm: monotonicity_Theta}, and the uniform bound from  Remark \ref{rmk: monotonicity_Theta}, to deduce that, for any $N>1$ and $t\in (\tau_0 - \rho, \tau_0)$, with $\rho:=\min\{1,\tau_0\}$,
\begin{equation}
\label{ineq: invoked_monotonicity}
    \Theta(t) \leqslant c\Theta(\tau_0 - \rho) + c\left(N^{n/2}(\varepsilon+\sqrt{\varepsilon})+\frac{1}{\ln^2{N}}\right)\rho.
\end{equation} 
    Since $G_{(0,\tau_0)}(x,\tau_0-\rho)\leqslant (4\pi \rho)^{-n/2}$, $0\leqslant\phi\leqslant 1$ and $\mathcal{D}(\xi(\tau_0-\rho))\leqslant \mathcal{D}(\xi_0)=\varepsilon$, the first term on the right-hand side is bounded as follows:
\[
\Theta(\tau_0 - \rho)\leqslant c\rho^{1-n/2}\int_{B_1(0)} {e(\xi)}(x,\tau_0 - \rho)\sqrt{\det{(g_{\alpha\beta})}}dx \leqslant c\rho^{1-n/2}\varepsilon,
\] 
    so \eqref{ineq: invoked_monotonicity} gives
\[
\Theta(t)\leqslant c\rho^{1-n/2}\varepsilon + c\rho N^{n/2}(\varepsilon + \sqrt{\varepsilon}) + \frac{c\rho}{\ln^2{N}}.
\]
    Assuming $\varepsilon<1$, and recalling that $N>1$ and $\rho\in (0,1]$, we obtain
\begin{equation}
\label{ineq: crucial_monotonicity_bound}
    \Theta(t)\leqslant cN^{n/2}\rho^{1-n/2}\sqrt{\varepsilon} + \frac{c}{\ln^2 N},\quad\forall t\in(\tau_0-\rho,\tau_0),
    \quad\forall N>1.
\end{equation} 
    Note that $c>0$ is a uniform constant, depending only on $(M^n,g)$, and in particular independent of $N$.

    Since $\tau_0 = t_i + \delta \lambda_i^2$, at any time $-\min\{\delta/2,\delta/c\}<t<\delta/2$ we have
\begin{equation}\label{ineq: minor_step_1}
\frac{1}{2}\delta\lambda_i^2<\tau_0 - (t_i+\lambda_i^2t) = \lambda_i^2(\delta - t)<\frac{3}{2}\delta\lambda_i^2,
\end{equation} 
    so that if furthermore $| {\tilde{x}}|\leqslant \lambda_i$, then
\begin{equation}\label{ineq: minor_step_2}
G_{(0,\tau_0)}( {\tilde{x}},t_i + \lambda_i^2t) = (4\pi(\tau_0 - t_i -\lambda_i^2t))^{-n/2}\exp{\left(-\frac{| {\tilde{x}}|^2}{4(\tau_0 - t_i -\lambda_i^2t)}\right)}\geqslant c\delta^{-n/2}\lambda_i^{-n}\exp\left(-\frac{1}{2\delta}\right).
\end{equation}
    Together with \eqref{ineq: crucial_monotonicity_bound}, this yields the following upper bounds:
\begin{align}
    \int_{B_1(0)} {e(\tilde{\xi})}( {x},t)\sqrt{\det{( {\tilde{g}}_{\alpha\beta})}}dx
    &= \lambda_i^{2-n}\int_{B_{\lambda_i}(0)} {e(\xi)}( {\tilde{x}},t_i+\lambda_i^2t) \sqrt{\det{(g_{\alpha\beta})}}d {\tilde{x}}\nonumber\\
    & {\leqslant 2\delta^{-1}(\tau_0 - t_i - \lambda_i^2t)\lambda_i^{-n}\int_{B_{\lambda_i}(0)}e(\xi)(\tilde{x},t_i+\lambda_i^2t) \sqrt{\det{(g_{\alpha\beta})}}d\tilde{x}\quad\text{(by \eqref{ineq: minor_step_1})}}\nonumber\\
    &\leqslant c\delta^{\frac{n-2}{2}}\exp{\left(\frac{1}{2\delta}\right)}\Theta(t_i+\lambda_i^2t)\quad\text{ {(by \eqref{ineq: minor_step_2})}}\nonumber\\
    &\leqslant cN^{n/2}\rho^{1-n/2}\sqrt{\varepsilon} + \frac{c}{\ln^2 N},\quad\forall N>1.\label{ineq: crucial_upper_bound}
\end{align} 
    Combining the estimates \eqref{ineq: crucial_lower_bound} and \eqref{ineq: crucial_upper_bound}, we conclude that there is a uniform constant $c>0$, depending only on $(M^n,g)$ and $H$, such that
\[
1\leqslant cN^{n/2}\rho^{1-n/2}\sqrt{\varepsilon} + \frac{c}{\ln^2 N},\quad\forall N>1.
\] 
    Choosing $N:=\exp{(\sqrt{2c})}$, we get
\begin{equation}\label{ineq: pre_final}
\rho^{\frac{n-2}{2}}\leqslant c_1\sqrt{\varepsilon},
\end{equation} where $c_1:=2c\exp{(n\sqrt{2c}/2)}>0$ is a uniform constant, depending only on $(M^n,g)$ and $H$. Now define $\varepsilon_1:=\min\{1,c_1^{-2}\}$ and assume that $\varepsilon<\varepsilon_1$. Then by \eqref{ineq: pre_final} and the definition of $\rho:=\min\{1,\tau_0\}$, it follows that $\rho=\tau_0$. Finally, since $\tau_0:=t_i+\lambda_i^2\delta\to\tau$ as $i\to\infty$, the inequality \eqref{ineq: pre_final} implies the desired result \eqref{ineq: max_time_energy_bound}.
\end{proof}

We are now in position to prove the main results of this section.
\begin{proof}[Proof of Theorem \ref{thm: long-time_existence}]
Define
\begin{equation}\label{def: e(K)_1}
\varepsilon(\kappa):=\min\left\{\varepsilon_0,\varepsilon_1,c_1^{-2}\left(\delta\arctan{\frac{1}{2\kappa^2}}\right)^{n-2}\right\},
\end{equation} where $\varepsilon_0$ is given by Proposition \ref{prop: energy_gap}, $\varepsilon_1$ and $c_1$ are given by Lemma \ref{lem: unbdound_case}, and $\delta$ is given by Lemma \ref{lem: time_bounds}. Let $\{\xi(t)\}_{[0,\tau)}$ be the unique solution to the harmonic $H$-flow with initial condition $\xi_0$ satisfying (i) $\|\nabla\xi_0\|_{L^{\infty}(M)}\leqslant \kappa$ and (ii) $\cD(\xi_0)<\varepsilon(\kappa)$. If \eqref{ineq: torsion_blows_up} was true, then Corollary \ref{cor: at_least} would contradict the maximal time upper bound \eqref{ineq: max_time_energy_bound} of Lemma \ref{lem: unbdound_case}. Therefore, the uniform bound \eqref{ineq: torsion_uniform_bound} holds, and by Lemma \ref{lem: unif_bound_case} it follows that the harmonic $H$-flow $\{\xi(t)\}$ exists for all time $t\geqslant 0$ and subconverges as $t\to\infty$ to a torsion-free $H$-structure $\xi_{\infty}$.
\end{proof}

\begin{proof}[Proof of Theorem \ref{thm: blow_up}]
Define
\[
\varepsilon_{\ast}:=\min\{\varepsilon_0,\varepsilon_1\},
\] where $\varepsilon_0$ is given by Proposition \ref{prop: energy_gap}, and $\varepsilon_1$ is given by Lemma \ref{lem: unbdound_case}. Let $\{\xi(t)\}_{[0,\tau)}$ be the unique solution to the harmonic $H$-flow with initial condition $\xi_0$, such that the isometric homotopy class $[\xi_0]$ contains no torsion-free $H$-structure and $\varepsilon:=\cD(\xi_0)<\varepsilon_{\ast}$. If the uniform bound \eqref{ineq: torsion_uniform_bound} held, then by Lemma \ref{lem: unif_bound_case} the flow  would exist for all time and smoothly subconverge to a torsion-free $H$-structure $\xi_{\infty}$ in the same homotopy class of $\xi_0$, which is a contradiction. Therefore the opposite condition \eqref{ineq: torsion_blows_up} holds, and by Lemma \ref{lem: unbdound_case} we get the maximal time upper bound \eqref{ineq: max_time_energy_bound}. We conclude that the flow has a finite-time singularity at $\tau$, and $\tau\to 0$ as $\varepsilon\to 0$. This behaviour in higher dimensions contrasts with the $n=2$ case, cf. Remark \ref{rmk: 2-torus}.
\end{proof}

\subsection{Uniqueness of the long-time limit}
\label{sec: uniqueness}

By a slight modification in the definition \eqref{def: e(K)_1} of the constant $\varepsilon(\kappa)$ in the proof of Theorem \ref{thm: long-time_existence}, which possibly makes it smaller and with a more complicated dependence on $\kappa$, {we can prove a uniqueness result for the long-time limit along the harmonic $H$-flow. For this purpose, we shall require the group $H$ to satisfy
\begin{equation}
\label{eq: hyp h-part ddt of T}
    \pi_{\fh}\left(\frac{\partial}{\partial t}T_m\right)=c_H\pi_{\fh}\left([T_m,\Div T]\right),
\end{equation}
for some constant $c_H\in \bR$ depending only on $H$, where $T$ is the torsion of any solution $\{\xi(t)\}$ of the harmonic $H$-flow \eqref{eq: harmonic_flow}. Notice that any of the groups $H=\U(m)$, $\rm G_2$, $\rm Spin(7)$, $\Sp(k)\Sp(1)$ or $\{1\}$ satisfies \eqref{eq: hyp h-part ddt of T} for some $c_H\in \bR$ (cf. Remark \ref{rm: h-part ddt of T}).}
%for some $c\leqslant 2$ depending of $H$. Notice that any of the groups $U(m)$, $\rm G_2$, $\rm Spin(7)$, $\Sp(k)\Sp(1)$ or $\{1\}$ satisfy \eqref{eq: hyp h-part ddt of T} (cf. Remark \ref{rm: h-part ddt of T}).}
\begin{proposition}[Uniqueness of the limit along the flow]
\label{prop: uniqueness}
   Under the hypotheses of Theorem \ref{thm: long-time_existence}, and the {additional assumption  \eqref{eq: hyp h-part ddt of T} on the group $H$}, possibly adopting a smaller constant $\varepsilon(\kappa)>0$, still depending only on $\kappa$, the geometry $(M^n,g)$ and $H$, the harmonic $H$-flow $\{\xi(t)\}$ has a unique smooth limit as $t\to\infty$.
\end{proposition}
In order to prove this result, we are going to combine the previous techniques with the following two key lemmas, the first of which generalises known counterparts for $H=\rm G_2$ \cite{dgk-isometric}*{Lemmas 5.10-11} and $H=\mathrm{Spin}(7)$  \cite{Dwivedi-Loubeau-SaEarp2021}*{Lemma 5.7}.

\begin{lemma}[Convexity of the energy under small torsion]
\label{lem: convexity}
    Along a solution $\{\xi(t)\}$ of the harmonic $H$-flow \eqref{eq: harmonic_flow} on a closed Riemannian  manifold $(M^n,g)$,  {under the assumption \eqref{eq: hyp h-part ddt of T}}, one has
    \begin{align}
        \frac{d^2}{dt^2}\cE(\xi(t))=-\frac{d}{dt}\int_M |\Div T|^2\vol_g\geqslant \int_M(\Lambda- {4(5+2|c_H|)}|T|^2)|\Div T|^2\vol_g,
   \end{align}
    where $\Lambda$ is the first non-zero eigenvalue of the rough Laplacian $\nabla^{\ast}\nabla=-\Delta:\Omega^2(M)\to\Omega^2(M)$ of $g$ on $2$-forms.
\end{lemma}
\begin{proof}
    Using equation \eqref{eq: ddt_norm_T} of Corollary \ref{cor: evo_torsion}, with $S=0$ and $C=\Div T$, and integrating by parts, we get:
    \[
    \frac{d^2}{dt^2}\cE(\xi(t))=-\frac{d}{dt}\int_M |\Div T|^2\vol_g.
    \]
    Applying \eqref{eq: m_part ddt_T}, \eqref{eq: hyp h-part ddt of T} and again integrating by parts, yields
\begin{align*}
    \frac{d^2}{dt^2}\cE(\xi(t))
    &= 2\int_M \langle \pi_\fm(\nabla_m\Div T), \nabla_m \Div T\rangle\vol_g+2\int_M\langle\pi_{\fm}([T_m,\Div T],\nabla_m\Div T\rangle\vol_g\\ & \quad  {+2c_H\int_M\langle\pi_{\fh}([T_m,\Div T]),\nabla_m\Div T\rangle\vol_g}\\
    &= 2\int_M \langle \pi_\fm(\nabla_m\Div T), \nabla_m \Div T\rangle\vol_g+2\int_M\langle [T_m,\Div T],\nabla_m\Div T\rangle\vol_g.\\
    & \quad  {-2\int_M\langle\pi_{\fh}([T_m,\Div T]),\nabla_m\Div T\rangle\vol_g+2c_H\int_M\langle\pi_{\fh}([T_m,\Div T]),\nabla_m\Div T\rangle\vol_g}\\
    &\geqslant 2\int_M (|\nabla \Div T|^2-|\pi_\fh(\nabla \Div T)|^2)\vol_g-2\int_M|[T_m,\Div T]||\nabla_m\Div T|\vol_g\\
    & \quad  {-2(1-c_H)\int_M\langle\pi_{\fh}([T_m,\Div T]),\nabla_m\Div T\rangle\vol_g.}
\end{align*}
    Now,  $\nabla^H:=\nabla+T$ defines an $H$-connection, so that $\nabla^H_m\Div T\in \Omega^2_\fm$ and we have
    \[
      |\pi_\fh(\nabla_m \Div T)|^2=|\pi_\fh([T_m,\Div T])|^2\leqslant 4|T_m|^2|\Div T|^2.
    \]
    Combining the above with Young's inequality, we get
    \begin{align*}
     \frac{d^2}{dt^2}\cE(\xi(t))&\geqslant 2\int_M |\nabla \Div T|^2\vol_g- {2(2+|c_H|)\int_M|\pi_{\fh}([T_m,\Div T])|^2\vol_g}-4\int_M|\nabla_m\Div T||T_m||\Div T|\vol_g\\
    &\geqslant 2\int_M (|\nabla \Div T|^2- {4(2+|c_H|)|T_m|^2|\Div T|^2})\vol_g-4\int_M|\nabla_m\Div T||T_m||\Div T|\vol_g\\
    &\geqslant \int_M |\nabla \Div T|^2\vol_g- {4(5+2|c_H|)}\int_M|T|^2|\Div T|^2\vol_g\\
    &\geqslant \int_M (\Lambda- {4(5+2|c_H|)}|T|^2)|\Div T|^2\vol_g.
\end{align*}
%\begin{align*}
 %   \frac{d^2}{dt^2}\cE(\xi(t))&= 2\int_M |\nabla \Div T|^2\vol_g {-2(2-c)\int_M|\pi_{\fh}([T_m,\Div T]|^2\vol_g}-4\int_M\nabla_m\Div T_{ac}T_{m;cb}\Div T_{ab}\vol_g\\
  %  &\geqslant 2\int_M |\nabla \Div T|^2- {4(2-c)}|T|^2|\Div T)|^2)\vol_g-4\int_M|\nabla\Div T||T||\Div T|\vol_g\\
   % &\geqslant \int_M |\nabla \Div T|^2\vol_g- {4(5-2c)}\int_M|T|^2|\Div T|^2\vol_g\\
    %&\geqslant \int_M (\Lambda- {4(5-2c)}|T|^2)|\Div T|^2\vol_g.
%\end{align*}
    The last step in the above inequality is justified because, on a closed manifold, the non-negative elliptic operator $\nabla^{\ast}\nabla=-\Delta:\Omega^2(M)\to\Omega^2(M)$ has discrete spectrum, and its kernel consists of parallel $2$-forms, therefore there is $\Lambda>0$ such that
    \[
    \int_M |\nabla\omega|^2 \vol_g = \int_M \langle\nabla^{\ast}\nabla\omega,\omega\rangle \vol_g \geqslant \Lambda\int_M|\omega|^2 \vol_g,\quad\forall\omega\in (\ker\Delta)^{\perp_{L^2}}.
    \] 
    One can readily verify, integrating by parts, that for each parallel $2$-form $\omega$ one has:
    \[
    \int_M\langle\Div T,\omega\rangle = -\int_M\langle T,\nabla\omega\rangle = 0,
    \] i.e. $\Div T\in(\ker\Delta)^{\perp_{L^2}}$.
\end{proof}

\begin{lemma}[Interpolation]
\label{lem: interpolation}
    Let $(M^n,g)$ be a closed oriented Riemannian manifold, and suppose that $\xi$ is a compatible $H$-structure with torsion $T$. Suppose that
\[
|\nabla T|\leqslant C<\infty.
\] %and let $v_0>0$ be a constant such that for every $x\in M$ and $0<r\leqslant 1$ we have
%\[
%\mathrm{Vol}_g(B_r(x))\geqslant v_0 r^n.
%\]
Then, for every $\nu>0$, there exists $\mu=\mu(\nu,C,M^n,g,H)>0$ such that
\[
\cD(\xi)<\mu \quad\Longrightarrow\quad |T|<\nu.
\] 
\end{lemma}
\begin{proof}
    Using  \eqref{eq: T_equiv_nablaxi}, the proof is the same as in \cite[Lemma 5.12]{dgk-isometric}.
\end{proof}

\begin{proof}[Proof of Proposition \ref{prop: uniqueness}]
    From the proof of Theorem \ref{thm: long-time_existence}, taking $\varepsilon(\kappa)$ as in \eqref{def: e(K)_1}, the unique harmonic flow $\{\xi(t)\}$ with initial condition $\xi_0$ satisfying $\|\nabla\xi_0\|_{L^{\infty}(M)}\leqslant \kappa$ and $\cD(\xi_0)<\varepsilon(\kappa)$ exists for all time $t\geqslant 0$, and for any given sequence $t_n\to\infty$ there is a subsequence of $\xi(t_n)$ converging smoothly to a torsion-free limit $\xi_{\infty}$. 

    We claim that, taking $\varepsilon(\kappa)$ perhaps even smaller than in \eqref{def: e(K)_1}, yet still depending only on $\kappa$, $(M,g)$ and $H$, we must have
\[
t_{\ast}:=\sup\left\{t\geqslant 0: \overline{e}(t)\leqslant 2\kappa^2+\frac{1}{\kappa^2}\right\}=\infty.
\] 
    Indeed, suppose on the contrary that $t_{\ast}<\infty$. By Lemma \ref{lem: time_bounds} we know that
\[
t_{\ast}>\delta\arctan{\frac{1}{2\kappa^2}}=:\sigma_\kappa.
\] 
    Applying the Shi-type estimates [Proposition \ref{prop: Shi-estimates}] over $[t_{\ast}-\sigma_\kappa,t_{\ast}]$,  we find a constant $c_\kappa>0$, depending only on $\kappa$, $(M,g)$ and $H$, such that
\[
|\nabla T_{\xi(t_{\ast})}|<c_\kappa.
\] 
    Hence, by Lemma \ref{lem: interpolation}, there exists $\gamma_\kappa>0$, depending only on $\kappa$, $(M^n,g)$ and $H$, such that $\overline{e}(t_{\ast})< 2\kappa^2 + \frac{1}{\kappa^2}$ whenever  $\cD(\xi(t_{\ast}))<\gamma_\kappa$, which in turn contradicts the maximality of $t_{\ast}$. Since $\cD(\xi(t))\leqslant \cD(\xi_0)$ along the flow, redefining the  $\varepsilon(\kappa)$ of \eqref{def: e(K)_1} by
\begin{equation}
\label{def: e(K)_2}
    \varepsilon(\kappa)
    :=\min\left\{\gamma_\kappa,\varepsilon_0,\varepsilon_1,c_1^{-2}\left(\delta\arctan{\frac{1}{2\kappa^2}}\right)^{n-2}\right\},
\end{equation} 
    would guarantee $t_{\ast}=\infty$, as claimed.

    Again by the Shi-type estimates, we actually get a uniform constant $c_\kappa=c_\kappa(M,g,H)>0$ such that
\[
|\nabla T_{\xi(t)}|<c_\kappa,\quad\forall t\geqslant 1.
\]
    Now set $\nu := \left( {\frac{\Lambda}{8(5+2|c_H|)}}\right)^{1/2}$, where $\Lambda>0$ is the first non-zero eigenvalue of the rough Laplacian $\nabla^{\ast}\nabla=-\Delta$ on $\Omega^2(M)$, and let $\mu = \mu(\nu,c_\kappa,M^n,g,H)>0$ be the constant given by Lemma \ref{lem: interpolation}. Then, redefining the $\varepsilon(\kappa)$ in \eqref{def: e(K)_2} by the possibly smaller constant $\min\{\varepsilon(\kappa),\mu\}$, which in turn still depends only on $\kappa$, $(M^n,g)$ and $H$, and recalling that $\cD(\xi(t))$ is decreasing along the flow, it follows from Lemma \ref{lem: interpolation} that $|T_{\xi(t)}|^2< {\Lambda/8(5+2|c_H|)}$, for all $t\geqslant 1$. Hence, using Lemma \ref{lem: convexity}, we have
\[
\frac{d}{dt}\int_M |\Div T_{\xi(t)}|^2 \leqslant -\frac{\Lambda}{2}\int_M |\Div T_{\xi(t)}|^2,\quad\forall t\geqslant 1,
\] 
    which implies the exponential decay estimate
\begin{equation}
\label{ineq: exp_decay}
    \int_M |\Div T_{\xi(t)}|^2 \leqslant e^{-\frac{\Lambda}{2}(t-1)}\int_M |\Div T_{\xi(1)}|^2,\quad\forall t\geqslant 1.
\end{equation} 
    Thus, for all $1\leqslant s_1<s_2$, we get
\begin{align}
    \int_M |\xi(s_1)-\xi(s_2)| &= \int_M\int_{s_1}^{s_2} |\partial_t\xi(s)|ds\text{ }\vol_g
    \leqslant c\int_{s_1}^{s_2}\int_M |\Div T_{\xi(s)}|\vol_g\text{ }ds\nonumber\\
    &\leqslant c\int_{s_1}^{s_2}\left(\int_M |\Div T_{\xi(s)}|^2\vol_g\right)^{1/2} ds\nonumber\\
    &\leqslant c\|\Div T_{\xi(1)}\|_{L^2(M)}\int_{s_1}^{s_2}e^{-\frac{\Lambda}{4}(s-1)}ds.\label{ineq: uniqueness}
\end{align} 
    In particular, $\xi(t)$ decays exponentially to a unique limit in $L^1$, as $t\to\infty$. This means that any subsequential limit $\xi_{\infty}$, guaranteed by Theorem \ref{thm: long-time_existence}, is independent of the subsequence and is in fact the unique smooth limit of the flow $\{\xi(t)\}$ as $t\to\infty$.    
\end{proof}

\subsection{Stability of torsion-free structures}\label{sec: stability}

In this final paragraph, we prove a version of dynamical stability of torsion-free $H$-structures under the harmonic $H$-flow, using some of the same techniques as in the previous paragraph.

We start with the following long-time existence result under small initial torsion, which generalises  \cite{dgk-isometric}*{Theorem 5.13} and \cite{Dwivedi-Loubeau-SaEarp2021}*{Theorem 5.9}, established  for $H=\rG_2$ and $H=\S7$, respectively. 

\begin{theorem}[Long time existence under small initial torsion]
\label{thm: long_time_small_torsion}
    Let $(M^n,g)$ be a closed, oriented Riemannian manifold admitting a compatible $H$-structure $\xi_0$,  {where $H$ satisfies the assumption \eqref{eq: hyp h-part ddt of T}}. Then, for every $\delta>0$, there exists $\varepsilon=\varepsilon(\delta,M^n,g,H)>0$ such that if $\|\nabla\xi_0\|_{L^{\infty}(M)}<\varepsilon$, then the harmonic $H$-flow starting from $\xi_0$ exists for all time $t\geqslant 0$ and converges smoothly to a harmonic $H$-structure $\xi_{\infty}$, which furthermore satisfies $\|\nabla\xi_{\infty}\|_{L^{\infty}(M)}<\delta$.
\end{theorem}
\begin{proof}
    Suppose that $\overline{e}_0:=\|\nabla\xi_0\|_{L^{\infty}(M)}^2<\varepsilon^2<1$. Then, by Lemma \ref{lem: time_bounds}, there is $\sigma>0$ depending only on $(M^n,g)$ and $H$ such that
    \[
    t_{\ast}:=\max\{t\geqslant 0: \overline{e}(t)\leqslant 3\}>\sigma.
    \] 
    Suppose that $t_{\ast}<\infty$. Then, applying the Shi-type estimates [Proposition \ref{prop: Shi-estimates}] over the interval $[t_{\ast}-\sigma, t_{\ast}]$ we get a constant $c>0$ depending only on $(M^n,g)$ and $H$ such that
    \[
    |\nabla T_{t_{\ast}}|<c.
    \] Therefore, by Lemma \ref{lem: interpolation}, there is $\gamma>0$ depending only on $(M^n,g)$ and $H$, such that if $\cD(\xi(t_{\ast}))<\gamma$ then $\overline{e}({t_{\ast}})<3$, which would contradict the maximality of $t_{\ast}$. Thus, taking $$
    \varepsilon<\min\{1,\gamma^{1/2}\mathrm{Vol}(M)^{-1/2}\},$$ 
    since $\cD(\xi(t))$ is decreasing along the flow, it follows that $t_{\ast}=\infty$. In particular, for such small $\varepsilon$, the flow exists for all time $t\geqslant 0$ and $\sup_{t\geqslant 0}\overline{e}(t)\leqslant 3$. Thus, Lemma \ref{lem: unif_bound_case} implies that for any $t_n\to\infty$ there is a subsequence of $\xi(t_n)$ which converges smoothly to a harmonic $H$-structure $\xi_{\infty}$. Moreover, by the Shi-type estimates [Proposition \ref{prop: Shi-estimates}], it follows that there is $c>0$, depending only on $(M^n,g)$ and $H$, such that
    \[
    |\nabla T_{\xi(t)}|<c,\quad\forall t\geqslant 1.
    \]
    
    We now set $\nu:=\left( {\frac{\Lambda}{8(5+2|c_H|)}}\right)^{1/2}$, where $\Lambda>0$ is the first non-zero eigenvalue of the rough Laplacian on $\Omega^2(M)$, and let $\mu=\mu(\nu,c,M^n,g,H)>0$ be the constant given by Lemma \ref{lem: interpolation}. Then, taking 
    $$
    \varepsilon<\min\{1,\mu^{1/2}\mathrm{Vol}(M)^{-1/2},\gamma^{1/2}\mathrm{Vol}(M)^{-1/2}\},
    $$ 
    it follows that $|T_{\xi(t)}|^2<\Lambda/8(5+2|c_H|)$, for all $t\geqslant 1$, and we get the exponential decay \eqref{ineq: exp_decay} from the convexity Lemma \ref{lem: convexity}. Hence one has \eqref{ineq: uniqueness}, for all $1\leqslant s_1<s_2$,  which in turn implies that $\xi(t)$ decays exponentially to a unique limit in $L^1$, as $t\to\infty$. This means that any subsequential limit $\xi_{\infty}$, guaranteed by Lemma \ref{lem: unif_bound_case}, is independent of the subsequence and is in fact the unique smooth limit of the flow $\xi(t)$ as $t\to\infty$. Finally, given $\delta>0$, by choosing $\varepsilon=\varepsilon(\delta, M^n,g,H)>0$ small enough we can also achieve $\|\nabla\xi_{\infty}\|_{L^{\infty}(M)}<\delta$, using Lemma \ref{lem: interpolation}.
\end{proof}
 {
\begin{remark}
    In the above proof of Theorem \ref{thm: long_time_small_torsion}, the hypothesis \eqref{eq: hyp h-part ddt of T} was only used to guarantee the uniqueness of the limit along the flow as $t\to\infty$.
\end{remark}}

    Combining the energy gap of Proposition \ref{prop: energy_gap} with the above Theorem \ref{thm: long_time_small_torsion}, we obtain the following stability feature of torsion-free structures.

\begin{theorem}[Stability of torsion-free structures under the harmonic flow]
\label{thm: stability}
    Let $(M^n,g)$ be a closed, oriented Riemannian manifold admitting a compatible $H$-structure,  {where $H$ satisfies the assumption \eqref{eq: hyp h-part ddt of T}}. Then the following hold:
\begin{itemize}
    \item[(i)] There is a constant $\kappa_0=\kappa_0(M,g,H)>0$ such that, if  $\xi_0$ is a compatible $H$-structure satisfying $\|\nabla\xi_0\|_{L^{\infty}(M)}<\kappa_0$, then the harmonic $H$-flow \eqref{eq: harmonic_flow} starting at $\xi_0$ exists for all $t\geqslant 0$ and converges smoothly to a torsion-free $H$-structure $\xi_{\infty}$, as $t\to\infty$.
    
    \item[(ii)] If $(M^n,g)$ admits a compatible torsion-free $H$-structure $\overline{\xi}$, then for every $\delta>0$, there exists  $\overline{\varepsilon}(\delta,M,g,H)>0$ such that, for any compatible $H$-structure $\xi_0$ with $\|\xi_0-\overline{\xi}\|_{ {C^2}(M)}<\overline{\varepsilon}$, the harmonic $H$-flow \eqref{eq: harmonic_flow} with initial condition $\xi_0$ exists for all $t\geqslant 0$, satisfies the estimate $\|\xi_t-\overline{\xi}\|_{C^1(M)}<\delta$ for all $t\geqslant 0$, and converges smoothly to a torsion-free $H$-structure $\xi_{\infty}$ as $t\to\infty$.
\end{itemize} 
\end{theorem}
\begin{proof}\quad
    \begin{itemize}
    \item[(i)] 
    Take $\delta:=\left(\frac{\varepsilon_0}{\mathrm{Vol}(M)}\right)^{1/2}$, where $\varepsilon_0=\varepsilon_0(M^n,g,H)>0$ is the constant given by Proposition \ref{prop: energy_gap}, and let $\kappa_0:=\varepsilon(\delta,M^n,g,H)$ be given by Theorem \ref{thm: long_time_small_torsion}. The latter result then implies that the harmonic $H$-flow starting at $\xi_0$ with $\|\nabla\xi_0\|_{L^{\infty}(M)}<\kappa_0$ must exist for all $t\geqslant 0$ and converge smoothly to a harmonic $H$-structure $\xi_{\infty}$ satisfying $|\nabla\xi_{\infty}|<\delta$. Thus, from the definition of $\delta$ and Proposition \ref{prop: energy_gap}, it follows that $\xi_{\infty}$ is torsion-free.
    
    \item[(ii)] 
    It follows from (i) that, taking $0<\overline{\varepsilon}<\kappa_0(M,g,H)$,  {and assuming $\|\xi_0-\overline{\xi}\|_{C^2(M)}<\overline{\varepsilon}$}, the harmonic $H$-flow $\xi(t)$ with initial condition $\xi_0$ exists for all $t\geqslant 0$ and converges smoothly to a torsion-free $H$-structure $\xi_{\infty}$. In fact, from the choice of $\kappa_0(M,g,H)$ in the above proof of (i), and the proof of Theorem \ref{thm: long_time_small_torsion}, we can also assume that
    \begin{equation}\label{ineq: bounded_torsion_for_all_time}
        \overline{e}(t)\leqslant 3,\quad\forall t\geqslant 0,
    \end{equation} and, by the Shi-type estimates [Proposition \ref{prop: Shi-estimates}], given any $\eta=\eta(M^n,g,H)>0$, there is a constant $c>0$ depending only on $(M^n,g)$ and $H$, so that
    \begin{equation}\label{ineq: bounded_torsion_for_t>1}
        \|\nabla T_{\xi(t)}\|_{L^{\infty}(M)}\leqslant c,\quad\forall t\geqslant\eta.
    \end{equation}
    It remains to show that, given any $\delta>0$, then taking $\overline{\varepsilon}>0$ {, perhaps even smaller than above}, depending only on $\delta$, $(M^n,g)$ and $H$, then we also have $\|\xi_t - \overline{\xi}\|_{C^1(M)}<\delta$, for all $t\geqslant 0$.  {In order to prove this, we first make a sequence of reductions to restate the problem in a more convenient way.}
    
    Assuming at first  {$\overline{\varepsilon}<\min\{\delta/2,\kappa_0(M,g,H)\}$}, then
    \[
    \|\xi_t - \overline{\xi}\|_{C^1}\leqslant \|\xi_t - \xi_0\|_{C^1} + \|\xi_0 - \overline{\xi}\|_{C^1} < \|\xi_t - \xi_0\|_{C^1}+\delta/2,\quad\forall t\geqslant 0.
    \] 
    So it suffices to show that, for $\overline{\varepsilon}=\overline{\varepsilon}(\delta,M,g,H)>0$, perhaps even smaller, one has $\|\xi_t - \xi_0\|_{C^1}<\delta/2$ for all $t\geqslant 0$.  {In fact, allowing $\overline{\varepsilon}>0$ to be possibly smaller, with same dependence, note that} one can also achieve $\|\nabla\xi_t\|_{C^0}<\delta/8$, for all $t\geqslant 0$, by the interpolation Lemma \ref{lem: interpolation} and the fact that $\|\nabla\xi_t\|_{L^2}^2$ decreases along the flow. Then, by also imposing $\overline{\varepsilon}<\delta/8$, it follows that $\|\nabla\xi_t - \nabla\xi_0\|_{C^0}<\delta/4$ for all $t\geqslant 0$.  {Using this last inequality and} the mean value theorem  {(see e.g. the proof of \cite[Lemma 5.12]{dgk-isometric})}, one can choose $\mu>0$ sufficiently small, depending only on $\delta$ and $(M^n,g)$, such that $\|\xi_t - \xi_0\|_{C^0}<\delta/4$ whenever $\|\xi_t-\xi_0\|_{L^1}<\mu$,  {which combined with the previous derivative bound would imply} $\|\xi_t - \xi_0\|_{C^1}<\delta/2$, as desired. These arguments show that the proof boils down to proving that, for $\overline{\varepsilon}=\overline{\varepsilon}(\delta,M,g,H)>0$ small enough, we can achieve
    \begin{equation}\label{ineq: stability_final_bound}
        \|\xi_t-\xi_0\|_{L^1}<\mu,\quad\forall t\geqslant 0.
    \end{equation} 
    In order to prove the $L^1$-bound \eqref{ineq: stability_final_bound}, we extend the bound \eqref{ineq: bounded_torsion_for_t>1} for all $t\geqslant 0$ by proving
    \begin{equation}\label{ineq: bounded_torsion_for_t<1}
        \|\nabla T_{\xi(t)}\|_{L^{\infty}(M)}\leqslant c,\quad\forall t\in [0,\eta].
    \end{equation} If we prove this, then proceeding as in the final parts of the proofs of Proposition \ref{prop: uniqueness} and Theorem \ref{thm: long_time_small_torsion}, we get that the following holds whenever $\overline{\varepsilon}=\overline{\varepsilon}(\delta,M,g,H)>0$ is small enough:
\begin{align}
    \|\xi_t - \xi_0\|_{L^1} &\leqslant c\|\Div T_{\xi_0}\|_{L^2}\int_0^t e^{{-\frac{\Lambda}{4}}s} ds = \frac{4c}{\Lambda}\|\Div T_{\xi_0}\|_{L^2}(1 - e^{-\frac{\Lambda}{4}t})\nonumber\\
    &<\frac{4c}{\Lambda}\overline{\varepsilon}\mathrm{Vol}_g(M)^{1/2},\quad\forall t\geqslant 0,
    \label{eq: a priori_bound}
\end{align} 
    where in the last step we used that  {$\|\Div T_{\xi_0}\|_{C^0}\leqslant c\|\nabla^2\xi_0\|_{C^0}<c\overline{\varepsilon}$, since $\|\xi_0 - \overline{\xi}\|_{C^2(M)}<\overline{\varepsilon}$ and $\nabla\overline{\xi}=0$. Hence,} by taking $\overline{\varepsilon}<\frac{\Lambda\mu}{4c\mathrm{Vol}_g(M)^{1/2}}$, we would get the $L^1$-bound \eqref{ineq: stability_final_bound} and therefore prove the desired result.

    Having made all the previous considerations, the rest of this proof is dedicated to obtain the estimate \eqref{ineq: bounded_torsion_for_t<1}. To avoid cumbersome notation, in what follows we suppress the $t$ and denote $\xi=\xi(t)$, $T=T(\xi(t))$ etc. We begin by recalling Lemma \ref{lem: torsion_equiv_nablaxi} and the hypothesis \eqref{eq: key_assumption}.  Note that not only $|\nabla\xi|^2=c|T|^2$, but also since $\nabla_k\xi = T_k\diamond\xi$, we can invert the $(\cdot{}\diamond\xi)|_{\Omega_{\fm}^2}$ operator [Lemma \ref{lem: ker_diamond}] and write schematically $T = \xi\circledast\nabla\xi$ so that
    \begin{align}
        \nabla T &= \nabla\xi\circledast\nabla\xi + \xi\circledast\nabla^2\xi,\nonumber\\
        \nabla^2 T &= \nabla^2\xi\circledast\nabla\xi + \xi\circledast\nabla^3\xi.\nonumber
    \end{align} In particular, since $\xi=\xi(t)$ has constant norm, depending only on the geometric data, and since we have a uniform bound $|T|\sim|\nabla\xi|\leqslant c$ for all $t\geqslant 0$ by \eqref{ineq: bounded_torsion_for_all_time}, it follows that
    \begin{align}
    |\nabla T| &\leqslant c+c|\nabla^2\xi|,\quad\text{and}\nonumber\\
    |\nabla^2 T| &\leqslant c|\nabla^2\xi| + c|\nabla^3\xi|,\quad\forall t\geqslant 0.\nonumber
    \end{align} Now, by a similar computation as we did in the proof of Lemma \ref{lem: dif_ineq}, recalling the equation \eqref{eq: rough_laplacian_of_harm_flow_sol} satisfied by the solution $\xi=\xi(t)$,
\[
\Delta\xi = \partial_t\xi + T_k\diamond(\nabla_k\xi),
\] we compute:
\begin{align}
    \frac{1}{2}(\partial_t-\Delta)|\nabla^2\xi|^2 &= \langle\nabla^2\partial_t\xi,\nabla^2\xi\rangle -\langle\Delta\nabla^2\xi,\nabla^2\xi\rangle - |\nabla^3\xi|^2\nonumber\\
    &=\langle\nabla^2\Delta\xi - \Delta\nabla^2\xi,\nabla^2\xi\rangle - \langle\nabla^2(T_k\diamond\nabla_k\xi),\nabla^2\xi\rangle - |\nabla^3\xi|^2,\quad\forall t\geqslant 0.\label{eq: bound_0_nabla2}
\end{align} On the other hand, using the Ricci identity we can write schematically
\[
\nabla^2\Delta\xi - \Delta\nabla^2\xi = R\circledast\nabla^2\xi + \nabla R\circledast\nabla\xi +\nabla^2R\circledast\xi.
\] Thus, since $(M^n,g)$ is closed and therefore has bounded geometry, combining all the above facts and Young's inequality, we get
\begin{equation}\label{ineq: bounds_1_nabla2}
\langle\nabla^2\Delta\xi - \Delta\nabla^2\xi,\nabla^2\xi\rangle \leqslant c|\nabla^2\xi| + c|\nabla^2\xi|^2\leqslant c + c|\nabla^2\xi|^2,    
\end{equation}
and 
\begin{align}
 \langle\nabla^2(T_k\diamond\nabla_k\xi),\nabla^2\xi\rangle &\leqslant c|\nabla^2 T ||\nabla^2\xi| + c|\nabla T||\nabla^2\xi|^2 + c|\nabla^3\xi||\nabla^2\xi| \nonumber\\
 &\leqslant c|\nabla^2\xi|^2 + c|\nabla^2\xi|^3 + c|\nabla^3\xi||\nabla^2\xi|\nonumber\\
 &\leqslant c|\nabla^2\xi|^2 + c|\nabla^2\xi|^3 + |\nabla^3\xi|^2.\label{ineq: bounds_2_nabla2}
\end{align} Thus, plugging inequalities \eqref{ineq: bounds_1_nabla2} and \eqref{ineq: bounds_2_nabla2} into equation \eqref{eq: bound_0_nabla2}, we estimate:
\begin{align}
    \frac{1}{2}(\partial_t-\Delta)|\nabla^2\xi|^2 &\leqslant c + c|\nabla^2\xi|^2 + c|\nabla^2\xi|^3\nonumber\\
    &\leqslant c + c|\nabla^2\xi|^2 + c|\nabla^2\xi|^4\nonumber\\
    &\leqslant c(|\nabla^2\xi|^4+1),\quad\forall t\geqslant 0.\nonumber
\end{align} In other words, considering the function $f(\xi):=|\nabla^2\xi|^2(x,t)$, we deduced that
\begin{equation}\label{ineq: heat_op_bound_nabla2}
    (\partial_t-\Delta)f(\xi) \leqslant c(f(\xi)^2 + 1),\quad\text{on }M\times [0,\infty),
\end{equation} which is the same sort of inequality as we have deduced in the Bochner estimate \eqref{ineq: Bochner_estimate}. Therefore, writing $\overline{f}(t):=\max_M f(\xi(t))$, we can use inequality \eqref{ineq: heat_op_bound_nabla2} to derive an analogue of Lemma \ref{lem: time_bounds} for $\overline{f}(t)$; in particular, if $\overline{f}_0:=\overline{f}(0)$ then we can infer
\[
\overline{f}(t)\leqslant\frac{\overline{f}_0 + \tan{ct}}{1-\overline{f}_0\tan{ct}},\quad\forall t\in\left[0,\frac{1}{c}\arctan\frac{1}{\overline{f}_0}\right).
\] Now, since we are assuming $\|\xi_0-\overline{\xi}\|_{C^2}<\overline{\varepsilon}<1$, and $\nabla\overline{\xi}=0$, it follows that $\overline{f}_0=\max_M|\nabla^2\xi_0|^2\leqslant 1$ is uniformly bounded and hence there is a time $\eta=\eta(M^n,g,H)>0$ such that $\overline{f}(t)\leqslant c$ for all $t\in [0,\eta]$, i.e. $\|\nabla^2\xi\|_{L^{\infty}(M)}\leqslant c$ for all $t\in [0,\eta]$. Finally, since we already observed earlier that $|\nabla T| \leqslant c+c|\nabla^2\xi|$ for all $t$, we then get inequality \eqref{ineq: bounded_torsion_for_t<1}, as we wanted. This completes the proof.
    \qedhere
    \end{itemize}
\end{proof}

 {We finish noting that part (i) of Theorem \ref{thm: stability} can also be rephrased as a gap result:
\begin{corollary}[$L^{\infty}$ torsion gap in the absence of torsion-free structures]
    Let $(M^n,g)$ be a closed, oriented Riemannian manifold admitting a compatible $H$-structure $\xi_0$. If there is no compatible torsion-free $H$-structure in the same isometric homotopy class of $\xi_0$ then
    \begin{equation}\label{eq: L_infty_gap}
    \|\nabla\xi_0\|_{L^{\infty}}\geqslant \kappa_0,
    \end{equation}
    where $\kappa_0=\kappa_0(M,g,H)>0$ is the constant given by Theorem \ref{thm: stability}. In particular, if $(M,g)$ admits no torsion-free compatible $H$-structure, then all compatible $H$-structures $\xi_0$ on $(M,g)$ satisfy the $L^{\infty}$ torsion gap \eqref{eq: L_infty_gap}.
\end{corollary}}

\newpage
\appendix 

\section{Global version of an almost-monotonicity formula}\label{ap: global_monotonicity}

We derive a Hamilton-type monotonicity formula along the harmonic $H$-flow \eqref{eq: harmonic_flow}, using the
backward heat kernel of the background Riemannian metric,  along the lines of \cite{Hamilton1993}. This  generalises known results in the cases $H=\U(m)$, $\rm G_2$ and $\rm Spin(7)$, respectively \cite[Lemma 3.1]{he2021}, \cite[Lemma 5.2 and Theorem 5.3]{dgk-isometric}, \cite[Theorem 6.1]{Grigorian2019} and \cite[Lemma 5.1 and Theorem 5.2]{Dwivedi-Loubeau-SaEarp2021}. Just as in the case of the local monotonicity formulas proved in §\ref{subsec: local_monotonicity}, the key ingredients are the evolution of the torsion along the flow, given by Corollary \ref{cor: evo_torsion}, and the Bianchi-type identity of Corollary \ref{cor: m-part_Bianchi_identity}.

Let $(M^n,g)$ be a connected and oriented Riemannian $n$-manifold of bounded geometry, admitting a compatible $H$-structure, for some closed and connected subgroup $H\subset\SO(n)$ of the form $\mathrm{Stab}_{\SO(n)}(\xi_{\circ})$, as in Section \ref{sec: harmonic_flow}. For $(x_0,t_0)\in M\times\mathbb{R}$, we let $G=G_{(x_0,t_0)}$ be the fundamental solution of the backward heat equation on $(M,g)$, starting with the delta function $\delta_{x_0}$ at time $t_0$:
\begin{align*}
    \Big( \frac{\partial}{\partial t} - \Delta \Big) G &= 0,\quad\forall t\in (-\infty,t_0).\\
    \lim_{t \to t_0^{-}} G &= \delta_{x_0}.
\end{align*} 
We also let $f=f_{(x_0,t_0)}\in C^{\infty}(M)$ be such that
\[
G = \frac{\exp(-f)}{\big( 4\pi (t_0 -t) \big)^\frac{n}{2 }}.
\] 
For instance, when $(M^n,g)=(\mathbb{R}^n,g_{\circ})$, then $G=G_{(x_0,t_0)}$ is given by \eqref{eq: heat_kernel}, and in particular $f(x)=|x-x_0|^2/4(t_0-t)$ in this case. 

Now, if $\{\xi(t)\}_{t\in [0,t_0)}$ is a family of compatible $H$-structures on $(M,g)$ solving the harmonic $H$-flow equation \eqref{eq: harmonic_flow}, define the function
\begin{align}\label{thetadefn}
  \Theta_{(x_0, t_0)} (\xi(t)) := (t_0 -t) \int_M |T(t)|^2 G \, \vol_g.  
\end{align} Note that $\Theta_{(x_0, t_0)}(\xi(t))$ is invariant under the parabolic rescalings of Lemma \ref{lem: rescaling}. Moreover, we can deduce the following evolution along the flow:

\begin{lemma}\label{lemma: ddt_theta_xi}
Let $\{\xi(t)\}_{t\in [0,t_0)}$ be a solution to the harmonic $H$-flow equation \eqref{eq: harmonic_flow} on $(M^n,g)$. If $(M^n,g)$ is noncompact, assume further that the torsion $T(t)$ of $\xi(t)$ has at most polynomial growth at infinity. Then the evolution of $\Theta_{(x_0,t_0)}(\xi(t))$ along the flow is given by
\begin{align}\label{eq: ddt_theta_xi}
    \frac{d}{d t} \Theta_{(x_0, t_0)}(\xi(t)) &+ 2(t_0-t)\int_M |\Div T-\nabla f\lrcorner T|^2 G\vol_g\nonumber\\
    &=-2(t_0-t)\int_M\left(\nabla_m \nabla_s G-\frac{\nabla_m G\nabla_s G}{G}+\frac{G g_{ms}}{2(t_0-t)}\right)g^{mn}g^{rs}\langle T_n,T_r\rangle\vol_g\nonumber\\
    &\quad-2(t_0-t)\int_M(\nabla_a\Ric_{bm} -\nabla_b\Ric_{am})T_{m;ij}g^{mn}g^{ai}g^{bj} G\vol_g\nonumber\\
    &\quad-(t_0-t)\int_M\langle R_{mr},-2[T_s,T_n] +\pi_\fm([T_s,T_n])+\pi_\fm(R_{ns}))\rangle g^{mn}g^{rs}G\vol_g.
\end{align}
\end{lemma}
\begin{proof}
    To justify all the arguments below in the noncompact case, note that the polynomial growth assumption on $T(t)$ implies, via Shi-type estimates [Proposition \ref{prop: Shi-estimates}], that $|\nabla^m T|(t)$ also grows at most polynomially  at infinity. This together with the well-known fact that the (backward) heat kernel $G$ of a Riemannian manifold of bounded geometry decays exponentially \cite[Corollary 3.1]{li1986parabolic} ensures that all the following integrals and integrations by parts are well-defined.

    Using the evolution equation in  \eqref{eq: ddt_norm_T} with $S=0$, we can compute:
\begin{align*}
    \frac{d}{dt}\Theta
    =&\int_M\Big(-|T|^2G+(t_0-t)\frac{\partial}{\partial t}|T|^2G+(t_0-t)|T|^2\frac{\partial G}{\partial t} \Big)\vol_g \\ =&\int_M\Big(\big(-|T|^2+2(t_0-t)\langle\nabla C,T\rangle\big)G-(t_0-t)|T|^2\Delta G \Big)\vol_g\\
    =&\int_M\Big(-|T|^2 G+2(t_0-t)\nabla_mC_{ab}T_{n;ij}g^{mn}g^{ai}g^{bj}G-(t_0-t)T_{m;ab}T_{n;ij}g^{mn}g^{ai}g^{bj}\Delta G\Big)\vol_g
\end{align*}
    Integrating by parts, we get
\begin{align*}
    \frac{d}{dt}\Theta=&\int_M\Big(-|T|^2G  -2(t_0-t) C_{ab}\nabla_mT_{n;ij}g^{mn}g^{ai}g^{bj}G 
    -2(t_0-t) C_{ab}T_{n;ij}g^{mn}g^{ai}g^{bj}\nabla_m G\\
    &+2(t_0-t)\nabla_rT_{m;ab}T_{n;ij}g^{mn})g^{ai}g^{bj}\nabla_s G g^{rs}\Big)\vol_g.
\end{align*}
    Now using the Bianchi-type identity \eqref{eq: m-part_Bianchi_identity} one has
\begin{align*}
    \frac{d}{dt}\Theta
    =&\int_M \Big(-|T|^2G-2(t_0-t)C_{ab}\Div T_{ij}g^{ai}g^{bj}G -2(t_0-t)C_{ab}(\nabla G\lrcorner T)_{ij}g^{ai}g^{bj}\\
    &+2(t_0-t)\big(\nabla_mT_{r;ab}+[T_m,T_r]_{ab}+R_{mrab}\big)T_{n;ij}g^{mn}g^{ai}g^{bj}\nabla_s G g^{rs}\Big)\vol_g.
    \end{align*}
    Notice that $[T_m,T_r]_{ab}T_{n;ij}g^{mn}g^{ai}g^{bj}=\langle[T_m,T_r],T_n\rangle g^{mn}=-\tr(T_mT_rT_n-T_rT_mT_n)g^{mn}=0$, then applying integrating by parts,
\begin{align*}
    \frac{d}{dt}\Theta
    =&\int_M \Big(-|T|^2G-2(t_0-t)\langle C,\Div T\rangle G -2(t_0-t)\langle C, \nabla G\lrcorner T\rangle\\
    &-2(t_0-t)T_{r;ab}\nabla_mT_{n;ij}g^{mn}g^{ai}g^{bj}\nabla_s G g^{rs}-2(t_0-t)T_{r;ab}T_{n;ij}g^{mn}g^{ai}g^{bj}\nabla_m\nabla_s G g^{rs}\\
    &-2(t_0-t)(\nabla_sR_{mrab}T_{n;ij}+R_{mrab}\nabla_sT_{n;ij})G g^{mn}g^{ai}g^{bj}g^{rs}\Big)\vol_g.
\end{align*}
    Thus, using $\nabla_i
    G=-G\nabla_i f$, the Bianchi-type identity \eqref{eq: m-part_Bianchi_identity}, the skew-symmetry $R_{mrab}=-R_{rmab}$ and the second Bianchi identity \eqref{eq: riem2ndBid} for $g^{rs}\nabla_s R_{rmab}$, we get
\begin{align*}
    \frac{d}{dt}\Theta
    =&\int_M \Big(-|T|^2G-2(t_0-t)\langle C,\Div T\rangle G +2(t_0-t)\langle C,\nabla f\lrcorner T\rangle G \\
    &+2(t_0-t)\langle \nabla f\lrcorner T,\Div T\rangle G-2(t_0-t)\nabla_m\nabla_s G\langle T_n,T_r\rangle g^{mn}g^{rs}\\
    &-2(t_0-t)(\nabla_a\Ric_{bm} -\nabla_b\Ric_{am})T_{m;ij}g^{mn}g^{ai}g^{bj} G\\
    & -(t_0-t)\langle R_{mr},-2[T_s,T_n] +\pi_\fm([T_s,T_n])+\pi_\fm(R_{ns})\rangle g^{mn}g^{rs}G\Big)\vol_g.
\end{align*}
    Finally, replacing $C=\Div T$ and completing the square using again $\nabla_i
    G=-G\nabla_i f$, we conclude
\begin{align*}
    \frac{d}{d t} \Theta 
    &=-2(t_0-t)\int_M \Big(\frac{g^{nr}}{2(t_0-t)}\langle T_n,T_r\rangle G +|\Div T|^2G-2\langle \nabla f\lrcorner T,\Div T\rangle G +\nabla_m\nabla_s G g^{mn}g^{rs}\langle T_n,T_r\rangle\\
    &\hspace{3cm} +(\nabla_a\Ric_{bm} -\nabla_b\Ric_{am})T_{m;ij}g^{mn}g^{ai}g^{bj} G\\
    &\hspace{3cm} +\frac{1}{2}\langle R_{mr},-2[T_s,T_n]  +\pi_\fm([T_s,T_n]) +\pi_\fm(R_{ns})\rangle g^{mn}g^{rs}G\Big)\vol_g\\
    &=-2(t_0-t)\int_M\Big(\big(\nabla_m \nabla_s G-\frac{\nabla_m G\nabla_s G}{G} +\frac{G g_{ms}}{2(t_0-t)}\big)g^{mn}g^{rs}\langle T_n,T_r\rangle +|\Div T-\nabla f\lrcorner T|^2G\\
    &\hspace{3cm} +(\nabla_a\Ric_{bm} -\nabla_b\Ric_{am})T_{m;ij}g^{mn}g^{ai}g^{bj} G\\
    &\hspace{3cm} +\frac{1}{2}\langle R_{mr},-2[T_s,T_n] +\pi_\fm([T_s,T_n])+\pi_\fm(R_{ns})\rangle g^{mn}g^{rs}G\Big)\vol_g.
    \qedhere
    \end{align*}
\end{proof}
%Using the identity $\nabla^rR_{rmab}=\nabla_aR_{bm}-\nabla_bR_{am}$, obtained from second Bianchi identity for the Riemannian curvature, we have:    

\begin{theorem}[Hamilton-type almost-monotonicity]
    Let $(M^n,g)$ be an oriented Riemannian manifold of bounded geometry, admitting a compatible $H$-structure, with $H=\mathrm{Stab}_{\SO(n)}(\xi_{\circ})$, and let $\{\xi(t)\}_{t\in [0,t_0)}$ be a solution to the harmonic $H$-flow \eqref{eq: harmonic_flow}. If $(M^n,g)$ is noncompact, assume further that the torsion $T(t)$ of $\xi(t)$ grows at most polynomially fast at infinity. 

    Then, for any $x_0\in M$ and $\max\{0,t_0-1\}<t_1\leqslant t_2<t_0$, there is a constant $c>0$ depending only on the geometry $(M^n,g)$ (and possibly $H$) such that the following holds:
\begin{equation}
\label{eq: global_monotonicity}
    \Theta_{(x_0,t_0)}(\xi(t_2))\leqslant c\Theta_{(x_0,t_0)}(\xi(t_1)) +c(E_0+1)(t_2-t_1),
\end{equation} 
    where $E_0:=\cE(\xi(0))$.
    Moreover, if $(M^n,g)=(\mathbb{R}^n,g_{\circ})$ is the flat Euclidean space, then for all $x_0\in\mathbb{R}^n$ and $0<t_1\leqslant t_2<t_0$, we have
\begin{equation}\label{eq: global_monotonicity_Rn}
\Theta_{(x_0,t_0)}(\xi(t_2))\leqslant\Theta_{(x_0,t_0)}(\xi(t_1)),
\end{equation} and equality holds if and only if 
\begin{equation}
\label{eq: equality_monotonicity}
    \Div T(t) =  \frac{x-x_0}{2(t_0-t)}\lrcorner T(t),
    \quad\forall 
    t\in [t_1,t_2].
\end{equation}
\end{theorem}
\begin{proof}
    We shall bound each of the terms on the right-hand side of  \eqref{eq: ddt_theta_xi} in Lemma \ref{lemma: ddt_theta_xi}.
    
    When $(M^n,g)=(\mathbb{R}^n,g_{\circ})$, the curvature $R \equiv 0$ vanishes identically, and $G$ is given by \eqref{eq: heat_kernel}, so that the term
\[
\nabla_m\nabla_s G -\frac{\nabla_m G \nabla_s G}{G} + \frac{G g_{ms}}{2(t_0-t)}
\] 
    also vanishes identically, and furthermore $\nabla f = \frac{x-x_0}{2(t_0-t)}$. Therefore, in this case we get from Lemma \ref{lemma: ddt_theta_xi} that
\[
\frac{d}{dt}\Theta_{(x_0,t_0)}(\xi(t)) = -2(t_0-t)\int_M |\Div T-\frac{x-x_0}{2(t_0-t)}\lrcorner T|^2 G\vol_g\leqslant 0,
\] 
    which upon integration directly implies the strict monotonicity \eqref{eq: global_monotonicity_Rn}, with equality if and only if \eqref{eq: equality_monotonicity} holds.

    In the general case,  Hamilton's matrix Harnack estimate,  combining \cite[Theorem 4.3]{hamilton1993matrix}  with \cite[Corollary 1.3]{Hamilton1993}), yields  constants $c,c'>0$ depending only on $(M^n,g)$ such that, for all $t\in(t_0-1,t_0)$,
\[
\nabla_m\nabla_s G -\frac{\nabla_m G \nabla_s G}{G} + \frac{G g_{ms}}{2(t_0-t)} \geqslant -c\left(1+G\ln\left(\frac{c'}{(t_0-t)^{n/2}}\right)\right)g_{ms}.
\] 
    Observing that $t_0-t\leqslant 1$, the first term on the right-hand side of \eqref{eq: ddt_theta_xi} can be estimated as follows:
\begin{align*}
    -2(t_0-t)&\int_M\left(\nabla_m \nabla_s G-\frac{\nabla_m G\nabla_s G}{G}+\frac{G g_{ms}}{2(t_0-t)}\right)g^{mn}g^{rs}\langle T_n,T_r\rangle\vol_g\\
    &\leqslant c(t_0-t)\int_M|T|^2\vol_g + c\ln\left(\frac{c'}{(t_0-t)^{n/2}}\right)(t_0-t)\int_M|T|^2 G\vol_g\\
    &\leqslant c\cE(\xi(t)) + c\ln\left(\frac{c'}{(t_0-t)^{n/2}}\right)\Theta_{(x_0,t_0)}(\xi(t))\\
    &\leqslant cE_0 + c\ln\left(\frac{c'}{(t_0-t)^{n/2}}\right)\Theta_{(x_0,t_0)}(\xi(t)),
\end{align*} 
    where in the last step we used the fact that $\cE(\xi(t))\leqslant c\cD(\xi(t))\leqslant c\cD(\xi(0))\leqslant cE_0$, since $\{\xi(t)\}$ is a solution to \eqref{eq: harmonic_flow}, which is the negative gradient flow of $\cD$ restricted to isometric structures [Lemma \ref{lem: gradient_dirichlet_along_iso_var}], and also $\cD(\xi)\sim\cE(\xi)$, by  \eqref{ineq: T_equiv_nablaxi}. As to the second term in the right-hand side of \eqref{eq: ddt_theta_xi}, using that $(M,g)$ has bounded geometry, $\int_M G\vol_g = 1$, $t_0-t\leqslant 1$, and  Young's inequality, we can estimate: 
\begin{align*}
    -2(t_0-t) &\int_M(\nabla_a\Ric_{bm} -\nabla_b\Ric_{am}) T_{m;ij}g^{mn}g^{ai}g^{bj} G\vol_g\\
    &\leqslant c(t_0-t)\int_M |\nabla\Ric|^2 G\vol_g + c(t_0-t)\int_M|T|^2 G\vol_g\\
    &\leqslant c(t_0-t)\int_M G\vol_g + c\Theta_{(x_0,t_0)}(\xi(t))\\
    &\leqslant c\left(1+\Theta_{(x_0,t_0)}(\xi(t))\right).
\end{align*} Next, using the same facts as before, we estimate the last term in the right-hand side of \eqref{eq: ddt_theta_xi} as follows:
\begin{align*}
    -(t_0-t)&\int_M\langle R_{mr},-2[T_s,T_n] +\pi_\fm([T_s,T_n])+\pi_\fm(R_{ns}))\rangle g^{mn}g^{rs}G\vol_g\\
    &\leqslant c(t_0-t)\int_M |R||T|^2G\vol_g + c(t_0-t)\int_M |R|^2G\vol_g\\
    &\leqslant c\Theta_{(x_0,t_0)}(\xi(t)) + c.
\end{align*} 
    In summary, and using Lemma \ref{lemma: ddt_theta_xi}, we get
\begin{equation}
    \frac{d}{dt}\Theta_{(x_0,t_0)}(\xi(t)) \leqslant c\left(1+\ln\left(\frac{c'}{(t_0-t)^{n/2}}\right)\right)\Theta_{(x_0,t_0)}(\xi(t)) + c(E_0+1).
\end{equation} 
    
    Since the function
\[
F(t):=\left(-1-\ln c' + \frac{n}{2}\ln(t_0-t)-\frac{n}{2}\right)(t_0-t)
\] satisfies $F'(t) = 1+\ln\left(\frac{c'}{(t_0-t)^{n/2}}\right)$, it follows that
\begin{align*}
    \frac{d}{dt}\left(e^{-cF(t)}\Theta_{(x_0,t_0)}(\xi(t))\right) \leqslant ce^{-cF(t)}(E_0+1).
\end{align*} Now observe that $F(t)$ is uniformly bounded for $\max\{0,t_0-1\}<t<t_0$, and thus integrating for $\max\{0,t_0-1\}<t_1<t_2<t_0$ we get
\begin{align*}
    \Theta_{(x_0,t_0)}(\xi(t_2)) &\leqslant e^{c(F(t_2)-F(t_1))}\Theta_{(x_0,t_0)}(\xi(t_1)) + c(E_0+1)(t_2-t_1)\\
    &\leqslant c\Theta_{(x_0,t_0)}(\xi(t_1)) + c(E_0+1)(t_2-t_1).
    \qedhere
\end{align*}
\end{proof}

\begin{remark}
    The equality case \eqref{eq: equality_monotonicity} of the monotonicity in $\mathbb{R}^n$ is attained precisely by self-similar solutions $\xi(t)$ induced by a specific kind of shrinking soliton described in Example \ref{ex: special_solitons}: namely, those compatible with the Euclidean metric $g_{\circ}$ and which satisfy equation \eqref{eq: special_soliton} with $c=1$. 
\end{remark}

\bibliography{references}
%===============================================================================


\bigskip

\noindent
(DF)  {\textit{Current affiliation:} Institute of Mathematics (IM), Federal University of Rio de Janeiro (UFRJ), 21941-909 Rio de Janeiro - RJ, Brazil.\\
\textit{Former:}} Institute of Mathematics, Statistics and Scientific Computing (IMECC), University of Campinas (Unicamp), 13083-859 Campinas-SP, Brazil.\\
\& Univ. Brest, CNRS UMR 6205, LMBA, F-29238 Brest, France.\\ \href{fadel.daniel@gmail.com}{fadel.daniel@gmail.com}

\medskip

\noindent
(EL) Univ. Brest, LMBA, France.\\ \href{loubeau@univ-brest.fr}{loubeau@univ-brest.fr}

\medskip

\noindent
(AM) IMECC, Unicamp, Brazil.\\ \href{amoreno@unicamp.br}{amoreno@unicamp.br}

\medskip

\noindent
(HSE) IMECC, Unicamp, Brazil.\\ \href{henrique.saearp@ime.unicamp.br}{henrique.saearp@ime.unicamp.br}

\end{document}